\documentclass[a4paper,11pt]{article}

\usepackage{amsmath,amsthm}
\usepackage{amssymb}
\usepackage{breqn}
\usepackage{enumerate}
\usepackage{graphicx}
\usepackage{bm}
\usepackage{bbm}
\usepackage[affil-it]{authblk}
\usepackage{tabu}
\usepackage{bold-extra}
\usepackage[hang,flushmargin]{footmisc}
\usepackage[driverfallback=dvipdfm]{hyperref}
\usepackage{mathtools}
\usepackage{relsize}
\usepackage{scalerel}
\usepackage{xcolor}
\usepackage{pagecolor}
\usepackage{titlesec}
\usepackage{apptools}
\usepackage{appendix}

\newtheorem{theorem}{Theorem}
\numberwithin{theorem}{section}
\newtheorem{corollary}[theorem]{Corollary}
\newtheorem{lemma}[theorem]{Lemma}
\newtheorem{proposition}[theorem]{Proposition}

\newtheorem{claim}[theorem]{Claim}
\newtheorem*{claim*}{Claim}

\theoremstyle{definition}
\newtheorem{defin}[theorem]{Definition}

\newtheorem{fact}[theorem]{Fact}

\titleformat{\section}[hang]{\scshape\large\bfseries\filcenter}{\S\thesection}{4pt}{}
\titleformat{\subsection}[hang]{\scshape\bfseries}{\thesubsection.}{4pt}{}

\allowdisplaybreaks

\newcommand\id{\mathbbm{1}}	
\newcommand{\tss}[1]{\textsuperscript{#1}}
\newcommand{\on}[1]{
	\operatorname{#1}
}
\newcommand{\dhor}{\mathsf{D}_{\mathsf{hor}}}
\newcommand{\dver}{\mathsf{D}_{\mathsf{ver}}}
\newcommand{\uc}[1]{
    \|#1\|_{\mathbb{T}}
}

\newcommand{\lr}[1]{
    \langle #1\rangle
}

\newcommand{\blr}[1]{
    \Big\langle #1\Big\rangle
}

\newcommand{\tn}[1]{ 
	\|#1\|_{\mathbb{T}}
}

\newcommand\range{\on{\# Im}}

\newcommand{\tdt}{\times\cdots\times}

\newcommand{\tightoverset}[2]{
  \mathop{#2}\limits^{\vbox to -.5ex{\kern-1.15ex\hbox{$#1$}\vss}}}

\newcommand\blfootnote[1]{%
  \begingroup
  \renewcommand\thefootnote{}\footnote{#1}%
  \addtocounter{footnote}{-1}%
  \endgroup
}

\renewenvironment{thebibliography}[1]
{
  \begin{oldthebibliography}{#1}
    \setlength{\itemsep}{0em  plus 0.3ex}
    \setlength{\parskip}{0em}
}
{
  \end{oldthebibliography}
}

\newcommand\ssk[1]{
	\substack{#1}
}

\newcommand\ex{\mathop{\mathbb{E}}}

\newcommand{\exx}{
  \mathop{
    \mathchoice{\vcenter{\hbox{\larger[4]$\mathbb{E}$}}}
               {\kern0pt\mathbb{E}}
               {\kern0pt\mathbb{E}}
               {\kern0pt\mathbb{E}}
  }\displaylimits
}

\makeatletter
\newcommand*\bcdot{\mathpalette\bigcdot@{0.5}}
\newcommand*\bigcdot@[2]{\mathbin{\vcenter{\hbox{\scalebox{#2}{$\m@th#1\bullet$}}}}}
\makeatother

\makeatletter
\def\blfootnote{\gdef\@thefnmark{}\@footnotetext}
\makeatother

\newcommand{\blc}{\bm{\mathsf{C}}}

\newcommand{\bsc}{\bm{\mathsf{c}}}

\newcommand{\upd}[1]{\overset{\bcdot}{#1}}

\begin{document}

\begin{center}\Large\noindent{\bfseries{\scshape General inverse theory for the $\mathsf{U}^4$ norm}}\\[24pt]\normalsize\noindent{\scshape Luka Mili\'cevi\'c\tss{\dag}}
\end{center}
\blfootnote{\noindent\dag\ Mathematical Institute of the Serbian Academy of Sciences and Arts\\\phantom{\dag\ }Email: luka.milicevic@turing.mi.sanu.ac.rs}

\footnotesize
\begin{changemargin}{1in}{1in}
\centerline{\sc{\textbf{Abstract}}}
\phantom{a}\hspace{12pt}~In this paper, we develop a quantitative inverse theory for the Gowers uniformity norm $\|\bcdot\|_{\mathsf{U}^4}$ in general finite abelian groups. The norm $\|\bcdot\|_{\mathsf{U}^4}$ has played an important role in the inverse theory of uniformity norms as, in the works in this area, the proofs of the inverse theorem in that case typically provided an overall strategy that could be generalized to the higher order case. We identify a new type of obstructions to uniformity, which we call almost-cubic polynomials. An almost-cubic polynomial $q$ on a Bohr set $B(\Gamma, \rho_0)$ is a function such that, for each $\rho \leq \on{min}\{\rho_0, 1/8\}$, we have
\[\|\Delta_{a,b,c,d} q(x)\|_{\mathbb{T}} \leq  2^{10} \rho\]
for all $x, a,b,c,d \in B(\Gamma, \rho)$.\\[6pt]
\phantom{a}\hspace{12pt}~Let $f : G \to \mathbb{D}$ be a function with $\|f\|_{\mathsf{U}^4} \geq c$. We prove quasipolynomial inverse theorems: namely that\\
\phantom{a}\hspace{18pt}$\bullet$ when $(|G|, 6) = 1$, there exists an almost-cubic $q : B(\Gamma, \rho)$ for $|\Gamma| \leq \log^{O(1)} c^{-1}$ and $\rho \geq \exp(-\log^{O(1)} c^{-1})$, and an element $t \in G$ such that 
$$\Big|\sum_{x \in G} \id_{B}(x) f(x + t) \on{e}(q(x))\Big| \geq \exp(-\log^{O(1)} c^{-1})|G|,$$
\phantom{a}\hspace{18pt}$\bullet$ when $G = (\mathbb{Z}/2^d\mathbb{Z})^n$, there exists a cubic polynomial $q : G \to \mathbb{T}$ such that
$$\Big|\sum_{x \in G} f(x)\on{e}(q(x))\Big| \geq \exp(-\log^{O_d(1)} c^{-1})|G|.$$
Almost-cubic polynomials are rather rigid and we exhibit a strong connection with generalized polynomials in the case of cyclic groups, as well as with polynomials in the classical sense in the case of finite vector spaces. Thus, the theory in this paper gives a unified treatment of the inverse theorems of Green, Tao and Ziegler and of Bergelson, Tao and Ziegler for the $\mathsf{U}^4$ norm, which were proved in the above-mentioned cases of ambient groups. Simultaneously, we answer a question of Jamneshan, Shalom and Tao concerning the inverse theory in groups of bounded torsion.\\
\phantom{a}\hspace{12pt}~The central result from which the inverse theorems follow is a structural result for Freiman bihomomorphisms, which are an approximate variant of bilinear maps, in general finite abelian groups. In our proof, we generalize methods of our previous work in the case of finite vector spaces, relying on the algebraic regularity method and the abstract Balog-Szemer\'edi-Gowers theorem, and introduce novel ideas concerning extensions of Freiman bihomomorphisms. In the problem of extension of Freiman bihomomorphisms, genuinely new phenomena appear in general finite abelian groups, which are not present in the finite vector space case.\\[12pt]
\end{changemargin}
\normalsize

\section{Introduction}

\hspace{17pt}Let us begin by recalling the definition of the uniformity norms, introduced by Gowers~\cite{Gow4AP, GowerskAP} in in his quantitative proof of Szemer\'edi's theorem on arithmetic progressions.

\begin{defin}
    Let $f : G \to \mathbb{C}$ be a function on a finite abelian group $G$. The \textit{discrete multiplicative derivative} with \textit{shift} $a$ is the operator that maps $f$ to the function $\partial_a f$, given by the formula $\partial_a f(x) = f(x + a)\overline{f(x)}$. With this notation, the uniformity norm $\|f\|_{\mathsf{U}^k}$ is defined as 
    \[\Big(|G|^{-k - 1}\sum_{x, a_1, \dots, a_k \in G} \partial_{a_1} \dots \partial_{a_k} f(x)\Big)^{2^{-k}}.\]
\end{defin}

It is a well-known fact that $\|\cdot\|_{\mathsf{U}^k}$ is a norm for $k \geq 2$.\\

These norms quantify the amount of algebraic structure present in functions on abelian groups. A key insight of Gowers, which is crucial for the proof of Szemer\'edi's theorem, is that if one considers a set $A \subseteq G$ of density $\frac{|A|}{|G|} = \delta$, if the value $\|\id_A - \delta\|_{\mathsf{U}^{k}}$ is sufficiently small, then the number of arithmetic progressions of length $k + 1$ inside $A$ is approximately the same as in the case when $A$ is a randomly chosen subset of $G$ of density $\delta$, that is $\delta^{k + 1}|G|^2$. To complete the proof of Szemer\'edi's theorem, one needs to answer the inverse problem for uniformity norms which is to describe functions $f : G \to \mathbb{D} = \{z \in \mathbb{C} : |z| \leq 1\}$ with large uniformity norm. This question has become one of the central problems of additive combinatorics and is the starting point of higher order Fourier analysis.\\

In his work, Gowers provided a partial answer to the inverse problem, which was sufficient to complete the proof of Szemer\'edi's theorem. Namely, in the case of the cyclic group $\mathbb{Z}/N\mathbb{Z}$, he showed that, for any function $f : \mathbb{Z}/N\mathbb{Z} \to \mathbb{D}$ with $\|f\|_{\mathsf{U}^k} \geq c$, there exist arithmetic progressions $P_1, \dots, P_m$ of lengths at least $N^{\delta_{k,c}}$, where $\delta_{k,c} \in (0,1)$ is a constant depending on $k$ and $c$ only, which partition  the group $\mathbb{Z}/N\mathbb{Z}$, and there are polynomials $p_1, \dots, p_m$ of degree $k - 1$, such that $f$ locally correlates with phases of these polynomials, meaning that
\[\sum_{i \in [m]}\Big|\sum_{x \in P_i} f(x) \on{e}(p_i(x))\Big| \geq \Omega_c(N).\]
This description is incomplete in the sense that not all such piece-wise phase polynomials have large uniformity norm.\\

To obtain a complete description, for a given group $G$, order of the norm $k$ and norm bound $c$, we want to find a family of functions $\mathcal{F} = \mathcal{F}_{G, k ,c}$, whose elements we refer to as the \textit{obstructions to uniformity}, such that every $f : G \to \mathbb{D}$ with $\|f\|_{\mathsf{U}^k} \geq c$ correlates \textit{globally} with a function $g \in \mathcal{F}$, namely $|\ex_{x \in G} f(x) \overline{g(x)}| \geq c'$. Here $c'$ is a parameter that depends on $c$ and $k$, that we shall refer to as the \textit{correlation bound} in the rest of the introduction, and $\ex_{x \in G}$ is the standard shorthand for the average $\frac{1}{|G|}\sum_{x \in G}$. The family of obstructions should be in a certain sense minimal, the obstruction functions are expected to have a polynomial-like structure and their description should allow a reasonably efficient exploitation of such structure. Thus, an important part of the answer is to describe the family of obstructions and the answers to the inverse question typically depends on the ambient group. Since the introduction of uniformity norms, a significant body of work concerning their inverse theory has been produced and we next outline some of the key results.\\

\noindent\textbf{Brief history.} The first non-trivial case is the $\mathsf{U}^3(G)$ norm, for which the inverse theorem was proved by Green and Tao~\cite{GreenTaoU3}, in the case of groups $G$ of odd order, and by Samorodnitsky~\cite{SamorU3} for the case of $\mathbb{F}_2^{n}$. More recently, Jamneshan and Tao~\cite{JamTao} developed a unified theory for this norm by giving a proof that works in all finite abelian groups.\\
\indent For $k \geq 4$, the inverse theory for uniformity norms $\mathsf{U}^k$ is considerably harder. In particular, the norm $\mathsf{U}^4$ is in some sense key as its proof typically provides a strategy for the proof of the inverse theorem in the higher order cases as well. For instance, unlike the $\mathsf{U}^3$ case, the known results treat separately finite vector spaces and cyclic groups, (or, slightly more generally, low rank groups and bounded torsion groups), and, up to this paper, there have been no inverse theorems that unify the strong inverse theorems in those special cases in general abelian groups, even in a qualitative sense.\\
\indent For general $k$, the inverse theorem for $\mathsf{U}^k(\mathbb{F}_p^n)$ norm, in the so-called case of the high characteristic where $k \leq p$, where the obstruction functions can be taken to be polynomial phases, is a major result of Bergelson, Tao and Ziegler~\cite{BTZ, TaoZieglerCorr}. Tao and Ziegler later extended that result~\cite{TaoZiegler}, to also include the low characteristic case ($k > p$). In contrast to the high characteristic case, in the latter situation a more general family of obstructions is required, known as the \textit{non-classical polynomials}. Such functions arise precisely as the solution of the extremal problem of determining $g : \mathbb{F}_p^n \to \mathbb{D}$ such that $\|g\|_{\mathsf{U}^k} = 1$. At the opposite end of the spectrum, when the ambient group is cyclic, the inverse theorem for the $\mathsf{U}^k(\mathbb{Z}/N\mathbb{Z})$ norm is an remarkable result of Green, Tao and Ziegler~\cite{GTZU4, GTZ}, which is a key part of the programme of Green and Tao for obtaining asymptotic estimates for the counts of linear configurations in primes~\cite{GTprimes1, GTprimes2}. For cyclic groups, one can take nilsequences as the family of obstructions.\\

\indent The above-mentioned results only provided qualitative bounds on the correlation with the obstruction functions. For example, the proof of Bergelson, Tao and Ziegler, in which the ambient group is a finite vector space, relied on ergodic theory. Furthermore, in the extension to the low-characteristic case, Tao and Ziegler use the multidimensional Szemer\'edi theorem. In the case of cyclic groups, the proof of Green, Tao and Ziegler involves regularization procedures and they note that their proof leads to Ackermannian bounds, even if the ultrafilters language was avoided\footnote{According to Leng, Sah and Sawhney~\cite{LengSahSawhney}, it might be the case that with more care, the work of Green, Tao and Ziegler already gives quantitative bounds, involving $O(k^2)$ exponentials for $\mathsf{U}^k$ norm.}.\\
\indent Before discussing some quantitative results, let us note that the theory of nilspaces, originating in papers by Szegedy~\cite{Szeg} and Camarena and Szegedy~\cite{CamSzeg}, provides another approach to the inverse theory of uniformity norms. The nilspace theory is a rich subject, with further developments by Candela~\cite{CandelaNotes1, CandelaNotes2}, Candela and Szegedy~\cite{CandelaSzegedy1, CandelaSzegedy2} and Gutman, Manners and Varj\'u~\cite{GMV1, GMV2, GMV3}. More recently, using this theory, Candela, Gonz\'alez-S\'anchez and Szegedy~\cite{nilspacesCharp} gave a new proof the Tao-Ziegler inverse theorem, and made progress in the case of the groups of bounded torsion~\cite{nilspacesBoundedTorsion} and general abelian groups~\cite{nilspacesGeneralAbelian}. However, the nilspace theory approach is decisively infinitary as well and therefore only qualitative.\\

\noindent\textbf{Quantitative inverse theorems.} In the light of numerous applications in arithmetic combinatorics, it is of major value to obtain inverse theorems with good bounds on the corelation with the obstruction functions. The question of bounds appears in Green's list of open problems~\cite{Green100}, and is mentioned as one the biggest problems in additive combinatorics.\\
\indent The inverse theorems for $\mathsf{U}^3(G)$ norms of Green and Tao, and of Samorodnitsky, mentioned above, were already quantitative. The bounds were roughly a single exponential (more precisely, $c' \geq \exp(-c^{-O(1)})$). Sanders~\cite{Sanders} improved their bounds to quasipolynomial (bounds of the shape $c' \geq \exp(-\log^{O(1)}(2c^{-1}))$). Almost-periodicity result of Croot and Sisask~\cite{CrootSisaskPaper} played an important role in that result. Finally, Gowers, Green, Manners and Tao~\cite{Marton1, Marton2} obtained the polynomial bounds ($c' \geq c^{O(1)}$) in the setting of finite vector spaces as a conseqeunce of their resolution of Marton's conjecture.\\
\indent When it comes to higher norms, in the case finite vector spaces, first quantitative bounds were obtained by Gowers and the author~\cite{U4paper} for the case of $\mathsf{U}^4$ for $p \geq 5$, and were of the form $c' \geq \exp^{(3)}(\log^{O(1)}(2c^{-1}))$, so approximately doubly exponential. Kim, Li and Tidor~\cite{KimLiTidor}, and independently Lovett, improved the bounds by a single exponential by noting that a step of the proof in~\cite{U4paper}, where an inefficient use of the Inclusion-Exclusion principle led to exponential dependence, can be improved. A quantitative version of the inverse theorem for $\mathsf{U}^k$ norms in the high characteristic was proved by Gowers and the author~\cite{FMulti}, with a bound involving a bounded number of exponentials, depending on $k$ only. There is also progress in the low characteristic~\cite{LukaU56, Tidor}.\\
\indent Finally, the author found a new proof of the inverse theorem for the $\mathsf{U}^4$ norm with a quasipolynomial corelation bound.\\
\indent When the ambient group is cyclic, we have breakthroughs of Manners, who proved doubly exponential bounds in the inverse theorem~\cite{MannersUk} for all $k$, of Leng~\cite{LengNil2}, who proved quasipolynomial bounds in the case of $\mathsf{U}^4$ norm, followed by a work of Leng, Sah and Sawhney~\cite{LengSahSawhney}, building upon the work of Green, Tao and Ziegler, and using Leng's improved equidistribution theory for nilsequences~\cite{LengNil1}, who proved quasipolynomial bounds for all orders.\\

\noindent\textbf{New results.} As it is follows from the previous discussion, there is still no inverse theory for uniformity norms for general abelian groups which simultaneously generalizes the best known descriptions in cyclic groups and finite vector spaces case. In this paper, we develop a general quantitative inverse theory for finite abelian groups in the case of $\mathsf{U}^4$ norm. To do so, we also define a new class of obstruction functions, which we call \textit{almost-cubic polynomials}: these are functions $\phi : B = B(\Gamma, \rho_0) \to \mathbb{T}$, where $B(\Gamma, \rho_0)$ is a Bohr set, with the property that, for each $\rho \leq \min\{\rho_0,1/8\}$, we have
\[\|\Delta_{a,b,c,d} q(x)\|_{\mathbb{T}} \leq  2^{10} \rho\]
for all $x, a,b,c,d \in B(\Gamma, \rho)$.\\

Let us state our two main inverse results. The first one concerns the finite abelian groups of order coprime to 6.

\begin{theorem}[General inverse $\mathsf{U}^4$ theorem]\label{genu4introduction}
     Let $G$ be a finite abelian group of order coprime to 6. Let $f : G \to \mathbb{D}$ be such that $\|f\|_{\mathsf{U}^4} \geq c$. Then there exists a Bohr set $B$ of codimension $(2 \log c^{-1})^{O(1)}$ and radius $\exp(-(2 \log c^{-1})^{O(1)})$, an element $t \in G$ and an almost-cubic polynomial $\phi : B  \to \mathbb{T}$ such that 
    \[\Big|\exx_{x \in G} \id_{B}(x) f(x + t) \on{e}(q(x))\Big| \geq \exp(-(2 \log c^{-1})^{O(1)}).\]
\end{theorem}

Given that the inverse theory for uniformity norms stems from a quantitative proof of Szemer\'edi's theorem, let us remark that it is plausible that novel variants with bounds matching those of Leng, Sah and Sawhney~\cite{LengSahSawhneykAP} would follow from our work, for example, proving that sets of density $\frac{1}{k}$ inside $(\mathbb{Z}/k\mathbb{Z})^n$ for an integer $k$ contain an arithmetic progression of length 5. Such a pairing of an ambient group and density is interesting, as Szemer\'edi's theorem in such a case cannot be trivially deduced from the cases of vector spaces or cyclic groups.\\

The second inverse theorem concerns groups whose exponent is a power of two.\\

\begin{theorem}[Inverse $\mathsf{U}^4$ theorem in $(\mathbb{Z}/2^d\mathbb{Z})^n$ with polynomial obstructions]\label{abelian2groupscaseintroduction}
    Fix $d \in \mathbb{N}$ and let $G = (\mathbb{Z}/2^d\mathbb{Z})^n$. Suppose that $f : G \to \mathbb{D}$ satisfies $\|f\|_{\mathsf{U}^4} \geq c$. Then there exists a cubic polynomial $q : G \to \mathbb{T}$ such that 
    \[\Big|\exx_{x} f(x) \on{e}(q(x))\Big| \geq \exp(-\log^{O_d(1)}(2c^{-1})).\]
\end{theorem}

Let us remark that in~\cite{JamneshanShalomTao} (Question 1.9), Jamneshan, Shalom and Tao asked whether one can take polynomials of degree $k-1$ as obstructions for the $\mathsf{U}^k$ norm in the case of bounded torsion groups. In particular, these two theorems essentially resolve their question for the $\mathsf{U}^4$ norm. With some more work, one could obtain a fully unified answer, but, for the reasons of emphasizing the additional difficulties with the symmetry argument in the case of low characteristic, we opted for this way of presenting the inverse results. Theorem~\ref{abelian2groupscaseintroduction} is arguably the hardest case of that question for the given norm.\\[12pt]

The key result that implies the inverse theorems above is the structural theorem for Freiman bihomomorphism in finite abelian groups, whose definition we now recall. Throughout the paper, $G_1, G_2$ and $H$ are arbitrary finite abelian groups. A map  $\phi : A \to H$ from a subset $A$ of $G_1 \times G_2$ to $H$ is a \textit{Freiman bihomomorphism}, if it respects all directional additive quadruples, namely if $\phi(x_1, y) + \phi(x_2, y) = \phi(x_3, y) + \phi(x_4,y)$ whenever $x_1 + x_2 = x_3 + x_4$ and all four points $(x_i, y)$ belong to $A$, and similarly when roles of $x$ and $y$ are swapped.\\
To state the structural result, we introduce another notion, of an $E$-bihomomorphism. Namely, if, for some subset $E \subseteq H$, instead we have $\phi(x_1, y) + \phi(x_2, y) - \phi(x_3, y) - \phi(x_4,y) \in E$ for additive quadruples above (and analogously when $x$ and $y$ interchange the roles), we say that $\phi$ is an $E$-bihomomorphism. Finally, we say that $E$ has \textit{rank at most $r$} if, for some elements $a_1, \dots, a_r \in H$, every element of $E$ can be written as $\sum_{i \in [r]} \varepsilon_i a_i$ where $\varepsilon \in \{0,1\}^r$. Note that such a map is a Freiman bihomomorphism when $E = \{0\}$.\\
Our principal result is that Freiman bihomomorphisms on dense domains are essentially restrictions of $E$-bihomomorphisms defined on products of Bohr sets, for small sets $E$.

\begin{theorem}\label{maininversetheorem}
    Suppose that $\phi : A \to H$ is a Freiman bihomomorphism on a set $A \subseteq G_1 \times G_2$ of size $c |G_1||G_2|$. Then there exist a set $E$ of rank $(2\log c^{-1})^{O(1)}$, Bohr sets $B_1 \subseteq G_1, B_2 \subseteq G_2$ of codimension $(2\log c^{-1})^{O(1)}$ and radius $\exp(-(2\log c^{-1})^{O(1)})$, elements $s \in G_1,t\in G_2$ and an $E$-bihomomorphism $\Phi : B_1 \times B_2 \to H$ such that $\Phi(x,y) = \phi(x + s,y + t)$ holds for at least $\exp(-(2\log c^{-1})^{O(1)})|G_1|G_2|$ points $(x,y) \in B_1 \times B_2$.
\end{theorem}

(In the case of groups of bounded torsion, we may replace the $E$-bihomomorphism by Freiman bihomomorphism, as an example, see Proposition~\ref{maininversetheoremspectorsion}.)\\

Unlike the case of finite vector spaces, when the structural result implies that $\phi$ can be related to a global bilinear map, we require a more general notion of an $E$-bihomomorphism in the case of general finite abelian groups. We shall discuss this in more detail in the proof overview. Before that, let us turn our attention to almost-cubic polynomials.\\

\subsection{Obstruction functions}

Let us firstly note that almost-cubic polynomials arising in the proof of Theorem~\ref{genu4introduction} reduce to usual polynomials in the case of bounded torsion groups. Let $q: B \to \mathbb{T}$ be an almost-cubic polynomial. If $G$ has exponent $m$, in the proof of Theorem~\ref{genu4introduction}, $q$ takes values in $\frac{\mathbb{Z}}{m} \pmod{1}$. Then we have $\|\Delta_{a,b,c,d}q(x)\|_{\mathbb{T}} \leq \varepsilon$ on $B_{\varepsilon}$. However, the restriction on values of $q$ implies that $\Delta_{a,b,c,d}q(x)$ vanishes. Moreover $B_\varepsilon$ contains a dense subgroup $U$ of $G$, and in such groups it is possible find a further dense subgroup $U' \leq U$ from which $q$ can be extended to a global polynomial (see Proposition~\ref{directsummandstorsion}).\\

We could in principle prove a unified theorem for all abelian groups, but that would obfuscate the details of the symmetry argument in the case of groups whose exponent is a power of two, where additional difficulties arise even compared to the case of low characteristic (see Proposition~\ref{correlationforsymmetrytorsionLemma}).\\

When it comes to the case of cyclic groups, we prove that almost-trilinear forms, which are a multilinear variant of almost-cubic polynomials and arise by taking triple discrete additive derivate, can be approximated on lower-order sets by generalized trilinear polynomials. This is the content of Theorem~\ref{almosttrilineartogenpolys} in Appendix~\ref{cyclicgroupsappendix}.\\

Finally, it is plausible that, with appropriate constructions of nilmanifolds from our obstruction functions, one can deduce prove the general inverse conjecture of Jamneshan and Tao~\cite{JamTao}, but, given the length of this paper, we opt not to pursue that here.\\

However, rather than solely focusing on the precise form of the final obstruction functions, one of the key takeaways of our work is that there are three natural classes of functions, which are all roughly equivalent:
\begin{itemize}
    \item \textit{almost-trilinear forms}, which are functions $\phi : B \times B \times B \to \mathbb{T}$ for a Bohr set $B = B(\Gamma, \rho_0) \subseteq G$, such that $\|\phi(x_1 + x_2, y, z) - \phi(x_1, y,z) - \phi(x_2, y, z)\|_{\mathbb{T}} \leq O(\rho)$ when $x_1, x_2, y, z \in B(\Gamma, \rho)$ for $\rho \leq \rho_0$, and similarly for other variables (and in our paper almost-cubic polynomials arise as $x \mapsto \phi(x, x, x)$), 
    \item $E$-bilinear maps $\Phi: B \times B\to \hat{G}$, for a Bohr set $B$,
    \item Freiman-bilinear maps $\Phi : V \to \hat{G}$, for a bilinear Bohr variety $V$.
\end{itemize}

Passing between these classes varies in difficulty. Given an almost-trilinear form $\phi$, it is not too difficult to get an $E$-bilinear maps $\Phi: B' \times B'\to \hat{G}$ related to $\phi$. That is the matter of studying almost-linear maps on Bohr sets and obtaining a relatively precise structure of large Fourier coefficients of Bohr sets. The reverse direction is essentially trivial, simply define $\phi(x,y,z) = \Phi(x,y)(z)$ on the cube of the Bohr set $B' \cap B(E, \rho)$.\\
Getting from $E$-bilinear maps $\Phi: B' \times B'\to \hat{G}$ to a Freiman-bilinear map on a bilinear Bohr variety can be carried out using the arguments\footnote{This rough implication most likely has a more direct proof, but this was the shortest way to phrase it.} in Sections~\ref{fbihomsmallranksection}--\ref{exttovarsection} and the bilinear Bogolyubov argument, and the reverse direction is the content of Section~\ref{passingtobihomonprodsection}. 
   
\subsection{Proof overview and paper organization}

Before discussing the organization of the paper, let us briefly describe two important ingredients originating in previous works.\\

In our proof of the quasipolynomial inverse theorem in finite vector spaces~\cite{newU4}, an abstract version of the Balog-Szemer\'edi-Gowers theorem was introduced. In a nutshell, we consider a dense set $A$ of an abelian group $G$ and a sequence of sets $\mathcal{Q}_1 \subseteq \mathcal{Q}_2 \subseteq \dots \subseteq \mathcal{Q}_{36}$ of additive quadruples in $A$, which satisfy three conditions: refer to a \textit{largeness}, meaning that $\mathcal{Q}_1$ is sufficiently dense, \textit{symmetry}, meaning that all $\mathcal{Q}_i$ are preserved under certain permutations of elements in additive quadruples and \textit{weak-transitivity}, meaning that if an additive quadruple $(a_1, a_2, a_3, a_4)$ in $A$ satisfies $(a_1, a_2, b, b') \in \mathcal{Q}_i$ and $(b, b', a_3, a_4) \in \mathcal{Q}_j$ for a few choices of $(b, b')$, then $(a_1, a_2, a_3, a_4) \in \mathcal{Q}_{i+j}$. If all those conditions hold, then we may find a large subset $A'$ of $A$ with the property that each (not necessarily additive) $k$-tuple in $A'$ can be related to a large number of $3k$-tuples in $A$ using additive quadruples in $\mathcal{Q}_{36}$ as `bridges'. The flexibility in the definition of sets of additive quadruples is crucial, we shall touch upon that in the description of the argument. In Section~\ref{absbsgsection}, we prove a slightly more general statement (Theorem~\ref{absg}) which applies to approximate subgroups as ambient sets in place of groups.\\
\indent Secondly, in that proof of the inverse theorem, (bilinear) algebraic regularity method\footnote{Higher order algebraic regularity method was used in~\cite{FMulti}, but it is more complicated to describe. The key idea behind it is to use effective equidistribution theory of multilinear forms~\cite{BhowLov, GreenTaoPolys, Janzer2, LukaRank, MoshZhu} in place of the traditional (hypergraph) regularity lemmas.} played a major role, which we now briefly describe (see~\cite{newU4} for a longer discussion). Recall that Szemerédi regularity lemma~\cite{SzemAP, SzemReg} partitions the vertex set of a given graph into a bounded number of parts such that most pairs of parts induce quasirandom bipartite graphs. Together with the counting lemma, this lemma is a basis of the (combinatorial) regularity method. However, as it is well-known, this method is inherently very inefficient~\cite{TimRegConstruction}.\\
\indent In the arithmetic setting, we study graphs with strong algebraic structure resembling a bilinear variety and due to this structure, an efficient algebraic regularity lemma (Theorem~\ref{algreglemmaintro}) is available. The  algebraic regularity method was generalized to general finite abelian groups in~\cite{generalBilBog}. The statement there is slightly imprecise, but it does not affect the main result of that paper. We revisit and rectify the statement in Section~\ref{algreggensection}.\\
\indent We now proceed with a proof of overview and explanation of organization of the paper.\\

The paper is divided into three chapters. The first one is of preliminary nature, and concerns the tools needed in the paper, which we divide into three sections. Section~\ref{sectiontools} concerns what we refer to as 'approximate linear algebra'. In that section, we gather various facts and key objects that will be used frequently in the paper. Thus, we discuss in detail Bohr sets and coset progressions, Fourier analysis on Bohr sets, quantitative lattice theory and various maps on Bohr sets and coset progressions, such as Freiman homomorphisms, $E$-homomorphisms and almost-linear maps. Next, we have Section~\ref{algreggensection} in which algebraic regularity method for general abelian groups is revisited, and the chapter ends with a Section~\ref{absbsgsection} on the abstract Balog-Szemer\'edi-Gowers theorem.\\

The second chapter is devoted to the proof of our main result, Theorem~\ref{maininversetheorem}, 
which we now briefly overview.\\
\indent The proof starts by following a strategy of our previous proof in the case of finite vector spaces~\cite{newU4}. In that paper, very briefly put, we started with a Freiman bihomomorphism $\phi : A \to H$ on a dense set $A \subseteq G_1 \times G_2$. The proof begins with a change of perspective where we pass to a system of linear maps $\phi_x: G_2 \to H$, indexed by a set $X \subseteq G_1$, which is approximately linear in the sense that
\begin{equation}
    \on{rank} \Big(\phi_{x_1} - \phi_{x_2} + \phi_{x_3} - \phi_{x_4}\Big) \leq O(1),\label{firstconditionrespFpn}
\end{equation}
for many additive quadruples $(x_1, x_2, x_3, x_4)$ in $X$. Abstract Balog-Szemer\'edi-Gowers theorem is then applied for the first time to pass to a subset $X' \subseteq X$ in which all additive quadruples are respected in the sense of~\eqref{firstconditionrespFpn}. Robust Bogolyubov-Ruzsa theorem then allows us to replace $X'$ by a subspace $V$.\\
\indent After that, we make a crucial change of viewpoint, and consider partially defined linear maps $\phi: U_x \to H$, where $U_x \leq G_2$ is a low-codimensional subspace, rather than full space $G_2$. On its own, that would be only be a loss in structure, but at the same time, we change the notion of respectedness to 
\begin{equation}
   (\forall y \in U_{x_1} \cap U_{x_2} \cap U_{x_3} \cap U_{x_4})\,\, \Big(\phi_{x_1} - \phi_{x_2} + \phi_{x_3} - \phi_{x_4}\Big)(y) = 0,\label{secondconditionrespFpn}
\end{equation}
which is stronger than the previous one. We may pass to a system of partially defined linear maps $\phi_x : U_x \to H$, for all $x \in V$, where a vast majority of additive quadruples is respected in the second sense. A variant of bilinear Bogolyubov argument is then applied to ensure that $U_x$ depend linearly on $x$, at the cost of passing to a subset of indices of $V$, and we may only assume that a positive proportion of additive quadruples are respected in the sense of~\eqref{secondconditionrespFpn}. Another application of the abstract Balog-Szemer\'edi-Gowers theorem, this time combined with algebraic regularity method, allows us to pass to a set of indices in which all additive quadruples are respected in the sense of~\eqref{secondconditionrespFpn}.\\
\indent Viewing the system of partially linear maps as a single map defined on a subset of $G_1 \times G_2$, we get a Freiman bihomomorphism on a dense set of columns of a bilinear variety. We note that, in each of the remaining steps, algebraic regularity method is applied. Robust Bogolyubov-Ruzsa theorem allows us to pass to a Freiman bihomomorphism defined on a full bilinear variety. Finally, we using some older extension results, we get a global bilinear map.\\
\indent Let us now describe how the strategy above is modified in the present paper.\\

\noindent\textbf{Main differences to finite vector spaces.} As it is clear from the proof above, it relies heavily on linear algebra (in particular, linear maps between subspaces play an important role), abstract Balog-Szemer\'edi-Gowers theorem and algebraic regularity method, so the starting point is obtain their generalizations to arbitrary abelian groups. Abstract Balog-Szemer\'edi-Gowers theorem in its original form in~\cite{newU4} already has a proof that works in general abelian groups, but the more general context requires a slight generalization. Algebraic regularity method was previously generalized in~\cite{generalBilBog}. Finally, in place of linear maps $\phi_x: U_x \to H$, we have Freiman-linear maps on Bohr sets. This is a natural generalization, as such maps arise in the structure theorem for approximate homomorphisms (see Proposition~\ref{cosettobohrset} and Theorem~\ref{approxFreimanHom}).\\
\indent However, the most problematic part of the proof, which requires novel ideas, is the final extension step, which we now describe in more detail.\\

\noindent\textbf{Extension issues.} The question of extending maps is the most difficult obstacle and obtaining a product structure in Theorem~\ref{maininversetheorem} is essential in order to complete the symmetry argument and the deduction of the inverse theorems for the $\mathsf{U}^4$ norm (Theorems~\ref{genu4introduction} and~\ref{abelian2groupscaseintroduction}). In the finite vector spaces, we can always extend a linear map from a subspace to the full space, but this is no longer true in general abelian groups. Namely, if we are given a symmetric proper progression $[-L_1, L_1]\cdot a_1 + \dots + [-L_d, L_d] \cdot a_d$ of rank $d\geq2$, the Freiman-linear maps are defined by the values at $a_1, \dots, a_d$. Elements $a_1, \dots, a_d$ typically satisfy some non-trivial linear combination in the group, which prevents extensions. Furthermore, even if we have a Freiman-linear map on a subgroup (i.e. a homomorphism), it might not extend to the whole group.\\
\indent The way we overcome this issue is by \textbf{ignoring exactness}, which we insisted on, and \textbf{allowing small and controlled errors} in our maps. More precisely, let us go back to example of a Freiman-linear map $\phi$ on a dense progression $C = [-L_1, L_1] \cdot a_1 + \dots + [-L_d, L_d] \cdot a_d$ we considered previously. Let us assume that $[-kL_1, kL_1] \cdot a_1 + \dots + [-kL_d, kL_d] \cdot a_d= G$, which we expect to happen for a small $k$. We \textbf{naively} extend $\phi$; for each $x \in G$, we take arbitrary $y_1, \dots, y_k \in C$ such that $x = y_1 + \dots + y_k$ and put $\tilde{\phi}(x) = \sum_{i \in [k]} \phi(y_i)$. Of course, such a procedure will not give a homomorphism on the group, but let us investigate how badly the homomorphism property fails. While $\tilde{\phi}$ depends on the choice the tuple adding up to $x$, the possible values of $\tilde{\phi}$ are closely related. To understand the relationship between possible values, we need to address the question of value of 
\[\phi(y_1) + \dots + \phi(y_k) - \phi(z_1) - \dots - \phi(z_k)\]
in the case when $y_1 + \dots + y_k = z_1 + \dots + z_k$. Using the coordinates on the progression, we end up with 
\[\lambda_1 \phi(a_1) + \dots + \lambda_d \phi(a_d)\]
where $\lambda_i \in [-2kL_i, 2kL_i]$ and $\sum_{i \in [d]} \lambda_i a_i = 0$. A simple covering argument shows that there are only $O(1)$ such linear combinations. The question of possible values of $\Delta_{a,b}\tilde{\phi}$ is answered in the same spirit, where the only difference being that it concerns additive $4d$-tuples. This motivates the notion of an $E$-homomorphism.

\begin{defin}
    For a set $E \subseteq H$, a function $\phi : A \to H$ is an $E$-homomorphism if $\phi(x_1) + \phi(x_2) - \phi(x_3) - \phi(x_4) \in E$ for all $x_1 + x_2 = x_3 + x_4$. It is \textit{$E$-linear} if the additive quadruples condition is replace by additive triples.
\end{defin}

Hence, we may extend Freiman-linear maps on Bohr sets to $E$-linear maps defined on the whole group.\\

Using this idea, we may relate a Freiman bihomomorphism defined on a bilinear Bohr variety to an $E$-bihomomorphism on a product of Bohr sets. That means that we need to control many naive extensions simultaneously, which relies on algebraic regularity method, as well as the bilinear Bogolyubov argument in its original form, which is different from the one in the proof overview above, and which we now recall.\\

\noindent\textbf{Bilinear Bogolyubov argument -- original form.} Recall that the classical Bogolyubov argument~\cite{bogoriginal} states that $2A - 2A$ contains a Bohr set of large radius and small codimension, when $A$  is a dense subset of the cyclic group. This fact played an important role in Ruzsa's new proof~\cite{ruzsafreimanproof} of Freiman's theorem. A bilinear generalization of this fact was obtained in~\cite{BienvenuLe, BilinearBog} for the case of finite vector spaces, with a quantitative improvement in~\cite{HosseiniLovett}, which was subsequently generalized to finite abelian groups~\cite{generalBilBog}. We now describe the latter, general version. We write $\dhor A$ for the \emph{horizontal difference set} of $A \subset G_1\times G_2$ which is defined as $\dhor A = \{(x_1 - x_2, y) \colon (x_1, y), (x_2, y) \in A\}$. Similarly, we write $\dver A$ for the \emph{vertical difference set} of $A$ which is defined as $\dver A = \{(x, y_1 - y_2) \colon (x, y_1), (x, y_2) \in A\}$. In other words, horizontal difference set is obtained by taking the difference set inside each row and the vertical difference set is obtained by taking the difference set inside each column. Iterated directional difference sets are obtained by the obvious compositions, as $\dhor$ and $\dver$ are simply maps from the power-set $\mathcal{P}(G \times G)$ to itself.

\begin{theorem}[Bilinear Bogolyubov argument in finite abelian groups]\label{bogruzsabilinearintro}Let $G$ and $H$ be finite abelian groups and let $A \subseteq G \times H$ be a set of density $\delta$. Then there exist a positive quantity $\rho \geq \exp\Big(-\log^{O(1)} (10 \delta^{-1})\Big)$, sets $\Gamma \subseteq \hat{G},  \Psi \subseteq \hat{H}$ of size at most $\log^{O(1)} (10\delta^{-1})$ and Freiman-linear maps $L_1, \dots,$ $L_r \colon B(\Psi; \rho) \to \hat{G}$ for some positive integer $r \leq \log^{O(1)} (10 \delta^{-1})$ such that the bilinear Bohr variety
\[\Big\{(x,y) \in B(\Gamma; \rho) \times B(\Psi; \rho) \colon x \in B(L_1(y), \dots, L_r(y); \rho) \Big\}\]
is contained inside $\dhor\dver \dver \dhor \dver \dhor \dhor A$.\end{theorem}

If we refer to a bilinear Bohr variety implicitly, we say that it has \textit{codimension at most $d$ and radius $\rho$} if $|\Psi|, r \leq d$.\\

\noindent\textbf{Proof overview.} We may now give a proof overview for Theorem~\ref{maininversetheorem}. Let $\varphi: A \to H$ be a Freiman bihomomorphism on a dense set $A \subseteq G_1 \times G_2$.

\begin{itemize}
    \item[\textbf{Step 1.}] In Section~\ref{fbihomsmallranksection}, we pass from $\varphi$, which is a function of two variables, to a system of Freiman-linear maps $\phi_x : B_x \to H$, defined on Bohr sets $B_x$ of low codimension and large radius, indexed by $x \in X \subseteq G_1$ for a dense set $X$, such that many additive quadruples in $X$ are respected in the sense that 
    \[\phi_{x_1} - \phi_{x_2} + \phi_{x_3} - \phi_{x_4}\]
    takes a bounded number of values on the intersection of domains $\cap_{i \in [4]} B_{x_i}$, which is a variant of~\eqref{firstconditionrespFpn} working in general finite abelian groups.
    \item[\textbf{Step 2.}] In Section~\ref{almostalladditive16tuplesABSG1}, we use abstract Balog-Szemer\'edi-Gowers theorem to pass to a subset $X' \subseteq X$ inside which almost all additive 16-tuples are respected in the above sense.
    \item[\textbf{Step 3.}] In Section~\ref{firstBRstepSection}, we use robust Bogolyubov-Ruzsa theorem to replace indexing set $X'$ by a symmetric proper coset progression $C$ of small rank and large density.
    \item[\textbf{Step 4.}] In Section~\ref{obtmanybohrrespsection}, we make a crucial change of perspective and pass to the second notion of respectedness. Namely, similarly to~\eqref{secondconditionrespFpn}, we say that an additive quadruple $(x_1, \dots, x_4)$ is \textit{Bohr-respected} if 
    \[(\forall y \in B_{x_1} \cap B_{x_2} \cap B_{x_3} \cap B_{x_4})\,\, \Big(\phi_{x_1} - \phi_{x_2} + \phi_{x_3} - \phi_{x_4}\Big)(y) = 0.\]
    In this step, we use a dependent random choice argument to redefine domains of the maps $\phi_x$ so that a vast majority of $2k$-tuples, for some $k = O(1)$, are Bohr-respected, without affecting the indexing set $C$.
    \item[\textbf{Step 5.}] In Section~\ref{bilbogsection}, we use a variant of the bilinear Bogolyubov argument to ensure that $B_x$ exhibits linear behaviour in $x$. Namely, we find some Freiman-linear maps $\Theta_1, \dots, \Theta_r : C \to \hat{G}_2$ such that $B_x$ can be replaced by $B(\Theta_1(x), \dots, \Theta_r(x); \rho)$. This step comes at the cost of losing the structure in the indexing set, and once again it becomes merely a dense subset $X \subseteq C$, with a positive proportion of additive quadruples Bohr-respected (for the new choice of Bohr sets).
    \item[\textbf{Step 6.}] In Section~\ref{bilbohrmanybohrrespsection}, we use the abstract Balog-Szemer\'edi-Gowers theorem another time to pass to a subset in which all additive quadruples are Bohr-respected. The flexibility of the abstract Balog-Szemer\'edi-Gowers theorem is essential, as the sets of additive quadruples that we consider will be those that are Bohr-respected after the domains are intersected with smaller and smaller fixed Bohr set. In order to show that conditions of the abstract Balog-Szemer\'edi-Gowers theorem are satisfied, we need to use the algebraic regularity method.
    \item[\textbf{Step 7.}] In Section~\ref{exttovarsection}, we return to our initial viewpoint of functions in two variables. Thus, we now have a Freiman bihomomorphism defined on a dense set of columns of a bilinear Bohr variety. Using algebraic regularity method and the robust Bogolyubov-Ruzsa theorem, we may extend the map to the full bilinear Bohr variety.
    \item[\textbf{Step 8.}] In Section~\ref{passingtobihomonprodsection}, we carry out the novel extension argument, based on the naive extensions idea, where we pass to an $E$-bihomomorphism on a product of Bohr sets. This argument consists of three sub-tasks. Recall that, so far, we have a Freiman-bilinear map $\phi : V \to H$ on some bilinear Bohr variety $V$. Our strategy to extending $\phi$ is to extend $\phi$ naively in the vertical direction. However, it is not obvious that naive extensions in different columns are related.\\
    \phantom{a}\hspace{12pt} Firstly, in Subsection~\ref{controllingerrorsubsec}, we show that this is the case, namely that, after shrinking the Bohr variety in the domain of $\phi$ somewhat, the error sets arising in different columns can be assumed to be the same. This part of the argument relies on the bilinear Bogolyubov argument in its original form (Theorem~\ref{bogruzsabilinearintro}) and a construction of an auxiliary bilinear Bohr variety in groups larger than $G_1$ and $G_2$. However, due to requirement of passing to further bilinear Bohr variety, we cannot yet assume that columns are the full group $G_2$.\\
    \phantom{a}\hspace{12pt} Secondly, in Subsection~\ref{subgroupsincolumnssubsec}, we show that naive extensions in columns lead to a new map, whose domain is another bilinear Bohr variety, but whose columns are dense subgroups of $G_2$ intersected with some fixed Bohr set independent of the column index $x$. This part of the argument is based on an efficient regularization procedure using a quantitative lattice theory (see Lemma~\ref{nestedLattices} and Theorem~\ref{quantLattice}), and similar in spirit to the proof of the algebraic regularity lemma. \\
    \phantom{a}\hspace{12pt} Finally, in Subsection~\ref{extendingtoproductsubsec}, we carry out the final extension to the product of Bohr sets. Having subgroups for the columns of the bilinear Bohr variety is crucial, and the proof that the final extension is well-defined and an $E$-bihomomorphism uses approximate versions of certain cocycle identities (see Claim~\ref{cocycleEqnBil}) previously used in extensions of multilinear maps from multilinear varieties.\\
    \phantom{a}\hspace{12pt} We remark that the algebraic regularity method is heavily used in all three steps above.
\end{itemize}
Finally, we put everything together in Section~\ref{putevtogsection}.\\

Final chapter concerns applications. It comprises Section~\ref{sectionequidistributiontheory}, in which we develop equidistribution theory of almost trilinear maps, relying on the bilinear Bogolyubov argument, followed by Section~\ref{generalabelianinverseunifsection}, in which Theorem~\ref{genu4introduction} is proved, and Section~\ref{abelian2groupssection}, devoted to Theorem~\ref{abelian2groupscaseintroduction}. Once the required equidistribution theory is in place, the deduction of the first inverse theorem involves standard arguments and a symmetry argument of Green and Tao used in a mostly classical fashion. For the second inverse theorem, things are more subtle due to low characteristic and are further complicated by nesting of subgroups of lower exponents, leading to somewhat convoluted description of derivatives of cubics (Theorem~\ref{trilinearintegration}). This results in a more complicated symmetry argument (Proposition~\ref{correlationforsymmetrytorsionLemma}).\\

\vspace{\baselineskip}

\noindent\textbf{Acknowledgements.} This research was supported by the Ministry of Science, Technological Development and Innovation of the Republic of Serbia through the Mathematical Institute of the Serbian Academy of Sciences and Arts, and by the Science Fund of the Republic of Serbia, Grant No.\ 11143, \textit{Approximate Algebraic Structures of Higher Order: Theory, Quantitative Aspects and Applications} - A-PLUS.\\

\pagebreak

\begin{center}\noindent{\large\bfseries{\scshape Chapter 1: Preliminaries and Tools}}
\end{center}\normalsize

In this chapter, we gather the important tools and the frequently used auxiliary facts in the paper. Firstly, we begin with 'approximate linear algebra', in which we generalize various linear-algebraic concepts to general abelian groups. The key objects are Bohr sets and coset progressions, which serve us as approximate versions of subspaces, as well as Freiman and E-homomorphisms, that generalize linear maps. That section is a combination of classical results and theory of additive combinatorics, some important more contemporary results of additive combinatorics, some earlier work concerning general abelian groups, and new material as well. Secondly, we treat the algebraic regularity method, which was previously introduced in the context of the finite abelian groups in~\cite{generalBilBog}. Finally, we prove an abstract Balog-Szemer\'edi-Gowers theorem inside sets of small doubling. The last theorem originates in the~\cite{newU4} and was one of the key new ingredients in the proof of quasipolynomial bounds for the $\mathsf{U}^4(\mathbb{F}_p^n)$ norm.\\

\noindent\textbf{General notation.} We denote the unit disk as $\mathbb{D} = \{z \in \mathbb{C}\colon |z| \leq 1\}$. We use the standard expectation notation $\ex_{x \in X}$ as shorthand for the average $\frac{1}{|X|} \sum_{x \in X}$, and we simply write $\ex_x$ when no confusion is caused. Instead of denoting a sequence of length $m$ by $(x_1, \dots, x_m)$, we write $x_{[m]}$, and for $I\subset[m]$ we write $x_I$ for the subsequence with indices in $I$. This applies to products as well: $G_{[k]}$ stands for $\prod_{i \in [k]} G_i$ and $G_I = \prod_{i \in I} G_i$. Given a set $A \subseteq G_1 \times G_2$, we write $A_{x \bcdot}$ for the \textit{column} indexed by $x$, which is $\{y \in G_2: (x,y) \in A\}$, and similarly, we write $A_{\bcdot y} = \{x \in G_1 : (x,y) \in A\}$ for the \textit{row} indexed by $y$.\\
In certain parts of the paper, we use the use non-standard asymptotic notation and write $\blc$ (and $\bsc$) as a placeholder for a positive constant whose value is not important. For example,
\[(\forall x, y > 1) \hspace{3pt} \Big(x \geq \blc y^{\blc}\hspace{3pt}\implies\hspace{3pt}x \geq 100 y^2 \log y\Big)\]
is a shorthand for 
\[(\exists C_1, C_2 > 0)(\forall x, y > 1) \hspace{3pt} \Big(x \geq C_1 y^{C_2}\implies x \geq 100 y^2 \log y\Big).\]
We write $\blc$ and $\bsc$ to indicate that we think of a sufficiently large and sufficiently small positive constants respectively, though they are formally logically equivalent.

\section{Approximate linear algebra}\label{sectiontools}

In this section we set up the tools required for working with approximate subgroups in arbitrary abelian groups. As it is well-known, Bohr sets and coset progressions are a particularly useful class of such objects. Furthermore, we cover approximate versions of linear maps, articulated as Freiman and $E$-homomorphisms. We begin with a subsection containing the main definitions, most of which are standard. Five more subsections follow, on the topics of Bohr sets theory, coset progressions, variants of Freiman's theorem, lattice theory and, finally, Freiman and $E$-homomorphisms.

\subsection{Basic definitions}

\noindent\textbf{Characters and Bohr sets.} Let $G$ be a finite abelian group. We shall rely heavily on Fourier analysis on finite abelian groups. We view the dual group in the following, explicit way. Namely, due to the fundamental theorem for finitely generated abelian groups, we may assume that $G$ is of the form $G = \mathbb{Z}/q_1 \mathbb{Z}\, \oplus\, \mathbb{Z}/q_2 \mathbb{Z} \,\oplus\, \dots \,\oplus\,\mathbb{Z}/q_d \mathbb{Z}$ for some natural numbers $q_1, \dots, q_d$ such that $q_1 | q_2 | \dots | q_k$. The dual group of $G$, denoted by $\hat{G}$, is the groups of \textit{characters} of $G$, which are homomorphisms from $G$ to the circle group $\mathbb{T} = \mathbb{R}/\mathbb{Z}$. Addition of characters induces group structure on $\hat{G}$. For our purposes, $\hat{G}$ can be given an explicit form as follows. The dual group $\hat{G}$ has the structure $\mathbb{Z}/q_1 \mathbb{Z}\, \oplus\, \mathbb{Z}/q_2 \mathbb{Z} \,\oplus\, \dots \,\oplus\,\mathbb{Z}/q_d \mathbb{Z}$ as well. To see this, notice that for each $\chi \in \hat{G}$, there exist unique elements $\chi_i \in \mathbb{Z}/q_i\mathbb{Z}$ for each $i \in [d]$ such that 
\[\chi(x) = \sum_{i \in [d]} \frac{|\chi_i x_i|_{q_i}}{q_i} + \mathbb{Z},\]
where $|\cdot|_q \colon \mathbb{Z}/ q\mathbb{Z} \to \{0, 1, \dots, q-1\} \subseteq \mathbb{Z}$ is a map which maps each residue to the unique integer among $\{0,1, \dots, q-1\}$ that projects to it inside $\mathbb{Z}/ q\mathbb{Z}$.\\
\indent Write $\on{e}(t) = \exp(2 \pi i t)$ for $t \in \mathbb{R}$. For $x \in \mathbb{R}/\mathbb{Z}$ let $\tn{x}$ be the element $d \in [0, 1/2]$ such that $x \in \{-d,d\} + \mathbb{Z}$, i.e.\ the distance from 0. We may now define Bohr sets.

\begin{defin}[Bohr sets]
    Given a set of characters $\Gamma \subseteq \hat{G}$ and a map $\rho : \Gamma \to [0, 1)$ we define the \emph{Bohr set} with \emph{frequency set} $\Gamma$ and \emph{radius function} $\rho$ as $B(\Gamma; \rho) = \{x \in G \colon (\forall \chi \in \Gamma) \tn{\chi(x)} \leq \rho(\chi)\}$. Frequently, the function $\rho$ will be constant, in which case we simply refer to $\rho$ as the \textit{radius}. We refer to the size of $\Gamma$ as the \textit{codimension} of the Bohr set. 
\end{defin}

Most of the time we will consider Bohr sets of the constant radius, as $B(\Gamma; \rho)$ contains $B(\Gamma; \mu)$, where $\mu = \min_{\gamma \in \Gamma} \rho(\gamma)$, and the properties we are interested in can typically, with some rare exceptions, be deduced for $B(\Gamma; \rho)$ from the more general case of $B(\Gamma; \mu)$.\\

\noindent \textbf{Coset progressions.}  Coset progressions were introduced by Green and Ruzsa in their generalization of Freiman's theorem to general abelian groups~\cite{greenRuzsaFreiman}, and, as their theorem confirms, these objects are a correct generalization of (cosets of) subgroups from an additive-combinatorial perspective.\\

\begin{defin}[Coset progressions]
    A \emph{coset progression} in an abelian group $G$ is a set $C$ of the form $L_1 + \dots + L_r + H$, where $L_i$ are arithmetic progressions and $H \leq G$ is a subgroup. We say that $L_1 + \dots + L_r + H$ is a \emph{canonical form} of $C$. The number $r$ is the \emph{rank} of the coset progression. We say that a coset progression $C$ is \emph{proper} if $|C| = |L_1| \cdots |L_r||H|$, (thus all sums of $(r + 1)$-tuples are distinct), and we say that $C$ is \emph{symmetric} if all $L_i$ are symmetric, i.e. $L_i = -L_i$. We refer to sizes of arithmetic progressions $L_1, \dots, L_r$ as the \textit{lengths} of $C$ and to $H$ as the \textit{subgroup} of $C$.
\end{defin}

We shall make use of approximate variants of homomorphisms. While the notion of Freiman homomorphism is standard, we shall also need \textit{$E$-homomorphisms} defined below. The latter are maps where the error on additive quadruples is allowed ($E$ stands for the \textit{error set}), but tightly controlled. Such maps will be crucial in the later stages of the paper. 

\begin{defin}[Freiman and $E$-homomorphisms]\label{FEhommDefin}
    Let $G$ and $H$ be abelian groups, and let $A \subseteq G$. A map $\phi : A \to H$ is a \textit{Freiman homomorphism} if $\phi(a) + \phi(b) = \phi(c) + \phi(d)$ holds for all $a, b, c, d \in A$ such that $a + b = c + d$. If $0 \in A$ and $\phi(0) = 0$, we say that $\phi$ is Freiman-linear.\\
    \indent If $E \subseteq H$, we say that $\phi$ is an \emph{$E$-homomorphism} if $\phi(a) + \phi(b) - \phi(c) - \phi(d) \in E$ holds for all $a,b, c,d \in A$ such that $a + b = c + d$.
\end{defin}

Thus, a map being a Freiman homomorphism is the same as being a $\{0\}$-homomorphism. In order to get good bounds, we need to be efficient in describing error sets. With this in mind, we say that a set $E$ has \textit{rank at most $r$} if $E = \langle S\rangle_{\{-1,0,1\}}$ for a set $S$ of size $r$.

\subsection{Bohr sets theory}

Let us recall the most basic fact concerning Bohr sets. Firstly, we have easy estimates on the sizes of Bohr sets. Later, we shall derive much stronger results.

\begin{lemma}[Lemma 4.20 in~\cite{TaoVuBook}]\label{basicbohrsizel}For $\Gamma \subseteq \hat{G}$ and $\rho : \Gamma \to (0,1)$ we have
\[|B(\Gamma; \rho)| \geq \Big( \prod_{\gamma \in \Gamma} \rho(\gamma) \Big) |G|,\]
and
\[|B(\Gamma; 2\rho)| \leq 4^{|\Gamma|}|B(\Gamma; \rho)|.\]
\end{lemma}

The second inequality is stated in~\cite{TaoVuBook} only for the Bohr sets of constant radius, but the same covering proof works in the more general case.\\

Next, we show that the difference set of a very dense subset of a Bohr set contains a slightly shrunk version of that Bohr set.

\begin{lemma}\label{almostfullBohr} Let $A \subset B(\Gamma; \rho)$ be a set of size at least $(1 - 4^{-k-1})|B(\Gamma; \rho)|$ where $k$ is the codimension of the Bohr set. Then $A - A \supseteq B(\Gamma; \rho/2)$.\end{lemma}

\begin{proof}
    By the previous lemma, $|B(\Gamma; \rho/2)| \geq 4^{-k}|B(\Gamma; \rho)|$. Let $x \in B(\Gamma; \rho/2)$ be arbitrary. Then, $|(A - x) \cap B(\Gamma; \rho/2)| = |A \cap (x + B(\Gamma; \rho/2))| \geq \frac{3}{4}|B(\Gamma; \rho/2)|$. In particular, the same holds for $x = 0$, so we have
    \begin{align*}|A\cap (A - x) \cap B(\Gamma; \rho/2)| = & |B(\Gamma; \rho/2)| - |B(\Gamma; \rho/2) \setminus (A \cup A - x)|\\
    \geq &|B(\Gamma; \rho/2)| - |B(\Gamma; \rho/2) \setminus A| - |B(\Gamma; \rho/2) \setminus (A - x)| \geq \frac{1}{2}|B(\Gamma; \rho/2)|.\end{align*}
    In particular, $A\cap (A - x) \not= \emptyset$, so $x \in A - A$.
\end{proof}

Next elementary result tells us that characters efficiently separate group elements.

\begin{lemma}\label{charseparation}
    Let $S \subseteq G$ be a set of size $k$. Then there exists a Bohr set $B = B(\Gamma, 1/10)$ with $|\Gamma| \leq O(\log k)$ such that all translates $s + B$ for $s \in S$ are disjoint. In other words, homomorphism $\chi : G \to \mathbb{T}^d$, given by $\{\chi_1, \dots, \chi_d\} = \Gamma$, $\frac{1}{10}$-separates the points of $S$.
\end{lemma}

\begin{proof}
    Let $\chi \in \hat{G}$ be a character chosen uniformly at random. Observe that for any $s \not= t$ in $G$ we have $\mathbb{P}(\uc{\chi(s) - \chi(t)} \geq 1/5) = \mathbb{P}(\uc{\chi(s - t)} \geq 1/5) \geq 1/2$. We iteratively find $\Gamma$ by setting initially $P = S^{(2)}$ to be the set of all pairs in $S$ and perform $\log_2 \binom{ |S|}{2}$ steps, in each of which we take a new character $\chi$ uniformly at random and those pairs in $P$ separated by $\chi$ are removed. By linearity of expectation, we can decrease the size of $P$ by a factor of 2 each time, so the procedure terminates in the claimed number of steps. 
\end{proof}

\noindent\textbf{Weak regularity.} When working with Bohr sets, it is frequently necessary to pass to well-behaved radius. Typically, this is done by finding the so-called regular Bohr sets. However, for the purposes of this paper, we use a simpler notion, which we refer to as the weak regularity. We say that a Bohr set $B(\Gamma, \rho)$ is $(\eta, \varepsilon)$-\textit{weakly regular} if 
\[|B(\Gamma, \rho + \eta) \setminus B(\Gamma, \rho)| \leq \varepsilon |G|.\]
As we shall see soon, we require a Bohr set to be weakly-regular in order to obtain a useful approximation to its Fourier coefficients.\\

Our first result on weak regularity shows that Bohr set can be made weakly regular, without affecting it as a set.

\begin{lemma}[Weak regularity I]\label{bohrwreg1}
    Let $\Gamma \subseteq \hat{G}$ of size $d$ be given and let $\rho : \Gamma \to [0,1]$ be the radius function. Let $\varepsilon > 0$ and let $\eta \leq \frac{\varepsilon}{8d}$. Then there exists $\rho' : \Gamma \to [0,1]$, with $\rho'(\chi) \in [\rho(\chi) - \eta, \rho(\chi) + \eta]$ such that $B(\Gamma, \rho') = B(\Gamma, \rho)$ and
    \[|B(\Gamma, \rho' + \eta) \setminus B(\Gamma, \rho' - \eta)| \leq \varepsilon |G|.\]
    \indent Moreover, if $\chi \in \Gamma$ has order at least $4d \varepsilon^{-1}$, then we may take $\rho'(\chi) = \rho(\chi)$.
\end{lemma}

\noindent\textbf{Remark.} Natural choice of $\varepsilon$ is comparable to the density of $B$ in $G$, which is about $\rho^d$. Notice that we might modify the radius function, but the Bohr set remains unchanged.

\begin{proof}
    Let $\Gamma = \{\chi_1, \dots, \chi_d\}$. Set $K = 4d \varepsilon^{-1}$ so $\eta \leq \frac{1}{2K}$. Let $I \subseteq [d]$ be the set of indices $i$ such that $|\on{Im} \chi_i| \leq K$. For each such $i$, since $\on{Im} \chi_i$ is a cyclic group of size $s_i \leq K$, we have an interval $[a_i, b_i)$ between rationals with denominator $s_i$, containing $\rho(\chi_i)$, such that $\|\chi_i(x)\|_{\mathbb{T}} < b_i$ implies $\|\chi_i(x)\|_{\mathbb{T}} \leq a_i$. We define $\rho'(\chi_i) = \rho(\chi_i) + \eta$ if $\rho(\chi_i) \leq a_i + \eta$, $\rho'(\chi_i) = \rho(\chi_i) - \eta$ if $\rho(\chi_i) \geq b_i - \eta$ and $\rho'(\chi_i) = \rho(\chi_i)$ otherwise. For $i \notin I$, simply set $\rho'(\chi_i) = \rho(\chi_i)$. Write also $\rho'_i = \rho'(\chi_i)$ to simplify the notation.\\
    
    To complete the proof, it suffices to have
    \begin{equation}
        \Big|\chi_i^{-1}\Big((\rho'_i - \eta, \rho'_i + \eta]\Big)\Big| \leq \frac{\varepsilon}{2d}|G|\label{weakregstepeqindiviudalineq}
    \end{equation} for each $i \in [d]$. If all these inequalities hold, then
    \begin{align*}|B(\Gamma, \rho' + \eta) \setminus B(\Gamma, \rho' - \eta)| \leq &\sum_{i \in [d]} |B(\Gamma \setminus \{\chi_i\}, \rho'|_{\Gamma \setminus \{\chi_i\}} + \eta) \cap \{x \in G : |\chi_i(x)| \in (\rho'_i - \eta, \rho'_i + \eta]\}|\\
    \leq & \sum_{i \in [d]} |\{x \in G : |\chi_i(x)| \in (\rho'_i - \eta, \rho'_i + \eta]\}|
\\
\leq & \varepsilon |G|.\end{align*}

    Suppose that inequality~\eqref{weakregstepeqindiviudalineq} fails for some $i \in [d]$. If it happens that $|\on{Im} \chi_i| > K$, since $\on{Im} \chi_i$ is a cyclic subgroup of $\mathbb{T}$ and $2\eta K < 1$, we get at least $K$ $2\eta$-separated elements $t_1, \dots, t_{K}$ in $\on{Im} \chi_i$ and thus $|G| \geq\sum_{j \in [K]} |\chi_i^{-1}(t_j - \eta, t_i + \eta)|$, which is a contradiction. Hence, if the inequality above fails, it follows that $|\on{Im} \chi_i| \leq K$. However, by the choice of $\rho'$, we in fact have $\chi_i^{-1}\Big((\rho_i' - \eta, \rho_i' + \eta]\Big) = \emptyset$, so the proof is complete.
\end{proof}

If we insist on having a constant radius, then we may no longer guarantee that the Bohr set itself is unchanged, but we can still get an akin result. The proof is very similar.

\begin{lemma}[Weak regularity II]\label{bohrwreg2}
    Let $\Gamma \subseteq \hat{G}$ of size $d$ be given and let $\rho> 0$ be the radius. Let $\varepsilon > 0$ and let $\eta \leq \Big(\frac{\varepsilon}{8d}\Big)^2$. Then there exists $\rho' \in [\rho - \eta, \rho + \eta]$ such that
    \[|B(\Gamma, \rho' + \eta) \setminus B(\Gamma, \rho' - \eta)| \leq \varepsilon |G|.\]
    \indent Moreover, if image of $\chi \in \Gamma$ has at least $4d \varepsilon^{-1}$ elements in $\mathbb{T}$, then we may take $\rho'(\chi) = \rho(\chi)$.
\end{lemma}

\begin{proof}
    Let $\Gamma = \{\chi_1, \dots, \chi_d\}$. Set $K = 4d \varepsilon^{-1}$. By assumptions $\eta \leq \frac{1}{4K^2}$. Let $I \subseteq [d]$ be the set of indices $i$ such that $|\on{Im} \chi_i| \leq K$. For each such $i$, since $\on{Im} \chi_i$ is a cyclic group of size $s_i \leq K$, we have an interval $[a_i, b_i)$ between rationals with denominator $s_i$, containing $\rho$, such that $\|\chi_i(x)\|_{\mathbb{T}} < b_i$ implies $\|\chi_i(x)\|_{\mathbb{T}} \leq a_i$. We set $a = \max_{i \in I} a_i$ and $b = \min_{i \in I} b_i$. We define $\rho'$ as follows
    \[\rho' = \begin{cases} 
        a + \eta,\text{ if } \rho < a + \eta\\
        b - \eta,\text{ if } \rho > b - \eta\\
        \rho,\text{ otherwise}.
    \end{cases}\]

    Note that $b \geq a + \frac{1}{K^2}$, as $a$ and $b$ are distinct rationals with denominators among $s_1, \dots, s_d$.\\
    
    As in the previous proof, it suffices to have
    \begin{equation}
        \Big|\chi_i^{-1}\Big((\rho' - \eta, \rho' + \eta]\Big)\Big| \leq \frac{\varepsilon}{2d}|G|\label{weakregstepeq}
    \end{equation} for each $i \in [d]$. If all these inequalities hold, then
    \begin{align*}|B(\Gamma, \rho' + \eta) \setminus B(\Gamma, \rho' - \eta)| \leq &\sum_{i \in [d]} |B(\Gamma \setminus \{\chi_i\}, \rho' + \eta) \cap \{x \in G : |\chi_i(x)| \in (\rho' - \eta, \rho' + \eta]\}|\\
    \leq & \sum_{i \in [d]} |\{x \in G : |\chi_i(x)| \in (\rho' - \eta, \rho' + \eta]\}|
\\
\leq & \varepsilon |G|.\end{align*}

    Suppose that inequality~\eqref{weakregstepeq} fails for some $i \in [d]$. If it happens that $|\on{Im} \chi_i| > K$, since $\on{Im} \chi_i$ is a cyclic subgroup of $\mathbb{T}$, we get at least $K$ $2\eta$-separated elements $t_1, \dots, t_{2K}$ in $\on{Im} \chi_i$ and thus $|G| \geq\sum_{j \in [K]} |\chi_i^{-1}(t_j - \eta, t_i + \eta)|$, which is a contradiction. Hence, if the inequality above fails, it follows that $|\on{Im} \chi_i| \leq K$. However, by the choice of $\rho'$, we in fact have $\chi_i^{-1}\Big((\rho' - \eta, \rho' + \eta]\Big) = \emptyset$, so the proof is complete.
\end{proof}

\noindent\textbf{Large spectrum of Bohr sets.} We now turn to the issue of the size of Bohr sets in more detail. In the paper, we shall frequently require more precise information than provided by Lemma~\ref{basicbohrsizel}. To that end, we derive an approximation for the values of the Fourier transform of the indicator function of a Bohr set. The value for the zero character will give the approximation on the size of the Bohr set.\\

For $\rho, \eta > 0$, let us introduce an auxiliary bump function $\mathsf{b}_{\rho, \eta} : \mathbb{T} \to [0,1]$, given by 
\begin{equation}\mathsf{b}_{\rho, \eta}(x) = \begin{cases}
1,&\text{ when }x\in [-\rho, \rho]\\
\frac{\rho + \eta - \uc{x}}{\eta}, &\text{ when }x\in [-\rho -\eta, -\rho] \cup [\rho, \rho + \eta]\\
0, & \text{ when }x\notin [-\rho-\eta, \rho+\eta]
\end{cases},\label{bumpdefinition}\end{equation}
which is equal to the convolution $\eta^{-1} I \ast J(x)$, where $I$ and $J$ are the indicator functions of the intervals $[-\rho -\eta/2, \rho + \eta/2]$ and $[-\eta/2, \eta/2]$. The Fourier transform of the bump function is very well-behaved.

\begin{fact}[Claims 24 and 25 in \cite{generalBilBog}]\label{bumpfunctionsfact}
    For any $\xi \in \mathbb{Z}\setminus \{0\}$, we have $|\widehat{\mathsf{b}_{\rho, \eta}}(\xi)| \leq \eta^{-1}/|\xi|^2$. Furthermore, if $S \subseteq \mathbb{Z}$ is any set containing the interval $[-L, L]$ then 
    \[\Big|\mathsf{b}_{\rho, \eta}(x) - \sum_{\xi \in S} \widehat{\mathsf{b}_{\rho, \eta}}(\xi) \on{e}(\xi x)\Big| \leq 2\eta^{-1}/L\]
    holds for all $x \in \mathbb{T}$.
\end{fact}

We may now state the result on approximations of the Fourier coefficients of Bohr sets. We use a slightly different notation for Bohr sets this time, namely for $\gamma_1, \dots, \gamma_r \in \hat{G}$ and $\rho_1, \dots, \rho_r > 0$ we write $B(\gamma_{[r]}, \rho_{[r]})$ for the set of all $x$ such that $\tn{\gamma_i(x)} \leq \rho_i$. This is closely related to the functional notation, but the one used here is more convenient as we shall be interested in vanishing linear combinations of characters.

\begin{proposition}[Spectrum of Bohr sets]\label{bohrsizeLargeFC} There exists an absolute constant $\mathsf{C}_{\on{spec}}$ with the following property. Let $r\in \mathbb{N}$ and let $\eta, \varepsilon > 0$. Then, there exist a positive integer $K \leq \mathsf{C}_{\on{spec}} r \varepsilon^{-1} \eta^{-1}$ for which the following holds.\\
\indent Suppose that $\gamma_1, \dots, \gamma_r \in \hat{G}$ and $\rho_1, \dots, \rho_r > 0$ satisfy 
\begin{equation}|B(\gamma_{[r]}; (\rho_i + \eta)_{i \in [r]}) \setminus B(\gamma_{[r]}; \rho_{[r]})| \leq \varepsilon |G|/2.\label{weakregforbohrfcapprox}\end{equation}
Then for any $\tau \in \hat{G}$ and any $L_1, \dots, L_r \geq K$
\[\Big|\widehat{\id_{B(\gamma_{[r]}; \rho_{[r]})}}(\tau) - \sum_{\ssk{\lambda_1 \in [-L_1, L_1], \dots, \\\lambda_r \in [-L_r, L_r]}} \id\Big(\sum_{i \in [r]} \lambda_i \gamma_i = \tau\Big) \prod_{i \in [r]} \widehat{\mathsf{b}_{\rho_i, \eta}}(\lambda_i)\Big| \leq \varepsilon.\]
\end{proposition}

\begin{proof}
    Write $\mathsf{b}_i$ for the bump function $\mathsf{b}_{\rho_i, \eta}$. From the definition of Bohr sets, the product
    \[\mathsf{b}_1(\gamma_1(x))\cdots \mathsf{b}_r(\gamma_r(x))\]
    is 1 when $x \in B(\gamma_{[r]}; \rho_{[r]})$ and vanishes when $ x \notin B(\gamma_{[r]}; (\rho_i + \eta)_{i \in [r]})$. From the assumption~\eqref{weakregforbohrfcapprox} we obtain
    \[\Big|\widehat{\id_{B(\gamma_{[r]}; \rho_{[r]})}}(\tau) - \exx_{x \in G} \mathsf{b}_1(\gamma_1(x))\cdots \mathsf{b}_r(\gamma_r(x)) \on{e}(-\tau(x))\Big| \leq \frac{|B(\gamma_{[r]}; (\rho_i + \eta)_{i \in [r]}) \setminus B(\gamma_{[r]}; \rho_{[r]})| }{|G|} \leq \varepsilon/2.\]

    We estimate $\ex_{x \in G} \mathsf{b}_1(\gamma_1(x))\cdots \mathsf{b}_r(\gamma_r(x)) \on{e}(-\tau(x))$ using Fourier analysis on the unit circle $\mathbb{T}$. Let $K = 4 e \eta^{-1} \varepsilon^{-1} r$, which is chosen so that $\frac{2\eta^{-1}}{K} \leq \frac{\varepsilon}{2 e r}$. For each $i \in [r]$, let $s_i(x) = \sum_{\xi \in [-L_i, L_i]} \widehat{\mathsf{b}_i}(\xi) \on{e}(\xi x)$. By Fact~\ref{bumpfunctionsfact} we have $\|\mathsf{b}_i - s_i\|_{L^\infty} \leq 2\eta^{-1}/{L_i} \leq 2\eta^{-1}/K$. In particular, $\|s_i\|_{L^\infty} \leq 1 + 2\eta^{-1}/K$. Hence, by the triangle inequality
    \begin{align*}\Big|\exx_{x \in G} & \mathsf{b}_1(\gamma_1(x))\cdots \mathsf{b}_r(\gamma_r(x)) \on{e}(-\tau(x)) - \exx_{x \in G} s_1(\gamma_1(x))\cdots s_r(\gamma_r(x)) \on{e}(-\tau(x))\Big|\\
    = & \Big|\sum_{i \in [r]} \exx_{x \in G} \mathsf{b}_1(\gamma_1(x))\cdots \mathsf{b}_{i - 1}(\gamma_{i-1}(x))\Big(\mathsf{b}_i(\gamma_i(x)) - s_i(\gamma_i(x))\Big)\Big)s_{i + 1}(\gamma_{i + 1}(x))\cdots s_r(\gamma_r(x)) \on{e}(-\tau(x))\Big|\\
    \leq &  \sum_{i \in [r]}  \|\mathsf{b}_1 \|_{L^\infty} \cdots \|\mathsf{b}_{i - 1}\|_{L^\infty} \|\mathsf{b}_i  - s_i \|_{L^\infty} \|s_{i + 1}\|_{L^\infty} \cdots \|s_r \|_{L^\infty}\\
    \leq & \sum_{i \in [r]} \frac{2\eta^{-1}}{K}\Big(1 + \frac{2\eta^{-1}}{K}\Big)^{r - i} \leq r \frac{\varepsilon}{2e r} (1 + 1/r)^r \leq \varepsilon/2.\end{align*}

    Hence, we conclude that 
    \[\Big|\widehat{\id_{B(\gamma_{[r]}; \rho_{[r]})}}(\tau) - \exx_{x \in G} s_1(\gamma_1(x))\cdots s_r(\gamma_r(x)) \on{e}(-\tau(x))\Big| \leq \varepsilon,\]
    and it remains to rewrite the approximant in the desired form. Expanding the definition of functions $s_i$, we have
    \begin{align*}&\exx_{x \in G} s_1(\gamma_1(x))\cdots s_r(\gamma_r(x)) \on{e}(-\tau(x)) = \exx_{x \in G} \prod_{i \in [r]} \Big(\sum_{\xi_i \in [-L_i, L_i]} \widehat{\mathsf{b}_i}(\xi_i) \on{e}(\xi_i \gamma_i(x))\Big) \on{e}(-\tau(x))\\
    &\hspace{2cm}=\sum_{\ssk{\xi_1 \in [-L_1, L_1], \dots,\\\xi_r \in [-L_r, L_r]}} \Big(\prod_{i \in [r]} \widehat{\mathsf{b}_i}(\xi_i)\Big) \exx_{x \in G} \on{e}\Big(\sum_{i \in [r]} \xi_i \gamma_i(x)-\tau(x)\Big)\\
    &\hspace{2cm}=\sum_{\ssk{\xi_1 \in [-L_1, L_1], \dots,\\\xi_r \in [-L_r, L_r]}} \Big(\prod_{i \in [r]} \widehat{\mathsf{b}_i}(\xi_i)\Big) \id\Big(\sum_{i \in [r]} \xi_i \gamma_i = \tau\Big).\qedhere\end{align*}
\end{proof}

Later in the paper, we shall need to look at the subgroups generated by Bohr sets. As expected, such subgroups can be described using characters defining the Bohr sets.\\

\begin{proposition}\label{subgroupgeneratedbybohrset}
    Let $B = B(\Gamma, \rho)$ be a Bohr set of constant radius $\rho$ and let $G' \leq G$ be the subgroup generated by $B$. Then there are characters $\chi_1, \dots, \chi_r \in \langle \Gamma \rangle_K$, where $r \leq O(|\Gamma| \log \rho^{-1})$ and $K \leq (\rho/2)^{-O(d)}$, such that $|\on{Im} \chi_i| \leq \rho^{-d}$ and $B' = \cap_{i \in [r]} \ker \chi_i$.
\end{proposition}

\begin{proof}
    Let $d = |\Gamma|$. By classification of finite abelian groups, there exists an injective homomorphism $\chi : G/ G' \to \mathbb{T}^m$, where $m$ is the rank of $G/G'$ and hence $m \leq O(\log (|G|/ |G'|)).$ Note that each component $\chi_i$ extends to a character on $G$ which vanishes on $G'$ and in fact $G' = \cap_{i \in [m]} \ker \chi_i$. Clearly, $|\on{Im} \chi_i| \leq |G|/|G'| \leq \rho^{-d}$, by Lemma~\ref{basicbohrsizel}. It remains to prove that each $\chi_i$ is a linear combination of elements in $\Gamma$.\\

    Our goal is apply Proposition~\ref{bohrsizeLargeFC}. However, we need to find a weakly regular Bohr set first. To that end, write $\varepsilon = \frac{1}{4} 2^{-d}\rho^{d}$ and set $\eta = \frac{\varepsilon^2}{2^8d^2}$. Apply Lemma~\ref{bohrwreg2} to the Bohr set $B(\Gamma, \rho/2)$ with $\eta$. Hence, there exists $\rho' \in [\rho/2 - \eta, \rho/2 + \eta]$ such that $|B(\Gamma, \rho' + \eta) \setminus B(\Gamma, \rho')| \leq \frac{\varepsilon}{2} |G|$.\\
    \indent Write $B' = B(\Gamma, \rho/2)$ and take any $\chi_i$. In particular, $\chi_i$ vanishes on $B'$ so $|\widehat{\id_{B'}}(\chi_i)| = |B'| / |G| \geq 2^{-d}\rho^{d} = 4\varepsilon$. Since $B'$ is weakly regular, Proposition~\ref{bohrsizeLargeFC} gives us a quantity $K \leq O(d \varepsilon^{-1} \eta^{-1})$ such that, if $\chi_i \notin \langle \Gamma\rangle_{[-K, K]}$, then $|\widehat{\id_{B'}}(\chi_i)| \leq \varepsilon$. But $|\widehat{\id_{B'}}(\chi_i)|  \geq 4 \varepsilon$, which implies $\chi_i \in \langle \Gamma\rangle_{[-K, K]}$,as desired.
\end{proof}

We now recall a result from~\cite{generalBilBog}, which is a generalization of the fact that the dual of the sum of two subspaces of a vector space is given by the intersection of their duals. For a set $\Gamma = \{\gamma_1, \dots, \gamma_k\}$ of elements of an abelian group and a set of integers $S$, we write $\langle \Gamma \rangle_S$ for the set of all linear combinations $\lambda_1\gamma_1 + \dots + \lambda_k \gamma_k$, where $\lambda_1, \dots, \lambda_k \in S$. The radii of the Bohr sets in the next theorem are constant.

\begin{theorem}[Theorem 27 in~\cite{generalBilBog}]\label{bohrSum}Let $\Gamma_1, \Gamma_2 \subset \hat{G}$ and $\rho_1, \rho_2 \in (0,1)$. Then there is a positive integer $R \leq (2\rho_1^{-1})^{O(|\Gamma_1|)} + (2\rho_2^{-1})^{O(|\Gamma_2|)}$ such that
\[B(\langle \Gamma_1 \rangle_{[-R, R]} \cap \langle \Gamma_2 \rangle_{[-R, R]}, 1/4) \subseteq B(\Gamma_1, \rho_1) + B(\Gamma_2, \rho_2).\]
\end{theorem}
\vspace{\baselineskip}

As a special case of the theorem, we conclude that $B(\langle \Gamma \rangle_{[-R, R]} , 1/4) \subseteq B(\Gamma, \rho)$, allowing us to replace radius by $1/4$ at the expense of using the span of the frequency set.

\subsection{Coset progressions}

In this subsection, we derive several results concerning coset-progressions. Let $C = [-L_1, L_1] \cdot v_1 + \dots + [-L_d, L_d] \cdot v_d + H$ be a proper symmetric coset progression of rank $d$ in its canonical form. We say that $C'$ is a \textit{shrinking} if $C$ if $C' = [-L'_1, L'_1] \cdot v_1 + \dots + [-L'_d, L'_d] \cdot v_d + H$ for some $L_i' \leq L_i$. Moreover, by $\varepsilon \cdot C$ we denote the choice of $C'$ where $L'_i = \lfloor \varepsilon L_i \rfloor$.\\

One way to obtain coset progressions is through Bohr sets. This is standard fact, see Theorem 4.22 in~\cite{TaoVuBook}. In this paper, we shall need a slightly stronger fact that a Bohr sets of different radii can be intertwined with shrinkings of the same coset progression, so we postpone this result to the subsection on lattice theory, where it more naturally belongs. We now quote the related fact from~\cite{generalBilBog} that coset progressions contain large Bohr sets.

\begin{proposition}[Proposition 28 in~\cite{generalBilBog}]\label{cosettobohrset}Let $C$ be a symmetric coset progression of rank $r$ and density $\alpha$ inside $G$. Then there is a positive integer $R \leq (r \log (\alpha^{-1}))^{O(1)}$ and characters $\chi_1, \dots, \chi_R \in \hat{G}$ such that $B(\chi_1, \dots, \chi_R; 1/(4R)) \subseteq C$.\end{proposition}

We say that $A \subseteq B$ is a \emph{Freiman-subgroup} of $B$ if whenever $a, b \in A$ and $a - b \in B$ then $a - b \in A$. Another lemma from~\cite{generalBilBog} shows that Freiman-subgroups of a coset progression $C$ are closely related to $C$.

\begin{lemma}[Lemma 8 in~\cite{generalBilBog}]\label{progsbgp}Let $G$ be a finite abelian group. Let $C = [-N_1, N_1] \cdot v_1 + \dots + [-N_d, N_d] \cdot v_d + H$ be a proper symmetric coset progression of rank $d$ in its canonical form. Let $A \subseteq C$ be a Freiman-subgroup of size $|A| \geq \alpha |C|$ in $G$. Then $A$ contains $[-M_1, M_1] \cdot \ell_1 v_1 + \dots + [-M_d, M_d] \cdot \ell_d v_d + H'$ for some positive integers $\ell_1, \dots, \ell_d \leq 20\alpha^{-1}$, $M_i = \lfloor N_i/\ell_i\rfloor$ and a subgroup $H' \leq H$ of size $|H'| \geq \alpha |H|$.\end{lemma}

At some point, we shall need a tiling result about coset progressions which generalizes the fact that a group can be partitioned into cosets of a given subgroup. Here is an easy lemma that achieves that.

\begin{lemma}\label{tilingLemma}
    Let $H_0 \leq G$ and $C = a + [0, N_1 -1 ] \cdot v_1 + \dots + [0, N_d - 1] \cdot v_d + H_0$ be a proper coset progression. Let $\eta > 0$ and $S = [0, N'_1 - 1] \cdot \ell_1 v_1 + \dots + [0, N'_d - 1] \cdot \ell_d v_d + H_0$ be another coset progression.\\
    If for each $i \in [d]$ we have $\ell_i N'_i \leq \frac{\eta}{d}N_i$ or $N'_i = 1$, then there exists a set of translates $T \subseteq G$, which is a progression of rank $d$, such that $S + T \subseteq C$, $|S+ T| = | S||T|$ and $|S + T| \geq (1 - \eta) |C|$.
\end{lemma}

We use the sum notation in the case of sets as well: $\sum_{i \in I} A_i$ denotes the set of all sums $\sum_{i \in I} a_i$ with $a_i \in A_i$.

\begin{proof}
    Let $I$ be the set of all $i \in [d]$ such that $N'_i > 1$ and set $T_0 = a + \sum_{i \in I} [0, \ell_i - 1] \cdot v_i$. Note that when $i \notin I$, we have $N'_i = 1$ so $v_i$ does not affect the definition of $S$, so we may assume $\ell_i = 1$ in that case. Thus $T_0 + S$ equals $a + [0, \ell_1 N'_1 -1] \cdot v_1 + \dots + [0, \ell_d N'_d -1] \cdot v_d + H_0$ (note that $ \ell_i N'_i -1 = 0$ when $i \notin I$). Set $k_i = \lfloor N_i / (\ell_i N'_i)\rfloor$ for $i \in I$ and $T = T_0 + \sum_{i \in I} \{0, \ell_i N'_i, \dots, (k_i - 1) \ell_i N'_i\} \cdot v_i + \sum_{i \notin I} \{0, 1, \dots, N_i - 1\}\cdot v_i$. It is clear that $S + T \subseteq C$ and $|S+ T| = | S||T|$. On the other hand, $C \setminus (S + T)$ consists of linear combinations $a + \sum_{i \in [d]} \lambda_i v_i + h$ where $\lambda_i \geq k_i \ell_i N'_i$ for some $i \in I$ and $h \in H_0$. The total number of such linear combinations is at most
    \[|H_0| \sum_{i \in I} (N_i -k_i \ell_i N'_i ) \prod_{j \not= i} N_j \leq \sum_{i \in I} \frac{N_i -k_i \ell_i N'_i }{N_i} |C| \leq \sum_{i \in I} \frac{\ell_i N'_i }{N_i} |C| \leq \eta |C|.\qedhere\]
\end{proof}

A simple corollary allows us to pass to symmetric coset progressions whose iterated sumset lies in a translate of the given one.

\begin{corollary}\label{shrunksymmcpCor}
    Let $C$ be a proper coset progression of rank $d$ and let $X \subseteq C$ be a set of size $|X| \geq \alpha |C|$. Let $k$ be a positive integer. Then there exist a symmetric proper progression $S$ or rank $d$ and size $|S| \geq (\alpha/100 kd)^d |C|$ and an element $t \in G$ such that $|(t + S) \cap X| \geq \frac{\alpha}{2} |X|$ and $t + k S \subseteq C$.
\end{corollary}

\begin{proof}
    Let $a + [0, N_1 -1 ] \cdot v_1 + \dots + [0, N_d - 1] \cdot v_d + H_0$ be the canonical form of $C$. For each $i \in [d]$, define $N'_i = \frac{\alpha}{100 kd} N_i$ when $N_i \geq 200 kd \alpha^{-1}$ and $N'_i = 0$ otherwise. Define $S = [-N'_1, N'_1] \cdot v_1 + \dots + [-N'_d, N'_d] \cdot v_d + H_0$, which is symmetric and proper.\\   
    \indent Let $C' = a + [kN'_1, N_1 - 1 - kN'_1] \cdot v_1 + \dots + [kN'_d, N_d - 1 - kN'_d] \cdot v_d + H_0$, which is a slightly shrunk version of $C$ with the property that $C' + k S \subseteq C$. Note that $|C \setminus C'| \leq \sum_{i \in [d]} 2kN'_i \prod_{j \in [d] \setminus \{i\}} N_j |H_0| \leq 2k \sum_{i \in [d]} \frac{N'_i}{N_i} |C| \leq \frac{\alpha}{4} |C|$, so we get $|X \cap C'| \geq \frac{3\alpha}{4} |C|$.\\
    Apply Lemma~\ref{tilingLemma} to find a set of translates $T \subseteq G$ such that $S + T \subseteq C'$, $|S+ T| = | S||T|$ and $|S + T| \geq (1 - \alpha/4) |C'|$. Thus, there exists a translate $t + S \subseteq C'$ such that $|X \cap (t + S)| \geq \frac{\alpha}{2} |C|$, as desired.
\end{proof}

Next three facts are further elementary properties of coset progressions appearing in~\cite{generalBilBog}.

\begin{proposition}[Corollary 13 in~\cite{generalBilBog}]Let $G$ be a finite abelian group with a subgroup $H$. Then there exists a coset progression $C$ of rank at most $\log_2 |H|$ and size at least $|H|^{-2}|G|$ such that $|C + H| = |C| |H|$.\end{proposition}

\begin{theorem}[Partial projectivity in abelian groups, Theorem 9 in~\cite{generalBilBog}]\label{partialProjThm}Let $G$ and $H$ be finite abelian groups. Let $K \leq H$ be a subgroup. Suppose that $\phi \colon G \to H/K$ is a homomorphism. Let $s \geq 2$. Then there exist a proper coset progression $C \subseteq G$ of rank at most $\log_2 |K|$ and size $|C| \geq s^{- \log_2 |K|}|G|$ and a Freiman $s$-homomorphism $\psi \colon C \to H$ such that for all $x \in C$ we have $\phi(x) = \psi(x) + K$.\end{theorem}

\begin{proposition}[Corollary 12 in~\cite{generalBilBog}]\label{almostinjectivitycor}Let $G$ and $H$ be finite abelian groups and let $C \subseteq G$ be a proper coset progression with canonical form $C = P + K$ where $P = v_0 + [0, N_1 - 1] \cdot v_1 + \dots + [0, N_r - 1] \cdot v_r$ and $K \leq G$. Let $\phi \colon C \to H$ be a Freiman homomorphism such that $|\phi(C)| \geq \alpha |C|$. Then there exist a further proper coset progression $D \subseteq K$ of rank $\log_2 \alpha^{-1}$ and size $\alpha^2 |K|$, a positive integer $m \leq (8r \alpha^{-1})^r$ and a partition $P = P_1 \cup \dots \cup P_m$ into proper progressions such that $\phi$ is injective on each proper coset progression of the form $k + P_i + D$ for $i \in [m]$ and $k \in K$.\end{proposition}

We now prove a few further properties of Freiman homomorphisms on coset progressions. Firstly, we show that non-quasirandom Freiman homomorphisms to the unit circle are almost constant. Secondly, we show that, more generally, such maps come from characters. Recall the standard estimate on the exponential sum.

\begin{fact}\label{stdexpsumest} For all $\alpha \in \mathbb{R}$
    \[\Big|\Big(\sum_{k = 1}^{N} \on{e}(k \alpha)\Big)\Big| \leq \min\Big\{N, \frac{1}{\|\alpha\|_{\mathbb{T}}}\Big\}.\]
\end{fact}

\begin{lemma}[Biased forms on coset progressions]\label{biasedformsprog}
    Let $\phi : P + H \to \mathbb{T}$ be a Freiman homomorphism on a symmetric proper coset progression of rank $r$. Suppose that $\Big|\sum_{x \in P + H}(x) \on{e}(\phi(x)) \Big| \geq \delta |P + H|$. Then there exists a $P'$ is a shrinking of $P$ of size $|P'| \geq \delta \, (10^{5}r)^{-r}|P|$ with the following property.\\
    \indent Let $\varepsilon > 0$ be given. Then $\|\phi(x) - \phi(0)\|_{\mathbb{T}} < \varepsilon$ holds for $x \in \varepsilon \cdot P' + H$, where $\varepsilon \cdot P'$ is further shrinking by a factor of $\varepsilon$.
\end{lemma}

\begin{proof}
    Without loss of generality, $\phi$ is Freiman-linear and $P = [-L_1, L_1] \cdot a_1 + \dots + [-L_r, L_r] \cdot a_r$. Then, for $\alpha_i = \phi(a_i)$, we have
    \[\Big|\Big(\sum_{x \in H} \on{e}(\phi(x))\Big)\prod_{i \in [r]} \Big(\sum_{k = -L_i}^{L_i} \on{e}(k \alpha_i)\Big)\Big| \geq \delta |H| \prod_{i \in [r]} (2L_i + 1).\]
    By Fact~\ref{stdexpsumest}, $\prod_{i \in [r]} \max\{\|\alpha_i\|_{\mathbb{T}}, 1/(2L_i+1)\} \leq \delta^{-1} 10^r / |P|$ and $\phi = 0$ on $H$. Set $L'_i = \on{min}\Big\{\Big\lfloor\frac{1}{1000 r \|\alpha_i\|_{\mathbb{T}}}\Big\rfloor, L_i\Big\}$ and $P' = [-L'_1, L'_1] \cdot a_1 + \dots + [-L'_r, L'_r] \cdot a_r$. It is easy to check that $P'$ has the desired property.
\end{proof}

We now prove approximation for Freiman homomorphisms on coset progressions.

\begin{lemma}\label{homapproxcosetp}
    Let $P + H$ be a symmetric proper coset progression of rank $r$ and density $c$. Let $\psi : P + H \to \mathbb{T}^d$ be a Freiman homomorphism. Then there exists a $P'$ is a shrinking of $P$ of size $|P'| \geq c^2(10^5r)^{-r}|P|$ with the following property.\\
    \indent Let $\varepsilon > 0$. Then there exists a homomorphism $\tilde{\psi} : G \to \mathbb{T}^d$ and an element $u \in \mathbb{T}^d$ such that $\|\psi(x) - \tilde{\psi(x)} - u\|_\infty < \varepsilon$ holds for all $x \in \varepsilon \cdot P' + H$.
\end{lemma}

\begin{proof}
    It suffices to prove the claim for the case $d = 1$. Consider an auxiliary function $f : G \to \mathbb{C}$ given by $f(x) = \id_{P + H}(x) \on{e}(\psi(x))$. Then, since $\psi$ is a Freiman homomorphism
    \[\|f\|_{\mathsf{U}^2}^4 = \exx_{x, a, b} \partial_{a,b}\id_{P + H}(x) \geq c^4.\]
    By the $\mathsf{U}^2$-inverse theorem, we have a homomorphism $\tilde{\psi} :G \to \mathbb{T}$ such that
    \[\Big|\exx_x \id_{P+ H}(x) \on{e}(\psi(x) - \tilde{\psi}(x))\Big| \geq c^2.\]
    We are done by Lemma~\ref{biasedformsprog}.
\end{proof}

\subsection{Variants of Freiman's theorem}

In this paper, we make heavy use of the Bogolyubov-Rusza lemma due to Sanders~\cite{Sanders}.

\begin{theorem}[Sanders's Bogolyubov-Rusza lemma, Theorem 1.1 in~\cite{Sanders}]\label{bogRuzsa1}Suppose that $G$ is a finite abelian group and that $A, B \subseteq G$ are subsets such that $|A + B| \leq K \min \{|A|, |B|\}$. Then there exists a symmetric proper coset progression $C \subseteq A - A + B - B$ of rank at most $\log^{O(1)} (K)$ and size $|C| \geq \exp\Big(-\log^{O(1)}(K)\Big)|A+B|$.\end{theorem}

The next variant is the robust Bogolyubov-Rusza lemma of Schoen and Sisask~\cite{SchSis}. The formulation below follows from their work and appears as Theorem 15 in~\cite{generalBilBog}.

\begin{theorem}[Robust Bogolyubov-Rusza lemma]\label{robustBogRuzsa}Let $G$ be a finite abelian group. Let $A \subset G$ be a set such that $|A + A| \leq K |A|$. Then there exists a symmetric proper coset progression $C$ of rank at most $\log^{O(1)} (K)$ and size at least $|C| \geq \exp(-\log^{O(1)} (K))|A|$ such that for each $x \in C$ there are at least $\frac{1}{64K}|A|^3$ quadruples $(a_1, a_2, a_3, a_4) \in A^4$ such that $x = a_1 + a_2 - a_3 - a_4$.\end{theorem}

The next corollary follows from the observation that coset progressions of small rank have small doubling constant. It appears as Corollary 16 in~\cite{generalBilBog}. 

\begin{corollary}\label{robustBogRuzsaCP}Let $G$ be a finite abelian group and let $C$ be a coset progression of rank $d$. Suppose that $A \subseteq C$ has size $|A| \geq \alpha |C|$. Then there exists a symmetric proper coset progression $C'$ of rank at most $(d\log (\alpha^{-1}))^{O(1)}$ and size at least $|C'| \geq  \exp\Big(-(d\log (\alpha^{-1}))^{O(1)}\Big) |C|$ such that for each $x \in C'$ there are at least $\frac{\alpha}{2^{d + 6}}$ quadruples $(a_1, a_2, a_3, a_4) \in A^4$ such that $x = a_1 + a_2 - a_3 - a_4$.\end{corollary}

Another closely related result that we need is about approximate homomorphisms between finite abelian groups. The result below is Theorem 17 from~\cite{generalBilBog}. Result in this spirit was first proved by Gowers~\cite{Gow4AP} in the case of cyclic groups (with weaker bounds).

\begin{theorem}\label{approxFreimanHom}Let $G$ and $H$ be finite abelian groups, let $A \subset G$ be a subset and let $\phi \colon A \to H$ be a function. Suppose that the number of additive quadruples respected by $\phi$ is at least $\delta |G|^3$. Then there exist a proper coset progression $Q \subseteq G$ of rank at most $\log^{O(1)} (\delta^{-1})$ and a Freiman homomorphism $\Phi \colon Q \to H$ such that $\Phi(x) = \phi(x)$ holds for at least $\exp(-\log^{O(1)} (\delta^{-1})) |G|$ elements $x \in A \cap Q$.\end{theorem}

Unlike the previous results, which concerned approximate subgroups, we need an further consequence of the Sanders's Bogolyubov-Ruzsa theorem that allows us to efficiently generate actual subgroups. For a set $A$, we write $[L] \cdot A = \{\ell a: \ell \in [L], a \in A\}$.

\begin{proposition}\label{efficientGroupGeneration}
    Let $A \subseteq G$ be a symmetric set of density $\alpha$. Then $r ([L] \cdot A)$ (which is a sum of $r$ copies of $[L]\cdot A$) is a subgroup for some $r \leq O(\log^{O(1)} \alpha^{-1})$ and $L \leq \exp(r)$. 
\end{proposition}

In other words, elements of the subgroup generated by $A$ can be expressed in the form $\lambda_1 a_1 + \dots + \lambda_r a_r$ for $a_1, \dots, a_r \in A$, where $r$ is quite small and $\lambda_1, \dots, \lambda_r$ are reasonably small. The example of $[0, \alpha N]$ inside the cyclic group $\mathbb{Z}/N\mathbb{Z}$ shows that somewhat large coefficients in the representation above are necessary.\\

Before proving the proposition, we prove a couple of special cases. The first one is the classical fact that a sets of very small doubling is very dense inside a subgroup.

\begin{lemma}\label{verysmalldoublingsbugroup}
        Let $X \subseteq G$ be a set such that $|X - X| \leq 1.01 |X|$. Then $X - X$ is a group.
\end{lemma}

    \begin{proof} Let $d \in X - X$ be an arbitrary difference. Then $d = x - y$ for some $x,y \in X$. Note that $x - X, y - X \subseteq X - X$ are two sets of size $|X|$ so we have
    $|x - X \cap y - X| \geq 0.99 |X|$. Hence, there are at least $0.99 |X|$ elements $z \in G$ such that $x - z, y- z \in X$. The difference of each such pair is again $x - y = d$, so each difference in $X - X$ has at least $0.99|G|$ representations.\\ 
    \indent Take any $d_1, d_2 \in X - X$. Let $Y$ be the set of all $y \in X$ such that $y - d_1 \in X$, and let $Z$ be the set of all $z \in X$ such that $z - d_2 \in X$. By the work above, $Y$ and $Z$ are two subsets of $X$ of size at least $0.99 |G|$. Hence, $Y \cap Z$ is non-empty, so take any $w$ inside it. Then $w - d_1, w - d_2 \in X$, so we have $d_1 - d_2 = (w - d_2) - (w - d_1) \in X - X$, showing that $X - X$ is a finite set closed under taking differences, making it a group.
    \end{proof}

The second easy lemma concerns the case of very small ambient group.

\begin{lemma}\label{smallGroupsGeneration}
    Let $K$ be a finite abelian group of size at most $M$ with a generating set $A$. Then $r' ([0, M-1] \cdot A) = K$ for some $r' \leq \log_2 M$.
\end{lemma}

\begin{proof}
    Let $a_1, \dots, a_s$ be the elements of $A$. By reordering and removing the if necessary, we may assume that $a_i \notin \langle a_1, \dots, a_{i-1}\rangle$ for each $i \in [s]$ and $\langle a_1, \dots, a_s\rangle = K$. In particular, this means that $|\langle a_1, \dots, a_i\rangle| \geq 2^i$ for each $i \in [s]$ and so $s \leq \log_2 M$. Moreover, since $|K|\leq M$, every element has order at most $M$. Thus, every element in $K$ can be written as a linear combination $\lambda_1 a_1 + \dots + \lambda_s a_s$ with $\lambda_i \in [0, M - 1]$, completing the proof.
\end{proof}

We may now prove Proposition~\ref{efficientGroupGeneration}.

\begin{proof}
    We apply Theorem~\ref{bogRuzsa1} to find a proper symmetric coset progression $P + K$ inside $2A - 2A$, of rank $d \leq \log^{O(1)}(\alpha^{-1})$ and size $|P + K| \geq \exp\Big(-\log^{O(1)}(\alpha^{-1})\Big)|A|$. By symmetry of $A$, we have $2A - 2A = 4A$. Let $L = \Big\lceil\frac{|G|}{|P + K|}\Big\rceil$. Let us first show that $2d([L] \cdot (P + K))$ is a subgroup.\\

    Write $P = P_1 + \dots + P_d$, for some symmetric arithmetic progressions $P_i = [-\ell_i, \ell_i] \cdot a_i$ (so $a_i$ is a step of $P_i$) such that $|P + K| = |K| \prod_{i \in [d]}|P_i|$. Observe firstly that $m (P + K) \subseteq (d + 1) ([m] \cdot P) + K$ for each $m \geq 1$. Indeed, if $x_1, \dots, x_m \in P$ and $k_1, \dots, k_m \in K$ then we may write each $x_i = \sum_{j \in [d]} y_{ij}$ for $y_{ij} \in P_j$. In particular $x_1 + \dots + x_m = \lambda_1 a_1 + \dots + \lambda_d a_d$ for integers $\lambda_i$ with $|\lambda_i| \leq m \ell_i$. For each $j \in [d]$, take $\mu_j \in [-\ell_j, \ell_j]$ such that $|\lambda_j - m \mu_j| < m$. Write $\nu_j = \lambda_j - m \mu_j$. Then
    \[x_1 + \dots + x_m = m(\mu_1 a_1 + \dots + \mu_d a_d) + \nu_1 a_1 + \dots + \nu_d a_d,\]
    so \[(x_1 + k_1) + \dots + (x_m + k_m) = (x_1 + \dots + x_m) + (k_1 + \dots + k_m) \in (d + 1)([m] \cdot Q) + K,\]
    showing that $m (P + K) \subseteq (d + 1) ([m] \cdot P) + K$.\\

    Therefore, it suffices to show that $m (P + K)$ is a subgroup for a reasonably small $m$. Let us consider $m = 2^k$. Note that by symmetry of $P + K$ we have $2^{k + 1}(P + K) = 2^k(P + K) - 2^k (P + K)$. If it happens that $|2^{k + 1}(P + K)| \leq 1.01 |2^{k}(P + K)|$,  we are done by Lemma~\ref{verysmalldoublingsbugroup}. Clearly, this must happen for some $k \leq \log_{1.01} L$. Hence, for some $m \leq 2^{\log_{1.01} L}$ we have that $m(P + K)$ is a subgroup and thus $(d + 1) ([m] \cdot (P + K))$ contains a subgroup of size at least $|P + K|$. Going back to the set $A$, we have that $4(d + 1) ([m] \cdot A)$ contains a subgroup $H \leq G$ of density at least $L^{-1}$.\\

    Let $G'$ be the subgroup generated by $A$. Hence, the projection of $A$ in $G/H$ generates $G' / H$, and so by Lemma~\ref{smallGroupsGeneration} $r' ([L] \cdot (A + H)/H) = K$ for some $r' \leq \log_2 L$. But this means that $G' = r' ([L] \cdot (A + H)) \subseteq (r'(4d + 5))([Lm] \cdot A) \subseteq G'$, as desired.  
    \end{proof}

\subsection{Lattice theory}

It is a standard fact that lattices have $\mathbb{Z}$-basis. The volume of the parallelepiped spanned by a $\mathbb{Z}$-basis of a lattice $\Lambda$ of full rank is the \textit{covolume} of $\Lambda$, denoted $\on{cov} \Lambda$, and is independent of the choice of a $\mathbb{Z}$-basis as it equals the absolute value of the determinant of the basis.\\

In this paper, we also need a result about images of homomorphisms $\chi : G \to \mathbb{T}^r$. Phrased in terms of Bohr sets, the next results determines which generalized Bohr sets, which are allowed to be centered about non-zero values, are non-empty. However, the result is essentially about lattices.

\begin{proposition}
    Let $\chi : G \to \mathbb{T}^r$ be a homomorphism. Let $\Lambda = \{\lambda \in \mathbb{Z}^r : \lambda \cdot \chi = 0\}$ be the lattice of vanishing linear combinations of $\chi$. Then $\on{Im} \chi = \{v \in \mathbb{T}^r : (\forall \lambda \in \Lambda) \lambda \cdot v = 0\}$.
\end{proposition}

\begin{proof}
    Since $\chi_i$ is a homomorphism from $G$ to $\mathbb{T}$ we have $\on{Im} \chi \subseteq \{(a_1/|G|, \dots, a_r / |G|) : a_1, \dots, a_r \in [0, |G|-1]\}$. Hence, we may instead consider $\chi : G \to (\mathbb{Z}/n\mathbb{Z})^r$, where $n = |G|$. Write $\Lambda^\perp = \{(a_1, \dots, a_r) \in (\mathbb{Z}/n\mathbb{Z})^r : (\forall \lambda \in \Lambda)\,\lambda \cdot a = 0\}$. Let $a = \chi(g) \in \on{Im} \chi$ and $\lambda \in \Lambda$. Then $\lambda \cdot a = \lambda \cdot \chi(g) = 0$ so $\on{Im}\chi \subseteq \Lambda^\perp$. To prove the claim, it suffices to show that $|\on{Im}\chi| = |\Lambda^\perp|$.\\
    \indent By the first isomorphism theorem for groups, we have $|\on{Im}\chi| = |G| / |\on{ker} \chi|$. To find the size of the kernel, we use exponential sums
    \begin{align*}|\on{ker} \chi| = &\sum_{x \in G} \id(\chi(x) = 0) = \sum_{x \in G} \exx_{\lambda_1, \dots, \lambda_r \in \mathbb{Z}/n\mathbb{Z}}\on{e}\Big(\frac{\lambda_1 \chi_1(x) + \dots + \lambda_r \chi_r(x)}{n}\Big)\\
    = & \exx_{\lambda_1, \dots, \lambda_r \in \mathbb{Z}/n\mathbb{Z}} \sum_{x \in G} \on{e}\Big(\frac{\lambda_1 \chi_1(x) + \dots + \lambda_r \chi_r(x)}{n}\Big)\\
    = & |G| \exx_{\lambda_1, \dots, \lambda_r \in \mathbb{Z}/n\mathbb{Z}} \id(\lambda \cdot \chi = 0)\\
    = & \frac{|\{ \lambda \in [0, n-1]^r : \lambda \cdot \chi = 0\}|}{n^r} |G|.
    \end{align*}

Note that $n e_i \in \Lambda$ for all standard basis vectors $e_i$. Hence, 
\[\frac{|\{ \lambda \in [0, n-1]^r : \lambda \cdot \chi = 0\}|}{n^r} = \frac{|\{ \lambda \in [0, Rn-1]^r : \lambda \cdot \chi = 0\}|}{R^r n^r}\]
for all $R \in \mathbb{N}$, and the right hand side tends to $\frac{1}{\on{cov} \Lambda}$.\\
\indent Finally, if $a_1, \dots, a_r \in \mathbb{Z}^r$ is a $\mathbb{Z}$-basis of $\Lambda$, its Smith normal form is diagonal matrix with entries $d_1, \dots, d_r$ dividing $n$, from which we deduce $|\Lambda^\perp| = d_1 d_2 \cdots d_r = \det (a_1, \dots, a_r) = \on{cov} \Lambda$, which gives $|\on{Im} \chi| = |\Lambda^\perp|$, as desired.\end{proof}

\vspace{\baselineskip}
We also need a variant of this result which tells us that if we are only interested in the approximate values of images of a homomorphism $\chi : G \to \mathbb{T}^r$, then we only need to know the bounded linear combinations in $\Lambda$.

\begin{proposition}\label{approximatelatticeimage}
    There exists an absolute constant $C$ such that the following holds. Let $\chi : G \to \mathbb{T}^r$ be a homomorphism. Let $\Lambda = \{\lambda \in \mathbb{Z}^r : \lambda \cdot \chi = 0\}$ be the lattice of vanishing linear combinations of $\chi$. Let $K$ be a given positive integer. Let $v \in \mathbb{T}^r$ be any element such that $\lambda_1 v_1 + \dots + \lambda_r v_r = 0$ holds for all $\lambda \in \Lambda \cap [-K, K]^r$. Then, provided $K \geq C r  (2\delta^{-1})^{2r + 1}$, there exists an element $x \in G$ such that $\|\chi(x) -v\|_{\mathbb{T}} \leq \delta$.
\end{proposition}

\begin{proof}
    Let $\eta = \frac{1}{4}(\delta/2)^{r+1}$. Let $\rho \in [\delta/2, \delta - \eta]$ be such that 
    \[|B(\chi_1, \dots, \chi_r, \rho + \eta) \setminus B(\chi_1, \dots, \chi_r, \rho)| \leq \frac{1}{4} (\delta/2)^r |G|.\]
    Such a $\rho$ exists by pigeonhole principle. We need this property as we shall apply Proposition~\ref{bohrsizeLargeFC} to complete the proof.\\
    
    \indent Assume the contrary, namely that there is no $x \in G$ such that $\|\chi(x)-v\|_{\mathbb{T}} \leq \delta$.\\
    \indent Recall the definition of bump functions~\eqref{bumpdefinition} and let $\mathsf{b} = \mathsf{b}_{\rho, \eta}$. Since no $x \in G$ satisfies $\|\chi(x) -v\|_{\mathbb{T}} \leq \rho + \eta$, we have
    \[\exx_{x \in G} \mathsf{b}(\chi_1(x) - v_1) \cdots \mathsf{b}(\chi_r(x) - v_r) = 0.\]
    Let $g(x) = \sum_{a \in [-K, K]} \hat{\mathsf{b}}(a) \on{e}(a x)$. By Fact~\ref{bumpfunctionsfact}, we have $|\mathsf{b}(x) - g(x)| \leq 2\eta^{-1}/K$ for all $x \in \mathbb{T}$. Using this inequality, we have $\|g\|_{\infty} \leq 1 + 2\eta^{-1}/K \leq 1 + 1/r$ as long as $K \geq 2\eta^{-1} r$, and we conclude that
    \[\Big|\exx_{x \in G} g(\chi_1(x) - v_1) \cdots g(\chi_r(x) - v_r)\Big| \leq 2\eta^{-1} r (1+1/r)^r /K.\]
    Write $\varepsilon$ for the bound on the right hand side. Hence, 
    \begin{align*}\varepsilon \geq &\Big|\exx_{x \in G} \sum_{a_1, \dots, a_r \in [-K, K]} \hat{\mathsf{b}}(a_1) \cdots \hat{\mathsf{b}}(a_r)\on{e}(a_1(\chi_1(x) - v_1) + \dots + a_r(\chi_r(x) - v_r)) \Big| \\
    = &\Big| \sum_{a_1, \dots, a_r \in [-K, K]} \hat{\mathsf{b}}(a_1) \cdots \hat{\mathsf{b}}(a_r) \on{e}(- a \cdot v) \Big(\exx_{x \in G} \on{e}(a_1\chi_1(x) + \dots + a_r \chi_r(x))\Big)\Big|\\
    = &\Big| \sum_{a_1, \dots, a_r \in [-K, K]} \hat{\mathsf{b}}(a_1) \cdots \hat{\mathsf{b}}(a_r) \on{e}(- a \cdot v) \id(a_1\chi_1 + \dots + a_r \chi_r = 0)\Big|\\
    = &\Big| \sum_{a_1, \dots, a_r \in [-K, K]} \hat{\mathsf{b}}(a_1) \cdots \hat{\mathsf{b}}(a_r) \id(a_1\chi_1 + \dots + a_r \chi_r = 0)\Big|,
    \end{align*}
    using the fact that $a \cdot v = 0$ for all $v \in \Lambda \cap [-K, K]^r$ in the last step. Let $C$ be implicit constant from the bound on $K$ in Proposition~\ref{bohrsizeLargeFC}. By Proposition~\ref{bohrsizeLargeFC}, the final term is an approximation for $|B(\chi_1, \dots, \chi_r; \rho)| / |G|$ up to error $\frac{1}{2} (\delta/2)^r$, provided $K \geq 8 C r \eta^{-1} (\delta/2)^{-r})$. From Lemma~\ref{basicbohrsizel}, Bohr set $B(\chi_1, \dots, \chi_r; \rho)$ is of size at least $\rho^r |G|$, leading to inequality $\varepsilon \geq \rho^r - \frac{1}{2} (\delta/2)^r$. This is a contradiction, as $K \geq 32 C r  (2\delta^{-1})^{2r + 1}$.
\end{proof}

The following two results from~\cite{generalBilBog} concern quantitative theory of lattices.

\begin{lemma}[Lemma 29 in~\cite{generalBilBog}]\label{nestedLattices}Let $k$ and $K$ be positive integers. Suppose that $\Lambda_1 \subseteq \Lambda_2 \subseteq \dots \subseteq \Lambda_r \subseteq \mathbb{Z}^k$ are lattices such that for each $i \in [r-1]$ the set-difference $\Lambda_{i + 1} \setminus \Lambda_i$ contains an element of the box $[-K, K]^k$. Then $r \leq O(k^2 (\log k + \log K))$.\end{lemma}

\begin{theorem}[Quantitative fundamental theorem of lattices, Theorem 30 in~\cite{generalBilBog}]\label{quantLattice}Let $G$ be a finite abelian group, let $a_1, \dots, a_k \in G$ and let $R$ be a positive integer. Suppose that $B \subset \langle a_1, \dots, a_k \rangle_R$ is a non-empty set. Then, there exist positive integers $\ell \leq O(k^2 (\log k + \log R))$, $S \leq O((2Rk)^{k+ 3})$ and elements $b_1, \dots, b_\ell \in B$ such that $B \subseteq \langle b_1, \dots, b_\ell\rangle_{S}$.\end{theorem}

Let us recall the Discrete John's theorem. For a sequence $L = (L_1, \dots, L_r)$, the notation below $(-L, L) \cdot w$ stands for all sums $\sum_{i \in [r]} \lambda_i w_i$ with integer $\lambda_i \in (-L_i, L_i)$.

\begin{proposition}[Discrete John's theorem, Lemma 3.36 in~\cite{TaoVuBook}]\label{discJohn}
    Let $B$ be a convex symmetric body in $\mathbb{R}^d$ and let $\Gamma$ be a lattice in $\mathbb{R}^d$ of rank $r$. Then there exists an $r$-tuple $w = (w_1, \dots, w_r) \in \Gamma^r$ of linearly independent vectors in $\Gamma$ and an $r$-tuple $L = (L_1, \dots, L_r)$ of positive integers such that
    \[(r^{-2r} \cdot B) \cap \Gamma \subseteq (-L, L) \cdot w \subseteq B \cap \Gamma \subseteq (-r^{2r} L, r^{2r} L) \cdot w.\]
\end{proposition}

We deduce that Bohr sets contain coset progressions and that their shrinkings can be intertwined. The proof is the same as that of Lemma 4.22 of~\cite{TaoVuBook}. 

\begin{proposition}\label{cpsinbohrsets}
    Let $B(\Gamma, \rho)$ be a Bohr set of codimension $r$. Then there exists a proper symmetric coset progression $C$ of rank at most $r$ such that $r^{2r} C$ is also proper and
    \[B(\Gamma, r^{-4r} \rho) \subseteq C \subseteq B(\Gamma, r^{-2r} \rho)  \subseteq r^{2r} C \subseteq B(\Gamma, \rho).\]
    Moreover, if $C = [-L_1, L_1] \cdot a_1 + \dots + [-L_r, L_r] \cdot a_r + H$ is its canonical form, then for all $c \in (0,1/2)$ we have
    \[B(\Gamma, cr^{-4r} \rho) \subseteq [-2c L_1, 2cL_1] \cdot a_1 + \dots +[-2cL_r, 2cL_r] \cdot a_r + H.\]
\end{proposition}

\begin{proof}
    Let $\Gamma = \{\chi_1, \dots, \chi_r\}$. The image of the homomorphism $\chi = (\chi_1, \dots, \chi_r): G \to \mathbb{T}^r$ is a finite subgroup $\chi(G) \leq \mathbb{T}^r$ and, for any radius $\sigma$, the Bohr set $B(\Gamma, \sigma)$ is the inverse image of the cube $Q_\sigma = \{x \in \mathbb{T}^r : \uc{x}\leq \sigma\}$.\\
    \indent Let $\Lambda \leq \mathbb{R}^r$ be the lattice $\chi(G) + \mathbb{Z}^r$. Applying Proposition~\ref{discJohn} to the intersection $\Lambda \cap [-r^{-2r}\rho, r^{-2r}\rho]^r$ we find a progression $\tilde{P} = (-L, L)\cdot w$ for some linearly independent $w_1, \dots, w_r \in \Lambda$ such that
    \[[-r^{-4r}\rho, r^{-4r}\rho]^r \cap \Lambda \subseteq (-L, L) \cdot w \subseteq [-r^{-2r}\rho, r^{-2r}\rho]^r \cap \Lambda \subseteq (-r^{2r} L, r^{2r} L) \cdot w.\]
    Note that we also have $(-r^{2r} L, r^{2r} L) \cdot w \subseteq [-\rho, \rho]^r \cap \Lambda$. We take $C = (-L, L) \cdot v + H$, where $v_i$ is an arbitrary element of $\chi^{-1}(w_i)$ and $H = \on{ker}\chi$.\\
    \indent To complete the proof, let $c \in (0,1/2)$, set $\sigma = cr^{-4r}\rho$. We first show that $Q_{\sigma} \cap \Lambda \subseteq (-2cL, 2cL) \cdot w$. Namely, we have $x \in Q_\sigma \cap \Lambda \setminus (-2cL, 2cL) \cdot w$. But, for $k = \lceil c^{-1} \rceil$, $x, kx \in [-r^{-4r}\rho, r^{-4r}\rho]^r \cap \Lambda$, so $x, kx \in (-L, L) \cdot w$. Hence, there are integers $\lambda_i, \mu_i \in (-L_i, L_i)$ such that $x = \sum_{i \in [r]} \lambda_i w_i$ and $kx = \sum_{i \in [r]} \mu_i w_i$. However, $w_1, \dots, w_r$ are linearly independent in $\mathbb{R}^r$, so $k \lambda_i = \mu_i$. Thus, $|\lambda_i| < L_i/k \leq 2c L_i$.\\
    Since $B(\Gamma, \sigma) = \chi^{-1}(Q_\sigma \cap \Lambda)$ and $[-2c L_1, 2cL_1] \cdot a_1 + \dots +[-2cL_r, 2cL_r] \cdot a_r + H$ contains the inverse image of $(-2cL, 2cL) \cdot w$, we are done.
\end{proof}

As a corollary, we show how Bohr sets can be used to efficiently generate subgroups.

\begin{corollary}\label{subgroupgenerationiteratedsumsbohr}
    Let $B = B(\Gamma, \rho)$ be a Bohr set of codimension $d$ and let $r \in \mathbb{N}$. Then every element of $ B(\Gamma, \rho/2)$ can be written as $x_1 + 2x_2 + \dots + 2^{r-1}x_{r}$ for $x_1, \dots, x_r \in B$ in at least $2^{-d(r+1)^2}\rho^{-dr}|G|^{r-1}$ many ways.\\
    \indent Furthermore, let $\rho' = d^{-2d} \rho$ and $r' = 1000 (d^2 \log d + d \log \rho^{-1})$. Let $G'$ is the subgroup of $G$ generated by $B' = B(\Gamma, \rho')$. Then every element of $G'$ can be written as $x_1 + 2x_2 + \dots + 2^{r'-1}x_{r'}$ for $x_1, \dots, x_{r'} \in B'$, in at least $2^{-O(d{r'}^2)}d^{-2d^2r'}\rho^{dr'}|G|^{r'-1}$ many ways.
\end{corollary}

\begin{proof}
    For the first part of the corollary, take any $y \in B(\Gamma, \rho/2)$ and $x_i \in B(\Gamma, 2^{-2i-1}\rho)$ for $i = 2, \dots, r$. Then set $x_1 = y - (2x_2 + \dots + 2^{r-1}x_{r})$. By the choice of elements, we have $x_1 \in B(\Gamma, \rho)$. Due to Lemma~\ref{basicbohrsizel}, we can write $y$ as $x_1 + 2x_2 + \dots + 2^{r-1}x_{r}$ for $x_1, \dots, x_r \in B$ in at least $2^{-d(r+1)^2}\rho^{-dr}|G|^{r-1}$ many ways.\\
    
    For the second part of the conclusion, due to the way we chose $\rho'$, by Proposition~\ref{cpsinbohrsets}, we may find coset progressions $C$ such that $C$ and $d^{2d} C$ are proper, symmetric and of rank at most $d$, and such that $B(\Gamma, d^{-2d}\rho') \subseteq C\subseteq B(\Gamma, \rho') \subseteq  d^{2d} C\subseteq B(\Gamma, \rho)$. Note that inclusion $C\subseteq B(\Gamma, \rho') \subseteq d^{2d} C$ implies that $C$ also generates $G'$.\\

    Let $C = [-L_1, L_1] \cdot a_1 + \dots + [-L_d, L_d] \cdot a_d + H$ be its canonical form. Set $L'_i = \lceil L_i/2 \rceil$ and $L''_i = \lfloor L_i/2\rfloor$ and define additional symmetric proper coset progressions $C'$ and $C''$ as $C' = [-L'_1, L'_1] \cdot a_1 + \dots + [-L'_d, L'_d] \cdot a_d + H$  and $C'' = [-L''_1, L''_1] \cdot a_1 + \dots + [-L''_d, L''_d] \cdot a_d + H$. Hence, $C' + C'' = C$, and $C' \subseteq C \subseteq 2C'$. In particular, the subgroup generated by $C'$ is again $G'$.\\
    
    Let us check that $C' + 2\cdot C' + \dots + 2^r \cdot C' = G'$ for $r = 100 (d^2 \log d + d \log \rho^{-1})$. Using Lemma~\ref{verysmalldoublingsbugroup}, and noting that $C$ is a symmetric set, we get that either $2^k C'$ is a subgroup of $G$ or has size at least $1.01^k |C'|$. Using the elementary inequality $1.01^{100} > 2$ and Lemma~\ref{basicbohrsizel}, we obtain $n C' = G'$ for some $n \leq (|G| / |C'|)^{100} \leq d^{400 d^2}\rho^{-100 d}$. Since $C'$ is a coset progression and $r \geq \log_2 n$, we conclude that $n C' \subseteq C' + 2\cdot C' + \dots + 2^r \cdot C'$.\\
    
    It remains to show that there are many representations of elements of $G'$. Let $y \in G'$ and let $z_0, \dots, z_r \in C'$ be such that $y = z_0 + 2z_1 + \dots + 2^r z_r$. Take any $x_0, \dots, x_r \in C''$ such that $x_0 + 2x_1 + \dots + 2^rx_{r} = 0$. Each such a choice gives another representation for $y$, by putting $x_i + z_i$. On the other hand, by choosing $x_2 \in \frac{1}{8} C'', \dots, x_r \in \frac{1}{2^{2r + 1}} C''$, we get $x_1 = -2x_2 - \dots - 2^r x_r \in C''$, finish the proof.
\end{proof}

\vspace{\baselineskip}

The next result is about sets that arise as inverse images of Freiman homomorphisms from a Bohr set to the unit circle, which can therefore be thought of as `Bohr-Bohr sets'. It is shown that such a set is actually close to a usual Bohr set.

\vspace{\baselineskip}

\begin{proposition}[Bohr-Bohr sets are Bohr]\label{bohrbohr}
    Let $B = B(\Gamma, \rho)$ be a Bohr set and let $r = |\Gamma|$. Let $\phi : B \to \mathbb{T}^d$ be a Freiman-linear map and let $\sigma > 0$. Then the set
    \[X = \{x \in B: \uc{\phi(x)} \leq \sigma\}\]
    contains a Bohr set of codimension at most $d + (2r \log (\sigma^{-1} \rho^{-1}))^{O(1)}$ and radius at least $\sigma\,(2r \log (\sigma^{-1} \rho^{-1}))^{-O(1)}$.
\end{proposition}

\begin{proof}
    By Proposition~\ref{cpsinbohrsets}, $B$ contains a proper symmetric coset progression $C$ of size $|C| \geq \rho^r r^{-2r^2}|G|$ and rank at most $r$. The map $\phi$ is a Freiman homomorphism on $C$ so by Lemma~\ref{homapproxcosetp} there exist further symmetric proper coset progression $C' \subseteq C$ of rank at most $r$ and size $|C'| \geq \sigma^r \rho^{3r} (10^5 r)^{-7r^2} |G|$, and a homomorphism $\tilde{\phi} : G \to \mathbb{T}^d$ such that $\uc{\phi(x) - \tilde{\phi}(x)} \leq \sigma/2$ for all $x \in C'$. Let $B' = B(\Gamma', 1/4R)$ Bohr set inside $C'$ given by Proposition~\ref{cosettobohrset}, where $R \leq (2r \log (\sigma^{-1} \rho^{-1}))^{O(1)}$ and $|\Gamma'| \leq R$. Then $X$ contains Bohr set $B' \cap B(\tilde{\phi}_1, \dots, \tilde{\phi}_r, \sigma/2)$.
\end{proof}

\vspace{\baselineskip}

Another property concerns the iterated sums of intersections of Bohr sets. 

\vspace{\baselineskip}

\begin{proposition}[Sums of generalized Bohr sets]\label{sumsofgenbohrs}
    There exists an absolute constant $C > 0$ for which the following holds. Let $\Gamma, \Gamma' \subseteq \hat{G}$ be sets of characters of sizes $r$ and $r'$ respectively, and let $\rho, \rho' > 0$ be radii. Let $k \in \mathbb{N}$. Suppose that $\langle\Gamma\rangle_{[-K_0, K_0]} \cap \langle\Gamma'\rangle_{[-K_0, K_0]} = \{0\}$ holds for $K_0 = 2^{10}(k+4)^2\mathsf{C}_{\on{spec}} (r + r') (\varepsilon^{2} \rho\rho')^{-2}$.\\
    \indent Let $x \in G$ be an element that can be written in at least $\varepsilon |G|^{k + 3}$ ways as $a_1 + a_2 + a_3 + a_4 + 2b_1 + \dots + 2^k b_k$ for $a_1, \dots, a_4, b_1, \dots b_k \in  B(\Gamma; \rho)$, and analogously for $ B(\Gamma'; \rho')$. Then $x \in 4(B(\Gamma; 2\rho) \cap B(\Gamma'; 2\rho')) + 2\cdot (B(\Gamma; 2\rho) \cap B(\Gamma'; 2\rho')) + \dots + 2^k \cdot (B(\Gamma; 2\rho) \cap B(\Gamma'; 2\rho'))$.
\end{proposition}

\begin{proof}
    Let $A \subseteq G$ be any set. Then, the number of $a_1, \dots, a_4, b_1, \dots b_k \in A$ that satisfy $x = a_1 + a_2 + a_3 + a_4 + 2b_1 + \dots + 2^k b_k$ equals
    \begin{align*}&|G|^{k + 3} \exx_{a_2, a_3, a_4, b_1, \dots, b_k \in G} \id_A(x - a_2 - a_3 - a_4 - 2b_1 - \dots - 2^k b_k) \id_A(a_2) \dots \id_A(b_k)\\
    &\hspace{2cm}= |G|^{k + 3} \exx_{a_2, a_3, a_4, b_1, \dots, b_k \in G} \sum_{\gamma_1, \gamma_2, \gamma_3, \gamma_4, \xi_1, \dots, \xi_k\in \hat{G}} \widehat{\id_A}(\gamma_1)\on{e}(\gamma_1(x - a_2 - a_3 - a_4 - 2b_1 - \dots - 2^k b_k))\\
    &\hspace{4cm}\widehat{\id_A}(\gamma_2)\on{e}(\gamma_2(a_2))\widehat{\id_A}(\gamma_3)\on{e}(\gamma_3(a_3))\widehat{\id_A}(\gamma_4)\on{e}(\gamma_4(a_4))\prod_{i \in [k]} \widehat{\id_A}(\xi_i)\on{e}(\chi_i(b_i))\\
    &\hspace{2cm}= |G|^{k + 3}  \sum_{\gamma_1, \gamma_2, \gamma_3, \gamma_4, \xi_1, \dots, \xi_k\in \hat{G}}  \widehat{\id_A}(\gamma_1)\widehat{\id_A}(\gamma_2)\widehat{\id_A}(\gamma_3)\widehat{\id_A}(\gamma_4)\widehat{\id_A}(\xi_1)\dots\widehat{\id_A}(\xi_k)\\
    &\hspace{4cm}\id(\gamma_1= \gamma_2)\id(\gamma_1 = \gamma_3)\id(\gamma_1 = \gamma_4)\id(\xi_1 = 2\gamma_1)\dots \id(\xi_k = 2^k \gamma_1)\on{e}(\gamma_1(x))\\
    &\hspace{2cm}= |G|^{k + 3} \sum_{\gamma \in \hat{G}} \widehat{\id_A}(\gamma)^4 \widehat{\id_A}(2\gamma)\dots \widehat{\id_A}(2^k\gamma)\on{e}(\gamma(x)).
    \end{align*}

    Thus, we need to show that if $\sum_{\gamma \in \hat{G}} \widehat{\id_A}(\gamma)^4 \widehat{\id_A}(2\gamma)\dots \widehat{\id_A}(2^k\gamma)\on{e}(\gamma(x))$ is at least $\varepsilon$ for $A =  B(\Gamma; \rho),  B(\Gamma'; \rho')$, then it is positive for $A = B(\Gamma; 2\rho) \cap B(\Gamma'; 2\rho')$.\\

    Let $c > 0$ be a parameter to be chosen later, and let $S \subseteq \hat{G}$ be any set containing all characters $\gamma \in \hat{G}$ such that $|\widehat{\id_A}(\gamma)| \geq c$ (and possibly some additional characters). Then
    \begin{align}&\Big|\sum_{\gamma \in \hat{G}} \widehat{\id_A}(\gamma)^4 \widehat{\id_A}(2\gamma)\dots \widehat{\id_A}(2^k\gamma)\on{e}(\gamma(x)) - \sum_{\gamma \in S} \widehat{\id_A}(\gamma)^4 \widehat{\id_A}(2\gamma)\dots \widehat{\id_A}(2^k\gamma)\on{e}(\gamma(x))\Big|\nonumber \\
    &\hspace{2cm} \leq \sum_{\gamma \notin S} |\widehat{\id_A}(\gamma)|^4 \leq c^2 \sum_{\gamma \notin S} |\widehat{\id_A}(\gamma)|^2 \leq c^2.\label{passingtolargeFCsInsumsofgenbohrs}\end{align}

    Let us first find a $\tilde{\rho} \in [\rho, 2\rho]$ and $\tilde{\rho}' \in [\rho', 2\rho']$ such that 
    \begin{align*}
    \Big|B\Big(\Gamma, \tilde{\rho} + \frac{\varepsilon^2}{10000(k+4)}\rho\rho'\Big) \setminus B(\Gamma, \tilde{\rho})\Big| & \leq \frac{\varepsilon^2}{2000(k+4)}|G|,\\
    \Big| B\Big(\Gamma', \tilde{\rho}' + \frac{\varepsilon^2}{10000(k+4)}\rho\rho'\Big) \setminus B(\Gamma, \tilde{\rho}')\Big| & \leq \frac{\varepsilon^2}{2000(k+4)}|G|,\text{ and}\\
    \Big|B\Big(\Gamma, \tilde{\rho} + \frac{\varepsilon^2}{10000(k+4)}\rho\rho'\Big) \cap  B\Big(\Gamma', \tilde{\rho}' + \frac{\varepsilon^2}{10000(k+4)}\rho\rho'\Big) \setminus (B(\Gamma, \tilde{\rho}) \cap B(\Gamma', \tilde{\rho}')) \Big| & \leq \frac{\varepsilon^2}{2000(k+4)}|G|.
    \end{align*}
    We try taking $\tilde{\rho} = \rho + i \frac{\varepsilon^2}{10000}\rho\rho'$ and $\tilde{\rho}' = \rho' + j \frac{\varepsilon^2}{10000}\rho\rho'$ for $i, j \in [0, 10000 \varepsilon^{-2}]$. Each of the inequalities above can fail for at most a fifth of all pairs $(i,j)$, so we can find such a $\tilde{\rho}$ and $\tilde{\rho}'$. Let $B = B(\Gamma, \tilde{\rho})$ and $B' = (\Gamma', \tilde{\rho}')$.\\

    Let $\mathsf{f} = \widehat{\mathsf{b}_{\tilde{\rho}, \frac{\varepsilon^2}{10000(k+4)}\rho\rho'}}, \mathsf{g} = \widehat{\mathsf{b}_{\tilde{\rho}', \frac{\varepsilon^2}{10000(k+4)}\rho\rho'}}$ be the Fourier transforms of bump functions, which are used in Proposition~\ref{bohrsizeLargeFC}. Applying Proposition~\ref{bohrsizeLargeFC} (with $\eta = \frac{\varepsilon^2}{10000(k+4)}\rho\rho'$ in that proposition), we obtain $K \leq 10^9 (k+4)^2 \mathsf{C}_{\on{spec}} (r + r')(\varepsilon^{2} \rho\rho')^{-2}$ such that 
    \begin{align*}\Big|\widehat{\id_B}(\tau) - & \sum_{\lambda \in [-K, K]^\Gamma} \id(\sum_{\gamma \in \Gamma} \lambda(\gamma)\gamma = \tau) \prod_{\gamma \in \Gamma} \mathsf{f}(\lambda(\gamma)) \Big|\leq \frac{\varepsilon^2}{1000(k+4)},\\
    \Big|\widehat{\id_{B'}}(\tau) - & \sum_{\mu \in [-K, K]^{\Gamma'}} \id(\sum_{\gamma' \in \Gamma'} \mu(\gamma')\gamma' = \tau) \prod_{\gamma' \in \Gamma'} \mathsf{g}(\mu(\gamma')) \Big|\leq \frac{\varepsilon^2}{1000(k+4)},\text{ and}\\
    \Big|\widehat{\id_{B \cap B'}}(\tau) - & \sum_{\ssk{\lambda \in [-K, K]^\Gamma\\\mu \in [-K, K]^{\Gamma'}}} \id\Big(\sum_{\gamma \in \Gamma}  \lambda(\gamma)\gamma + \sum_{\gamma' \in \Gamma'} \mu(\gamma')\gamma' = \tau\Big) \prod_{\gamma \in \Gamma} \mathsf{f}(\lambda(\gamma))\prod_{\gamma' \in \Gamma'} \mathsf{g}(\mu(\gamma'))  \Big|\leq \frac{\varepsilon^2}{1000(k+4)}\end{align*}
    hold for all $\tau \in \hat{G}$.\\

    Comparing the choice of $K$ with the assumptions, we have $\langle\Gamma\rangle_{[-10K, 10K]} \cap \langle\Gamma'\rangle_{[-10K, 10K]} = \{0\}$.\\
    
    Note that inequalities above show that $|\widehat{\id_B}(\tau)| \geq \varepsilon/2$ implies $\tau \in \langle \Gamma \rangle_{[-K, K]}$ and $|\widehat{\id_B'}(\tau)| \geq \varepsilon/2$ implies $\tau \in \langle \Gamma' \rangle_{[-K, K]}$ . In particular, by~\eqref{passingtolargeFCsInsumsofgenbohrs}, for $x \in X$, we obtain 
    \begin{align*}\Big|\sum_{\gamma \in \langle \Gamma \rangle_{[-K, K]}} &\widehat{\id_B}(\gamma)^4 \widehat{\id_B}(2\gamma)\dots \widehat{\id_B}(2^k\gamma)\on{e}(\gamma(x))\Big| \geq \varepsilon/2,\\
    \Big|\sum_{\gamma' \in \langle \Gamma' \rangle_{[-K, K]}} &\widehat{\id_{B'}}(\gamma')^4 \widehat{\id_{B'}}(2\gamma')\dots \widehat{\id_{B'}}(2^k\gamma')\on{e}(\gamma'(x))\Big| \geq \varepsilon/2.\end{align*}

    Furthermore, 
    \begin{align*}&\Big|\sum_{\gamma \in \hat{G}} \widehat{\id_{B \cap B'}}(\gamma)^4 \widehat{\id_{B \cap B'}}(2\gamma)\dots \widehat{\id_{B \cap B'}}(2^k\gamma)\on{e}(\gamma(x)) \\
    &\hspace{2cm}- \sum_{\gamma \in \langle \Gamma \cup \Gamma'\rangle_{[-K, K]}} \widehat{\id_{B \cap B'}}(\gamma)^4 \widehat{\id_{B \cap B'}}(2\gamma)\dots \widehat{\id_{B \cap B'}}(2^k\gamma)\on{e}(\gamma(x))\Big| \leq \varepsilon^2/1000(k+4).\end{align*}

    To complete the proof, we relate the Fourier coefficients $\widehat{\id_{B}}$, $\widehat{\id_{B'}}$ and $\widehat{\id_{B \cap B'}}$.

    \begin{claim}
        Let $\xi \in \langle \Gamma, \Gamma' \rangle_{[-K, K]}$. Then there exist unique $\xi_1 \in \langle \Gamma \rangle_{[-K, K]}$ and $\xi_2 \in \langle \Gamma' \rangle_{[-K, K]}$ such that $\xi = \xi_1 + \xi_2$, and these characters satisfy
        \[\Big| \widehat{\id_{B \cap B'}}(\xi) -\widehat{\id_{B}}(\xi_1) \widehat{\id_{B'}}(\xi_2)\Big| \leq \frac{\varepsilon^2}{100(k + 4)}.\]
    \end{claim}

    \begin{proof}
        We use the condition $\langle\Gamma\rangle_{[-2K, 2K]} \cap \langle\Gamma'\rangle_{[-2K, 2K]} = \{0\}$  in the statement to show uniqueness of decomposition $\xi = \xi_1 + \xi_2$. Use the triangle inequality, and note that the approximant for $\widehat{\id_{B \cap B'}}(\xi)$ is product of approximants for $\widehat{\id_{B}}(\xi_1)$ and $\widehat{\id_{B'}}(\xi_2)$.
    \end{proof}

    Using the claim above we have
    \begin{align*}&\sum_{\gamma \in \langle \Gamma \cup \Gamma'\rangle_{[-K, K]}} \widehat{\id_{B \cap B'}}(\gamma)^4 \widehat{\id_{B \cap B'}}(2\gamma)\dots \widehat{\id_{B \cap B'}}(2^k\gamma)\on{e}(\gamma(x))\\
    &\hspace{2cm} =\sum_{\ssk{\gamma_1 \in \langle \Gamma\rangle_{[-K, K]}\\\gamma_2 \in \langle \Gamma'\rangle_{[-K, K]}}} \widehat{\id_{B \cap B'}}(\gamma_1 + \gamma_2)^4 \widehat{\id_{B \cap B'}}(2\gamma_1 + 2\gamma_2)\dots \widehat{\id_{B \cap B'}}(2^k\gamma_1 + 2^k \gamma_2)\on{e}(\gamma_1(x))\,\on{e}(\gamma_2(x)).\end{align*}

    By triangle inequality, we have

    \begin{align*}
        &\Big|\sum_{\ssk{\gamma_1 \in \langle \Gamma\rangle_{[-K, K]}\\\gamma_2 \in \langle \Gamma'\rangle_{[-K, K]}}} \widehat{\id_{B \cap B'}}(\gamma_1 + \gamma_2)^4 \widehat{\id_{B \cap B'}}(2\gamma_1 + 2\gamma_2)\dots \widehat{\id_{B \cap B'}}(2^k\gamma_1 + 2^k \gamma_2)\on{e}(\gamma_1(x))\,\on{e}(\gamma_2(x))\\
        &\hspace{2cm}-\sum_{\ssk{\gamma_1 \in \langle \Gamma\rangle_{[-K, K]}\\\gamma_2 \in \langle \Gamma'\rangle_{[-K, K]}}} \widehat{\id_{B}}(\gamma_1)^4 \widehat{\id_{B'}}(\gamma_2)^4 \widehat{\id_{B}}(2\gamma_1) \widehat{\id_{B'}}(2\gamma_2)\dots \widehat{\id_{B}}(2^k\gamma_1) \widehat{\id_{B'}}(2^k\gamma_2)\on{e}(\gamma_1(x))\,\on{e}(\gamma_2(x))\Big|\\
        &\hspace{1cm}\leq (k+4) \frac{\varepsilon^2}{100(k+4)} = \frac{\varepsilon^2}{100}.
    \end{align*}

    But the final expression is 
    \begin{align*}&\sum_{\ssk{\gamma_1 \in \langle \Gamma\rangle_{[-K, K]}\\\gamma_2 \in \langle \Gamma'\rangle_{[-K, K]}}} \widehat{\id_{B}}(\gamma_1)^4 \widehat{\id_{B'}}(\gamma_2)^4 \widehat{\id_{B}}(2\gamma_1) \widehat{\id_{B'}}(2\gamma_2)\dots \widehat{\id_{B}}(2^k\gamma_1) \widehat{\id_{B'}}(2^k\gamma_2)\on{e}(\gamma_1(x))\,\on{e}(\gamma_2(x)) \\
    &\hspace{2cm} = \Big(\sum_{\gamma_1 \in \langle \Gamma\rangle_{[-K, K]}}\widehat{\id_{B}}(\gamma_1)^4 \widehat{\id_{B}}(2\gamma_1) \dots \widehat{\id_{B}}(2^k\gamma_1)\on{e}(\gamma_1(x)) \Big) \\
    &\hspace{4cm}\Big(\sum_{\gamma_2 \in \langle \Gamma'\rangle_{[-K, K]}} \widehat{\id_{B'}}(\gamma_2)^4  \widehat{\id_{B'}}(2\gamma_2)\widehat{\id_{B'}}(2^k\gamma_2)\on{e}(\gamma_2(x))\Big)\end{align*}

    which has modulus at least $\varepsilon^2/4$. Combining the above inequalities, we obtain 
    \[\Big|\sum_{\gamma \in \hat{G}} \widehat{\id_A}(\gamma)^4 \widehat{\id_A}(2\gamma)\dots \widehat{\id_A}(2^k\gamma)\on{e}(\gamma(x))\Big| \geq \varepsilon^2/8.\qedhere\]

    Finally, let us record another lemma in the similar spirit. Let $\mathsf{C}_{\on{lattice}}$ be the constant from Proposition~\ref{approximatelatticeimage}.

    \begin{lemma}\label{qrsumbohrs}
        Let $\Lambda, \Lambda', \Gamma \subseteq \hat{G}$ be sets of characters of size at most $d$. Let $\rho > 0$. There exists $K = 3\mathsf{C}_{\on{lattice}}d (20\rho^{-1})^{6d + 1}$ such that the following holds. Suppose that $\sum_{\gamma \in \Gamma} a(\gamma) \gamma + \sum_{\lambda \in \Lambda} b(\lambda) \lambda  + \sum_{\lambda' \in \Lambda'} c(\lambda') \lambda' = 0$ for coefficients $a, b, c \in [-K, K]$ implies that $b$ and $c$ vanish. Then
        \[B(\Gamma, \rho/10) \subseteq B(\Gamma \cup \Lambda, \rho) + B(\Gamma \cup \Lambda, \rho').\]
    \end{lemma}

    \begin{proof}
        We rely on Proposition~\ref{approximatelatticeimage}. Let $x \in B(\Gamma, \rho/10)$. List $\Gamma = \{\gamma_1, \dots, \gamma_d\}$, $\Lambda = \{\lambda_1, \dots, \lambda_d\}$ and $\Lambda' = \{\lambda'_1, \dots, \lambda'_d\}$. Let $u = (\gamma_1(x), \dots, \gamma_d(x))$, $v = (\lambda_1(x), \dots, \lambda_d(x))$ and $v' = (\lambda'_1(x), \dots, \lambda'_d(x))$. Thus, $\|u\|_{\mathbb{T}} \leq \rho/10$.\\

        We claim that there exists $y \in G$ such that $\|\gamma(y)\|_{\mathbb{T}}, \|\lambda(y) - v\|_{\mathbb{T}}, \|\lambda'(y)\|_{\mathbb{T}} \leq \rho/2$. Apply Proposition~\ref{approximatelatticeimage} to the vector $(0, v, 0)$ and approximation parameter $\rho/2$. By assumptions  of the lemma, the vector satisfies conditions of Proposition~\ref{approximatelatticeimage}, so there exists desired $y$. Then $y \in B(\Gamma \cup \Lambda', \rho/2)$. On the other hand, $x - y$ satisfies $\|\gamma(x-y)\|_{\mathbb{T}}, \|\lambda(x - y)\|_{\mathbb{T}} \leq \rho$, so $x -y \in B(\Gamma \cup \Lambda', \rho)$.
    \end{proof}

\end{proof}

\subsection{Freiman and $E$-homomorphisms} 

\noindent \textbf{Freiman homomorphisms on Bohr sets.} A considerable part of the proof of our main result will be about systems of Freiman homomorphisms whose domains are Bohr sets.\\

\noindent\textbf{Notation.} Let $\phi_1 : B_1 \to H, \dots, \phi_k : B_k \to H$ be Freiman homomorphisms on Bohr sets $B_1, \dots, B_k \subseteq G$. We interpret expression $\psi = \pm \phi_1 \pm \phi_2 \dots \pm \phi_k$ as a Freiman homomorphism on the intersection of domains $\cap_{i \in [k]} B_i$, which is itself a Bohr set. In other words, we consider the commutative monoid of Freiman homomorphisms $\phi : B \to H$, with operation $(\phi_1, B_1) + (\phi_2, B_2) = ((\phi_1 + \phi_2)|_{B_1 \cap B_2}, B_1 \cap B_2)$.\\

A particularly important role will be played by such maps that take few values, or, almost equivalently, vanish frequently.

Let us define $\range \phi = |\phi(B)|$, which is the size of the image of $\phi$.

\begin{proposition}Let $\phi : B(\Gamma, \rho) \to H$ be a Freiman-linear map on a Bohr set $B = B(\Gamma, \rho)$ of codimension $d$ and radius function $\rho : \Gamma \to \hat{G}$ such that $\phi(x) = 0$ holds for at least $\frac{2}{3}|B|$ elements $x \in B$. Then $\phi = 0$ on $ B(\Gamma, \eta) $ for $\eta = \frac{1}{100d} \prod_{\gamma \in \Gamma} \rho(\gamma)$.
\end{proposition}

\begin{proof}
    By Proposition~\ref{basicbohrsizel}, we have $|B(\Gamma, \rho)| \geq \Big(\prod_{\gamma \in \Gamma} \rho(\gamma) \Big)|G|$. Let $\eta = \frac{1}{100d} \prod_{\gamma \in \Gamma} \rho(\gamma)$. By Proposition~\ref{bohrwreg1}, there exists radius function $\rho' : \Gamma \to \hat{G}$ such that $B(\Gamma, \rho') = B$, $\|\rho' - \rho\|_\infty \leq \eta$ and $|B(\Gamma, \rho') \setminus B(\Gamma, \rho' - \eta)| \leq \frac{1}{10}|B(\Gamma, \rho')|$.\\
    
    Let $Z = \{x \in B : \phi(x) = 0 \}$ and write $B' = B(\Gamma, \eta)$. Let $x \in B'$ be arbitrary. Then $x + B(\Gamma, \rho' - \eta) \subseteq B$ so 
    \[|Z \cap x + Z| \geq |(Z \cap B(\Gamma, \rho' - \eta)) \cap x + (Z \cap B(\Gamma, \rho' - \eta))| \geq |B| - 2|B \setminus (Z \cap B(\Gamma, \rho' - \eta))| > 0,\]
    so there exist $z_1, z_2 \in Z$ such that $x = z_1 - z_2$. By Freiman-linearity, $\phi(x) = \phi(z_1) - \phi(z_2) = 0$.
\end{proof}

Let us now show that, analogously to linear maps in linear algebra, large kernel implies small range.

\begin{lemma}[Many zeroes imply small range]\label{smallrangelemma}
     Suppose that $B = B(\Gamma, \rho)$ is a Bohr set of codimension $d$ and radius function $\rho$. Let $\phi  : B \to H$ be a Freiman-linear map such that $\phi(x) = 0$ holds for at least $c |B|$ elements $x \in B$. Then $\phi$ takes at most $(200d c^{-1})^{d+1} \Big(\prod_{\gamma \in \Gamma} \rho(\gamma)\Big)^{-d}$ different values on $B(\Gamma, \rho/2)$.
\end{lemma}
    
\begin{proof}
    Let $\eta = \frac{c}{200d} \prod_{\gamma \in \Gamma} \rho(\gamma)$. By Proposition~\ref{basicbohrsizel}, we have $\eta \leq \frac{c|B|}{2 |G|}$. By Proposition~\ref{bohrwreg1}, there exists radius function $\rho' : \Gamma \to \hat{G}$ such that $B(\Gamma, \rho') = B$, $\|\rho' - \rho\|_\infty \leq \eta$ and $|B(\Gamma, \rho') \setminus B(\Gamma, \rho' - \eta)| \leq \frac{c}{20}|B(\Gamma, \rho')|$.\\
    
    Let $Z = \{x \in B: \phi(x) = 0\}$. Observe that
    \begin{align*}\sum_{t \in B(\Gamma, \rho' - \eta/2)}& |Z \cap (t + B(\Gamma, \eta/2))| = \sum_{t, z \in G} \id_{B(\Gamma, \rho' - \eta/2)}(t) \id_Z(z) \id_{B(\Gamma, \eta/2)}(z - t)\\
    \geq &\sum_{t, z \in G} \id_{B(\Gamma, \rho' - \eta/2)}(z - s) \id_{Z \cap B(\Gamma, \rho' - \eta)}(z) \id_{B(\Gamma, \eta/2)}(s) \\
    = &\sum_{t, z \in G} \id_{Z \cap B(\Gamma, \rho' - \eta)}(z) \id_{B(\Gamma, \eta/2)}(s) = |Z \cap B(\Gamma, \rho' - \eta)| |B(\Gamma, \eta/2)|.\end{align*}

    By averaging, there exists some $t \in B(\Gamma, \rho' - \eta/2)$ such that 
    \[|Z \cap (t + B(\Gamma, \eta/2))| \geq \frac{|Z \cap B(\Gamma, \rho' - \eta)|}{|B(\Gamma, \rho' - \eta/2)|}|B(\Gamma, \eta/2)|.\]

    Note that $|Z \cap B(\Gamma, \rho' - \eta)| \geq |Z| - |B \setminus B(\Gamma, \rho' - \eta)| \geq |Z| - \frac{c}{20}|B|$, so $\frac{|Z \cap B(\Gamma, \rho' - \eta)|}{|B(\Gamma, \rho' - \eta/2)|} \geq \frac{c}{2}$. Hence $|Z \cap (t + B(\Gamma, \eta/2))| \geq \frac{c}{2} |B(\Gamma, \eta/2)|$.\\
    
    Suppose that $\phi$ takes $m$ values on $B(\Gamma, \rho/2)$, let them be $\phi(x_1), \dots, \phi(x_m)$. Consider the sets $X_i = x_i + (Z \cap (t + B(\Gamma, \eta/2)))$. Each $X_i$ is contained inside $B$ and, due to Freiman-linearity, $\phi$ equals $\phi(x_i)$ on $X_i$. Hence, the sets $X_1, \dots, X_m$ are disjoint and thus $m \frac{c}{2}|B(\Gamma, \eta/2)| \leq |B|$, implying $m \leq 2c^{-1} \eta^{-d}$ and proving the claim. 
\end{proof}

Next, we show that if a Freiman-linear map on a coset progression takes few values then it vanishes on a further coset progression closely related to the domain.

\begin{proposition}[Freiman-linear maps with few values on coset progression]\label{fewvalsCPFLin}
    Let $C = [-N_1, N_1] \cdot v_1 + \dots + [-N_d, N_d] \cdot v_d + K$ be a symmetric proper coset progression of rank $d$ in its canonical form and let $\phi : C \to H$ be a Freiman-linear map. Suppose that there exists some $h \in H$ such that $\phi(x) = h$ for at least $\alpha |C|$ elements $x \in C$. Then there exists a further symmetric proper coset progression $C' = [-M_1, M_1] \cdot \ell_1 v_1 + \dots + [-M_d, M_d] \cdot \ell_d v_d + K' \subseteq C$ of rank $d$, with $M_i \ell_i \leq N_i$, and size at least $2^{-O(d^2)}\alpha^{d + 1}|C|$ such that $\phi = 0$ on $C'$.
\end{proposition}

\begin{proof}
   Let $N'_i = \lfloor \alpha N_i /2d \rfloor$ if $N_i \geq 2\alpha^{-1} d$ and $N'_i = 0$ otherwise. Let $S = [-N'_1, N'_1] \cdot v_1 + \dots + [-N'_d, N'_d] \cdot v_d + K$, which is also a symmetric proper coset progression of rank at most $d$. Coset progressions $C$ and $S$ satisfy the conditions of Lemma~\ref{tilingLemma} so we may find a set of translates $T$ such that $T + S$ almost tiles $C$ with $|T+ S| \geq (1-\alpha/2)|C|$. By averaging, there exists a translate $t +S $ such that $\phi(t + s) = v$ holds for at least $\frac{\alpha}{2}|S|$ elements $s \in S$. As $S$ has rank $d$, we have $\frac{\alpha}{2}|S| \geq \frac{\alpha}{2^{d + 1}}|S - S|$. By fixing single such element $s_0$ and using Freiman-linearity and observing that $S - S \subseteq C$ (as $2N'_i \leq N_i$ for each $i$), we get at least $\frac{\alpha}{2^{d + 1}}|S - S|$ elements of $S-S$ mapping to zero. But the set $Z = \{x \in S-S: \phi(x) = 0\}$ is a Freiman-subgroup of $S - S$, so by Lemma~\ref{progsbgp} we get the desired coset progression.
 \end{proof}

\hspace{\baselineskip}

\noindent \textbf{$E$-homomorphisms.} Unlike the vector space case, where we can easily extend linear maps from subspaces to the whole space, in general abelian groups it is not possible to extend homomorphisms defined on subgroups. However, if we allow a small error, such extensions become possible. This is the motivation for introducing $E$-homomorphisms (recall Definition~\ref{FEhommDefin}).

\begin{proposition}[Naive extensions]\label{generalextensionserror}
    Let $\phi : B(\Gamma, \rho) \to H$ be a Freiman-linear map. Let $d = |\Gamma|$, $\rho' = d^{-2d} \rho$ and let $G' = \langle B(\Gamma, \rho') \rangle$. Let $m$ be a positive integer such that $G' = m B(\Gamma, \rho')$.\\
    \indent Then there exist a set $S \subseteq H$ of size at most $O(d (d \log d + \log m + 1))$ with the following property. Define $\tilde{\phi} : G' \to H$ by $\tilde{\phi}(y) = \phi(x_1) + \dots + \phi(x_m)$ for arbitrary choice of $x_1, \dots, x_m \in B(\Gamma, \rho')$ with $x_1 + \dots + x_m = y$. Then we have
    \begin{equation}
        \tilde{\phi}(x_1 + \dots + x_m) - \phi(x_1) - \dots - \phi(x_m) \in \langle S\rangle_{\{-1, 0, 1\}} \label{tildePhiDefinIngenextn}
    \end{equation}
    for all $x_1, \dots, x_m \in B(\Gamma, \rho')$.
\end{proposition}

\textbf{Remark.} When defining $\tilde{\phi}$, in order to prevent exponential losses in the bounds, we may use many copies of $y_i$, so take $x = z_1 + z_2 + z_3 + z_4 + 2 y_1 + \dots + 2^r y_r$, and define $\tilde{\phi}(x) = \sum_{i \in [4]} \phi(z_i) +  \sum_{i = 1}^m 2^i \phi(y_i)$.\\
\indent To see that we need a few summands, consider $X = \{0,1\}^r \times H$ inside $G = (\mathbb{Z}/2^m\mathbb{Z})^r \times H$. Then $a_1 \cdot X + \dots + a_k \cdot X = G$ requires at least logarithmic number of summands for trivial counting reasons. 

\vspace{\baselineskip}

\begin{proof}
    Due to the way we chose $\rho'$, by Proposition~\ref{cpsinbohrsets}, we may find proper symmetric coset progressions $C, C' = d^{2d} C$ of rank at most $d$ such that $B(\Gamma, d^{-2d}\rho') \subseteq C\subseteq B(\Gamma, \rho') \subseteq C' \subseteq B(\Gamma, \rho)$. Note that inclusion $C\subseteq B(\Gamma, \rho') \subseteq C'$ implies that $C$ also generates $G'$. Note also that Freiman linearity is particularly pleasant property when working with coset progressions; namely, if $e_1, \dots, e_{d'}$ are the generators of progression in $C' = d^{2d}P + H$, then 
    
    \begin{equation}\label{iteratedFreimanLinCP}\phi(\lambda_1 e_1 + \dots + \lambda_{d'} e_{d'} + h) = \lambda_1\phi(e_1) + \dots + \lambda_{d'}\phi(e_{d'}) + \phi(h)\end{equation}
    
    easily follows by induction.\\

    Let $L_i$ be the lengths of $P$. Consider the lattice $\Lambda = \{\lambda \in \mathbb{Z}^{d'} : \lambda_1 e_1 + \dots + \lambda_{d'} e_{d'} \in H\}$, which has full rank. Write $Q \subseteq \mathbb{Z}^{d'}$ for the set $[-L_1, L_1 ] \times \dots \times [-L_{d'}, L_{d'}]$. Notice that $\Lambda \cap Q = \{0\}$ as the coset progression $C$ is proper.\\
    
    \indent Let $\mu_1, \dots, \mu_\ell$ be a maximal dissociated set in $\Lambda \cap (2md^{2d})Q$, i.e. no non-trival $\{-1,0,1\}$-linear combinations of $\mu_i$ vanishes. 

    \begin{claim}
        \begin{itemize}
            \item[\textbf{(i)}] We have $\ell \leq  2d (2d \log d + 10 + \log m)$.
            \item[\textbf{(ii)}] We have $\Lambda \cap (2md^{2d})Q \subseteq \langle \mu_1, \dots, \mu_\ell\rangle_{\{-1,0,1\}}$.
        \end{itemize}
    \end{claim}

    \begin{proof}
        \textbf{Proof of (i).} Writing $\frac{1}{2}Q$ for the box $(-L_1/2, L_1/2] \times \dots \times (-L_{d'}/2, L_{d'}/2]$, we claim that translates $\sum_{i \in I} \mu_i + \frac{1}{2}Q$ are disjoint for all $I \subseteq [\ell]$. Otherwise, we have $\sum_{i \in I} \mu_i - \sum_{i \in J} \mu_j \in Q$ for some $I\not=J$, but $Q \cap \Lambda = \{0\}$ gives a contradiction. These translates are contained in $(2m d^{2d}+ 1)Q$, so we have $2^\ell \leq (4md^{2d} + 2)^{d'}$, so $\ell \leq 2d (2d \log d + 10 + \log m)$.\\
        \textbf{Proof of (ii).} This follows from the fact that $\mu_1, \dots, \mu_\ell$ is a maximal dissociated set inside $\Lambda \cap (2m)Q$.
    \end{proof}
    
    For each $i \in [\ell]$, define $\sigma_i = \sum_{j \in [d']} \mu_{ij}\,\phi(e_j)$ and $\sigma'_i = \phi\Big(\sum_{j \in [d']} \mu_{ij}\,e_j\Big)$. Note that the argument in the definition of $\sigma'_i$, $\sum_{j \in [d']} \mu_{ij}\,e_j$, belongs to $H$. Finally, set $S = \{\sigma_1, \dots, \sigma_\ell, \sigma'_1, \dots, \sigma'_\ell\}$.\\
    
    Recall that $C = P + H$, for $P = [-L_1, L_1]e_1 + \dots + [-L_{d'}, L_{d'}]e_{d'}$. Let $x_1, \dots, x_{2m} \in B(\Gamma, \rho')$ satisfy $\sum_{i \in [2m]} x_i = 0$. Recall that also that $B(\Gamma, \rho') \subseteq C' = d^{2d}C$, so we may write $x_i = \sum_{j \in [d']} \lambda_{ij} e_j + h_i$, where $|\lambda_{ij}| \leq d^{2d} L_j$ and $h_i \in H$. Hence,
    \begin{equation}\sum_{j \in [d']} \Big(\sum_{i \in [2m]} \lambda_{ij}\Big) e_j = -\sum_{i \in [2m]} h_i \in H.\label{basiselcinH}\end{equation}
    
    Due to bounds on $\lambda_{ij}$, this gives us linear combinations of $e_j$ inside $H$, with coefficient of $e_j$ being at most $2m d^{2d} L_j$ in absolute value. Thus, each linear combination in~\eqref{basiselcinH} gives a point in $\Lambda \cap (2m d^{2d})Q$. By the way we constructed $\mu_1, \dots, \mu_\ell$, there are elements $\xi_1, \dots, \xi_\ell \in \{-1, 0, 1\}$ such that, for all $j \in [d']$, $\sum_{i \in [2m]} \lambda_{ij} = \sum_{i \in [\ell]} \xi_i \mu_{i j}$. Using Freiman-linearity on $C'$~\eqref{iteratedFreimanLinCP} a few times we get
    \begin{align*}\sum_{i \in [2m]} \phi(x_i) = &\sum_{i \in [2m]} \phi\Big(\sum_{j \in [d']} \lambda_{ij}e_j + h_i\Big) = \bigg(\sum_{i \in [2m]} \Big(\sum_{j \in [d']} \lambda_{ij} \phi(e_j)\Big)\bigg) + \phi\Big(\sum_{i \in [2m]} h_i\Big) \\
    = &\sum_{j \in [d']} \Big(\sum_{i \in [2m]}  \lambda_{ij} \phi(e_j)\Big)+ \phi\Big(\sum_{i \in [2m]} h_i\Big) .\end{align*}
    Using the relationship between $\lambda$ and $\mu$, we get
    \[\sum_{j \in [d']} \Big(\sum_{i \in [2m]}  \lambda_{ij} \phi(e_j)\Big) = \sum_{j \in [d']} \Big(\sum_{i \in [\ell]}  \xi_i \mu_{ij} \phi(e_j)\Big) =  \sum_{i \in [\ell]} \xi_i  \Big(  \sum_{j \in [d']}\mu_{ij} \phi(e_j)\Big) = \sum_{i \in [\ell]} \xi_i  \sigma_i.\]

    Similarly, 
    \[\sum_{i \in [2m]} h_i = -\sum_{j \in [d']} \Big(\sum_{i \in [2m]} \lambda_{ij}\Big) e_j = -\sum_{j \in [d]'} \sum_{i \in [\ell]} \xi_i \mu_{i j} e_j = - \sum_{i \in [\ell]}\xi_i  \Big(\sum_{j \in [d']} \mu_{ij} e_j\Big)\]
    and each $\sum_{j \in [d']} \mu_{ij} e_j$ belongs to $H$. Using Freiman-linearity on $H$, we get
    \[\phi\Big(\sum_{i \in [2m]} h_i \Big) = - \sum_{i \in [\ell]}\xi_i \phi\Big(\sum_{j \in [d']} \mu_{ij} e_j\Big) = - \sum_{j \in [d']} \xi_i \sigma'_i.\]
    Therefore, 
    \begin{equation}\sum_{i \in [2m]} \phi(x_i) = \sum_{i \in [\ell]} \xi_i (\sigma_i - \sigma'_i) \in \langle S \rangle_{\{-1, 0, 1\}}.\label{generalextensionscocycleapprox}\end{equation}\

    Take any $\tilde{\phi}$ as described in the proposition. Let $y \in G'$ be arbitrary. Thus, there are $y_1, \dots, y_m \in B(\Gamma, \rho')$ such that $y_1 + \dots + y_m = y$ for which $\tilde{\phi}(y) = \sum_{i \in [m]} \phi(y_i)$.\\
    \indent It remains to check property~\eqref{tildePhiDefinIngenextn}. Take any $x_1, \dots, x_m \in B(\Gamma, \rho')$. The value $\tilde{\phi}(x_1 + \dots + x_m)$ is defined as  $\sum_{i \in [m]} \phi(y_i)$ for some $y_1, \dots, y_m \in B(\Gamma, \rho')$ such that $y_1 + \dots + y_m = x_1 + \dots + x_m$. Hence
    \[\tilde{\phi}(x_1 + \dots + x_m) - \sum_{i \in [m]} \phi(x_i) = \sum_{i \in [m]} \phi(y_i) + \sum_{i \in [m]} \phi(-x_i).\]
    But $y_1 + \dots + y_m + (-x_1) + \dots (-x_m) = 0$, so by~\eqref{generalextensionscocycleapprox}  $\sum_{i \in [m]} \phi(y_i) + \sum_{i \in [m]} \phi(-x_i) \in \langle S \rangle_{\{-1, 0, 1\}}$.
\end{proof}

\vspace{\baselineskip}

We also need a result about extensions of $E$-homomorphisms from subgroups to the whole group.

\begin{proposition}\label{ehomextensionsfromsubgroups}
    Let $H = \ker \chi \leq G$ be a subgroup of index $M$ given by the kernel of a homomorphism $\chi \in \hat{G}^r$. Let $B(\Gamma, \rho)$ be a Bohr set of codimension $d$ such that $\langle\chi_1, \dots, \chi_r \rangle_{[-M, M]} \cap \langle \Gamma \rangle_R = \{0\}$, for $R = \mathsf{C}_{\on{spec}} d 2^{10d + 8} \rho^{-2d-1} M^{4d + 2r}$.\\
    \indent Let $\phi : H \cap B(\Gamma, \rho) \to K$ be an $E$-homomorphism. Then $\phi$ extends to an $E'$-homomorphism on $B(\Gamma, \rho/16)$ where
    \[E' = m (E + \{0, \phi(0)\}) - m (E + \{0, \phi(0)\}) + \langle k_1, \dots, k_m \rangle_{\{-1, 0,1\}},\] for $m \leq O(\log M)$ and some elements $k_1, \dots, k_m \in K$.
\end{proposition}

\begin{proof} We may assume that $r \leq \log_2 M$. Write $\ell_i$ for the order of $\chi_i$, which is at most $M$. We first show that $B(\Gamma, \rho / 16M^2)$ intersects every coset of the subgroup $H$. The proof will proceed by Fourier analysis and Proposition~\ref{bohrsizeLargeFC} will be applied, so we need to ensure that the Bohr set is weakly regular. To that end, let $\varepsilon = \frac{\rho^d }{ 2^{5d +4}M^{2d + r}}, \eta = \rho \varepsilon / 4$. Lemma~\ref{basicbohrsizel} implies that $|B(\Gamma, \rho / 16M^2)| \geq \frac{\rho^d}{2^{4d}M^{2d}} |G|$. By the pigeonhole principle, there exists $\rho' \in [\rho / 32M^2, \rho / 16M^2]$ such that $|B(\Gamma, \rho' + \eta) \setminus B(\Gamma, \rho')| \leq \varepsilon |G| / 2$. Note that $R \geq \mathsf{C}_{\on{spec}} d \eta^{-1} \varepsilon^{-1}$. Applying Proposition~\ref{bohrsizeLargeFC} we deduce that the inequality $|\widehat{\id_{B(\Gamma, \rho')}}(\tau)| > \varepsilon$ implies $\tau \in \langle \Gamma \rangle_{[-R, R]}$. Let us now prove that $B(\Gamma, \rho')$ (and thus $B(\Gamma, \rho / 16M^2)$) meets every coset of $H$.\\

Suppose, on the contrary, that $t + H \cap B(\Gamma, \rho') = \emptyset$. Then
\begin{align*}0 = &\exx_{x \in G} \id_{B(\Gamma, \rho')}(x) \id_H(x-t) =  \exx_{x \in G} \id_{B(\Gamma, \rho')}(x) \prod_{i \in [r]} \id(\chi_i(x - t) = 0) \\
= &\exx_{x \in G} \id_{B(\Gamma, \rho')}(x) \prod_{i \in [r]} \Big(\ell_i^{-1} \sum_{\lambda_i \in [0, \ell_i-1]} \on{e}\Big(-\lambda_i \chi_i(x -t)\Big)\Big)\\
= & \frac{1}{\ell_1\cdots\ell_r}\sum_{\lambda \in [0, \ell_1 - 1] \tdt [0, \ell_r - 1]} \exx_{x \in G} \id_{B(\Gamma, \rho')}(x) \on{e}\Big(-\sum_{i \in [r]}\lambda_i \chi_i(x-t)\Big)\\
= & \frac{1}{\ell_1\cdots\ell_r}\sum_{\lambda \in [0, \ell_1 - 1] \tdt [0, \ell_r - 1]} \widehat{\id_{B(\Gamma, \rho')}}\Big(\sum_{i \in [r]}\lambda_i \chi_i\Big)\on{e}\Big(\sum_{i \in [r]}\lambda_i \chi_i(t)\Big).
\end{align*}

The contribution to the sum above from $\lambda$ such that $\lambda \cdot \chi = 0$ is $\frac{1}{\ell_1 \dots \ell_r}\frac{|B(\Gamma, \rho')|}{|G|} \geq 2^{-5d} M^{-2d-r} \rho^d$. By averaging over the remaining $\lambda$, there exists $\lambda$ such that $\lambda \cdot \chi \not= 0$ with $\Big|\widehat{\id_{B(\Gamma, \rho')}}\Big(\sum_{i \in [r]}\lambda_i \chi_i\Big)\Big| \geq 2^{-5d} M^{-2d-r} \rho^d$. The property above implies that $\lambda_1\chi_1 + \dots + \lambda_r \chi_r \in \langle \Gamma \rangle_{[-R, R]}$ for some $\lambda_1, \dots, \lambda_r$. But this sum belongs to $\langle\chi_1, \dots, \chi_r \rangle_{[-M, M]} \cap \langle \Gamma \rangle_{[-R, R]}$ which consists only of the zero. Hence $\lambda_1\chi_1 + \dots + \lambda_r \chi_r  = 0$, which is a contradiction. Thus, $B(\Gamma, \rho / 16M^2)$ meets every coset of $H$, as claimed.\\

By the classification of finite abelian groups, we have $G/H \cong C_{d_1} \oplus \dots \oplus C_{d_m}$ for some direct sum of cyclic groups. Note that $m \leq \log_2 M$. Take generators $x_i + H$ of the cyclic groups in the direct sum. Note that we may modify $x_i$ by adding any element of $H$, so we may in particular replace $x_i$ by an element of $x_i + H \cap B(\Gamma, \rho / 16M^2)$, which we know is non-empty, and thus, without loss of generality, guarantee that $x_i \in B(\Gamma, \rho / 16M^2)$. We have thus found elements $x_1, \dots, x_m \in B(\Gamma, \rho / 16M^2)$ such that $\lambda_1 x_1 + \dots + \lambda_m x_m + H$ cover all cosets of $H$ precisely once as $\lambda_i$ ranges over $[0, d_i - 1]$ and $d_i x_i \in H$ for all $i$.\\

In particular, for each $g \in G$, there exist unique coefficients $\lambda_i(g) \in [0, d_i - 1]$ such that $g - \sum_{i \in [m]} \lambda_i(g) x_i \in H$. Denote this difference by $\pi(g)$. Write $\tilde{h}_i = d_i x_i$ and note that $\tilde{h}_i \in B(\Gamma, \rho / 16M)$. While the map $\pi : G \to H$ is not necessarily a homomorphism, it turns out to be nearly a homomorphism, as we shall now show. Note also that $\sum_{i \in [m]} \lambda_i(g) x_i \in B(\Gamma, \rho / 16)$, so if $g \in B(\Gamma, \rho / 16)$, then $\pi(g) \in B(\Gamma, \rho/8)$.

    \begin{claim*}
        Let $g, g' \in G$. Then $\pi(g + g') \in \pi(g) + \pi(g') + \langle \tilde{h}_1, \dots, \tilde{h}_m\rangle_{\{0,1\}}$.
    \end{claim*}
    
    \begin{proof}
        Let $g, g' \in G$. Then $g + g'$ belongs to the coset $(g + H) + (g' + H)$, which is the same as $\sum_{i \in [m]} \Big(\lambda_i(g) + \lambda_i(g')\Big) x_i + H$. This means that $\lambda_i(g + g')$ is the unique integer in $[0, d_i - 1]$ congruent to $\lambda_i(g) + \lambda_i(g')$ modulo $d_i$. Hence, $\lambda_i(g + g') = \lambda_i(g) + \lambda_i(g') - \varepsilon_i d_i$, for some $\varepsilon_i \in \{0,1\}$. Thus,
        \[\pi(g + g') = g + g' -  \sum_{i \in [m]} \lambda_i(g + g') x_i = g + g' - \sum_{i \in [m]} (\lambda_i(g) + \lambda_i(g') - \varepsilon_i d_i) x_i = \pi(g) + \pi(g') + \sum_{i \in [m]} \varepsilon_i \tilde{h}_i.\qedhere\]
    \end{proof}

    We also record a partial linearity for a higher number of terms in $\phi$.

    \begin{claim*}
        For any $g_1, \dots, g_\ell \in B(\Gamma, \rho) \cap H$ be such that all sums of the first $i \in [\ell]$ elements also belong to $B(\Gamma, \rho)$. Then we have
        \[\phi(g_1 + \dots + g_\ell) \in \sum_{i \in [\ell]} \phi(g_i) - (\ell - 1) \phi(0) + (\ell - 1) E.\]
    \end{claim*}
 
    \begin{proof}
        We prove the claim by induction on $\ell$. The base case is trivial. Suppose the claim holds for $\ell$ and let $g_1, \dots, g_{\ell + 1} \in G$ be given. Note also that all sums of $g_i$ belong to $H$. Then
        \[\phi(g_1 + \dots + g_{\ell + 1}) - \phi(g_1 + \dots + g_{\ell}) - \phi(g_{\ell + 1}) + \phi(0) \in E,\]
        so by inductive hypothesis
        \[\phi(g_1 + \dots + g_{\ell + 1}) \in \phi(g_1 + \dots + g_{\ell}) + \phi(g_{\ell + 1}) - \phi(0) + E \subseteq \sum_{i \in [\ell+1]} \phi(g_i) - \ell \phi(0) + \ell E.\qedhere\]
    \end{proof}

    Consider $\tilde{\phi} : B(\Gamma, \rho / 16) \to K$ given by $\tilde{\phi} = \phi \circ \pi$. This is a well-defined map as $\pi(x) \in H$ for all $x \in G$ and $\pi(x) \in B(\Gamma, \rho / 8)$ whenever $x \in B(\Gamma, \rho/16)$. Let $x, a, b \in G$ be elements such that $x + a + b, x + a, x + b, x \in B(\Gamma, \rho / 2)$. By the claim on $\pi$ above, we have $\varepsilon_i, \varepsilon'_i \in \{0,1\}$ such that
    \[\pi(x  +a + b) -\pi(x+a) - \pi(b) = \sum_{i \in [m]} \varepsilon_i \tilde{h}_i,\hspace{1cm} \pi(x + b) -\pi(x) - \pi(b) = \sum_{i \in [m]} \varepsilon'_i \tilde{h}_i.\]

    Let $I = \{i \in [m] : \varepsilon_i = 1\}$ and $I' = \{i \in [m] : \varepsilon'_i = 1\}$. Note that, since $\pi(x + a) , \pi(b) \in B(\Gamma, \rho / 8)$ and $\tilde{h}_i \in B(\Gamma, \rho / 16M)$, all partial sums in $\pi(x + a) + \pi(b) + \sum_{i \in [m]} \varepsilon_i \tilde{h}_i$ belong to $B(\Gamma, \rho)$ and the second claim above implies
    \begin{align*}\phi(\pi(x  +a + b)) - \phi(\pi(x  + a)) = &\phi\Big(\pi(x + a) + \pi(b) + \sum_{i \in [m]} \varepsilon_i \tilde{h}_i\Big) - \phi(\pi(x  + a)) \\
    \in &\phi(\pi(b)) + \sum_{i \in I} \phi(\tilde{h}_i) - (|I| + 1)\phi(0) + (|I|  + 1) E.\end{align*}
    Note that $0 \in E$ (due to $\partial_{0,0}\phi(0) = 0$), so we may replace $(|I|  + 1)\phi(0) + (|I|  + 1) E$ by $(m+1) (E + \{0, \phi(0)\})$.
    
    Then 
    \begin{align*}\Delta_{a,b} \tilde{\phi}(x) = &\Big(\phi(\pi(x  +a + b)) - \phi(\pi(x  + a)) - \phi(\pi(b))\Big) - \Big(\phi(\pi(x + b)) - \phi(\pi(x)) - \phi(\pi(b))\Big)\\
    \in&  \sum_{i \in I} \phi(\tilde{h}_i) -  \sum_{i \in I'} \phi(h_i) + (m+1)(E + \{0, \phi(0)\}) - (m+1)(E + \{0, \phi(0)\}).\end{align*} 

    Misuse the notation and replace $m + 1$ by $m$, which does not affect the bounds, and write $k_i = \phi(\tilde{h}_i)$.
\end{proof}

\section{Algebraic regularity method in general abelian groups}\label{algreggensection}

The (bilinear) algebraic regularity method originates in~\cite{U4paper}, in which quantitative bounds for the $\mathsf{U}^4(\mathbb{F}_p^n)$ norm were first obtained. In additive combinatorics, in the bilinear setting, it is natural to consider graphs that are defined algebraically. For example, in the case of finite vector spaces, we may consider a bipartite graph on vertex classes being copies of $\mathbb{F}_p^n$, with $xy$ an edge if $\beta(x,y) = 0$ for a bilinear map $\beta: \mathbb{F}_p^n \times \mathbb{F}_p^n \to \mathbb{F}_p^r$. It turns out that one can prove an algebraic regularity lemma for such graphs, partitioning its vertex classes into a bounded number of pieces, each being a coset of some fixed subspace of low codimension, pairs of which induce quasirandom graphs. Such a lemma replaces Szemer\'edi's regularity lemma~\cite{SzemAP, SzemReg} in this setting, and can be used analogously to the combinatorial regularity method. Crucially, due to algebraic nature of the graph, we still get good bounds for the regularity lemma as well as a few other additional properties. A version of the algebraic regularity lemma for general abelian groups was obtained in~\cite{generalBilBog}.\\

In order to make use of the algebraic regularity lemma, we need the theory of quasirandom bipartite graphs, which we revisit briefly in the following subsection. In the subsection after that, we revisit the algebraic regularity lemma (Theorem~\ref{algreglemmaintro}).\\ 

\subsection{Quasirandomness of bipartite graphs}

The formulations of the main concepts and relevant results are taken from~\cite{generalBilBog}.\\
\indent Throughout this subsection, $G$ denotes a bipartite graph on vertex classes $X$ and $Y$. We also view $G$ simultaneously as a $\{0,1\}$-valued function on the product $X \times Y$. Let $\|\cdot\|_{\square} = \|\cdot\|_{\square(X, Y)}$ stand for the box norm, which is defined by 
\[\|f\|_{\square(X, Y)} = \Big(\exx_{x_0, x_1 \in X} \exx_{y_0, y_1 \in Y} f(x_0, y_0)\,\overline{f(x_1, y_0)}\,\overline{f(x_0, y_1)}\,f(x_1, y_1)\Big)^{1/4}.\]

\begin{lemma}[Lemma 41 in~\cite{generalBilBog}]\label{basicgcs}Let $f \colon X \times Y \to \mathbb{C}$, $u \colon X \to \mathbb{C}$ and $v \colon Y \to \mathbb{C}$ be functions. Then
\[\Big|\exx_{x \in X, y \in Y} f(x,y)u(x)v(y) \Big| \leq \|f\|_\square \|u\|_{L^2} \|v\|_{L^2}.\]
\end{lemma}

Given a bipartite graph $G$ on vertex classes $X$ and $Y$, its density $\delta$ is given by $\ex_{x \in X, y \in Y} G(x,y)$. We say that $G$ is $\varepsilon$-\emph{quasirandom} if $\|G - \delta\|_\square \leq \varepsilon$. Quasirandom graphs behave like randomly chosen graphs of the given density.

\begin{lemma}[Lemma 42 in~\cite{generalBilBog}]\label{appendknhoods}Let $k$ be a positive integer and let $G$ be an $\varepsilon$-quasirandom bipartite graph of density $\delta$ on vertex classes $X$ and $Y$. Pick $x_1, \dots, x_k \in X$ uniformly and independently at random. Let $\eta > 0$ be a positive real. Then 
\[\mathbb{P}\Big(\Big||N_{x_1} \cap \dots \cap N_{x_k}| - \delta^k |Y|\Big| \geq \eta |Y|\Big) \leq 4k \eta^{-2} \varepsilon.\]
\end{lemma}

\begin{lemma}[Lemma 43 in~\cite{generalBilBog}]\label{genappendknhoods}Let $k$ and $m$ be positive integers and let $G$ be an $\varepsilon$-quasirandom bipartite graph of density $\delta$ on vertex classes $X$ and $Y$. Let $M \subseteq Y^m$ be a set of $m$-tuples in $Y$. Pick $x_1, \dots, x_k \in X$ uniformly and independently at random. Let $\eta > 0$ be a positive real. Then 
\[\mathbb{P}\Big(\Big||N^m_{x_1} \cap \dots \cap N^m_{x_k} \cap M| - \delta^{mk} |M|\Big| \geq \eta |Y|^m\Big) \leq  4km \eta^{-2}\varepsilon.\]\end{lemma}

\begin{lemma}[Lemma 44 in~\cite{generalBilBog}]\label{appendonesided}Let $\delta,\varepsilon \in [0,1]$. Suppose that $G$ is a bipartite graph with vertex classes $X$ and $Y$ such that
\begin{equation}\label{qrCond1}\exx_{x \in X} \Big||N_x| - \delta |Y|\Big| \leq \varepsilon |Y|\end{equation}
and 
\begin{equation}\exx_{x, x' \in X} \Big||N_x \cap N_{x'}| - \delta^2 |Y|\Big| \leq \varepsilon |Y|.\label{qrCond2}\end{equation}
Then the density $\delta'$ of $G$ satisfies $|\delta - \delta'| \leq \varepsilon$ and the graph $G$ is $3 \sqrt[8]{\varepsilon}$-quasirandom.\end{lemma}

\subsection{General algebraic regularity lemma}

We now prove algebraic regularity lemma for general abelian groups. The following result appears in~\cite{generalBilBog}, but is slightly imprecise. Namely, in the statement of Theorem~\ref{algreglemmaintro} below, in the original version in~\cite{generalBilBog}, the Bohr sets $B(\chi_1, \dots, \chi_r; \rho')$ did not depend on the translate $t$ of the coset progression $C'$. Here we make a correction, although the main result of that paper, the bilinear Bogolyubov argument in general finite abelian groups, remains unchanged.\footnote{In fact, in that paper, we may use the algebraic regularity lemma for $t= 0$, giving a symmetric coset progression $C'$ and then the proof of bilinear Bogolyubov argument remains unchanged.} Additionally, in contrast to the previous version, in the present paper, we need Bohr sets with radius functions, instead of just a single radius. Using radius function will be critical in applications where we want simultaneous quasirandomness of two bilinear Bohr varieties, about which we shall say more in Section~\ref{passingtobihomonprodsection}.

\begin{theorem}[Algebraic regularity lemma for bilinear Bohr varieties]\label{algreglemmaintro}Let $G$ and $H$ be finite abelian groups. Let $C$ be a proper coset progression of rank $d$ inside the group $H$, and let $L_1, \dots, L_r \colon C \to \hat{G}$ be Freiman homomorphisms. Let $\rho : [r] \to (0,1)$ be a radius function and $\eta > 0$ be given.\\
\indent Then, we may find a further proper coset progression $C'$ of rank at most $d$ and a set $T$, which is a progression of rank $d$, of size at most
\[\exp\Big(d^{O(1)} r^{O(1)} \log^{O(1)} (\eta^{-1}) \log^{O(1)} (\rho(1)^{-1}\cdots\rho(r)^{-1}) \Big),\]
such that $C' + T \subseteq C$, $|C'| |T| = |C' + T| \geq (1 - \eta) |C|$ and every $t + C'$ induces a quasirandom piece of the bilinear Bohr variety in the following sense.\\
\indent Given $t \in T$, taking characters $\chi_1 = L_1(t), \dots, \chi_r = L_r(t)$, there exist reals $\delta > 0$ and $\rho' : [r] \to (0,1)$ with $\rho'(i) \in [\rho(i)/4, \rho(i)/2]$ for all $i \in [r]$, such that
\begin{itemize}
\item[\textbf{(i)}] for at least $1 - \eta$ proportion of all elements $x \in t + C'$ we have 
\[\Big||B(\chi_1, \dots, \chi_r, \rho') \cap B(L_1(x)-\chi_1, \dots, L_r(x)-\chi_r; \rho')| - \delta |B(\chi_1, \dots, \chi_r; \rho')|\Big| \leq \eta |G|,\]
\item[\textbf{(ii)}] for at least $1 - \eta$ proportion of all pairs $(x, x') \in ( t + C') \times (t + C')$ we have
\begin{align*}\Big||B(\chi_1, \dots, \chi_r, \rho') \cap B(L_1(x)-&\chi_1, \dots, L_r(x)-\chi_r, \rho') \cap B(L_1(x')-\chi_1, \dots, L_r(x')-\chi_r; \rho')| \\
&- \delta^2 |B(\chi_1, \dots, \chi_r; \rho')|\Big| \leq \eta |G|.\end{align*}
\end{itemize}
Moreover, if $C = a + [0, N_1 -1 ] \cdot v_1 + \dots + [0, N_d - 1] \cdot v_d + K$ is a canonical form of $C$, then we may take $C'$ of the form $[-N'_1, N'_1] \cdot \ell_1 v_1 + \dots + [-N'_d, N'_d] \cdot \ell_d v_d + K'$, $\ell_i N'_i \leq N_i$ and $K' \leq K$, and if $C$ is additionally symmetric, then we may assume that $C'$ is symmetric and $0 \in T$.
\end{theorem}

The proof is very similar to the one proposed in~\cite{generalBilBog}, but with corrections added.

\begin{proof} Let $C = a + [0, N_1 -1 ] \cdot v_1 + \dots + [0, N_d - 1] \cdot v_d + H_0$ be a canonical form of $C$. Let us first linearize maps $L_i$. Define $\tilde{a} = a + \sum_{i \in [d]} \lfloor N_i /2 \rfloor v_i$ and maps $L_i' : C - \tilde{a} \to \hat{G}$ by $L'_i(x) = L_i(x + \tilde{a}) - L_i(\tilde{a})$, which are Freiman-linear.\\\\

The proof will be iterative and in each step we shall keep track of a symmetric coset progression $S$ and a lattice $\Lambda$. Here are their key properties.

\begin{itemize}
    \item The coset progression $S$ will be of the form $[-N'_1, N'_1] \cdot \ell_1 v_1 + \dots + [-N'_d, N'_d] \cdot \ell_d v_d + H'$, $\ell_i N'_i \leq N_i$ and $H' \leq H_0$, and will be a candidate for the choice of $C'$.
    \item Every element $\lambda$ of the lattice $\Lambda$ will satisfy $\lambda \cdot L'(x) = 0$ for all $x \in S$, though $\Lambda$ might not contain all such linear combinations.
\end{itemize}

In order to be able to almost tile $C$ with $S$ eventually, we need to be slightly careful about the lengths of the progressions. Initially, we define $N'_i = \lfloor \eta N_i /2d \rfloor$ when $N_i \geq 2\eta^{-1}d$ and $N'_i = 0$ otherwise. We also set $S = a + [-N'_1, N'_1] \cdot v_1 + \dots + [-N'_d, N'_d] \cdot v_d + H_0$ and $\Lambda = \{0\}$.\\
\indent During the proof, the coset progression will only be replaced by its subprogressions. In particular, at each step of the proof, $S$ will satisfy the conditions of Lemma~\ref{tilingLemma} and we will be able to almost tile $C$ by $S$.\\
\indent Furthermore, we shall enlarge the lattice $\Lambda$ at each step of the iteration. The latter property will guarantee that the procedure terminates quickly.\\

Suppose that we have completed a step of the procedure and let us consider the described objects. Take an arbitrary $t \in C$ such that $t + S \subseteq C$. Let $\chi_1 = L_1(t), \dots, \chi_r = L_r(t)$ and write $\Gamma = \{\chi_1, \dots, \chi_r\}$. Let $M = \{\lambda \in \mathbb{Z}^r : \lambda \cdot \chi = 0\} \subseteq \mathbb{Z}^r$ be the annihilator lattice of $\chi_1, \dots, \chi_r$.\\

\noindent\textbf{Ensuring weak regularity.} Let $\mu = \min_{j \in [r]} \rho(j)$. We first find a suitable radius function $\rho' \in [\rho/4, \rho/2]$. Consider $\rho/2 - \frac{j \eta^2 \mu}{2000}$ for $j \in [500 \eta^{-2}]$ as candidates for $\rho'$. By the pigeonhole principle, there is such a value of $\rho'$ such that for at least $1 - \frac{\eta}{10}$ proportion of $x \in S$ we have
\begin{align}\Big|\Big(B\Big(\Gamma, \rho' + \frac{\eta^2 \mu}{2000}\Big) &\cap B\Big(L'_1(x), \dots, L'_r(x); \rho'+ \frac{\eta^2 \mu}{2000}\Big)\Big)\nonumber\\
&\setminus \Big(B\Big(\Gamma; \rho'\Big) 
\cap B\Big(L'_1(x), \dots, L'_r(x); \rho'\Big)\Big)\Big| \leq \frac{\eta}{10} |G|,\label{wreg1piece}\end{align}
for at least $1-\frac{\eta}{10}$ proportion of the pairs $(x,x') \in S \times S$ we have 
\begin{align}&\Big|\Big(B\Big(\Gamma, \rho' + \frac{\eta^2 \mu}{2000}\Big) \cap B\Big(L'_1(x), \dots, L'_r(x); \rho'+ \frac{\eta^2 \mu}{2000}\Big)\cap B\Big(L'_1(x'), \dots, L'_r(x'); \rho'+ \frac{\eta^2 \mu}{2000}\Big)\Big)\nonumber\\
&\hspace{2cm}\setminus \Big(B\Big(\Gamma; \rho'\Big) \cap B\Big(L'_1(x), \dots, L'_r(x); \rho'\Big)\cap B\Big(L'_1(x'), \dots, L'_r(x'); \rho'\Big)\Big)\Big| \leq \frac{\eta}{10} |G|\label{wreg2piece}\end{align}  
and 
\begin{equation}\Big|B\Big(\Gamma; \rho' + \frac{\eta^2 \mu}{2000}\Big) \setminus B(\Gamma; \rho')\Big| \leq \frac{\eta}{10} |G|.\label{wreg3piece}\end{equation}

\phantom{a}\hspace{\baselineskip}

\noindent\textbf{Finding new vanishing linear combination.} The next claim shows that $t + S$ gives rise to a quasirandom piece unless we obtain new vanishing linear combinations of characters given by values of maps $L'_1, \dots, L'_r$ on $S$.\\

Let $K = O(r\eta^{-3}\mu^{-1})$ be the quantity provided by Proposition~\ref{bohrsizeLargeFC} such that whenever $\gamma_1, \dots, \gamma_\ell \in \hat{G}$ are characters for $\ell \leq 3r$ and $\sigma : [\ell] \to (0,1)$ is a radius function with the weak regularity property
\begin{equation}\Big||B(\gamma_1, \dots, \gamma_\ell; \sigma + \eta^2 \mu/2000)| - |B(\gamma_1, \dots, \gamma_\ell; \sigma)|\Big| \leq \frac{\eta}{10} |G|\label{wregconditionQR}\end{equation}
then 
\begin{equation}\label{QRbohrconclusion}\Big||B(\gamma_1, \dots, \gamma_\ell; \sigma)| - \sum_{a_1, \dots, a_\ell \in [-K, K]} \id(a_1 \gamma_1 + \dots + a_\ell \gamma_\ell = 0) c_{1, a_1} \dots c_{\ell, a_\ell} |G|\Big| \leq \frac{\eta}{5}|G|,\end{equation}
where we write $c_{i, a}$ for the value $\widehat{\mathsf{b}_{\sigma(i), \eta^2 \mu/2000}}(a)$ of the Fourier coefficient at $a$ of bump functions defined in~\eqref{bumpdefinition}.\\

\indent We remark that $K$ is the same in all stages of the argument, in particular it does not depend on $S$, nor its size. This fact will be crucial in showing that the procedure terminates reasonably quickly.\\

\indent Furthermore, we shall our choices for $\sigma$ will be $\rho'$, or concatenations $(\rho', \rho')$ and $(\rho', \rho', \rho')$, so in fact $c_{i, a}$ becomes $\widehat{\mathsf{b}_{\rho'(i'), \eta^2 \mu/2000}}(a)$, where $i' \in [r]$ is the integer congruent to $i$ modulo $r$. Note that $c_{i, a}$ have the same meaning for all choices of radius function, the only difference is that we may use only the first $r$ or $2r$ instead of all $3r$ coefficients. In particular, we may write $c_{i, a}$ instead of $c_{i + r, a}$ and $c_{i + 2r, a}$.

\vspace{\baselineskip}

\begin{claim}\label{qrsuffconds}Suppose that at least $1 - \frac{\eta}{10}$ proportion of all pairs $(x, x') \in S \times S$ have the property that if the equality
\begin{equation}\nu_1 \chi_1 + \dots + \nu_r \chi_r + \lambda_1 L'_1(x) + \dots + \lambda_r L'_r(x) + \lambda'_1 L'_1(x') + \dots + \lambda'_r L'_r(x') = 0\label{latticeclaimproperty}\end{equation}
holds for some $\nu, \lambda, \lambda' \in [-K, K]^r$ then $\nu \in M$ and $\lambda, \lambda' \in \Lambda$. Then the properties \textbf{(i)} and \textbf{(ii)} hold on $t + S$ with $\delta = \sum_{\lambda \in \Lambda \cap [-K, K]^r} c_{1, \lambda_1} \dots c_{r, \lambda_r}$.
\end{claim}

\begin{proof}Define $\delta_0$ as
\[\delta_0 = \sum_{\nu \in M \cap [-K, K]^r} c_{1, \nu_1} \dots c_{r, \nu_r}\]
and $\delta_1$ as
\[\delta_1 = \sum_{\lambda \in \Lambda \cap [-K, K]^r} c_{1, \lambda_1} \dots c_{r, \lambda_r}.\]

We note two consequences of~\eqref{latticeclaimproperty}. Namely, taking any $(x,x')$ satisfying~\eqref{latticeclaimproperty}, putting $\lambda = \lambda' = 0$ and noting that $0 \in \Lambda$, we deduce that $\nu \in [-K, K]^r$ and $\nu_1 \chi_1 + \dots + \nu_r \chi_r = 0$ imply $\nu \in M$. Furthermore, by averaging, there are at least $(1 - \eta)|S|$ elements $x \in S$ for which there exists $x' \in S$ such that the pair $(x, x')$ has the property~\eqref{latticeclaimproperty}. Taking $\lambda' = 0$, we see that if the equality
\[\nu_1 \chi_1 + \dots + \nu_r \chi_r + \lambda_1 L'_1(x) + \dots + \lambda_r L'_r(x) = 0\]
holds for some $\nu, \lambda\in [-K, K]^r$ then $\nu \in M$ and $\lambda \in \Lambda$. We now turn to the proof of the claim.\\

From inequalities~\eqref{wreg3piece} and~\eqref{QRbohrconclusion} we conclude that
\[\Big||B(\Gamma; \rho')| - \sum_{\nu \in [-K, K]^r} \id(\nu_1 \chi_1 + \dots + \nu_r \chi_r = 0)c_{1, \nu_1} \dots c_{r, \nu_r} |G|\Big| \leq \frac{\eta}{5}|G|.\]

As previously discussed, a special case of~\eqref{latticeclaimproperty} implies that all $\nu$ contributing to the sum belong to $M$. Thus, since $\nu \cdot \chi = 0$ for all $\nu \in M$,
\[\Big||B(\Gamma; \rho')| - \delta_0 |G|\Big| \leq \frac{\eta}{5}|G|.\]

Note that at least $1-\eta/5$ proportion of all $x \in S$ obey~\eqref{wreg1piece} and the second property implied by~\eqref{latticeclaimproperty}. For such an $x$ the Bohr set $B(\Gamma \cup \{L'_1(x), \dots, L'_r(x)\}; \rho')$ satisfies the condition~\eqref{wregconditionQR} so the inequality~\eqref{QRbohrconclusion} gives
\begin{align*}\Big||B(\Gamma \cup \{L'_1(x), \dots, L'_r(x)\}; (\rho', \rho'))|\, - \sum_{a, \nu \in [-K, K]^{[r]}} &\id\Big(\sum_{i \in [r]} \nu_i \chi_i + a_i L'_i(x) = 0\Big)\\
&\hspace{2cm} c_{1, \nu_1} \dots c_{r, \nu_r} c_{r+ 1, a_1} \dots c_{2r, a_r} |G|\Big| \leq \frac{\eta}{5}|G|,\end{align*}
while, using a special case of~\eqref{latticeclaimproperty} and the fact that $\lambda \cdot L'(x) = 0$ holds for all $\lambda \in \Lambda$ and $x \in S$, we have the equality
\begin{align*}&\sum_{a ,\nu\in [-K, K]^r} \id\Big(\sum_{i \in [r]} \nu_i \chi_i + a_i L'_i(x) = 0\Big) c_{1, \nu_1} \dots c_{r, \nu_r} c_{1, a_1} \dots c_{r, a_r} \\
&\hspace{2cm}=\sum_{a, \nu\in [-K, K]^r} \id(a \in \Lambda) \id(\nu \in M) c_{1, \nu_1} \dots c_{r, \nu_r} c_{1, a_1} \dots c_{r, a_r}\\
&\hspace{2cm}= \delta_0 \delta_1.\end{align*}
Thus,
\[\Big||B(\Gamma \cup \{L'_1(x), \dots, L'_r(x)\}; (\rho', \rho'))| - \delta_1|B(\Gamma; \rho')|\Big| \leq \frac{\eta}{5}|G|.\]

Finally, recall that $\chi_i = L_i(t)$ and note that $L'_i(x) = L_i(x + \tilde{a}) - L_i(\tilde{a}) = L_i(t + x) - L_i(t) = L_i(t + x) - \chi_i$. Hence
\[B(\Gamma \cup \{L'_1(x), \dots, L'_r(x)\}; (\rho', \rho')) = B(\Gamma; \rho') \cap B(L_1(t + x)-\chi_1, \dots, L_r(t + x)-\chi_r; \rho')\]
proving the first part of the claim, with $\delta = \delta_1$. The second property follows similarly.\end{proof}

Suppose that $t + S$ is not quasirandom in the sense that the properties \textbf{(i)} and \textbf{(ii)} do not hold simultaneously for the radius function $\rho'$. Then at least one of the two assumptions in Claim~\ref{qrsuffconds} fails, but in either case, we conclude that there are at least $\frac{\eta}{10}|S|^2$ pairs $(x, x') \in S \times S$ for which we have an equality
\[\sum_{i \in [r]} \nu_i \chi_i + \lambda_i L'_i(x) + \lambda'_i L'_r(x') = 0\]
where at least one of $\nu \notin M, \lambda \notin \Lambda$ and $\lambda' \notin \Lambda$ holds and all coefficients belong to $[-K, K]$. By averaging, we obtain such a linear combination that holds for at least $\frac{\eta}{10 (2K + 1)^{3r}}|S|^2$ pairs $(x,x')$ of elements in $S$. We fix such a linear combination.\\ 

Suppose first that $\lambda, \lambda' \in \Lambda$. Then $\sum_{i \in [r]} \lambda_i L'_i(x) + \lambda'_i L'_i(x') = 0$ so $\sum_{i \in [r]} \nu_i \chi_i = 0$, which means that $\nu \in M$ which is a contradiction. Therefore, without loss of generality, we may assume that $\lambda \notin \Lambda$. By averaging over $x' \in S$, we get such an element with $\lambda\cdot  L'(x) = -(\nu \cdot \chi + \lambda'\cdot  L'(x'))$ holding for at least $\frac{\eta}{10 (2K + 1)^{3r}}|S|$ choices of $x \in S$.\\
\indent Apply Lemma~\ref{fewvalsCPFLin} to Freiman-linear map $\lambda' \cdot L'$ on $S$ to find a further symmetric coset progression $S' \subseteq S$ of the same shape and of size 
\[|S'| \geq 2^{-O(d^2)}\Big(\frac{\eta}{10 (2K + 1)^{3r}}\Big)^{d+1} |S|\]
on which $\lambda \cdot L'$ vanishes. Take $S'$ in place of $S$ and $\Lambda' = \Lambda + \langle \lambda \rangle_{\mathbb{Z}}$ in place of $\Lambda$.\\

\indent Recall that $K$ is the same in all steps of the argument. Since the set $(\Lambda' \setminus \Lambda) \cap [-K, K]^r$ is non-empty at every step (as it contains the $r$-tuple $\lambda$ chosen above), Lemma~\ref{nestedLattices} bounds the number of steps in the proof by $O(r^2(\log r + \log K))$.\\

Once the desired coset progression $S$ has been found, we declare $C' = S$ and, as remarked at the beginning of the proof, we use Lemma~\ref{tilingLemma} to find a set of translates $T \subseteq C$ such that $|S||T| = |S + T| \geq (1 - \eta) |C|$ and $S + T \subseteq C$.\end{proof}

\section{Abstract Balog-Szemer\'edi-Gowers theorem}\label{absbsgsection}

The classical Balog-Szemer\'edi-Gowers theorem~\cite{BalogSzem, Gow4AP} shows that if a finite set $A$ inside an ambient abelian group has many additive quadruples then it has a large piece of small doubling. Its abstract version, proved in~\cite{newU4}, allows us to consider more general objects than elements of a group, to which Balog-Szemer\'edi-Gowers theorem in its usual form cannot be applied. For example, in~\cite{newU4}, we consider quadruples of partially-defined linear maps $\phi_x : U_x \to \mathbb{F}_p^n$, indexed by elements of the ambient vector space, with the notion of additive quadruple $(x, y, z, w)$, meaning $x - y + z - w = 0$, being replaced by $\phi_x - \phi_y + \phi_z - \phi_w = 0$ on $U_x \cap U_y \cap U_z \cap U_w$. Thus, in place of a group, we have a monoid of partially-defined linear maps, with addition being $(\phi_1: U_1 \to \mathbb{F}_p^n) + (\phi_2: U_2 \to \mathbb{F}_p^n)= (\phi_1 + \phi_2: U_1 \cap  U_2 \to \mathbb{F}_p^n)$.\\

To make sure that we do not cause confusion when working with additive quadruples, we introduce the additional notation that indicates the choice of signs that we have in mind. We write a bold dot above variable to indicate that we take it with negative sign. Namely, writing $(\upd{a}, b, \overset{\bcdot}{c}, d)$ means that we have $a + c = b + d$ and $(\overset{\bcdot}{a}, b, c, \overset{\bcdot}{d})$ means that we have $a + d = b + c$.\\

In comparison with its original version, we consider dense sets of small doubling, rather than dense sets of groups. 

\begin{theorem}[Abstract Balog-Szemer\'edi-Gowers theorem]\label{absg}
    \indent Let $X \subseteq G$ be a set such that $|X - X| \leq K|X|$ and let $A \subseteq X$. Suppose that, for each $i \in [36]$, we have a collection $\mathcal{Q}_i$ of additive quadruples in $A$, satisfying the following properties:
    \begin{itemize}
    \item[\textbf{(i)}] (largeness) $|\mathcal{Q}_1| \geq c |X|^3$,
    \item[\textbf{(ii)}] (symmetry) for each $i\in [36]$, if $(\upd a_1, a_2, a_3, \upd a_4) \in \mathcal{Q}_i$, then 
    \begin{itemize}
        \item[\textbf{(S1)}] $(\upd a_3,a_4,$ $ a_1, \upd a_2) \in \mathcal{Q}_i$, and
        \item[\textbf{(S2)}] $(\upd a_2, a_1, a_4, \upd a_3) \in \mathcal{Q}_i$,
        \item[\textbf{(S3)}] $(\upd a_1, a_3, a_2, \upd a_4) \in \mathcal{Q}_i$,
    \end{itemize} 
    \item[\textbf{(iii)}] (weak transitivity) for all indices $i, j \in [36], i + j \leq 36$, for any additive quadruple $(\upd a_1, a_2, a_3, \upd a_4) \in A^4$, if there are at least $c' |X|$ pairs $(b, b') \in A^2$ such that $(\upd a_1, a_2, b, \upd b') \in \mathcal{Q}_i$ and $(\upd b, b', a_3, \upd a_4) \in \mathcal{Q}_j$, then $(\upd a_1, a_2, a_3, \upd a_4) \in \mathcal{Q}_{i+j}$.
    \end{itemize}

    Fix an integer $k$. Then, provided $c' \leq (cK^{-1}/2)^{\blc_k}$ and $|X| \geq 8c^{-1} K$, there exists a subset $A' \subseteq A$, of size $|A'| \geq (c/2K)^{O(1)}|X|$, with the following property. For $\ell \in [k]$, given an $\ell$-tuple $a_{[\ell]} \in {A'}^{\ell}$, define recursively a collection $\mathcal{Z}_\ell^{(a_{[\ell]})}$ of $(3\ell)$-tuples in $A$, (depending on elements $a_{[\ell]}$) as follows. For $\ell = 1$, $\mathcal{Z}_1^{(a_1)}$ consists of all triples $(y_1, y_2, y_3) \in A^3$ such that $(\upd{a}_1, y_1, y_3, \upd{y}_2) \in \mathcal{Q}_{36}$. For $\ell \geq 2$, the collection $\mathcal{Z}_\ell^{(a_{[\ell]})}$ consists of all $3\ell$-tuples $(y_1, \dots, y_{3\ell}) \in A^{3\ell}$ such that 
    \begin{itemize}
        \item $\sum_{i \in [\ell]} (-1)^i a_i = \sum_{i \in [3\ell]} (-1)^i y_i$,
        \item $(\upd{a_{\ell}}, a_{\ell} + y_{3\ell -1} - y_{3\ell}, y_{3\ell}, \upd{y_{3\ell - 1}})\in \mathcal{Q}_{36}$,
        \item $\Big((y_{3\ell-3})^\bcdot, y_{3\ell-3} - y_{3\ell -2} + y_{3\ell -1} - y_{3\ell} + a_{\ell}, y_{3\ell - 2}, (y_{3\ell - 1} - y_{3\ell} + a_{\ell })^\bcdot\Big)\in \mathcal{Q}_{36}$, and,
        \item $(y_1, \dots, y_{3\ell - 4}, a_\ell + y_{3\ell - 3} - y_{3\ell - 2} + y_{3\ell - 1} - y_{3\ell}) \in \mathcal{Z}_{\ell - 1}^{(a_{[\ell - 1]})}$.
    \end{itemize}
    Then, for all $\ell \in [k]$, we have $|\mathcal{Z}_\ell^{(a_{[\ell]})}| \geq (c/2K)^{O_\ell(1)} |X|^{3\ell - 1}$.
\end{theorem}

\noindent\textbf{Remark.} We use the weak transitivity property for pairs $(\mathcal{Q}_1, \mathcal{Q}_i)$ for $i \in [5]$, and $(\mathcal{Q}_6, \mathcal{Q}_{6i})$ for $i \in [5]$. It is likely possible to modify the proof so that we need only $(\mathcal{Q}_1, \mathcal{Q}_i)$ for $i \in [35]$, but we do not require precise control over pairs $(\mathcal{Q}_i, \mathcal{Q}_j)$ to which weak transitivity property is applied in this paper.
\vspace{\baselineskip}

The following lemma, stemming from the arguments of Gowers in~\cite{Gow4AP}, shows that dense graphs contain robustly connected subgraphs. Similar statements can be found in~\cite{SudSzeVu}. The formulation below is taken from~\cite{newU4}, where it appears as Lemma 6. 

\begin{lemma}\label{gowerspathssingle}Let $G$ be a graph on $n$ vertices with at least $cn^2$ edges. Then there exists a  subset of vertices $X$ of size at least $2^{-5} c n$ with the property that there are at least $2^{-35} c^9 n^5$ paths of length 6 between any two vertices in $X$.\end{lemma}

We proceed to prove Theorem~\ref{absg}.

\begin{proof}[Proof of Theorem~\ref{absg}]
    In the proof, we focus on the set of (ordered) pairs $A^2$. We say that a pair of pairs $\Big((x, x'), (y, y')\Big)$ is \textit{good} if $(\upd x, x', y, \upd{y}{}')$ is an additive quadruple in $\mathcal{Q}_1$. By the largeness assumption, there are at least $c|X|^3$ good pairs of pairs. For $d \in X - X$, let us write $\pi(d)$ for the number of good pairs of pairs of the shape $\Big((x' + d, x'), (y' + d, y')\Big)$, i.e. pairs of pairs of difference $d$. There are at most $|X|^2$ choices for $(x', y')$, so $\pi(d) \leq |X|^2$. On the other hand, since every good pair of pairs $\Big((x, x'), (y, y')\Big)$ can pe put in the form above with $d = x - x'$, we get $\sum_{d \in X- X} \pi(d) \geq c|X|^3$. Recalling the assumption $|X - X| \leq K |X|$, by averaging, we find a set $D \subseteq X - X$ of size $|D| \geq \frac{c}{2}|X|$ such that there are at least $\frac{c}{2K}|X|^2$ good pairs of pairs of the form $\Big((x + d, x), (y + d, y)\Big)$.\\

    Note that the symmetry assumption \textbf{(S1)} implies that if $\Big((x + d, x), (y + d, y)\Big)$ is a good pair of pairs, so is $\Big((y + d, y), (x + d, x)\Big)$. Hence, for any given $d \in D$, we may define the graph $\Gamma_d$, whose vertex set is $V_d = \{(x + d, x) : x \in A \cap A - d\}$ and edges are $(x + d, x)(y+d,y)$ for all good pairs of pairs $\Big((x + d, x), (y + d, y)\Big)$. In particular, $\Gamma_d$ is a graph on a vertex set of size at most $|X|$ consisting of some pairs in $A^2$. Due to the definition of $D$, there are at least $\frac{c}{4K}|X|^2 - |X|$ edges, as each edge comes from at most two pairs of pairs and we need to ignore pairs of the form $\Big((x + d, x), (x + d, x)\Big)$. The bound on the number of edges also ensures that $|V_d| \geq \frac{c}{4K}|X| - 1$. Since $|X| \geq 8c^{-1} K$, we may simplify the bounds above. Namely, we see that there are at least $\frac{c}{8K}|X|^2$ edges and $|X| \geq |V_d| \geq \frac{c}{8K} |X|$.\\

    For each $d \in D$, apply Lemma~\ref{gowerspathssingle} to graph $\Gamma_d$ to obtain a set of pairs $P_d \subseteq V_d$ of size $|P_d| \geq \Omega(c^2K^{-2})|X|$ such that, for any pairs $(x + d, x), (y + d, y) \in P_d$, there exists at least $\Omega(c^{14}K^{-14})|X|^5$ 6-paths $(x + d, x), (z_1 + d, z_1), \dots, (z_5 + d, z_5), (y + d, y)$ in the graph $\Gamma_d$. In other words, any two consecutive pairs in the sequence above give a good pair of pairs. 

    \begin{claim}\label{sixtupleschainclaim}
        Given any two pairs $(x + d, x), (y + d, y) \in P_d$, we have $(\upd x + d, x, y + d, \upd y) \in \mathcal{Q}_6$, provided $c' \leq \bsc\, c^{14}K^{-14}$.
    \end{claim}

    \begin{proof}
        By induction on $i \in [6]$ we show that there are least $\Omega(c^{14}K^{-14})|X|^{6-i}$  choices of $(x_i, x_{i + 1}, \dots, x_5) \in X^{6-i}$ such that $(\upd x + d, x, x_i + d, \upd x_i) \in \mathcal{Q}_i$ and $(\upd x_i + d, x_i, x_{i + 1} + d, \upd x_{i+1}), \dots, (\upd x_5 + d, x_5, y + d, \upd y) \in \mathcal{Q}_1$. The case $i = 6$ is interpreted as $(\upd x + d, x, y + d, \upd y) \in \mathcal{Q}_6$. Note that the base case holds by properties of $P_d$.\\

        Suppose now that the claim holds for some $i \leq 5$. Let $T$ be the collection of tuples $(x_i, x_{i + 1}, \dots, x_5) \in X^{6-i}$ above. By averaging, there are at least $\Omega(c^{14}K^{-14})|X|^{6-i - 1}$ choices of $(x_{i + 1}, \dots, x_5) \in X^{6-i - 1}$ such that $(x_i, x_{i + 1}, \dots, x_5) \in T$ for at least $\Omega(c^{14}K^{-14})|X|$ elements $x_i \in X$. But, then $(\upd x + d, x, x_i + d, \upd x_i) \in \mathcal{Q}_i$ and $(\upd x_i + d, x_i, x_{i + 1} + d, \upd x_{i+1}) \in \mathcal{Q}_1$ holds for at least $\Omega(c^{14}K^{-14})|X|$ choices of $x_i$. By the weak transitivity, we have $(\upd x + d, x, x_{i + 1} + d, \upd x_{i + 1}) \in \mathcal{Q}_{i + 1}$. This completes the proof of the inductive step.
    \end{proof}

    Next, we use the pairs we obtained above to define another graph, this time on the vertex set $A$. To ensure that edges are symmetric, we choose a single $d \in D$ for each pair $\{d, -d\}$ and then redefine $P_{-d}$ to be the set of all $(x, x + d)$ for all $(x + d, x) \in P_d$. Note that every pair of pairs in $P_{-d}$ still belongs to $\mathcal{Q}_6$ due to symmetry assumption \textbf{(S2)}. Thus, define $P$ to be the union of all $P_d$, which is symmetric, has size $|P| \geq \Omega(c^3K^{-3})|X|^2$, and 
    \begin{equation} \text{whenever }(x, x'), (y, y') \in P\text{ have the same difference }x - x' = y - y'\text{, then }(\upd x, x', y, \upd{y} {}') \in \mathcal{Q}_6.\label{ppairsproperty}\end{equation}

    Define a graph $\Pi$ whose vertex set is $A$ and edges are pairs in $P$. By Lemma~\ref{gowerspathssingle}, we find a set $B \subseteq A$ of size $c_1 |X|$ such that there are at least $c_2|X|^5$ 6-paths in graph $\Pi$ between any two vertices in $B$, where $c_1, c_2 \geq \Omega((cK^{-1})^{O(1)})$.\\

    The key property of the set $B$ is that is arithmetically rich everywhere, namely, that there are many additive quadruples belonging to $\mathcal{Q}_{36}$ in all of its sufficiently dense subsets.

    \begin{claim}\label{absBSGarithrich}
        Let $\beta_1, \beta_2 > 0$. Assume $c' \leq \bsc\,(\beta_1 \beta_2 c_2)^2K^{-6}$. Suppose that $B_1, B_2 \subseteq B$ are subsets of sizes $\beta_1|X|$ and $\beta_2|X|$. Then there exist at least $c_3(\beta_1\beta_2c_2)^2 K^{-6}|X|^3$ additive quadruples $(\upd{b_1}, b'_1, b_2, \upd{b_2} {}') \in (B_1 \times B_1 \times B_2 \times B_2) \cap \mathcal{Q}_{36}$, for some $c_3 \geq \Omega(1)$.
    \end{claim}

    \begin{proof}
        Take arbitrary elements $b_1 \in B_1$ and $b_2 \in B_2$. Since $b_1, b_2 \in B$ there are at least $c_2|X|^5$ 6-paths in the graph $\Pi$ between $b_1$ and $b_2$. Summing over $b_1$ and $b_2$, we conclude that there are at least $\beta_1 \beta_2 c_2|X|^7$ 7-tuples $(b_1, z_1, \dots, z_5, b_2) \in A^7$ with the property that $b_1 \in B_1, b_2 \in B_2$ and $(b_1, z_1),$ $(z_1, z_2), \dots,$ $(z_4, z_5),$ $(z_5, b_2) \in P$.\\

        We change the variables by setting $d_1 = b_1 - z_1, d_2 = z_1 - z_2, \dots, d_5 = z_4 - z_5, d_6 = z_5 - b_2$ and using them instead of $z_1, \dots, z_5, b_2$. Note that $d_1, \dots, d_6 \in X - X$. By Cauchy-Schwarz inequality and doubling assumption $|X - X| \leq K |X|$, there are at least $(\beta_1 \beta_2 c_2)^2K^{-6}|X|^8$ 8-tuples $(b_1, b'_1, d_1, \dots, d_6)$ such that $b_1, b'_1 \in B_1$, $b_1 - d_1 - \dots - d_6, b'_1 - d_1 - \dots - d_6 \in B_2$ and 
        \[(b_1 - d_1 - \dots - d_{i-1}, b_1 - d_1 - \dots - d_i), (b'_1 - d_1 - \dots - d_{i-1}, b'_1 - d_1 - \dots - d_i) \in P\]
        holds for all $i \in [6]$.\\

        Let us change the variables one more time, writing $z_i = b_1 - d_1 - \dots - d_i$, $z'_i = b'_1 - d_1 - \dots - d_i$, for $i \in [5]$ and $b_2 = b_1 - d_1 - \dots - d_6, b'_2 = b'_1 - d_1 - \dots - d_6$. By the work above, we obtain at least $(\beta_1 \beta_2 c_2)^2K^{-6}|X|^8$ 14-tuples $(b_1, b'_1, b_2, b'_2, z_{[5]}, z'_{[5]}) \in A^{14}$ such that 
        \begin{itemize}
            \item $b_1 - b'_1 = z_1 - z'_1 = z_2 - z_2'= \dots = z_5 - z_5' = b_2 - b_2'$, and
            \item $(b_1, z_1), (z_1, z_2), \dots, (z_4, z_5), (z_5, b_2), (b'_1, z'_1), \dots, (z'_5,b_2') \in P$.
        \end{itemize}

        By property~\eqref{ppairsproperty} of the set $P$, we have that
        \[(\upd{b_1},  z_1, b'_1,\upd{z_1} {}'), (\upd{z_1}, z_2, z'_1, \upd{z_2} {}'), \dots, (\upd{z_5}, b_2, z'_5, \upd{b_2} {}') \in \mathcal{Q}_6.\]
        Using symmetry condition \textbf{(S3)}, we get
        \[(\upd{b_1}, b'_1, z_1, \upd{z_1} {}'), (\upd{z_1}, z'_1, z_2, \upd{z_2} {}'), \dots, (\upd{z_5}, z'_5, b_2, \upd{b_2} {}') \in \mathcal{Q}_6.\]
        
        Averaging and using weak transitivity several times as in the proof of Claim~\ref{sixtupleschainclaim} proves the claim, provided $c' \leq \bsc\, (\beta_1 \beta_2 c_2)^2K^{-6}$.
    \end{proof}
    
\vspace{\baselineskip}

    Let $B' \subseteq B$ be the set of all elements $a \in B$ such that $(\upd a, b, c, \upd d) \in \mathcal{Q}_{36}$ for at least $2^{-10} c_3 c^4_1 c^2_2 K^{-6}|X|^2$ triples $(b,c,d) \in B^3$. We claim that $|B \setminus B'| \leq \frac{c_1}{2}|X|$. Otherwise, applying Claim~\ref{absBSGarithrich} to the set $B \setminus B'$, so $\beta_1, \beta_2 \geq c_1/2$, we see that there are at least  $2^{-10}c_1^64c_2^2 |X|^3$ additive quadruples in $\mathcal{Q}_{36} \cap (B \setminus B')^4$, which is a contradiction. Thus, $|B'| \geq |B| - \frac{c_1}{2}|X| \geq \frac{c_1}{2}|X|$. We claim that $B'$ has the desired properties (i.e., we may take $A' = B'$). Moreover, we restrict collections $\mathcal{Z}_\ell^{(a_{[\ell]})}$ to subsets of $B^{3\ell}$.\\

    To complete the proof, we show by induction on $\ell$ that $|\mathcal{Z}_\ell^{(a_{[\ell]})}| \geq (c/2K)^{O_\ell(1)} |G|^{3\ell - 1}$ holds for all $a_{[\ell]}$ in $B'$. The base case $\ell = 1$ is trivial by definition of $B'$. Suppose the claim holds for some $\ell$, and let $a_{[\ell + 1]}$ be an arbitrary $(\ell+1)$-tuple in $B'$. Let us simplify the notation and write $Z =\mathcal{Z}_\ell^{(a_{[\ell]})}$ and $Z' = \mathcal{Z}_1^{(a_{\ell+1})}$. By induction hypothesis, these two sets have sizes $|Z| \geq \delta |X|^{3\ell - 1}$ and $|Z'| \geq \delta |X|^{2}$ for some $\delta \geq (c/2K)^{O_\ell(1)}$.\\
    \indent Let $W$ be the set of all $y_{3\ell} \in B$ such that at least $\frac{1}{2}\delta|X|^{3\ell-2}$ of $3\ell$-tuples in $Z$ have $y_{3\ell}$ as their last element. Let $W'$ be the set of all $z_1 \in B$ which appear as the first element in at least $\frac{\delta}{2}|X|$ triples in $Z'$. Thus, $|W|, |W'| \geq \frac{\delta}{2}|X|$.\\
    \indent Apply Claim~\ref{absBSGarithrich} to these sets (leading to the dependency on $k$ in requirement $c' \leq (cK^{-1}/2)^{\blc_k}$ in the statement of the theorem) to find a collection $Q$ of additive quadruples 
    \begin{equation}(\upd{y_{3\ell} + t}, y_{3\ell}, z_1 + t, \upd{z_1}) \in (W \times W \times W' \times W') \cap \mathcal{Q}_{36}\label{wwprimeproperty}\end{equation}
    of size $|Q| \geq c_3 2^{-4} c_2^2 K^{-6} \delta^4 |X|^3$. Hence, we get at least $2^{-6} c_3 c_2^2 K^{-6} \delta^6|X|^{3\ell - 2 + 1 + 3}$ choices of $(3\ell + 4)$-tuples $(t, y_{[3\ell]}, z_1, z_2, z_3) \in B^{3\ell + 4}$ such that 
    \begin{itemize}
        \item $y_{[3\ell]} \in Z$,
        \item $z_{[3]} \in Z'$,
        \item $(\upd{y_{3\ell} + t}, y_{3\ell}, z_1 + t, \upd{z_1}) \in \mathcal{Q}_{36}$.
    \end{itemize}

    For each such choice, we claim that the $3(\ell + 1)$-tuple $y'_{[3\ell + 3]} = (y_1, \dots, y_{3\ell - 1}, y_{3\ell} + t, z_1 + t, z_2, z_3)$ belongs to $\mathcal{Z}_{\ell + 1}^{(a_{[\ell + 1]})}$. We need to check the properties described in the definition of the collection  $\mathcal{Z}_{\ell + 1}^{(a_{[\ell + 1]})}$ stated in the theorem.\\
    \indent Firstly, as $y_{[3\ell]} \in Z$, we have
    \[\sum_{i \in [\ell + 1]} (-1)^i a_i = \sum_{i \in [\ell]} (-1)^i a_i + (-1)^{\ell + 1} a_{\ell + 1} = \sum_{i \in [3\ell]} (-1)^i y_i + (-1)^{\ell + 1} a_{\ell + 1}.\]
    This further equals
    \[\sum_{i \in [3\ell - 1]} (-1)^i y_i + (-1)^\ell (y_{3\ell} - a_{\ell + 1}) = \sum_{i \in [3\ell - 1]} (-1)^i y_i + (-1)^{3\ell} (y_{3\ell} + t) + (-1)^{\ell + 1 }(t + z_1 - z_2 + z_3)\]
    where we used $z_{[3]} \in Z'$. This expression equals
    \[\sum_{i \in [3\ell - 1]} (-1)^i y_i + (-1)^{3\ell} (y_{3\ell} + t) + (-1)^{3\ell + 1 }(z_1 + t) +  (-1)^{3\ell + 2}z_2 + (-1)^{3\ell + 3}z_3 = \sum_{i \in [3\ell + 3]} (-1)^i y'_i,\]
    as desired.\\
    \indent Secondly, unpacking definitions and recalling that $z_{[3]} \in Z'$,
    \[(\upd{a_{\ell + 1}}, a_{\ell + 1} + y'_{3\ell + 2} - y'_{3\ell + 3}, y'_{3\ell + 3}, \upd{y'_{3\ell + 2}}) = (\upd{a_{\ell + 1}}, z_1, z_3, \upd{z_2}) \in \mathcal{Q}_{36},\]
    and, after algebraic manipulation and recalling~\eqref{wwprimeproperty},
    \[\Big((y'_{3\ell})^\bcdot, y'_{3\ell} - y'_{3\ell + 1} + y'_{3\ell + 2} - y'_{3\ell + 3} + a_{\ell + 1}, y'_{3\ell + 1}, (y'_{3\ell + 2} - y'_{3\ell + 3} + a_{\ell + 1})^\bcdot\Big) = (\upd{y_{3\ell} + t}, y_{3\ell}, z_1 + t, \upd{z_1}) \in \mathcal{Q}_{36}.\]
    Finally, $y_{3\ell} = a_{\ell + 1} + (y_{3\ell} + t) - (z_1 + t) + z_2 - z_3 = a_{\ell + 1} + y'_{3\ell} - y'_{3\ell + 1} + y'_{3\ell + 2} - y'_{3\ell + 3}$ and the $3\ell$-tuple $y_{[3\ell]}$ belongs to $\mathcal{Z}_\ell^{(a_{[\ell]})}$, with $y'_i = y_i$ for $i \in [3\ell - 1]$. All four conditions hold and the proof is complete.
\end{proof}

\pagebreak

\begin{center}\noindent{\large\bfseries{\scshape Chapter 2: Structure of Freiman bihomomorphisms}}
\end{center}\normalsize

In this chapter, we begin the study of Freiman bihomomorphisms. The proof follows the outline from the introduction. 

\section{Freiman bihomomorphisms and small rank condition}\label{fbihomsmallranksection}

Firstly, we pass from a Freiman bihomomorphisms to an approximately linear system of Freiman linear maps.

\begin{proposition}\label{freimanbihomomStep1passtolineaersystem}
    Let $\varphi : A \to H$ be a Freiman bihomomorphism on a set $A \subseteq G_1 \times G_2$ of density $c$. Then there exist quantities $c' \geq \exp(-\log^{O(1)} (2c^{-1})) $, $K \leq \exp(\log^{O(1)} (2c^{-1}))$, a set $X \subseteq G_1$, a collection of Freiman linear maps $\phi_x : B_x \to H$, where $B_x$ is a Bohr set in $G_2$ of codimension at most $\log^{O(1)} (2c^{-1})$ and radius at least $\exp(-\log^{O(1)} (2c^{-1}))$, such that
    \begin{itemize}
        \item[\textbf{(i)}] $|X| \geq c' |G_1|$,
        \item[\textbf{(ii)}] for each $x \in X$ and all $y \in B_x$, there are at least $c'|G_2|^3$ triples $(z_1, z_2, z_3) \in G_2^3$ such that points $(x, z_1), (x, z_2), (x, z_3)$ and $(x, z_1 + z_2 - z_3 - y)$ belong to $A$ and
        \[\phi_x(y) = \varphi(x, z_1) + \varphi(x, z_2) - \varphi(x, z_3) - \varphi(x, z_1 + z_2 - z_3 - y),\]
        \item[\textbf{(iii)}] for at least $c' |G|^3$ of additive quadruples $(\upd{x_1}, x_2, \upd{x_3}, x_4)$ in $X$ we have
        \[\range\Big(\phi_{x_1} - \phi_{x_2} + \phi_{x_3} - \phi_{x_4}\Big) \leq K.\]
    \end{itemize}
\end{proposition}

\begin{proof}
    Let $X \subseteq G_1$ be the set of columns of the set $A$ that are at least $c/2$-dense in $G_2$, namely the set of $x \in G_1$ with $|A_{x \bcdot}| \geq \frac{c}{2}|G_2|$. By averaging, $|X| \geq \frac{c}{2}|G_1|$. Since $\varphi$ is a Freiman bihomomorphism, for each $x \in X$, we may apply Theorem~\ref{approxFreimanHom} to the map $y \mapsto \varphi(x,y)$ with domain $A_{x \bcdot}$,  to find a proper coset progression $Q_x \subseteq G_2$ of rank at most $\log^{O(1)} (2c^{-1})$, a set $S_x \subseteq Q_x \cap A_{x \bcdot}$ of size $|S_x| \geq \exp(-\log^{O(1)} (2c^{-1})) |G_2|$ and a Freiman homomorphism $\psi_x \colon Q_x \to H$ such that $\varphi(x, y) = \psi_x(y)$ holds for all $x \in S_x$.\\
    \indent Applying Corollary~\ref{shrunksymmcpCor}, we may find a further coset progression $Q'_x$ of rank at most $\log^{O(1)} (2c^{-1})$, which is proper and symmetric such that $|S_x \cap (t + Q'_x)| \geq \exp(-\log^{O(1)} (2c^{-1})) |G_2|$ and $t + 4Q'_x \subseteq Q_x$. Define map $\phi_x : Q'_x \to H$ by $\phi_x(y) = \psi_x(t + y) - \psi_x(t)$, which is Freiman-linear. Apply Theorem~\ref{robustBogRuzsa} to the set $S_x \cap (t + Q'_x)$ and Lemma~\ref{cosettobohrset} to $S_x \cap (t + Q'_x)$ to find a 
    a Bohr set $B_x \subseteq G_2$ of codimension at most $\log^{O(1)} (2c^{-1})$ and radius at least $\exp(-\log^{O(1)} (2c^{-1}))$ such that for each $y \in B_x$ there are at least $\exp(-\log^{O(1)} (2c^{-1})) |G_2|^4$ quadruples $(z_1, z_2, z_3, z_4)$ of elements in $ S_x \cap (t + Q'_x)$ such that $y = z_1 + z_2 - z_3 z_4$. For each such a choice of $(z_1, z_2, z_3, z_4)$ we have
    \[\phi_x(y) = \psi_x(t + y) - \psi_x(t) = \psi_x(z_1) + \psi_x(z_2) - \psi_x(z_3) - \psi_x(z_4) = \varphi(x, z_1) + \varphi(x, z_2) - \varphi(x, z_3) - \varphi(x, z_4),\]
    giving us property \textbf{(ii)}. It remains to deduce property \textbf{(iii)}.\\
    
    To that end, for each $x \in X$, define $T_x$ to be the set of all quadruples $(z_1, z_2, z_3, z_4)$ of elements in $ S_x \cap (t + Q'_x)$ such that $z_1 + z_2 - z_3 - z_4 \in B_x$ and $\phi_x(z_1 + z_2 - z_3 - z_4) = \varphi(x, z_1) + \varphi(x, z_2) - \varphi(x, z_3) - \varphi(x, z_4)$. By the above argument, we have $|T_x| \geq \exp(-\log^{O(1)} (2c^{-1})) |G_2|^4$.\\
    \indent By using Cauchy-Schwarz inequality a few times, we have
    \begin{align*}\exp(-\log^{O(1)}& (2c^{-1})) \leq \Big(\exx_{x \in G_1} \id_X(x) \exx_{z_1, z_2, z_3, z_4 \in G_2} \id_{T_x}(z_1, z_2, z_3, z_4)\Big)^4 \\
    \leq &\Big( \exx_{z_1, z_2, z_3, z_4 \in G_2} \Big(\exx_{x \in G_1} \id_X(x) \id_{T_x}(z_1, z_2, z_3, z_4)\Big)\Big)^4 \\
    \leq & \Big( \exx_{z_1, z_2, z_3, z_4 \in G_2} \Big(\exx_{x \in G_1} \id_X(x) \id_{T_x}(z_1, z_2, z_3, z_4)\Big)^2\Big)^2\\
    = & \Big( \exx_{z_1, z_2, z_3, z_4 \in G_2}\exx_{x, a \in G_1} \id_X(x)  \id_X(x + a) \id_{T_x}(z_1, z_2, z_3, z_4)\id_{T_{x + a}}(z_1, z_2, z_3, z_4)\Big)^2\\
    \leq & \exx_{z_1, z_2, z_3, z_4 \in G_2}\Big(\exx_{x, a \in G_1} \id_X(x)  \id_X(x + a) \id_{T_x}(z_1, z_2, z_3, z_4)\id_{T_{x + a}}(z_1, z_2, z_3, z_4)\Big)^2\\
    = &\exx_{z_1, z_2, z_3, z_4 \in G_2}\exx_{x_1, x_2, a \in G_1} \id_X(x_1) \id_X(x_2) \id_X(x_1 + a) \id_X(x_2 + a)\\
    &\hspace{1cm}\id_{T_{x_1}}(z_1, z_2, z_3, z_4)\id_{T_{x_1 + a}}(z_1, z_2, z_3, z_4)\id_{T_{x_2}}(z_1, z_2, z_3, z_4)\id_{T_{x_2 + a}}(z_1, z_2, z_3, z_4).\end{align*}

    Hence, we have at least $\exp(-\log^{O(1)} (2c^{-1})) |G_1|^3$ additive quadruples $(\upd{x_1}, x_2, \upd{x_3}, x_4)$ in $X$, for which there exist at least $\exp(-\log^{O(1)} (2c^{-1})) |G_2|^4$ quadruples $(z_1, z_2, z_3, z_4)$ all $(x_i, z_j) \in A$ and such that 
    \begin{align*}&\sum_{i \in [4]} (-1)^i \phi_{x_i}(z_1 + z_2 - z_3 - z_4) = \sum_{i \in [4]} (-1)^i \Big(\varphi(x_i, z_1)  + \varphi(x_i, z_2) - \varphi(x_i, z_3) - \varphi(x_i, z_4))\Big)\\
    & = \Big(\sum_{i \in [4]} (-1)^i \varphi(x_i, z_1)\Big) + \Big(\sum_{i \in [4]} (-1)^i \varphi(x_i, z_2)\Big) - \Big(\sum_{i \in [4]} (-1)^i \varphi(x_i, z_3)\Big) -\Big(\sum_{i \in [4]} (-1)^i \varphi(x_i, z_4)\Big),\end{align*}
    which vanishes since $\varphi$ is a Freiman bihomomorphism. The proposition follows after an application of Lemma~\ref{smallrangelemma}.
\end{proof}

\section{A subset with almost all additive 16-tuples with small image}\label{almostalladditive16tuplesABSG1}

So far, we have obtained a system of Freiman-linear maps on Bohr sets  $\phi_x : B_x \to H$ in which many additive quadruples of indices give linear combinations of maps with small range. In this section, we use the abstract Balog-Szemer\'edi-Gowers theorem and a dependent random choice argument to pass to a subset of indices in which almost all additive 16-tuples give small range. The reason for having 16-tuples in the conclusion instead of quadruples is an application of the robust Bogolyubov-Ruzsa theorem in the next section.

\begin{proposition}\label{almostalladd16tupleshavesmallimage}
    Suppose that we are given a set $X \subseteq G_1$, a collection of Freiman-linear maps $\phi_x : B_x \to H$, where $B_x$ is a Bohr set in $G_2$ of codimension $d$ and radius $\rho$, such that 
    \[\range\Big(\phi_{x_1} - \phi_{x_2} + \phi_{x_3} - \phi_{x_4}\Big) \leq K\]
    holds for at least $c|G_1|^3$ additive quadruples $(\upd{x_1}, x_2, \upd{x_3}, x_4)$ in $X^4$. Let $\varepsilon > 0$ be given. Then there exists a subset $\tilde{X} \subseteq X$ of size $\exp(-(2d \log ( \rho^{-1}c^{-1} K))^{O(1)})|G_1|$ such that 
    \[\range\Big(\sum_{i \in [16]} (-1)^i \phi_{x_i}\Big|_{\cap_{i \in [16]} B'_{x_i}}\Big) \leq \exp((2 d\log (\varepsilon^{-1} \rho^{-1} K))^{O(1)})\]
    holds for all but at most $\varepsilon |G_1|^{15}$ additive 16-tuples $x_{[16]}$ in $\tilde{X}$, where $B'_x$ is the the Bohr set with the same frequency set as $B_x$, but with radius twice smaller.
\end{proposition}

\textbf{Remark.} It is crucial that the size $\tilde{X}$ does not depend on the parameter $\varepsilon$, which only affect the bound on the image size in the conclusion.

\begin{proof}
    Define $\mathcal{Q}_i$ to be the set of all additive quadruples $(\upd{x_1}, x_2, x_3, \upd{x_4})$ in $X$ such that 
    \[\range\Big(\phi_{x_1} - \phi_{x_2} - \phi_{x_3} + \phi_{x_4}\Big) \leq K^i.\]
    It is easy to check that all conditions of Theorem~\ref{absg} hold. In particular, weak transitivity is satisfied as long as $|G| \geq (2c^{-1})^{-\blc}$, which we may assume as otherwise the proposition is trivial. Thus, we may apply  Theorem~\ref{absg} with $k = 16$, to find a set $X'$, of size $|X'| \geq (c/2)^{O(1)} |G_1|$, which satisfies the conditions in the conclusion of that theorem.

    \begin{claim}
        For each additive $16$-tuple $x_{[16]} \in {X'}^{16}$, there are at least $(c/2)^{O(1)} |G_1|^{47}$ additive 48-tuples $y_{[48]} \in X^{48}$ such that 
        \[\range\Big(\sum_{i \in [16]} (-1)^i \phi_{x_i} - \sum_{i \in [48]} (-1)^i \phi_{y_i}\Big) \leq (2K)^{O(1)}.\]
    \end{claim}

    \begin{proof}
        Unpacking the conclusion of Theorem~\ref{absg}, we get $|\mathcal{Z}_{16}^{(x_{[16]})}|\geq (c/2)^{O(1)} |G_1|^{47}$. For each 48-tuple $y_{[48]} \in \mathcal{Z}_{16}^{(x_{[16]})}$ we have $\sum_{i \in [48]} (-1)^iy_i = \sum_{i \in [16]} (-1)^ix_i = 0$, making it an additive 48-tuple.\\
        To complete the proof, by induction on $\ell \in [16]$, we show that 
        \[\range\Big(\sum_{i \in [\ell]} (-1)^i \phi_{x_i} - \sum_{i \in [3\ell]} (-1)^i \phi_{y_i}\Big) \leq (2K)^{O(1)}\]
        holds for all $y_{[3\ell]} \in \mathcal{Z}_{\ell}^{(x_{[\ell]})}$.\\
        \indent The base case $\ell = 1$ holds by definition of $\mathcal{Q}_{36}$ and $\mathcal{Z}^{(x_1)}_1$.\\
        \indent Suppose now that the claim holds for some $\ell \geq 1$ and let $y_{[3\ell+3]} \in \mathcal{Z}_{\ell + 1}^{(x_{[\ell + 1]})}$. By definition of $\mathcal{Z}_{\ell + 1}^{(x_{[\ell + 1]})}$ we have that
        \[(y_1, \dots, y_{3\ell - 1}, x_{\ell + 1} + y_{3\ell} - y_{3\ell + 1} + y_{3\ell + 2} - y_{3\ell + 3}) \in \mathcal{Z}_{\ell}^{(x_{[\ell]})}\]
        so by induction hypothesis
        \[\range\Big(\sum_{i \in [\ell]} (-1)^i \phi_{x_i} - \Big(\sum_{i \in [3\ell - 1]} (-1)^i \phi_{y_i}\Big) - (-1)^{3\ell}\phi_{x_{\ell + 1} + y_{3\ell} - y_{3\ell + 1} + y_{3\ell + 2} - y_{3\ell + 3}}\Big) \leq (2K)^{O(1)}.\]
        We also have that 
        \[(\upd{x_{\ell + 1}}, x_{\ell + 1} + y_{3\ell + 2} - y_{3\ell + 3}, y_{3\ell + 3}, \upd{y_{3\ell + 2}})  \in \mathcal{Q}_{36}\]
        and
        \[\Big((y_{3\ell})^\bcdot, y_{3\ell} - y_{3\ell + 1} + y_{3\ell + 2} - y_{3\ell + 3} + x_{\ell + 1}, y_{3\ell + 1}, (y_{3\ell + 2} - y_{3\ell + 3} + x_{\ell + 1})^\bcdot\Big) \in \mathcal{Q}_{36}.\]
        Hence
        \begin{align*}&\range\Big(\sum_{i \in [\ell + 1]} (-1)^i \phi_{x_i} - \sum_{i \in [3\ell + 3]} (-1)^i\phi_{y_i}\Big)\\
        &\hspace{1cm}\leq \range\Big(\sum_{i \in [\ell]} (-1)^i \phi_{x_i} - \sum_{i \in [3\ell - 1]} (-1)^i \phi_{y_i} - (-1)^{3\ell}\phi_{x_{\ell + 1} + y_{3\ell} - y_{3\ell + 1} + y_{3\ell + 2} - y_{3\ell + 3}}\Big)\\
        &\hspace{2cm}\cdot\range\Big((-1)^{3\ell + 1}\Big(\phi_{y_{3\ell}} - \phi_{y_{3\ell} - y_{3\ell + 1} + y_{3\ell + 2} - y_{3\ell + 3} + x_{\ell + 1}} - \phi_{y_{3\ell + 1}} + \phi_{y_{3\ell + 2} - y_{3\ell + 3} + x_{\ell + 1}}\Big)\Big) \\
        &\hspace{2cm}\cdot\range\Big((-1)^{\ell + 1}\Big(\phi_{x_{\ell + 1}} - \phi_{x_{\ell + 1} + y_{3\ell + 2} - y_{3\ell + 3}} - \phi_{y_{3\ell + 3}} +  \phi_{y_{3\ell + 2}})\Big)\Big)\leq (2K)^{O(1)}.\qedhere\end{align*}
    \end{proof}
    
    For each additive 16-tuple $x_{[16]}$ in $X'$, let $Y_{x_{[16]}}$ be the set of additive 48-tuples in $X$ provided by the claim above. Let $Z_1, \dots, Z_m$ be a maximal collection of disjoint sets among $Y_{x_{[16]}}$, hence $m \leq (2c^{-1})^{O(1)}$. Since $Y_{x_{[16]}}$ meets some $Z_i$ for each additive 16-tuple $x_{[16]}$ in $X'$, we may partition additive 16-tuples in $X'$  into sets $Q_1, \dots, Q_m$, such that if $x_{[16]} \in Q_i$ then $Y_{x_{[16]}} \cap Z_i \not= \emptyset$. In particular, we obtain some $K' \leq (2K)^{O(1)}$ such that 
    \[\range\Big(\sum_{j \in [16]} (-1)^j \phi_{x_j} - \sum_{j \in [16]} (-1)^j \phi_{x'_j}\Big) \leq K'\]
    for any $x_{[16]}$ and $x'_{[16]}$ that belong to the same $Q_i$.\\

    \indent In order to obtain the desired set $\tilde{X}$, we employ a 3-step probabilistic argument. Let $r, s \in \mathbb{N}$ and $\eta > 0$ be parameters to be chosen later. Firstly, we shall choose a set $E$ of the form $E = \langle e_{1}, \dots, e_{r}\rangle_{\{0,1\}}$ for some $e_i \in G_2$, which will then be used to find an auxiliary subset $X'' \subseteq X'$. Secondly, we shall find a homomorphism $\chi: H \to \mathbb{T}^s$ which will separate certain sets of linear combinations of images $\phi_x(e_i)$ inside $\mathbb{T}^s$. Finally, we shall find the desired set $\tilde{X}$.\\
    \indent Write $B_x = B(\Gamma_x, \rho)$, for some $\Gamma_x \subseteq \hat{G}_2$ of size at most $d$.\\

    \noindent\textbf{Defining sets $E$ and $X''$.} Take elements $e_1, \dots, e_r \in G_2$, independently and uniformly at random. Depending on $E  = \langle e_{1}, \dots, e_{r}\rangle_{\{0,1\}}$, we define a set $X'' \subseteq X'$ as well as the notion of bad additive quadruples in $X'$, as follows.\\
    \indent Define $X''$ to be the set of all $x \in X'$ such that all $e_{i}$ belong to $B(\Gamma_x, \rho/100r)$.\\
    \indent We say that an additive 16-tuples $x_{[16]}$ of elements in $X'$ is \textit{bad} if $|Z(\sum_{i \in [16]} (-1)^i\phi_{x_i})| \leq \eta |G_2|$ and $\langle e_{i} : i \in [r]\rangle_{\{-1,0,1\}} \cap Z\Big(\sum_{i \in [16]} (-1)^i\phi_{x_i}\Big) \not=\{0\}$. Let $Q_{\on{bad}}$ be the set of all bad additive 16-tuples in $X'$. Note that we define this as a property of elements of $X'$, independently of the fact whether they end up in $X''$.\\

    \begin{claim} Provided $\eta \leq \frac{\varepsilon}{4\cdot 3^r} (\rho/100r)^{r d}$ and $|G_2| \geq 8 \cdot 3^r (\rho/100r)^{-r d} $, there exist $e_1, \dots, e_r$ such that $|E| = 2^r$, $|X''| \geq (\rho/200r)^{r d} |G_1|$ and $|Q_{\on{bad}}| \leq \frac{\varepsilon}{2} |G_1|^{15}$.
    \end{claim}

    \begin{proof}
        Note that the probability that $x \in X'$ becomes element of $X''$ is at least
        \[(|B(\Gamma_x, \rho/100r| / |G_2|)^{ r} \geq (\rho/100r)^{r d},\]
        by Lemma~\ref{basicbohrsizel}. One the other hand, if an additive 16-tuple $x_{[16]}$ in $X'$ satisfies
        \[\Big|Z\Big(\sum_{i \in [16]} (-1)^i\phi_{x_i}\Big)\Big| \leq \eta |G_2|\]
        then the probability that it becomes bad is at most $3^r \eta$. Thus, by linearity of expectation,
        \[\exx\Big(\frac{|X''|}{|G_1|} - 2\varepsilon^{-1} \frac{|Q_{\on{bad}}|}{|G_1|^{15}}\Big) \geq (\rho/100r)^{ r d} -  2\varepsilon^{-1} 3^r \eta \geq \frac{1}{2}(\rho/100r)^{r d},\]
        since $\eta \leq \frac{\varepsilon}{4 \cdot 3^r} (\rho/100r)^{ r d}$.\\

        Hence, $\mathbb{P}\Big(\frac{|X''|}{|G_1|} - 2\varepsilon^{-1} \frac{|Q_{\on{bad}}|}{|G_1|^{15}} \geq \frac{1}{4}(\rho/100r)^{r d}\Big) \geq \frac{1}{4}(\rho/100r)^{r d}$. On the other hand, 
        \[\mathbb{P}(|E| = 2^r) \geq 1 - \sum_{\lambda, \mu \in \{0,1\}^r, \lambda\not=\mu} \mathbb{P}(\lambda \cdot e = \mu \cdot e) \geq 1 - \frac{3^r}{|G_2|}.\]
        
        The claim follows from these two probability inequalities and assumption $|G_2| \geq 8 \cdot 3^r (\rho/100r)^{-r d}$.
    \end{proof}
    
\vspace{\baselineskip}
    Let $e_1, \dots, e_r$ be given by the claim, and let $E, X''$ and $Q_{\on{bad}}$ be induced by that choice.\\

    \noindent\textbf{Defining homomorphism $\chi$.} Let $I$ be the collection of all $i \in [m]$ such that $Q_i$ has at least one additive 16-tuple $x_{[16]}$ in $X''$ that is not bad. For each $i \in I$, fix such a $16$-tuple $y^{(i)}_{[16]}$.\\
    \indent Using Lemma~\ref{charseparation} we may find $s \leq O(r + \log m)$ such that $\chi : H \to \mathbb{T}^s$ $\frac{1}{10}$-separates all elements of $\cup_{i \in I} \Big(\sum_{j \in [16]} (-1)^j\phi_{y^{(i)}_j}\Big)(E)$.\\

    \noindent\textbf{Defining set $\tilde{X}$.} Take $t \in \mathbb{T}^{[s] \times [r]}$ uniformly and independently at random and set 
    \[\tilde{X} = \Big\{x \in X'' : (\forall i \in [s], j \in [r])\,\,\uc{\chi_i(\phi_x(e_{j})) - t_{i,j}} \leq \frac{1}{10000 r}\Big\}.\]
    By linearity of expectation, we can choose $t$ so that $|\tilde{X}| \geq (10000r)^{-rs} |X''| \geq (10000r)^{-rs}(\rho/200r)^{r d} |G_1|$.\\
    \indent To finish the proof, we need to check that a vast majority of additive 16-tuples in $\tilde{X}$ give a linear combination of maps with a small image.

    \begin{claim}
        If $2^r > K'$ and $x_{[16]}$ is an additive 16-tuple in $\tilde{X}$ that is not bad, then
        \[\Big|Z\Big(\sum_{i \in [16]} (-1)^i\phi_{x_i}\Big)\Big| \geq \eta |G_2| .\]
    \end{claim}

    \begin{proof}
        Assume that, on the contrary, $\Big|Z\Big(\sum_{i \in [16]} (-1)^i\phi_{x_i}\Big)\Big| \leq \eta |G_2|$. Let $i$ be the index such that $x_{[16]} \in Q_i$. Since $x_{[16]}$ is not bad and its elements belong to $X''$, we have $i \in I$. Furthermore, also due to the fact that $x_{[16]}$ is not bad, we have $\langle e_{i} : i \in [r]\rangle_{\{-1,0,1\}} \cap Z\Big(\sum_{i \in [16]} (-1)^i\phi_{x_i}\Big) = \{0\}$.\\
        \indent As $x_{[16]} \in Q_i$, we have 
        \[\range\Big(\sum_{j \in [16]} (-1)^j \phi_{x_j} - \sum_{j \in [16]} (-1)^j \phi_{y^{(i)}_j}\Big) \leq K'.\]

        The choice of $X''$ guarantees that $e_i \in B(\Gamma_x, \rho/100r)$. Therefore, $\langle e_1,\dots, e_r\rangle_{\{-1,0,1\}} \subseteq B(\Gamma_x, \rho/100)$. In particular, $\phi_x$ is defined at all elements of $\langle e_1,\dots, e_r\rangle_{\{-1,0,1\}}$ for each $x \in X''$.\\
        
        By the pigeonhole principle and inequality $|E| = 2^r > K'$, we have two distinct elements $u, u' \in E$ such that 
        \[\sum_{j \in [16]} (-1)^j \phi_{x_j}(u) - \sum_{j \in [16]} (-1)^j \phi_{y^{(i)}_j}(u) = \sum_{j \in [16]} (-1)^j \phi_{x_j}(u') - \sum_{j \in [16]} (-1)^j \phi_{y^{(i)}_j}(u').\]
        Rearranging, we get
        \[\sum_{j \in [16]} (-1)^j \phi_{x_j}(u) - \sum_{j \in [16]} (-1)^j \phi_{x_j}(u')  = \sum_{j \in [16]} (-1)^j \phi_{y^{(i)}_j}(u) - \sum_{j \in [16]} (-1)^j \phi_{y^{(i)}_j}(u').\]
        
        Write $u - u' = \sum_{i \in [r]} \lambda_i e_i$ for some $\lambda_i \in \{-1,0,1\}$. By the definition of $\tilde{X}$, for each $j \in [s]$ we have 
        \begin{align*}\Big\|\chi_j\Big(\sum_{k \in [16]} (-1)^k \phi_{x_k}(u - u')\Big)\Big\|_{\mathbb{T}} = &\Big\|\chi_j\Big(\sum_{k \in [16], \ell\in [r]} (-1)^k \lambda_\ell\phi_{x_k}(e_\ell)\Big) \Big\|_{\mathbb{T}}\\
        \leq &\sum_{\ell \in [r]} \Big\| \sum_{k \in [16]} (-1)^k \chi_j(\phi_{x_k}(e_\ell)) \Big\|_{\mathbb{T}}\\
        \leq & r \cdot 16 \cdot \frac{1}{10000r},\end{align*}
        where we used $\|\chi_j(\phi_{x_k}(e_\ell)) - t_{j, \ell}\|_{\mathbb{T}} \leq 1/10000r$ in the last inequality.\\
        \indent This implies 
        \[\Big\|\chi\Big(\sum_{j \in [16]} (-1)^j \phi_{x_j}(u) - \sum_{j \in [16]} (-1)^j \phi_{x_j}(u')\Big)\Big\|_\infty \leq \frac{1}{100}.\]
        
        Since $i \in I$, $\chi$ $\frac{1}{10}$-separates elements of $\Big(\sum_{j \in [16]} (-1)^j\phi_{y^{(i)}_j}\Big)(E)$, and the inequality above gives 
        \[\sum_{j \in [16]} (-1)^j \phi_{y^{(i)}_j}(u) = \sum_{j \in [16]} (-1)^j \phi_{y^{(i)}_j}(u')\]
        so
        \[\sum_{j \in [16]} (-1)^j \phi_{x_j}(u) = \sum_{j \in [16]} (-1)^j \phi_{x_j}(u')\]
        and thus $u - u' \in Z\Big(\sum_{j \in [16]} (-1)^j \phi_{x_j}\Big)$. However, that implies that $x_{[16]}$ is bad, which is a contradiction.  
    \end{proof}

    To complete the proof, choose $r = \lceil \log_2(2K') \rceil$ and $\eta = \frac{\varepsilon}{4\cdot 3^r} (\rho/100r)^{r d}$ so that the required conditions on these parameters above hold, recall that $s \leq O(r + \log m)$, and apply Lemma~\ref{smallrangelemma}.
\end{proof}

\section{Small images of quadruples along a coset progression}\label{firstBRstepSection}

In this section, we use the robust Bogolyubov-Ruzsa theorem to strengthen Proposition~\ref{almostalladd16tupleshavesmallimage} by having a structured set instead of a merely dense set $X$. 

\begin{proposition}\label{smallimgoncosetprogressionlinsys}
    Suppose that we are given a set $X \subseteq G_1$ of density $c$, a collection of Freiman linear maps $\phi_x : B_x \to H$, where $B_x$ is a Bohr set in $G_2$ of codimension $d$ and radius $\rho$, such that 
    \begin{equation}\range\Big(\sum_{i \in [16]} (-1)^i \phi_{x_i}\Big) \leq K \label{bad16tuplesrbrstep}\end{equation}
    holds for all but at most $\varepsilon |G_1|^{15}$ additive 16-tuples $x_{[16]}$ in $X$.\\ 
    \indent Assume that $\varepsilon \leq \exp(-(2\log c^{-1})^{\blc})$. Then there exists a symmetric proper coset progression $C \subseteq G_1$ of size $\exp(-(2\log c^{-1})^{O(1)})|G_1|$ and rank at most $(2\log c^{-1})^{O(1)}$, Bohr sets $B'_x$ of codimension $O(d)$ and radius at least $\Omega(\rho)$ and Freiman-linear maps $\psi_x : B_x' \to H$ for all $x \in C$ such that 
    \[\range\Big(\psi_{x_1} - \psi_{x_2} + \psi_{x_3} - \psi_{x_4}\Big) \leq \exp((2d \log (\rho^{-1}K))^{O(1)})\]
    holds for all additive quadruples $(\upd{x_1}, x_2, \upd{x_3}, x_4)$ in $C$ and, for all $x \in C$,
    \[\range\Big(\psi_x - \phi_{x_1} + \phi_{x_2} - \dots - \phi_{x_7} + \phi_{x_8}\Big) \leq \exp((2d \log (\rho^{-1}K))^{O(1)})\]
    holds for at least $\exp(-(2\log c^{-1})^{O(1)})|G_1|^7$ 8-tuples $(x_1, x_2, \dots, x_8) \in X^8$ with $\sum_{i \in [8]} (-1)^{i + 1}x_i = x$.
\end{proposition}

\begin{proof}
    Apply Theorem~\ref{robustBogRuzsa} to $X$ to find a symmetric proper coset progression $C$ of rank $r \leq (2\log c^{-1})^{O(1)}$ and size $|C| \geq \exp(-(2\log c^{-1})^{O(1)})|G_1|$ such that for each $x \in C$ there are at least $c_1 |G_1|^3$ quadruples $(x_1, x_2, x_3, x_4)$ in $X$ such that $x = x_1 - x_2 + x_3 - x_4$, where $c_1 = (c/2)^{O(1)}$. Let $Q_x$ be the set of all such quadruples.\\
    \indent Let $S$ be the collection of additive 16-tuples in $X$ for which~\eqref{bad16tuplesrbrstep} fails. Thus $|S| \leq \varepsilon |G_1|^{15}$.\\
    \indent For each $x \in C$, define $\theta_x$ and $B'_x$ by taking $(x_1, \dots, x_4) \in Q_x$ uniformly at random and setting $\theta_x = \phi_{x_1} - \phi_{x_2} + \phi_{x_3} - \phi_{x_4}$ and $B'_x = B_{x_1} \cap B_{x_2} \cap B_{x_3} \cap B_{x_4}$.

    \begin{claim}\label{randomchoicethetadefinclaim}
        There exist a choice of $B'_x$ and $\theta_x$ for $x \in C$ such that for all but at most $\sqrt{\varepsilon} \exp((2\log c^{-1})^{O(1)}) |G_1|^3$ additive quadruples $(\upd{x_1}, x_2, \upd{x_3}, x_4)$ in $C$ we have
        \begin{itemize}
        \item[\textbf{(i)}] inequality
        \begin{equation}\range\Big(\theta_{x_1} - \theta_{x_2} + \theta_{x_3} - \theta_{x_4}\Big) \leq K, \label{phithetapsiK}\end{equation} and
        \item[\textbf{(ii)}] for each $i\in[4]$, $\range(\theta_{x_i} - \phi_{y_1} + \phi_{y_2} - \phi_{y_3} + \phi_{y_4}) \leq K^2$ holds for at least $\exp(-(2\log c^{-1})^{O(1)}) |G_1|^3$ quadruples $(y_1, \dots, y_4)$ in $X$ such that $y_1 - y_2 + y_3 - y_4 = x_i$.
        \end{itemize}
    \end{claim}

    \begin{proof}
        Each additive quadruple $(\upd{x_1}, x_2, \upd{x_3}, x_4)$ in $C$ gives rise to $c_1^4|G_1|^{12}$  $16$-tuples $(y_{ij})_{i,j \in [4]}$ such that $(y_{i1}, y_{i2}, y_{i3}, y_{i4}) \in Q_{x_i}$ and $\sum_{i \in [16]} (-1)^{i + j} y_{ij} = 0$. Let $S_{x_{[4]}}$ be such 16-tuples with
        \[\range\Big(\sum_{i \in [16]} (-1)^{i + j} \phi_{y_{ij}}\Big) \leq K.\]

        By assumptions of the proposition, all but $\sqrt{\varepsilon} \exp((2\log c^{-1})^{O(1)}) |G_1|^3$ additive quadruples in $C$ have $|S_{x_{[4]}}| \geq (1 - \sqrt{\varepsilon}) |Q_{x_1}||Q_{x_2}||Q_{x_3}||Q_{x_4}|$. Take any such additive quadruple. Then the probability that $\range\Big(\theta_{x_1} - \theta_{x_2} + \theta_{x_3} - \theta_{x_4}\Big) \leq K$ is at least $1 - \sqrt{\varepsilon}$, proving the claim. The other property follows similarly.
    \end{proof}

    Let $C'$ be a shrinking of $C$, which is still symmetric and proper, and such that $16C' \subseteq C$ and $|C'| \geq \exp(-(2\log c^{-1})^{O(1)})|G_1|$. Note that $|C \cap a+C| \geq |8C'|$ for all $a \in 8C'$. Hence, for each $a \in 8C'$ we have at least $\exp(-(2\log c^{-1})^{O(1)})|G_1|^2$ additive quadruples of difference $a$ in $C$.\\

    We say that an additive quadruples $(\upd{x}, y, \upd{z}, w)$ in $C$ is $(\theta, s, K')$-\textit{respected} if 
    \[\range\Big(\theta_{x} - \theta_{y} + \theta_z - \theta_w|_{B'_x \cap B'_{y} \cap B'_z \cap B'_w \cap B(\Gamma, \rho)}\Big) \leq K',\]
    holds for some Bohr set $B(\Gamma, \rho)$ of codimension at most $s$ (note that radius $\rho$ is the same as in the statement).\\
    \indent We say that a pair $(x,y) \in 4C' \times 4C'$ is \emph{good} if all but at most $\sqrt[4]{\varepsilon} |G_1|$ additive quadruples of the form $(\upd{x}, y, \upd{z}, w)$ in $C$ are $(\theta, 0, K)$-respected. Otherwise, we say it is bad.\\

    \begin{claim}
        Let $(\upd{x}, y, \upd{z}, w)$ be an additive quadruple of elements in $4C'$ such that $(x,y)$ and $(z,w)$ are both good. Then $(\upd{x}, y, \upd{z}, w)$ in $C$ is $(\theta, 8d, K^2)$-respected.
    \end{claim}

    \begin{proof}
        Since $x, y, z, w \in 4C'$, we have $a = x - y = w - z \in 8C'$. Hence, there are at least $\exp(-(2\log c^{-1})^{O(1)})|G_1|$ elements $u \in C \cap C - a$. Thus, there exists $u$ such that $(\upd{x}, y, \upd{u}, u+a)$ and $(\upd{u}, u+a, \upd{w}, z)$ are $(\psi, 0, K)$-respected and so 
        \[\range\Big(\theta_{x} - \theta_{y} + \theta_z - \theta_w|_{B'_x \cap B'_{y} \cap B'_z \cap B'_w \cap B'_{u + a} \cap B'_u}\Big) \leq K^2\]
        proving the claim.
    \end{proof}

    By averaging, all but at most $\sqrt[4]{\varepsilon}\exp(-(2\log c^{-1})^{O(1)})|G_1|^2$ pairs in $4C'$ are good. Hence, there exists a set $S \subseteq 4C'$ such that $|4C' \setminus S| \leq \sqrt[8]{\varepsilon}\exp(-(2\log c^{-1})^{O(1)})|G_1|$ and every $x \in S$ belongs to at most $\sqrt[8]{\varepsilon}|G_1|$ bad pairs. Moreover, we may assume that each $x \in S$ has the property \textbf{(ii)} of Claim~\ref{randomchoicethetadefinclaim}.\\

    \begin{claim}
        Every additive $8$-tuple $x_{[8]}$ of elements of $S$ is $(\psi, O(r), K^{O(1)})$-respected.
    \end{claim}

    \begin{proof}
        We first show that all additive quadruples in $S$ are $(\psi, 24d, K^4)$-respected. Given an additive quadruple$(\upd{x_1}, x_2, x_3,  \upd{x_4})$ in $S$, pick $y, z \in 4C'$ such that $(x_1, y), (z, x_2), (x_3, y)$ and $(z, x_4)$ are good and $x_1 -y = x_3 - z$. The previous claim implies that $(\upd{x_1}, y, \upd{z}, x_2)$ and $(\upd{x_3}, y, \upd{z}, x_4)$ are $(\psi, 8d, K^2)$-respected, proving the claim.\\

        Once all additive quadruples are respected, it is not hard to show that longer tuples are as well.
    \end{proof}
    
    We may now define $\psi_a$ for all $a \in C'$. We define $\psi_a$ as $\theta_{x + a} - \theta_x$ with domain $B''_a = B'_{x + a} \cap B'_x$ for arbitrary $x$ such that $x, x+a \in S$. Since all additive $8$-tuples are respected in $S$, we are done. Apply Lemma~\ref{smallrangelemma} to remove the extra Bohr sets in the condition.
\end{proof} 

\section{Obtaining many Bohr-respected additive quadruples}\label{obtmanybohrrespsection}

In this section, we perform the key change of perspective. So far, we have been considering systems of Freiman-linear maps on Bohr sets $\phi_x : B_x \to H$ with the notion of respected additive quadruples given by $\range\Big(\sum_{i \in [4]} (-1)^i \phi_{x_i}\Big) \leq K$. We now use a different notion, where we say that an additive quadruple $x_{[4]}$ is \textit{Bohr-respected} if  $(-1)^i \phi_{x_i}$ vanishes on $\cap_{i \in [4]} B_{x_i}$.

\begin{proposition}\label{changeofperspectivestep1}
    Suppose that $\phi_x : B_x \to H$ is a Freiman-linear map for each $x \in C$, where $C$ is a symmetric proper coset progression, all of radius $\rho$ and codimension at most $d$, such that
    \[\range\Big(\sum_{i \in 2k]} (-1)^i \phi_{x_i}\Big) \leq K\]
    holds for all additive $2k$-tuples $x_{[2k]}$ in $C$. Let $\varepsilon > 0$. \\
    \indent For $\ell \in [2,2k]$, let $Q_\ell \subseteq C^{2\ell}$ be a collection of some additive $2\ell$-tuples in $C$ of size $|Q_\ell| \geq c |C|^{2\ell - 1}$.\\
    \indent Then there exist Bohr sets $B'_x \subseteq B_x$ of codimension $(2d\log(k K\varepsilon^{-1} c^{-1}\rho^{-1}))^{O(1)}$ and radius $(2k d\log(\rho^{-1}))^{-O(1)}$ such that, for each $2\leq \ell \leq k$, $1-\varepsilon$ proportion of all additive $(2\ell)$-tuples $x_{[2\ell]}$ in $Q_\ell$  satisfy
         \[\sum_{i \in [2\ell]} (-1)^i \phi_{x_i} = 0\]
    on $\cap_{i \in [2\ell]} B'_{x_i}$.
\end{proposition}

\noindent\textbf{Remark.} Note that the parameter $\varepsilon$ only affects the codimension of the Bohr sets produced by the proposition.

\begin{proof}
    Let $m \in \mathbb{N}$ be a positive integer to be chosen later. Take random characters $\chi_1, \dots, \chi_m \in \hat{H}$ uniformly and independently. For each $x\in C$, define further sets 
    \[U_x = \{y \in B_x : (\forall i \in [m]) \,\,\chi_i(\phi_x(y)) \in (-1/20k, 1/20k)\}.\]
    Note that these are not Bohr sets by themselves, we shall show later that they contain the desired Bohr sets $B'_x$. Furthermore, note that $C$ is not modified by this choice.\\

    \begin{claim}
        For all additive $2\ell$-tuples $x_{[2\ell]}$ in $C$ we have 
            \[\mathbb{P}\Big(\sum_{i \in [2\ell]} (-1)^i \phi_{x_i}(y) = 0\text{ holds on }\cap_{i \in [2\ell]}U_{x_i}\Big) \geq 1 - K 2^{-m}.\]
    \end{claim} 

    \begin{proof}
        Let $A = \on{Im}\Big(\sum_{i \in [2\ell]} (-1)^i \phi_{x_i}\Big)$, which by assumptions has size at most $K$. Take an arbitrary non-zero $h \in A$. If $h$ is taken by the map $\sum_{i \in [2\ell]} (-1)^i \phi_{x_i}$ on $\cap_{i \in [2\ell]}U_{x_i}$, then there exists some $y \in G_2$ such that $\chi(\phi_{x_i]}(y) \in (-1/20k, 1/20k)^m$ for all $i \in [2\ell]$ and $h = \sum_{i \in [2\ell]} (-1)^i \phi_{x_i}(y)$. It follows that $\chi(h) \in (-1/10, 1/10)^m$. In other words, if $\chi(h) \notin (-1/10, 1/10)^m$ then $h$ is not attained by the considered linear combination on $\cap_{i \in [2\ell]}U_{x_i}$. Thus, the desired probability is at least
        \[1 - \sum_{h \in A \setminus \{0\}} \mathbb{P}\Big(\chi(h) \in (-1/10,1/10)^m\Big) \geq 1 - K 2^{-m}.\qedhere\]
    \end{proof}

    Let $E_\ell$ be those additive $2\ell$-tuples $x_{[2\ell]}$ in $Q_\ell$ for which  $\sum_{i \in [2\ell]} (-1)^i \phi_{x_i}$ does not vanish on $\cap_{i \in [2\ell]} U_{x_i}$. By linearity of expectation, we have
    \[\exx\Big(\sum_{\ell \in [2,k]} |E_\ell|/|C|^{2\ell - 1}\Big) \leq \sum_{\ell \in [2,k]}  K 2^{-m} |Q_\ell|/|C|^{2\ell - 1} \leq k K 2^{-m}.\]

    We may choose $m = O(\log(k K \varepsilon^{-1} c^{-1}))$ so that the last bound above becomes smaller than $\varepsilon c$. Hence, the desired proportion of additive tuples have the vanishing property on $\cap_{i \in [2\ell]} U_{x_i}$.\\

    It remains to show that sets $U_x$ contain large Bohr sets. Note that the map $\chi_i \circ \phi_x$ is a Freiman-linear map from $B_x$ to $\mathbb{T}$. Apply Lemma~\ref{bohrbohr} to finish the proof.
\end{proof}

\section{Bilinear Bogolyubov argument}\label{bilbogsection}

In the next step of the proof, we use a bilinear Bogolyubov type of the argument. The key idea is that, given a system of Bohr sets $(B_x)_{x \in C}$, a new system given by $(B_{x_a+a} \cap B_{x_a}) + (B_{y_a+a} \cap B_{y_a})$ exhibits linear behaviour in $a$, i.e. the frequency set of the resulting Bohr set is given by Freiman-linear maps evaluated at $a$.\\

Before stating the main result, we need a preliminary lemma, which generalizes the simple linear-algebraic fact that two linear maps $\alpha_1$ and $\alpha_2$ defined on subspaces $U_1$ and $U_2$ of a vector space $V$ that agree on the intersections of their domains $U_1 \cap U_2$ can simultaneously be extended to a linear map on $U_1 + U_2$.\\
\indent Recall the notation $B^{(\delta)}$ for a Bohr set $B = B(\Gamma, \rho)$, which is the shorthand for $B(\Gamma, \delta \rho)$.

\begin{lemma}\label{flinearextn}
    Suppose that $\phi_1 : B_1 \to H$ and $\phi_2 : B_2 \to H$ are Freiman-linear maps on Bohr sets $B_1$ and $B_2$. Suppose that $\phi_1 = \phi_2$ holds on $B_1 \cap B_2$. Then there exists a Freiman-linear map $\psi : D = B_1^{(1/3)} + B_2^{(1/3)} \to H$ such that $\psi = \phi_1$ on $D \cap B_1^{(2/3)}$ and $\psi = \phi_2$ on $D \cap B_2^{(2/3)}$.
\end{lemma}

\begin{proof}
    We define $\psi$ as follows. Let $d \in D$ be a arbitrary. There exist $x \in B_1^{(1/3)}$ and $y \in B_2^{(1/3)}$ such that $d = x + y$. We set $\psi(d) = \phi_1(x) + \phi_2(y)$. In the rest of the proof we show that $\psi$ is well-defined, coincides with $\phi_1$ and $\phi_2$ on suitable sets and is Freiman-linear.\\
    \indent Suppose that $d = x' + y'$ for another choice of $x' \in B_1^{(1/3)}$ and $y' \in B_2^{(1/3)}$. Then $x - x' = y' - y \in B_1 \cap B_2$ so $\phi_1(x) - \phi_1(x') = \phi_1(x - x') = \phi_2(y' - y) = \phi_2(y') - \phi_2(y)$. Thus $\phi_1(x) + \phi_2(y) = \phi_1(x') + \phi_2(y')$, as desired.\\
    \indent Let now $z \in D \cap B_1^{(2/3)}$. Then $z = x + y$ for $x \in B_1^{(1/3)}$ and $y \in B_2^{(1/3)}$ and we have $\psi(z) = \phi_1(x) + \phi_2(y)$. But $y = z - x \in B_1$, so $\phi_2(y) = \phi_1(y)$ and thus $\psi(z) = \phi_1(x) + \phi_1(y) = \phi_1(x + y) = \phi_1(z)$. A similar argument shows that $\psi = \phi_2$ on $D \cap B_2^{(2/3)}$.\\
    \indent Let us now show that $\psi$ is Freiman-linear. To that end, let $d_1 + d_2 = d_3$ hold for some $d_1, d_2, d_3 \in D$. Hence, we have $x_1, x_2, x_3 \in B_1^{(1/3)}$ and $y_1, y_2, y_3 \in B_2^{(1/3)}$ such that $x_i + y_i = d_i$. In particular, $x_1 + x_2 - x_3 = y_3 - y_1 - y_2$ is an element of $B_1 \cap B_2$, so $\phi_1(x_1 + x_2 - x_3) = \phi_2(y_3 - y_1  - y_2)$. Using Freiman-linearity, we get
    \begin{align*}\phi_1(x_1) + \phi_1(x_2) - \phi_1(x_3) = &\phi_1(x_1 + x_2) - \phi_1(x_3) = \phi_1(x_1 + x_2 - x_3) \\
    =& \phi_2(y_3 - y_1 - y_2) = \phi_2(y_3) - \phi_2(y_1 + y_2) = \phi_2(y_3) - \phi_2(y_1) - \phi_2(y_2),\end{align*}
    so 
    \[\psi(d_1) + \psi(d_2) = \phi_1(x_1) + \phi_2(y_1) + \phi_1(x_2) + \phi_2(y_2) = \phi_1(x_3) + \phi_2(y_3) = \psi(d_3).\qedhere\]
\end{proof}

Furthermore, we need a restricted version of the structure theorem for approximate homomorphisms.

\begin{lemma}\label{restrictedfreiman} Suppose that $\psi_{[8]}, \psi'_{[8]}$ are maps from a set $C\subseteq G$ to $G'$. Suppose that we have a collection $Q$ of parameters $(a_1, a_2, a_3, x_1, x_2, x_3, x_4, y_1, y_2, y_3, y_4)$ of size at least $c|G|^{11}$ each of which satisfies
\begin{align*}&\psi_1(x_1 + a_1) - \psi_2(x_1) + \psi_3(x_2 + a_2) - \psi_4(x_2) + \psi_5(x_3 + a_3) - \psi_6(x_3) + \psi_7(x_4 + a_1 + a_2 - a_3) - \psi_8(x_4)\\
&\hspace{1cm}=\psi'_1(y_1 + a_1) - \psi'_2(y_1) + \psi'_3(y_2 + a_2) - \psi'_4(y_2) + \psi'_5(y_3 + a_3) - \psi'_6(y_3) + \psi'_7(y_4 + a_1 + a_2 - a_3) - \psi'_8(y_4).\end{align*}

Then there exist Freiman homomorphisms $\theta_i : C'_i \to G$, where $C'_i$ is a coset progression of size $\exp(-(2\log c^{-1})^{O(1)})|G|$ and rank $(2\log c^{-1})^{O(1)}$, for $i \in [8]$, and elements $u_{[8]}$ in $G$, such that the following 16 equalities hold
\begin{align*}&\psi_1(x_1 + a_1) = \theta_1(x_1 + a_1), \psi_2(x_1) = \theta_1(x_1) + u_1, \psi_3(x_2 + a_2) = \theta_2(x_2 + a_2), \psi_4(x_2) = \theta_2(x_2) + u_2,\\
&\psi_5(x_3 + a_3) = \theta_3(x_3 + a_3), \psi_6(x_3) = \theta_3(x_3) + u_3,\\
&\hspace{3cm}\psi_7(x_4 + a_1 + a_2 - a_3) = \theta_4(x_4 + a_1 + a_2 - a_3), \psi_8(x_4) = \theta_4(x_4) + u_4,\\
&\psi'_1(y_1 + a_1) = \theta_5(y_1 + a_1), \psi'_2(y_1) = \theta_5(y_1) + u_5, \psi'_3(y_2 + a_2) = \theta_6(y_2 + a_2), \psi'_4(y_2) = \theta_6(y_2) + u_6,\\
&\psi'_5(y_3 + a_3) = \theta_7(y_3 + a_3), \psi'_6(y_3) = \theta_7(y_3) + u_7,\\
&\hspace{3cm}\psi'_7(y_4 + a_1 + a_2 - a_3) = \theta_8(y_4 + a_1 + a_2 - a_3), \psi'_8(y_4) = \theta_8(y_4) + u_8,\end{align*}
for at least $\exp(-(2\log c^{-1})^{O(1)})|G|^{11}$ of 11-tuples in $Q$ (and arguments belong to domain of respective functions).
\end{lemma}

\begin{proof}
    Firstly, we prove that each $\psi_i$ can be assumed to be a Freiman homomorphism on a suitable domain, and then we conclude that some of these maps are closely related.\\
    Since $|Q| \geq c |G|^{11}$, the set $Z$ of values $z$ such that $z = x_1 + a_1$ for at least $\frac{c}{2}|G|^{10}$ 11-tuples $(a_1, \dots, y_4)$ in $Q$ has size at least $\frac{c}{2}|G|$. Pass to those 11-tuples $Q'$ whose $x_1 + a_1 \in Z$, so $|Q'| \geq \frac{c^2}{4}|G|^{11}$. By Cauchy-Schwarz inequality, we have at least $\frac{c^4}{16}|G|^{12}$ choices of $(a_1, a_2, a_3, x_1, x_1', x_2, \dots, y_4)$ such that 11-tuples $(a_1, \dots, y_4)$ and $(a_1, a_2, a_3, x_1', x_2, \dots, y_4)$ both belong to $Q'$, and in particular $x_1 + a_1, x'_1 + a_1 \in Z$. Subtracting equalities for these 11-tuples from one another we get
    \[\psi_1(x_1 + a_1) - \psi_2(x_1) = \psi_1(x'_1 + a_1) - \psi_2(x'_1).\]
    Hence, this equality holds for at least $\frac{c^4}{16}|G|^{3}$ choices of $x_1, x'_1, a_1$ and additionally $x_1 + a_1, x'_1 + a_1 \in Z$. Another Cauchy-Schwarz step shows that
    \[\psi_1(x_1 + a_1) - \psi_1(x_1 + b_1) = \psi_1(x'_1 + a_1) - \psi_1(x'_1 + b_1)\]
    holds for at least $2^{-8}c^8 |G|^{4}$ of $(x_1, x'_1, a_1, b_1) \in G^4$. Apply Theorem~\ref{approxFreimanHom} to find a Freiman homomorphism $\rho_1$ defined on a proper coset progression of rank at most $(2\log c^{-1})^{O(1)}$ which coincides with $\psi_1$ on a set $Z' \subseteq Z$ of size $|Z'| \geq \exp(-(2\log c^{-1})^{O(1)})|G|$. Pass to those 11-tuples whose $x_1 + a_1 \in Z'$, so we still have at least $\exp(-(2\log c^{-1})^{O(1)})|G|^{11}$ such 11-tuples.\\

    The same argument applies to any other choice of $\psi_i$ or $\psi'_i$, as the equation is symmetric. Hence, after 16 steps in total, we may assume that there exist Freiman homomorphisms on proper coset progression of rank at most $(2\log c^{-1})^{O(1)}$ $\rho_i : C_i'\to G'$ and $\rho'_i : C''_i \to G'$ and a collection of 11-tuples $Q' \subseteq Q$ such that $|Q'| \geq \exp(-(2\log c^{-1})^{O(1)})|G|^{11}$ and for each $(a_1, \dots, y_4) \in Q'$ we additionally have
    \[\psi_1(x_1 + a_1) = \rho_1(x_1 + a_1), \dots, \psi'_8(y_4) = \rho'_8(y_4).\]

    It remains to relate some pairs of $\rho_i, \rho'_i$. The same first step as above shows
    \[\rho_1(x_1 + a_1) - \rho_2(x_1) = \rho_1(x'_1 + a_1) - \rho_2(x'_1)\]
    for $\exp(-(2\log c^{-1})^{O(1)})|G|^3$ choices of $x_1, x'_1, a_1$. Hence $\on{rk} (\rho^{\text{lin}}_1 - \rho^{\text{lin}}_2) \leq O(\log c^{-1})$, where $\rho^{\text{lin}}_i$ is the Freiman-linear part, i.e., $\rho_i - \rho_i(0)$. Taking $K = \on{ker} (\rho^{\text{lin}}_1 - \rho^{\text{lin}}_2)$ and averaging, we may find a further subset $Q'' \subseteq Q'$ of size $|Q''| \geq (c/2)^{O(1)}|G|^{11}$, where additionally all $x_1$ belong to the same coset $s + K$. Set $u_1 = \rho_2(s) - \rho_1(s)$. Hence, whenever an 11-tuple $(a_1, \dots, y_4)$ belongs to $Q''$, we have
    \[\psi_2(x_1) = \rho_2(x_1) = \rho_2^{\text{lin}}(x_1 - s) + \rho_2(s) = \rho_1^{\text{lin}}(x_1 - s) + \rho_1(s) + u_1 = \rho_1(x_1) + u_1,\]
    using the fact that $x_1 - s \in K$, as claimed.\\
    \indent Finally, apply analogous argument to the remaining 7 pairs of affine maps among $\rho_3, \dots, \rho'_8$.
\end{proof}

We are now ready to state and prove the main result of this section. It show that we may pass to a further system of Freiman-linear maps on Bohr sets, in which the frequency sets of Bohr sets exhibit linear behaviour.

\begin{proposition}\label{bilinearBogolyubovStep}
    Let $\varepsilon \leq 10^{-10d}$. Let $C \subseteq G_1$ be a symmetric proper coset progression of rank $d$. Define a shrinking $C^0 = \frac{1}{10} \cdot C$, which is a symmetric proper coset progression. For $\ell \in \{2,3,4\}$, let $Q_\ell$ be the set of additive $2\ell$-tuples $x_{[2\ell]}$ in $C$, with the property that $x_{2i-1} - x_{2i} \in C^0$.\\
    \indent For each $x \in C$, let $B_x$ be a Bohr set of radius $\rho$ and codimension $d$ and $\phi_x : B_x \to H$ a Freiman-linear map such that, for each $\ell \in \{2,3,4\}$, $1 - \varepsilon$ proportion of additive $2\ell$-tuples in $Q_\ell$ satisfy
    \[\Big(\sum_{i \in [2\ell]} (-1)^i \phi_{x_i}\Big)|_{\cap_{i \in [2\ell]} B_{x_i}}= 0.\]
    Then, there exist quantities $d'$ and $c'$, a set $X$, Frieman homomorphisms $\theta_1, \dots, \theta_{d'}$ on a coset progression $C'$ of rank at most $d'$ and density at least $c'$, a set $X \subseteq C'$ of size at least $c'|G_1|$, Freiman-linear maps $\psi_a : V_a \to H$ for $a \in X$, where $V_a = B(\theta_i(a) : i \in [d'], c')$ such that 
    \begin{itemize}
        \item for each $a \in X$ we have $|Z(\psi_a - \phi_{z+a} - \phi_z)| \geq (\rho/4)^{4d}|G_2|$ for at least $c' |G_1|$ choices of $z \in C \cap C-a$, 
        \item we have at least $c' |X|^3$ additive quadruples $x_{[4]}$ in $X$ that are Bohr-respected for the system $(\psi_a : V_a \to H)$.
    \end{itemize}
\end{proposition}

\begin{proof}
    Let $B_x = B(\Gamma_x, \rho)$, where $\Gamma_x \subseteq \hat{G}_1$ is a set of characters of size $d$. Note also that $|Q_\ell| \geq 10^{-d \ell}|C|^{2\ell - 1}$ for each $\ell \in \{2,3,4\}$.\\

    Define $R$ to be the quantity from Theorem~\ref{bohrSum}, when that theorem is applies to two Bohr sets of codimension $8d$ and radius $\rho$, so $R \leq (2\rho^{-1})^{O(d)}$. During the proof, we shall keep track of a list of Freiman homomorphisms $\theta_1 : C_1 \to \hat{G}_2, \dots, \theta_m : C_m \to \hat{G}_2$, defined on coset progressions $C_i$, adding one new $\theta_i$ in each step, and, for each $x,y \in C$, we shall maintain a set of indices $I_{x,y} \subseteq [m]$ such that the set $\{ \theta_i(x - y) : i \in I_{x,y}\}$, with $x - y \in C_i$ for all $i \in I_{x,y}$, is $\{-1,0,1\}$-independent and
    \[\{ \theta_i(x - y) : i \in I_{x,y} \} \subseteq \langle \Gamma_x \cup \Gamma_y \rangle_R.\]
    
    This inclusion allows us to give a bound on the size of $I_{x,y}$. Write $s = |I_{x,y}|$. Namely, owing to $\{-1,0,1\}$-independence, we have $|\langle \theta_i(x - y) : i \in I_{x,y}\rangle_{\{0,1\}}| = 2^s$ and 
    \[\langle \theta_i(x - y) : i \in I_{x,y}\rangle_{\{0,1\}} \subseteq \langle \Gamma_x \cup \Gamma_y \rangle_{sR}.\]

    Thus, $2^s \leq (2sR + 1)^{2d}$. Writing $s_0$ for the least natural number with $2^{s_0} > (2s_0R + 1)^{2d}$, we get so $|I_{x,y}| \leq s_0$ and $s_0\leq (d + \log \rho^{-1})^{O(1)}$.
    
    \begin{claim}\label{claim1bilbogarg}
        Let $0 < \sigma < 1/8s_0$. Suppose that currently 
        \begin{equation}
        B(\theta_i(a) : i \in I_{x + a, x} \cup I_{y + a, y}, \sigma) \subseteq (B_{x + a} \cap B_x) + (B_{y + a} \cap B_y)
            \label{firstbogclaimeqn}
        \end{equation}
        \textbf{\upshape{fails}} for at least $\delta|C|^3$ triples $(x, y, a)$ such that $x,y, x+a, y+a \in C$. Then there exist a coset progression $C'$ of rank $(2d\log (\delta^{-1} R))^{O(1)}$ and a Freiman homomorphism $\theta' : C' \to \hat{G}_1$ such that 
        \[\theta'(x - y) \in \langle \Gamma_x \cup \Gamma_y \rangle_R \setminus \langle \theta_i(x - y) : i \in I_{x,y} \rangle_{\{-1,0,1\}}\]
        holds for at least $\exp(-(2d\log (\delta^{-1} R))^{O(1)})|C|^2$ pairs $(x,y)$ in $C^2$.
    \end{claim}

    \begin{proof}[Proof of Claim~\ref{claim1bilbogarg}]
        An application of Theorem~\ref{bohrSum} gives
        \[B(\langle \Gamma_{x + a} \cup \Gamma_x\rangle_R \cap \langle \Gamma_{y + a} \cup \Gamma_y\rangle_R, 1/4) \subseteq (B_{x + a} \cap B_x) + (B_{y + a} \cap B_y),\]
        (recalling that $R$ is defined with application of Theorem~\ref{bohrSum} in mind).\\
        
        By triangle inequality, since $\sigma \leq \frac{1}{4 (|I_{x + a, x}| + |I_{y + a, y}|)}$, we have
        \[B(\theta_i(a) : i \in I_{x + a, x} \cup I_{y + a, y}, \sigma) \subseteq B(\langle \theta_i(a) : i \in I_{x + a, x} \cup I_{y + a, y}\rangle_{\{-1,0,1\}}, 1/4).\]
        We conclude that 
        \[B(\langle \theta_i(a) : i \in I_{x + a, x} \cup I_{y + a, y}\rangle_{\{-1,0,1\}}, 1/4) \not\subseteq B(\langle \Gamma_{x + a} \cup \Gamma_x\rangle_R \cap \langle \Gamma_{y + a} \cup \Gamma_y\rangle_R, 1/4)\]
        and thus
        \[\langle \Gamma_{x + a} \cup \Gamma_x\rangle_R \cap \langle \Gamma_{y + a} \cup \Gamma_y\rangle_R \not\subseteq \langle \theta_i(a) : i \in I_{x + a, x} \cup I_{y + a, y}\rangle_{\{-1,0,1\}}.\]

        Define maps $\psi_1, \psi_2, \psi_3, \psi_4 : C \to \hat{G}_2$ as follows. For each $x \in C$, we choose a linear combination $\lambda \in [-R, R]^{\Gamma_x}$ uniformly at random, and independently for different elements $x$, and set $\psi_i(x) = \sum_{\gamma \in \Gamma_x} \lambda_\gamma \gamma$. Thus, when $x, x+a, y, y+a$ are distinct elements, the probability of 
        \begin{equation}\psi_1(x + a) - \psi_2(x) = \psi_3(y + a) - \psi_4(y) \notin \langle\theta_i(a) : i \in I_{x + a, x} \cup I_{y + a, y}\rangle_{\{-1,0,1\}}\label{addquadspsicondmissingspan}\end{equation}
        occurring is at least $(2R+1)^{-4d}$.\\

        Hence, there exists a choice of functions $\psi_1, \psi_2, \psi_3, \psi_4 : C \to \hat{G}_2$ such that~\eqref{addquadspsicondmissingspan} holds for at least $\frac{\delta}{2}(2R+1)^{-4d}|C|^3$ quadruples in $C$. The claim follows from Lemma~\ref{restrictedfreiman}.
    \end{proof}

    We need a similar result for more complicated arrangements of indexing elements.

    \begin{claim}\label{claim2bilbogarg}
        Let $0 < \sigma < 1/32s_0$. Suppose that currently 
        \begin{equation}\bigcap_{j \in [4]} B(\theta_i(a_j) : i \in I_{x_j + a_j, x_j} \cup I_{y_j + a_j, y_j}, \sigma) \subseteq \Big(\bigcap_{j \in [4]}(B_{x_j + a_j} \cap B_{x_j})\Big) + \Big(\bigcap_{j \in [4]}(B_{y_j + a_j} \cap B_{y_j})\Big)\label{secondbogclaimeqn}\end{equation}
        \textbf{\upshape{fails}} for at least $\delta|C|^{11}$ choices of 12-tuples $(x_{[4]}, y_{[4]}, a_{[4]})$ with $a_1 + a_2 = a_3 + a_4$. T Then there exist a coset progression $C'$ of rank $(2d\log (\delta^{-1} R))^{O(1)}$ and a Freiman homomorphism $\theta' : C' \to \hat{G}_1$ such that 
        \[\theta'(x - y) \in \langle \Gamma_x \cup \Gamma_y \rangle_R \setminus \langle \theta_i(x - y) : i \in I_{x,y} \rangle_{\{-1,0,1\}}\]
        holds for at least $\exp(-(2d\log (\delta^{-1} R))^{O(1)})|C|^2$ pairs $(x,y)$ in $C^2$.
    \end{claim}

    \begin{proof}
        The proof is similar to the previous one, so we only stress the modifications here. We apply Theorem~\ref{bohrSum} to get
        \[B\Big(\langle \bigcup_{j \in [4]}\Gamma_{x_j + a_j} \cup \Gamma_{x_j}\rangle_R \cap \langle \bigcup_{j \in [4]}\Gamma_{y_j + a_j} \cup \Gamma_{y_j}\rangle_R, 1/4\Big) \subseteq  \Big(\bigcap_{j \in [4]}(B_{x_j + a_j} \cap B_{x_j})\Big) + \Big(\bigcap_{j \in [4]}(B_{y_j + a_j} \cap B_{y_j})\Big).\]
        
        By triangle inequality, since $\sigma \leq \frac{1}{4 \Big(\sum_{j \in [4]} |I_{x_j + a_j, x_j}| + |I_{y_j + a_j, y_j}|\Big)}$, we have
        \[\bigcap_{j \in [4]} B(\theta_i(a_j) : i \in I_{x_j + a_j, x_j} \cup I_{y_j + a_j, y_j}, \sigma) \subseteq B(\langle \theta_i(a_j) : j \in [4], i \in I_{x_j + a_j, x_j} \cup I_{y_j + a_j, y_j}\rangle_{\{-1,0,1\}}, 1/4).\]
        We conclude that 
        \[B(\langle \theta_i(a_j) : j \in [4], i \in I_{x_j + a_j, x_j} \cup I_{y_j + a_j, y_j}\rangle_{\{-1,0,1\}}, 1/4) \not\subseteq B\Big(\langle \bigcup_{j \in [4]}\Gamma_{x_j + a_j} \cup \Gamma_{x_j}\rangle_R \cap \langle \bigcup_{j \in [4]}\Gamma_{y_j + a_j} \cup \Gamma_{y_j}\rangle_R, 1/4\Big)\]
        and thus
        \[\Big\langle \bigcup_{j \in [4]}\Gamma_{x_j + a_j} \cup \Gamma_{x_j}\Big\rangle_R \cap \Big\langle \bigcup_{j \in [4]}\Gamma_{y_j + a_j} \cup \Gamma_{y_j}\Big\rangle_R \not\subseteq \langle \theta_i(a_j) : j \in [4], i \in I_{x_j + a_j, x_j} \cup I_{y_j + a_j, y_j}\rangle_{\{-1,0,1\}}.\] 
        
        The rest of proof is analogous and follows from Lemma~\ref{restrictedfreiman}.
    \end{proof}

    Apply repeatedly Claims~\ref{claim1bilbogarg} and~\ref{claim2bilbogarg} until $1-\varepsilon$ proportion of additive quadruples in $C$  satisfy~\eqref{firstbogclaimeqn} and $1- \varepsilon$ proportion of the described 12-tuples in $C$ satisfy~\eqref{secondbogclaimeqn}.\\

    Recall that $1-\varepsilon$ proportion of additive quadruples in $Q_2$ are Bohr-respected. Note that $|C| \geq |C \cap C - a| \geq 2^{-d}|C|$ for $a \in C^0$. Hence, we can find a set $A \subseteq C^0$ of size at least $(1- 2^{2d}\sqrt{\varepsilon})|C^0|$ such that for each $a \in A$, the number of additive quadruples $(x + a, x, y + a, y)$ which are Bohr-respected is at least $(1-\sqrt[4]{\varepsilon})|C \cap C - a|^2$. Additionally, for $a \in A$, let $X_a$ be the set of all $x \in C \cap C - a$ such that $(x + a, x, y + a, y)$ is Bohr-respected for at least $(1-\sqrt[8]{\varepsilon})|C \cap C - a|$ of $y \in C \cap C - a$. Then $|X_a| \geq (1-\sqrt[8]{\varepsilon})|C \cap C - a|$.\\ 
    \indent Furthermore, let $A' \subseteq A$ be the set of $a \in A$ such that~\eqref{firstbogclaimeqn} holds for at least $(1 - 2\sqrt{\varepsilon})|C \cap C - a|^2$ of pairs $(x,y) \in (C \cap C -a)^2$. Hence, $|A'| \geq (1 - 2^{2d+1}\sqrt{\varepsilon})|C^0|$.\\
    
    Take random pair $(x_a, y_a)$ uniformly from $(C \cap C - a)^2$, independently for each $a \in A'$. We say that $(x_a, y_a)$ is \textit{good} if 
    \begin{itemize}
        \item $(x_a + a, x_a, y_a + a, y_a)$ is Bohr-respected,
        \item $x_a, y_a \in X_a$, and
        \item $B(\theta_i(a) : i \in I_{x_a + a, x_a} \cup I_{y_a + a, y_a}, \sigma) \subseteq (B_{x_a + a} \cap B_{x_a}) + (B_{y_a + a} \cap B_{y_a})$.
    \end{itemize} 
    The definition of $A'$ implies that the probability $(x_a, y_a)$ being good is at least $1 - 2^{2d + 3}\sqrt[8]{\varepsilon}$ for any given $a \in A'$.\\

    If these conditions hold, we may apply Lemma~\ref{flinearextn} to define a Freiman-linear map $\psi_a : U_a = B(\theta_i(a) : i \in I_{x + a, x} \cup I_{y + a, y}, \sigma/4) \to \hat{H}$ which coincides with $\phi_{x_a+a} - \phi_{x_a}$ on $B_{x_a + a}^{(2/3)} \cap B_{x_a}^{(2/3)} \cap U_a$ and with $\phi_{y_a+a} - \phi_{y_a}$ on $B_{y_a + a}^{(2/3)} \cap B_{y_a}^{(2/3)} \cap U_a$.\\

    \noindent\textbf{Relating $\psi_\bcdot$ back to $\phi_\bcdot$.} Since $x_a \in X_a$, we have that $(x_a + a, x_a, z + a, z)$ is Bohr-respected for $\frac{1}{2}|C \cap C - a|$ of $z \in C \cap C - a$, so it follows that $|Z(\psi_a - \phi_{z+a} - \phi_z)| \geq (\rho/4)^{4d}|G_2|$.\\

    Let us now show that the expected number of Bohr-respected additive quadruples in $(\psi_a, U_a)$ is large. Let $Q$ be the set of all quadruples $a_{[4]}$ in $A'$ such that $a_1 + a_2 = a_3 + a_4$, such that additive 8-tuples $(x_1 + a_1, x_1, \dots, x_4 + a_4, x_4)$ are Bohr-respected (with respect to system $\phi_\bcdot$) for at least $(1-\sqrt{\varepsilon})\prod_{i \in [4]} |C \cap C-a_i|$ quadruples $x_{[4]}$ in $\prod_{i \in [4]} (C \cap C-a_i)$, and such that~\eqref{secondbogclaimeqn} holds for at least $(1-\sqrt{\varepsilon})\prod_{i \in [4]} |C \cap C-a_i|^2$ of $x_{[4]}, y_{[4]}$ in $(\prod_{i \in [4]} |C \cap C-a_i|) \times (\prod_{i \in [4]} |C \cap C-a_i|)$. Thus $|Q| \geq (1-2^{8d}\sqrt[4]{\varepsilon})|C^0|^3$. It follows that the expected number of Bohr-respected additive quadruples in $(\psi_a, U_a)$ is at least $(1-2^{9d}\sqrt[4]{\varepsilon})|C^0|^3$. We choose $\psi_a : U_a \to H$ for $a \in A'$ so that the described properties hold.\\

    Finally, we pass to a subset $\tilde{A} \subseteq A'$ on which $U_a$ exhibits linear behaviour. Since $U_a = B(\theta_i(a) : i \in I_{x + a, x} \cup I_{y + a, y}, \sigma/4)$ and $| I_{x + a, x} \cup I_{y + a, y}| \leq 2s_0$, let us pick $J \subseteq [m]$ of size $8s_0$ uniformly at random among all such sets, and let $\tilde{A}$ be the collection of all $a \in A'$ such that $I_{x + a, x} \cup I_{y + a, y} \subseteq J$. Thus, probability that $a$ is chosen is at least $\binom{m}{8s_0}^{-1}$, so by the linearity of expectation, there is a choice of $J$ such that $|\tilde{A}| \geq \binom{m}{8s_0}^{-1}|A'| \geq m^{-8s_0}|A'|$. Let us take $U'_a = B(\theta_i(a) : i \in J, \sigma/4)$ for $a \in \tilde{A}$ as the domain of $\psi_a$ to finish the proof.
\end{proof}

\section{Bilinear Bohr variety with many Bohr-respected additive quadruples}\label{bilbohrmanybohrrespsection}

In this step, we use the abstract Balog-Szemer\'edi-Gowers theorem again to find a subset of indices in which all additive quadruples are Bohr-respected. However, this step is more complicated than the previous one in Section~\ref{almostalladditive16tuplesABSG1} involving the abstract Balog-Szemer\'edi-Gowers theorem, as we need to ensure that the domains of the Freiman-linear maps in our system form a quasirandom bilinear Bohr variety, so we rely on the algbraic regularity lemma.

\begin{proposition}\label{absgStep2bilBohrvar}
    Let $C \subseteq G_1$ be a coset progression of rank $d$, $B = B(\Gamma, \rho) \subseteq G_2$ a Bohr set and let $\Theta_1, \dots, \Theta_r : C \to \hat{G}_2$ be Freiman homomorphisms.
    Let $V_x = B \cap B(\Theta_1(x), \dots, \Theta_r(x); \rho)$.
    Suppose that we are also given Freiman-linear maps $\phi_x : V_x \to H$ for each $x \in X$, $X \subseteq C$, such that at least $c |C|^3$ additive quadruples in $X$ are Bohr-respected.
    Then there exists a set $X' \subseteq X$ of size $\exp(-(2dr\log \rho{^-1}))|C|$ and a further Bohr set $B'$ such that all additive quadruples in $X'$ are respected by the system $\phi_x|_{B' \cap V_x}$.
\end{proposition}

\begin{proof}
    Let $\eta > 0$ be a parameter to be chosen later. We begin the proof by applying the algebraic regularity lemma (Theorem~\ref{algreglemmaintro}) for error parameter $\eta$. We obtain a further proper coset progression $C'$ of rank at most $d$ and a set $T$ of size at most
\[\exp\Big(d^{O(1)} r^{O(1)} \log^{O(1)} (\eta^{-1}) \log^{O(1)} (\rho^{-1}) \Big),\]
such that $C' + T \subseteq C$, $|C'| |T| = |C' + T| \geq (1 - \eta) |C|$ and every $t + C'$ induces a quasirandom piece of the bilinear Bohr variety in the sense of the theorem. Let $X_t = X \cap t + C$ for $t \in T$. Let $Q$ be the set of all Bohr-respected additive quadruples in $X$. Let $\mathcal{T}$ be the set of all quadruples $(t_1, \dots, t_4) \in T^4$ such that there exists an element of $Q$ inside $(t_1 + C') \times (t_2 + C') \times (t_3 + C') \times (t_4 + C')$. We claim that $|\mathcal{T}| \leq 2^{5d} |T|^3$. Otherwise, we have a choice of $t_2, t_3, t_4 \in T$ such that $(t_1, t_2, t_3, t_4) \in \mathcal{T}$ for a set $T'$ of possible $t_1$ of size greater than $2^{5d}$. But then $t_1 - t_2 + t_3 - t_4 \in 4C'$ for each $t_1 \in T'$. Since $t_1 + C'$ are disjoint for different $t_1$, and are contained in $5C' + t_2 - t_3 + t_4$, we have $|T'| \leq 2^{5d}$, which is a contradiction. Hence, we have a choice of $t_1, t_2, t_3, t_4$ with at least $2^{-5d}c|C'|^3$ additive quadruples in $(t_1 + C') \times (t_2 + C') \times (t_3 + C') \times (t_4 + C')$. But redefining the Bohr set $B$ slightly allows us to assume that all $t_i$ are equal. Namely, take $x_3 \in t_3 + C'$ and $x_4 \in t_4 + C'$ such that $(x_1, x_2, x_3, x_4)$ is Bohr-respected for many choices of $(x_1, x_2)$. Then taking another such a pair, we have
\[\phi_{x_1} - \phi_{x_2} - \phi_{x'_1} + \phi_{x'_2} =\Big(\phi_{x_1} - \phi_{x_2} +\phi_{x_3} - \phi_{x_4}\Big)- \Big(\phi_{x'_1} - \phi_{x'_2} +\phi_{x_3} - \phi_{x_4}\Big) \]
vanishing on $V_{x_1} \cap V_{x_2} \cap V_{x'_1} \cap V_{x'_2} \cap V_{x_3} \cap V_{x_4}$, so just intersect $B$ with $V_{x_3} \cap V_{x_4}$. Apply another such step to get a single coset progression $t+ C$, and assume that the number of Bohr-respected additive quadruples is at least $2^{-20d}c^4|C'|^3$.\\
    
    Misusing the notation, we pass to a coset progression $C'$ on which the bilinear variety is sufficiently quasirandom. Our goal is to apply the abstract Balog-Szemer\'edi-Gowers theorem. For a radius $\sigma$, let $V^{(\sigma)}_x = B(\Gamma; \sigma) \cap B(\Theta_1(x), \dots, \Theta_r(x); \sigma)$.\\
    Let $\rho_i$ be defined as $3^{-i}\rho$. Simplify the notation and write $V^{(i)}_x = V^{(\rho_i)}_x$, which is a slightly shrunk Bohr set. Let $\mathcal{Q}_i$ be the set of all additive quadruples $(\upd a_1, a_2, a_3, \upd a_4)$ such that 
    \[\phi_{a_1} - \phi_{a_2} - \phi_{a_3} + \phi_{a_4} = 0\]
    on $V^{(i)}_{a_1} \cap V^{(i)}_{a_2} \cap V^{(i)}_{a_3} \cap V^{(i)}_{a_4}$. Obviously, $\mathcal{Q}_i \subseteq \mathcal{Q}_{i + 1}$ so all sets $\mathcal{Q}_i$ are large. They are also symmetric. In order to apply the theorem, we need to check weak transitivity.

    \begin{claim}
        Suppose that $x + a, x, y + a, y$ are elements of $X$ such that $(\upd{x + a}, x, z + a, \upd z), (\upd{z + a}, z, y + a, \upd y) \in \mathcal{Q}_i$ holds for at least $c' |C|$ choices of $z \in X$. Then $(\upd{x+a}, x, y + a, \upd{y}) \in \mathcal{Q}_{i+1}$.
    \end{claim}

    \begin{proof} Let $Z$ be the set of the described elements $z \in X$. For each such $z \in Z$, by Bohr-respectedness, we have
    \begin{align*}\phi_{x + a}(w) - \phi_{x}(w) - \phi_{y + a}(w) + \phi_{y}(w) = &\Big(\phi_{x + a}(w) - \phi_{x}(w) - \phi_{z + a}(w) + \phi_{z}(w)\Big) \\
    &\hspace{2cm}+ \Big(\phi_{z + a}(w) - \phi_{z}(w) - \phi_{y + a}(w) + \phi_{y}(w)\Big) = 0,\end{align*}
    whenever $w \in V^{(i)}_{x + a} \cap V^{(i)}_x \cap V^{(i)}_{y + a} \cap V^{(i)}_y \cap V^{(i)}_{z + a} \cap V^{(i)}_z$. Write $S = V^{(i + 1)}_{x + a} \cap V^{(i + 1)}_x \cap V^{(i + 1)}_{y + a} \cap V^{(i + 1)}_y$ and observe that $S \cap V^{(i + 1)}_z \subseteq V^{(i)}_{z + a}$. Indeed, if $w$ belongs to the set on the left, then $\uc{\chi(w)} \leq \rho_{i+1}$ for all the relevant characters so $\uc{\theta_i(z + a)(w)} \leq 3\rho_{i + 1} \leq \rho_i$. Hence, $\phi_{x + a}(w) - \phi_{x}(w) - \phi_{y + a}(w) + \phi_{y}(w) = 0$ on $\cup_{z \in Z} S \cap V^{(i + 1)}_z$.\\
    \indent Use quasirandomness to deduce that $\phi_{x + a}(w) - \phi_{x}(w) - \phi_{y + a}(w) + \phi_{y}(w) = 0$ holds for at least $1 - 4^{-k}$ proportion of $w \in V^{(i + 1)}_{x+a} \cap V^{(i + 1)}_x \cap V^{(i + 1)}_{y + a} \cap V^{(i + 1)}_y$. By Lemma~\ref{almostfullBohr}, we are done provided $\rho_{i + 1} \leq \frac{1}{2}\rho_i$.
    \end{proof}

    Apply Theorem~\ref{absg} giving a set $X' \subseteq X$. Similarly, apply Theorem~\ref{absg} another time to find a further subset $X'' \subseteq X'$. In particular, there is some $c' \geq (c/2^d)^{O(1)}|C'|$ such that $|X'|, |X''| \geq c' |C'|$, for each additive $12$-tuple $a_{[12]}$ in $X''$ we have at least $c'|C'|^{35}$ additive $36$-tuples $y_{[36]}$ in $X'$ such that 
    \[\sum_{i \in [12]}(-1)^i a_i = \sum_{i \in [36]} (-1)^iy_i\]
    and
    \begin{equation}\label{xsecundumtoxprime}\sum_{i \in [12]}(-1)^i \phi_{a_i} = \sum_{i \in [36]} (-1)^i \phi_{y_i}\end{equation}
    holds on $\Big(\bigcap_{i \in [12]} V^{(\sigma)}_{a_i}\Big) \cap \Big(\bigcap_{i \in [36]} V^{(\sigma)}_{y_i}\Big)$, for some $\sigma \geq \Omega(\rho)$ and a similar conclusion holds for all additive additive $36$-tuples $y_{[36]}$ in $X'$, which we relate to additive $108$-tuples $z_{[36]}$ in $X$. Here we used $V^{(\tau/3)}_{x_1} \cap V^{(\tau/3)}_{x_2} \cap V^{(\tau/3)}_{x_3} \subseteq V^{(\tau)}_{x_4}$ for additive quadruples $(x_1, \dots, x_4)$ several times.\\

    \begin{claim}
        Provided $\eta \leq (2^d\delta c c')^{\blc}$, the following holds. Let $Q$ be a collection of additive additive $36$-tuples in $X'$ of size $\delta |C'|^{35}$. Then there exist a real $\rho' \geq \Omega(\rho)$ and a Freiman-linear map $\psi : B(\Gamma, \rho') \to H$ such that 
        \[\Big(\sum_{i \in [36]} (-1)^i \phi_{y_i}\Big)\Big|_{\cap_{i \in [36]} V^{(\rho')}_{y_i}} = \psi|_{\cap_{i \in [36]} V^{(\rho')}_{y_i}} \]
        holds for at least $(c c' \delta)^{O(1)} |C'|^{35}$ of $y_{[36]} \in Q$.
    \end{claim}

    \begin{proof}
        Consider the bipartite graph whose vertex classes are $Q$ and the set of all additive 108-tuples in $X'$. We put an edge between $y_{[36]} \in Q$ and $z_{[108]}$ if 
        \[\sum_{i \in [36]} (-1)^i \phi_{y_i} = \sum_{i \in [108]} (-1)^i \phi_{z_i}\]
        holds on $\Big(\bigcap_{i \in [36]} V^{(\sigma)}_{y_i}\Big) \cap \Big(\bigcap_{i \in [108]} V^{(\sigma)}_{z_i}\Big)$. By the properties guaranteed by Theorem~\ref{absg}, every vertex in $Q$ has degree at least $c' |C'|^{107}$, so the graph is dense.\\

        Using quasirandomness of the bilinear Bohr variety and Lemma~\ref{qrsumbohrs}, we conclude that for a vast majority of $y_{[36]}, y'_{[36]} \in Q$ and additive 108-tuples $z_{[108]}, z'_{[108]}$ we have
        \begin{align*}
            B(\Gamma, \sigma/300) \subseteq &(\cap_{i \in [108]} V^{(\sigma/30)}_{z_i}) + (\cap_{i \in [108]} V^{(\sigma/30)}_{z'_i})\\
            (\cap_{i \in [108]} V^{(\sigma/10)}_{z_i}) \cap (\cap_{i \in [108]} V^{(\sigma/10)}_{z'_i}) \subseteq & \Big((\cap_{i \in [108]} V^{(\sigma)}_{z_i}) \cap (\cap_{i \in [108]} V^{(\sigma)}_{z'_i}) \cap (\cap_{i \in [36]} V^{(\sigma)}_{y_i})\Big) \\
            &\hspace{2cm} + \Big((\cap_{i \in [108]} V^{(\sigma)}_{z_i}) \cap (\cap_{i \in [108]} V^{(\sigma)}_{z_'i}) \cap (\cap_{i \in [36]} V^{(\sigma)}_{y'_i})\Big)\\
            (\cap_{i \in [36]} V^{(\sigma/3000)}_{y_i}) \subseteq & \Big((\cap_{i \in [108]} V^{(\sigma/300)}_{z_i}) \cap (\cap_{i \in [36]} V^{(\sigma/300)}_{y_i})\Big)\\&\hspace{2cm}  +  \Big((\cap_{i \in [108]} V^{(\sigma/300)}_{z_'i}) \cap (\cap_{i \in [36]} V^{(\sigma/300)}_{y_i})\Big)\\
            (\cap_{i \in [36]} V^{(\sigma/3000)}_{y'_i}) \subseteq & \Big((\cap_{i \in [108]} V^{(\sigma/300)}_{z_i}) \cap (\cap_{i \in [36]} V^{(\sigma/300)}_{y'_i})\Big) \\&\hspace{2cm} +  \Big((\cap_{i \in [108]} V^{(\sigma/300)}_{z_'i}) \cap (\cap_{i \in [36]} V^{(\sigma/300)}_{y'_i})\Big).
        \end{align*}

        We first pick  $z_{[108]}, z'_{[108]}$ for which the most of $y_{[36]}, y'_{[36]} \in Q$ satisfy the inclusions above and a positive proportion forms a cycle of length 4 in the bipartite graph the we consider. Let $K = \cap_{i \in [108]} V^{(\sigma)}_{z_i}, K' = \cap_{i \in [108]} V^{(\sigma)}_{z_'i}$, $L = $, $L' = $ and let $\alpha : K \to H$, $\alpha' : K' \to H$, $\beta: L \to H$, $\beta' : L' \to H$ be the maps defined by $\alpha = \sum_{i \in [108]} (-1)^i\phi_{z_i}$, $\alpha' = \sum_{i \in [108]} (-1)^i\phi_{z'_i}$, $\beta = \sum_{i \in [36]} (-1)^i\phi_{y_i}$ and $\beta' = \sum_{i \in [36]} (-1)^i\phi_{y'_i}$. We first show that $\alpha = \alpha'$ on $K \cap K'$. Note that $\alpha = \beta$ on $K \cap L$ and $\alpha' = \beta$ on $K \cap L$ so $\alpha = \alpha'$ on $K \cap K' \cap L$. Using $\beta'$ instead of $\beta$, we get $\alpha = \alpha'$ on $(K \cap K' \cap L) + (K \cap K' \cap L')$. Provided we take suitable $y_{[36]}, y'_{[36]}$, we get $\alpha = \alpha'$ on $(K \cap K')_{1/10}$ using the second inclusion above.\\

        Extend $\alpha$ and $\alpha'$ simultaneously to $\psi$ on $(K + K')_{1/30}$, containing $B(\Gamma, \sigma/300)$, using Lemma~\ref{flinearextn}. Hence $\psi = \beta$ on $\cap_{i \in [108]} V^{(\sigma/300)}_{z_i}) \cap (\cap_{i \in [36]} V^{(\sigma/300)}_{y_i})$ and on $\cap_{i \in [108]} V^{(\sigma/300)}_{z'_i}) \cap (\cap_{i \in [36]} V^{(\sigma/300)}_{y_i})$, so coincides on the sum which contains $\cap_{i \in [36]} V^{(\sigma/3000)}_{y_i}$.
    \end{proof}

    For each additive 12-tuple $q = x_{[12]}$ in $X''$, we thus have $c'|C'|^{35}$ additive 36-tuples in $X'$ such that~\eqref{xsecundumtoxprime} holds. Apply the claim above to this set of additive 36-tuples  to find a Freiman-linear map $\psi_{q} : B(\Gamma, \rho') \to H$ such that 
        \[\Big(\sum_{i \in [36]} (-1)^i \phi_{y_i}\Big)\Big|_{\cap_{i \in [36]} V^{(\rho')}_{y_i}} = \psi|_{\cap_{i \in [36]} V^{(\rho')}_{y_i}} \]
        holds for a set $\mathcal{Z}_q$ of additive 36-tuples $y_{[36]}$ among the considered ones above in $X'$, of size at least $c_2|C'|^{35}$ for some $c_2 \geq (c c')^{O(1)} $.\\ 
    
    Take a maximal collection $q_1, \dots, q_m$ of additive 12-tuple $q = x_{[12]}$ in $X''$ such that $|\mathcal{Z}_{q_i} \cap \mathcal{Z}_{q_j}| \leq \frac{1}{2}c_2^2 |C'|^{35}$. By elementary double-counting argument, we have $m \leq (c_2^{-1}/2)^{O(1)}$. Write also $\psi_i$ for $\psi_{q_i}$.

    \begin{claim}
        For each additive 12-tuple $x_{[12]}$ in $X''$, we have some $j \in [m]$ such that 
        \[\Big(\sum_{i \in [12]} (-1)^i \phi_{x_i}\Big) = \psi_j\]
        holds on $\cap_{i \in [12]} V^{(\rho'/10)}_{x_i}$.
    \end{claim}

    \begin{proof}
        By our choice of $q_1, \dots, q_m$, there exists $q_j$ such that $|\mathcal{Z}_{q_j} \cap \mathcal{Z}_{q}| \geq \frac{1}{2}c_2^2 |C'|^{35}$. For each additive 36-tuple $y_{[36]}$ in the intersection, we have
        \[\sum_{i \in [36]} (-1)^i \phi_{y_i} = \sum_{i \in [12]} (-1)^i \phi_{x_i}\]
        on the intersection $\Big(\cap_{i \in [36]} V^{(\rho')}_{y_i}\Big) \cap \Big(\cap_{i \in [12]} V^{(\rho')}_{x_i}\Big)$, as well as 
        \[\sum_{i \in [36]} (-1)^i \phi_{y_i} = \psi_j\]
        on $\cap_{i \in [36]} V^{(\rho')}_{y_i}$. Thus, 
        \[\Big(\sum_{i \in [12]} (-1)^i \phi_{x_i}\Big) = \psi_j\]
        holds on $\Big(\cap_{i \in [12]} V^{(\rho'/10)}_{x_i}\Big) \cap \Big(\cap_{i \in [36]} V^{(\rho')}_{y_i}\Big)$ for many $ \frac{1}{2}c_2^2 |C'|^{35}$ additive 36-tuples $y_{[36]}$. The claim follows by quasirandomness.
    \end{proof}

    \noindent\textbf{Dependent random choice argument.} To finish the proof, we use a probabilistic argument to essentially reduce to the case $\psi_j = 0$. Firstly, we define $(\chi_1, t_1), \dots , (\chi_\ell, t_\ell) \in \hat{H} \times B(\Gamma, \rho'/100)$ iteratively at random. Let $J \subseteq [m]$ be the collection of $j$ such that $\psi_j$ is not zero on the whole of $B(\Gamma, \rho'/10^6)$. Then $\psi_j$ has at least $\frac{1}{10}$ proportion of non-zero values on $B(\Gamma, \rho'/100)$. At each step, take random $\chi_i, t_i$ and remove those indices $j$ such that $\|\chi_i(\psi_j(t_i))\|_{\mathbb{T}} \geq 1/100$ and pass to subset $X'' \cap V^{\rho'/10}_{\bcdot t_1}\dots \cap V^{\rho'/10}_{\bcdot t_i}$. With positive probability $J$ decreases by constant factor and $X''$ remains reasonably large at each step, so we are done after a logarithmic number of steps. Let $\tilde{X}$ be the remaining elements of $X''$.\\
    To complete the argument take a random cube $I$ inside $\mathbb{T}^\ell$ of edge length $1/10000$ and take the set of all $x \in \tilde{X}$ such that $(\chi_1(\phi_x(t_1)), \dots, \chi_\ell(\phi_x(t_\ell)) \in I$. It turns out that all additive 12-tuples in the final set are Bohr-respected, as desired. 
\end{proof}
\section{Extending domain to the full bilinear variety}\label{exttovarsection}

From this section on, we make another important change of perspective. Rather than considering systems of Freiman-linear maps, we consider maps defined on bilinear Bohr varieties, which will have bilinear structure.\\

Now that all additive quadruples in some dense set are Bohr-respected, we apply robust Bogolyubov-Ruzsa theorem another time, to ensure that the index set is a coset progression. However, as in the previous section, the situation is more complicated than that in Section~\ref{firstBRstepSection} and we need to use the algebraic regularity lemma. \\

To be able to relate the newly obtained maps to the old ones, we introduce the notion of \textit{arrangements}, which are certain sequences of points in $G_1 \times G_2$, and their \textit{lengths}, which will be a single point. An \textit{$\emptyset$-arrangement} is just a singleton sequence consisting of a point $(x,y) \in G_1 \times G_2$ and $(x,y)$ is its \textit{length}. More generally, \emph{$(d_1, \dots, d_r)$-arrangement} of \textit{length} $(\ell_1, \ell_2)$ is a concatenation of $d_r$ $(d_1, \dots, d_{r-1})$-arrangements whose lengths are either $(a_1, \ell_2), \dots, (a_{d_r}, \ell_2)$ with $a_1 - a_2 + \dots + (-1)^{d_r - 1}a_{d_r} = \ell_1$, or $(\ell_1, b_1), \dots, (\ell_1, b_{d_r})$ with $b_1 - b_2 + \dots + (-1)^{d_r - 1}b_{d_r} = \ell_2$. Geometrically, arrangements arise in the process of directional convolutions and occur naturally in the process of extending almost bilinear maps. We denote the set of all $(d_1, \dots, d_r)$-arrangements in $G_1 \times G_2$ as $\mathcal{A}_{d_1, \dots, d_r}$ and those of lengths $(\ell_1, \ell_2)$ are denoted by $\mathcal{A}_{d_1, \dots, d_r}(\ell_1, \ell_2)$. Note that, as a sequence, any  $(d_1, \dots, d_r)$-arrangement consists of precisely $d_1\dots d_r$ points.

\begin{proposition}\label{secondrobustBRstep}
    Let $C \subseteq G_1$ be a symmetric proper coset progressions of rank $d$, let $X \subseteq C$ be a set of size $|X| \geq c |C|$. Let $B(\Gamma, \rho) \subseteq G_2$ be a Bohr set of codimension $r$ and let $\Theta_1, \dots, \Theta_s : C \to \hat{G}_2$ be Freiman-linear maps. Let $V$ be the bilinear Bohr variety
    \[V = \Big\{(x,y) \in G_1 \times G_2 : x \in C, y \in B \cap B(\Theta_1(x), \dots, \Theta_s(x); \rho)\Big\}.\]
    Suppose that $\phi : (X \times B) \cap V \to H$ is a map that respects all horizontal additive 36-tuples and is Freiman-linear in the vertical direction. \\
    \indent Then there exist a positive quantity
    \[c' \geq \exp\Big(-d^{O(1)} r^{O(1)} s^{O(1)} \log^{O(1)}(\rho^{-1}) \Big)\]
    a symmetric proper coset progression $C' \subseteq C$ of rank $d$ (same as $C$) and size $|C'| \geq c'|C|$, a positive quantity $\rho' \geq \Omega(\rho)$, and a map $\psi : V' \to H$, for a bilinear Bohr variety
    \[V' = \Big\{(x,y) \in G_1 \times G_2 : x \in C, y \in B \cap B(\Theta_1(x), \dots, \Theta_s(x); \rho')\Big\},\]
    which is a Freiman-bilinear and the identity 
    \begin{align}
    \psi(x,y)= & \sum_{i \in [36]} \nu_i \phi(a_i, b_i).\label{extensionvarid}
    \end{align}
    holds and all arguments belong to $V$ for at least $c'|C|^{29} |G_2|^6$ choices of (4,3,3)-arrangements of lengths $(x,y)$, where $\nu_i \in \{-1,0,1\}$ are some fixed coefficients such that in the tensor product $G_1 \otimes G_2$ we have equality 
    \[x \otimes y =  \sum_{i \in [36]} \nu_i a_i \otimes b_i.\]
\end{proposition}

\vspace{\baselineskip}

Before proceeding with the step, we record a preliminary lemma which will be useful at several places, which allows us to fill in gaps in domains for almost bilinear maps defined on bilinear Bohr varieties.\\

\begin{lemma}\label{almostfullbilinearmap}
    Let $C \subseteq G_1$ be a symmetric proper coset progression of codimension $d$, let $\Gamma \subseteq \hat{G}_2$ be a set of size $d$ and let $\Theta_1, \dots, \Theta_r : C \to \hat{G}_2$ be Freiman-linear maps. For a quantity $\sigma > 0$, define a bilinear Bohr variety $V^\sigma = \bigcup_{x \in C} \{x\} \times B(\Gamma, \Theta_1(x), \dots, \Theta_r(x); \sigma)$.\\
    \indent Let $X \subseteq V^\rho$ be a subset of the bilinear Bohr variety of size $|X| \geq (1-\varepsilon)|V^\rho|$, let $\phi : X \to H$ be a map and let $E \subseteq H$ be a set. Suppose that $\phi(x_1, y) - \phi(x_2, y) + \phi(x_3, y) - \phi(x_4, y) \in E$ holds for all but at most $\varepsilon |G_1|^3|G_2|$ horizontal additive quadruples in the domain $X$ and that $\phi(x, y_1) - \phi(x, y_2) + \phi(x, y_3) - \phi(x, y_4) \in E$ holds for all but at most $\varepsilon |G_1||G_2|^3$ vertical additive quadruples in the domain $X$. Then, provided $\varepsilon \leq 2^{-\blc(r + d)}\rho^{\blc r}$, there exists an $10000E$-bihomomorphism $\psi$ from domain $D = (\frac{1}{20000}C \times G_2) \cap V^{\rho/20000}$ to $H$ and a subset $X' \subseteq D \cap X$ of size $|X'| \geq (1 - \varepsilon^{1/8} (2\rho^{-1})^{O(r + d)})|D|$ such that for each $(x,y) \in X'$ we have $\psi(x,y) = \phi(x,y)$. \\
    \indent Moreover, for each $(x,y) \in D$, we have at least $2^{-O(r + d)} \rho^{O(r)}|C|^2|B(\Gamma, \rho)|^6$ choices of $(x_1, x_2, y_{1,1}, y_{1,2},$ $y_{2,1}, y_{2,2}, y_{3,1}, y_{3,2}) \in C^2 \times B(\Gamma, \rho)^6$ such that, for $x_3 = x_1 + x_2 - x$, $y_{i, 3} = y_{i, 1} + y_{i, 2} - y$,
    \begin{align}\psi(x,y) =& \phi(x_1, y_{1,1}) + \phi(x_1, y_{1,2}) - \phi(x_1, y_{1,3}) \nonumber\\
    &\hspace{1cm} + \phi(x_2, y_{2,1}) + \phi(x_2, y_{2,2}) - \phi(x_2, y_{2,3})\nonumber\\
    &\hspace{1cm} -  \phi(x_3, y_{3,1}) - \phi(x_3, y_{3,2}) + \phi(x_3, y_{3,3})\label{technicalextensionverydense}\end{align}
    and all points in arguments belong to $X$.
\end{lemma}

\begin{proof} We say that a horizontal or a vertical additive quadruple is \textit{$E$-respected} if the condition on the appropriate linear combination of values of $\phi$ as in the condition in the statement of the lemma holds. The proof of the lemma is elementary and proceeds in four stages.

\begin{itemize}
    \item[\textbf{Step 1.}] Firstly, we remove a small number of points from the domain until all vertical and horizontal additive quadruples are $E$-respected.
    \item[\textbf{Step 2.}] We remove a further small subset of points until all columns and rows of the new domain are quite dense.
    \item[\textbf{Step 3.}] We extend the map in columns and observe that the extension still has all directional additive quadruples $O(1)E$-respected.
    \item[\textbf{Step 4.}] We finally extend to the claimed domain.
\end{itemize}

In order to treat rows and columns simultaneously, we use an abstract version of Bohr sets, namely Bourgain systems introduced by Green and Sanders~\cite{GreenSanders}. We use a slightly modified version of their definition. A collection of sets $B^{(t)}$ in a finite abelian group $G$, indexed by reals $t \in [0,1]$, is a \textit{Bourgain system of density $c$ and codimension $k$} if it has the following properties:
\begin{itemize}
    \item (Nesting) $B^{(t_1)} \subseteq B^{(t_2)}$ for $t_1 \leq t_2$,
    \item (Zero) $0 \in B^{(t)}$,
    \item (Symmetry) $-B^{(t)} = B^{(t)}$,
    \item (Addition) $B^{(t_1)} + B^{(t_2)} \subseteq B^{(t_1 + t_2)}$ for $t_1 + t_2 \leq 1$,
    \item (Doubling) $|B^{(t)}| \geq c t^k |B^{(1)}|$.
\end{itemize}

We claim that rows and columns of a bilinear Bohr variety give Bourgain systems.

\begin{claim}
    For each $x \in C$, the collection $B_{\ssk{\on{col}\\x}}^{(t)} = B(\Gamma, \Theta_1(x), \dots, \Theta_r(x); t\rho)$, is a Bourgain system of density $(\rho/4)^{r + d}$ and codimension $r + d$.\\
    For each $y \in B(\Gamma, \rho)$, the collection $B_{\ssk{\on{row}\\y}}^{(t)} = \{x \in t \cdot C :(\forall i \in [r])\,\, \Theta_i(x)(y)\|_{\mathbb{T}} \leq t \sigma \}$, is a Bourgain system of density $4^{-r-d}\rho^{r}$ and codimension $r + d$.\\
\end{claim}

\begin{proof}\textbf{Case of columns.} Each column is just a Bohr set of codimension $r + d$, with density at least $\rho^{r + d}$ by Lemma~\ref{basicbohrsizel}, so we get a Bourgain system.\\
\indent \textbf{Case of rows.} First two properties are trivial. Symmetry and addition follow from the fact that $\Theta_i$ are Freiman-linear. For the doubling bound, partition $\mathbb{T}^r$ into $m \leq (4t^{-1}\rho^{-1})^r$ cuboids $I_1, \dots, I_m$ of edge length at most $\frac{1}{2}t\rho$. By the pigeonhole principle we have some $I_i$ with $A = \{x \in \frac{t}{2}C : \Theta(x)(y) \in I_i\}$ having size at least $(2t^{-1})^{-d}m^{-1}|C|$. Then, due to Freiman-linearity of $\Theta_i$, we have $A - A \subseteq B_{\ssk{\on{row}\\y}}^{(t)}$, so $|B_{\ssk{\on{row}\\y}}^{(t)}| \geq (2t^{-1})^{-d}m^{-1}|C| \geq \rho^r (t/4)^{d + r} |B_{\ssk{\on{row}\\y}}^{(1)}|$.
\end{proof}

Let us note that the proof also show that inside $V^\rho$ each row has size at least $2^{-d}4^{-r}\rho^r|C|$ and each column has size at least $2^{-d-r}\rho^r|B(\Gamma, \rho)|$.\\

\noindent\textbf{Step 1.} In the first step, we pass from almost $E$-homomorphisms on Bourgain systems to actual $E$-homomorphisms after erasing few points.

\begin{claim}\label{almostfullebihomclaim1}
    Let $(B^{(t)})_{t \in [0,1]}$ be a Bourgain system of density $c$ and codimension $k$. Let $D \subseteq B^{(1)}$ be a set of size $|D| \geq (1-\eta)|B^{(1)}|$ and let $\gamma : D \to H$ be a map which $E$-respects all but at most $\eta |B^{(1)}|^3$ additive quadruples in $D$. Then, provided $\varepsilon \leq c^3 2^{-(36k + 20)}$, there exists a subset $D' \subseteq D \cap B^{(1/1000)}$ such that $|D'| \geq (1 - \eta c^{-3} 2^{36k + 20})|B^{(1/1000)}|$ on which $\gamma$ $24E$-respects all additive 12-tuples.
\end{claim}

\begin{proof}
    Let $A \subseteq B^{(1/8)} \cap D$ be the set of all $a \in B^{(1/8)} \cap D$ which appear in at least $c^2 2^{-(6k + 4)}|B^{(1)}|^2$ non-$E$-respected additive quadruples. Thus $|A| \leq \eta c^{-2} 2^{6k + 4}|B^{(1)}|$.\\\
    \indent Let us observe that every non-$4E$-respected additive quadruple in $B^{(1/8)}\cap D$ has an element in $A$. Indeed, if $x_1, x_2, x_3,x_4$ is not $4E$-respected, for each $z \in B^{(1/8)} \cap D \cap (D - x_1 + x_2)$ we have $\gamma(x_1) - \gamma(x_2) + \gamma(z) - \gamma(z + x_1 - x_2) \notin 2E$ or $\gamma(x_3) - \gamma(x_4) + \gamma(z) - \gamma(z + x_3 - x_4) \notin 2E$. Without loss of generality, the first case holds at least half of the time, thus for at least 
    \[|B^{(1/8)} \cap D \cap (D - x_1 + x_2)| \geq |B^{(1/8)}| - 2|B^{(1)} \setminus D| \geq (c 2^{-3k} - 2\eta)|B^{(1)}|\]
    additive quadruples in $B^{(3/8)}\cap D$ with $x_1$ and $x_2$. We may rewrite this is as  $\gamma(x_1) - \gamma(z + x_1 - x_2)  + \gamma(z) - \gamma(x_2)\notin 2E$. Apply the same step to deduce that $x_1 \in A$ or $x_2 \in A$.\\
    We claim that $\gamma$ $10E$-respects all additive 12-tuples on $(B^{(1/1000)} \cap D) \setminus A$. Indeed, if $x_1, \dots, x_{12} \in (B^{(1/1000)} \cap D) \setminus A$ and $\sum_{i \in [12]} (-1)^i x_i = 0$, then we may find $z_1, \dots, z_6 \in (B^{(1/8)} \cap D) \setminus A$ such that $z_{i} - z_{i+1} = x_{2i} - x_{2i + 1}$ holds for $i \in \{1,2,\dots, 5\}$. Namely, we just need to find $z_1 \in (B^{(1/1000)} \cap D) \setminus A$ such that $z_2 = z_1 + x_3 - x_2, z_3 = z_1 + x_5 - x_4 + x_3 - x_2, \dots \in D\setminus A$. But, as $|B^{(1/1000)}| \geq 10 |B^{(1)} \setminus D| + 10 |B^{(1)} \setminus A|$, we may achieve this.\\
    
    Note that we have
    \[z_6 - z_1 = (z_6 - z_5) + (z_5 - z_4) + \dots + (z_2 - z_1) = x_{11} - x_{10} + x_{9} - x_{8} + \dots + x_3 - x_2 = x_{12} - x_1.\]
    
    Then, using $4E$-respectedness of additive quadruples
    \[\sum_{i \in [12]} (-1)^i \phi(x_i) \in \sum_{i \in [6]} (\phi(z_i) - \phi(z_{i+1}) + 4E) = 24  E.\qedhere\]
\end{proof}

Applying the claim above for all rows of size $|X_{\bcdot y}| \geq (1 - \sqrt{\varepsilon})|V^\rho_{\bcdot y}|$, and then for all sufficiently dense columns, leaves us with a subset $X_1 \subseteq X$ of size $|X_1| \geq (1 - \varepsilon_1)|V^{(\rho/1000)}|$ on which $\phi$ is an $E$-bihomomorphism, with $\varepsilon_1 \leq 2^{O(d + r)}\rho^{-O(r)}\sqrt{\varepsilon}$.\\

\noindent\textbf{Step 2.} Remove first all rows $(X_1)_{\bcdot y}$ which have size less than $(1 - \sqrt{\varepsilon_1})|(V^{(\rho/1000)})_{\bcdot y}|$. Each row removed has at least $\sqrt{\varepsilon_1}2^{-O(d + r)}\rho^{O(r)}|C|$ points, so provided $\varepsilon \leq 2^{-\blc(r + d)}\rho^{\blc r}$, we have removed at most $\sqrt[4]{\varepsilon} |B(\Gamma, \rho)|$ rows, decreasing the size of each column by at most $\sqrt[4]{\varepsilon} |B(\Gamma, \rho)|$ as well. Repeat the same argument to columns, and note that all rows have been affected negligeably. Thus, provided $\varepsilon \leq 2^{-\blc(r + d)}\rho^{\blc r}$, we obtain a set $X_2 \subseteq X$ on which $\phi$ is an $E$-bihomomorphism, and all non-empty columns and rows have density at least $1-\varepsilon_2$ inside the respective column and row of $V^{\rho/10000}$ and $\varepsilon_2 \leq \sqrt[8]{\varepsilon}$.\\

\noindent\textbf{Step 3.} We prove a general extension claim for $E$-homomorphisms on Bourgain systems.\\

\begin{claim}
    Let $(B^{(t)})_{t \in [0,1]}$ be a Bourgain system of density $c$ and codimension $k$. Let $D \subseteq B^{(1)}$ be a set of size $|D| \geq (1-\eta)|B^{(1)}|$ and let $\gamma : D \to H$ be a map which $E$-respects all additive 12-tuples in $D$. Suppose that $c4^{-k - 1} \geq \eta$. Then there exists an $5E$-homomorphism $\gamma^{\on{ext}} : B^{(1/2)} \to H$ extending $\gamma$.
\end{claim}

\begin{proof}
    Fix any $y_0 \in D \cap B^{(1/4)}$, which is non-empty as $|B^{(1/4)}| \geq c4^{-k} |B^{(1)}| > |B^{(1)} \setminus D|$. Let $x \in B^{(1/4)}$ be arbitrary. Then $x$ can be written as $x_1 - x_2 + y_0$ for $x_1 \in B^{(1/2)} \cap D$ and $x_2 \in D$ in $|D \cap ((B^{(1/2)} \cap D) - x + y_0)|$ many ways. Using properties of the Bourgain system, we get
    \begin{align*}|D \cap ((B^{(1/2)} \cap D) - x + y_0)| = &|D| + |B^{(1/2)} \cap D| - |D \cup ((B^{(1/2)} \cap D) - x + y_0)| \\
    \geq &|B^{(1)}| + |B^{(1/2)}| - 2|B^{(1)} \setminus D| - |B^{(1)}| \geq |B^{(1/2)}| - 2|B^{(1)} \setminus D| \\
    \geq &(c2^{-k} - 2\eta) |B^{(1)}| > 0.\end{align*}
    Hence, defining $\gamma^{\on{ext}}(x) = \gamma(x_1) - \gamma(x_2) + \gamma(y_0)$ for any  $x_1 \in B^{(1/2)} \cap D$ and $x_2 \in D$  such that $x = x_1 - x_2 + y_0$ gives a well-defined $E$-homomorphism extension to $B^{(1/4)}$.
\end{proof}

\noindent\textbf{Step 4.} Provided $\varepsilon \leq 2^{-\blc(r + d)}\rho^{\blc r}$, we have at least $(1 - \sqrt[16]{\varepsilon_1})|C|$ columns where we may apply the claim above, to get extension to $V^{\rho/20000}$. This extension is $120E$-homomorphism in the vertical direction. But, given any horizontal additive 12-tuple $x_{[12]}$ inside row indexed by some $y$, we have that all columns $(X_2)_{x_i}$ are very dense, so $\cap_{i \in [12]} (X_2)_{x_i}$ is of density at least $1 - \sqrt[32]{\varepsilon_1}$ inside $\cap_{i \in [12]} (V^{\rho/10000})_{x_i}$, on which $\sum_{i \in [12]} (-1) \phi_{x_i}$ can similarly be uniquely extended, showing that horizontal additive 12-tuple are $1000E$-respected. Use the claim in the horizontal direction and the same argument as above to conclude that we obtain a $10000E$-bihomomorphism. The final relationship~\eqref{technicalextensionverydense} follows from the two extensions in this step.\end{proof}

Similar arguments can be used to prove that $E$-bihomomorphisms can be related to $E$-bilinear maps.

\begin{lemma}\label{ebihomtoebilinearlemma}
    Let $C, \Gamma, V^\sigma$ be as in the previous lemma and let $\phi : V^\rho \to H$ be an $E$-bihomomorphism. Then there exists an $9E$-bilinear map $\psi : V^{\rho/10000} \to H$ such that for each $(x,y) \in V^{\rho/10000}$ we have at least $2^{-O(r + d)}\rho^{O(r)} |C||B(\Gamma, \rho)|^2$ choices of $(a, b_1, b_2)$ such that
    \[\psi(x,y) = \phi(a + x, b_2 + y) - \phi(a + x, b_2) - \phi(a, b_1 + y) + \phi(a, b_1)\]
    and all points in the arguments of $\phi$ belong to $V^{\rho}$.
\end{lemma}

\begin{proof}
    The same argument near the end of the proof of Claim~\ref{almostfullebihomclaim1} shows that on $V^{\rho/1000}$, we have all directional additive 8-tuples $3E$-respected. Hence, we may first convolve in the vertical direction, and define an $3E$-bihomomorphism $\phi_1$ such that additionally $\phi_1(x, 0) \in E$ holds for all $x \in \frac{1}{1000}C$. Another convolution in the horizontal direction, gives a further $E$-bihomomorphism $\phi_2$ such that $\phi_2(0, y) \in E$. Thus $\phi_2$ is $E$-bilinear and has the desired property.
\end{proof}

\begin{proof}[Proof of Proposition~\ref{secondrobustBRstep}]
    We begin the proof by applying the algebraic regularity lemma. Let $\eta$ be a positive quantity that will be specified later. Apply Theorem~\ref{algreglemmaintro} to the bilinear Bohr variety which is the same as $V$, except the radius $\rho$ is replaced by $\rho/100$, with error parameter $\eta$ giving a coset progression $C'$ of rank at most $d$ whose few translates essentially cover $C$ and partition the bilinear Bohr variety $V$ into quasirandom pieces. By averaging, there exists a coset progression $t + C'$ such that $|X \cap t + C'| \geq \frac{c}{2}|C'|$ and $|C'| \geq c'|C|$, where 
    \[c' \geq \exp\Big(-d^{O(1)} r^{O(1)} s^{O(1)} \log^{O(1)} (\eta^{-1}) \log^{O(1)}(\rho^{-1}) \Big).\]
    Let us misuse the notation and write $C'$ instead of $t  + C'$. Let $\sigma \in [\rho/200,\rho/100]$ be radius in the quasirandom bilinear Bohr variety $V' \subseteq C' \times B'$ stemming from $C'$ for a suitable Bohr set $B'$ of codimension $r$ and radius $\sigma$, and let $\delta$ be the relevant density. By Lemma~\ref{basicbohrsizel}, $\delta \geq {\rho/200}^{r + s}$, provided $\eta \leq \frac{1}{2}{\rho/200}^{r + s}$.\\
    \indent Apply Corollary~\ref{robustBogRuzsaCP} to $X \cap C'$ to find a a symmetric proper coset progression $C''$ of rank at most $(d\log c^{-1}))^{O(1)}$ and size $|C''| \geq  \exp\Big(-(d\log c^{-1})^{O(1)}\Big) |C'|$ such that for each $u \in C''$ there are at least $\frac{c}{2^{d + 7}}$ quadruples $(x_1, x_2, x_3, x_4) \in X^4$ such that $u = x_1 + x_2 - x_3 - x_4$. Let $T_u$ be the set of all such $(x_1, x_2, x_3) \in X^3$ for the given $u \in C''$. The rest of the proof, we extend the map to the whole bilinear Bohr variety $V'$, where $V' = \{(x,y) \in C'' \times B: y \in B(\Theta_1(x), \dots, \Theta_s(x); \sigma)\}$. This will be done in two steps, firstly we shall define the map on vast majority of $V'$ and then, in the second step, we shall extend it to the whole variety.\\

    Let $S_1 \subseteq V'$ be the set of all $(u, y) \in V'$ such that $|T_u \cap (V'_{\bcdot y})^3| \geq \frac{1}{2} \delta^3 |T_u|$. By Lemma~\ref{genappendknhoods}, we have $|S_1| \geq (1 - 200 \delta^{-6} \eta^{\frac{1}{8}}) |V'|$. For each $(u,y) \in S_1$ define
    \begin{equation}
        \psi(u, y) = \phi(x_1, y) + \phi(x_2, y) - \phi(x_3, y) - \phi(x_1 + x_2 - x_3 - u, y) \label{psidefinbilinearextnfinal}
    \end{equation})
    for $(x_1, x_2, x_3) \in T_u$. Since $\phi$ respects all horizontal additive 36-tuples, it follows that $\psi$ is well-defined and respects all horizontal additive 12-tuples.\\

    Let $u \in C''$ be arbitrary. We next show that $\psi$ respects a vast majority of vertical additive triples on $V'_{u \bcdot}$. To see that, fix any $y_1, y_2, y_3 \in V'_{u \bcdot}$ such that $y_1 + y_2 = y_3$ and such that $(u, y_1), (u, y_2), (u, y_3) \in S_1$. Suppose that $T_u \cap (V'_{\bcdot y_1})^3 \cap (V'_{\bcdot y_2})^3 \not= \emptyset$. Take a triple $(x_1, x_2, x_3)$ that belongs to this set and observe that $(x_1, y_3), (x_2, y_3), (x_3, y_3) \in V$, which follows from the triangle inequality and the choice of $\sigma \leq \rho/100$. By definition of $\psi$ we get
    \begin{align*}
        \psi(u, y_3) = &\phi(x_1, y_3) + \phi(x_2, y_3) - \phi(x_3, y_3) - \phi(x_1 + x_2 - x_3 - u, y_3)\\
        =&\phi(x_1, y_1) + \phi(x_2, y_1) - \phi(x_3, y_1) - \phi(x_1 + x_2 - x_3 - u, y_1)\\
        &\hspace{2cm}+\phi(x_1, y_2) + \phi(x_2, y_2) - \phi(x_3, y_2) - \phi(x_1 + x_2 - x_3 - u, y_2)\\
        = & \psi(u, y_1)  + \psi(u, y_2) ,
    \end{align*}
    so the vertical additive triple is respected.\\

    By quasirandomness of the Bohr variety $V'$, by Lemma~\ref{genappendknhoods}, the set $T_u \cap (V^\sigma_{\bcdot y_1})^3 \cap (V^\sigma_{\bcdot y_2})^3$ is non-empty vast for all but at most $1000 (\rho/100)^{-100d} \delta^{-12} \eta^{\frac{1}{8}})$ choices of $(U, y_1, y_2) \in C'' \times B' \times B'$.\\

    Apply Lemma~\ref{ebihomtoebilinearlemma}, with the error set $E = 0$ to finish the proof.
\end{proof}

\section{Obtaining $E$-bihomomorphism on structured product}\label{passingtobihomonprodsection}

The goal of this section is to pass from a Freiman-bilinear map defined on a bilinear Bohr variety to a $E$-bihomomorphism on a product of Bohr sets. Obtaining the desired $E$-bihomomorphism is carried out in three steps. Firstly, we show that, upon passing to a suitable subprogression in the index set of columns, we may assume that errors arising from naive extensions do not depend on columns. Using this information, we obtain an $E$-bihomomorphism on a bilinear Bohr variety whose columns are subgroups. Finally, we extend the domain to a product of Bohr sets.

\subsection{Controlling errors in columns}\label{controllingerrorsubsec}

\begin{proposition}\label{controllingerrorscolsprop}
    Let $C$ be a symmetric proper coset progression of rank $d$ and density $c$ in $G_1$. Let $\Theta_1, \dots, \Theta_d : C \to \hat{G}_2$ be Freiman-linear maps. Let $\Gamma \subseteq \hat{G}_2$ be a set of characters of size $d$ and let $\rho > 0$. Let $V$ be the bilinear Bohr variety given by
    \begin{equation}\label{detaileddomaindefn}V = \bigcup_{x \in C} \{x\} \times \Big(B(\Gamma, \rho) \cap B(\Theta_{[d]}(x), \rho)\Big).\end{equation}
    Let $\phi : V \to H$ be a Freiman-bilinear map on $V$.\\
    \indent Let $r$ be a positive integer, and let $u_1, \dots, u_r \in \mathbb{Z}$ be fixed coefficients, with $\on{gcd}(u_1, \dots, u_r) = 1$. Write $U = \sum_{i \in [r]} |u_i|$. For $x \in C$ and $\tau, \tau' > 0$, let $E_x(\tau, \tau')$ be the set of values attained by
    \[\sum_{i = 1}^r u_i \phi(x, y_i) - \sum_{i = 1}^r u_i \phi(x, z_i)\]
    ranging over all choices $y_i, z_i \in  B(\Theta_{[d]}(x), \tau) \cap B(\Gamma, \tau')$ such that $\sum_{i = 1}^r u_i y_i = \sum_{i = 1}^r u_i z_i$. Then, there exist:
    \begin{itemize}
        \item a coset progression $C' \subseteq C$ of rank at most $d$ and size $|C_1| \geq \exp(-(2rd\log(c^{-1} \rho^{-1} U))^{O(1)}) |G_1|$,
        \item an element $x_0 \in C'$,
        \item a radius $\tilde{\rho} \in [\rho / 100, \rho/20]$, and
        \item a radius $\tilde{\sigma} \in [\rho / 100U, \rho/20U]$, 
    \end{itemize} 
    such that $E_x(\tilde{\rho}, \tilde{\sigma}) \subseteq E_{x_0}(10\tilde{\rho}, 10\tilde{\sigma})$ for all $x \in C'$.
\end{proposition}

\noindent\textbf{Remark.} Notice that we use the same $d$ for rank of $C$, the number of maps $\Theta_i$ and size of $\Gamma$. These quantities belong to the same quantitative regime, so this is no loss of generality, but it simplifies the notation. Similarly, we use a single radius.\\
\indent Note also that the bound on $\tilde{\rho}$ is independent of $r$, which will be crucial later.

\vspace{\baselineskip}

\begin{proof}
    In order to control the error sets arising in columns, we define auxiliary bilinear Bohr variety $\tilde{V}$ inside $C \times \mathcal{D}$ for a large abelian group $\mathcal{D}$ in which the relevant vertical tuples naturally embed. We shall consider a Freiman-bilinear map on $\tilde{V}$ whose values will be given by errors on the tuples in columns. We shall be able to show that such a map takes few values on a suitable bilinear Bohr subvariety of $\tilde{V}$. The proposition will follow after reinterpreting the obtained conclusion in the context of the original bilinear Bohr variety.\\

    Consider the group $\mathcal{D}$, where 
        \[\mathcal{D} = \Big\{(y_{[r]}, z_{[r]})\in G_2^{2r} : \sum_{i\in [r]} u_i y_i = \sum_{i\in [r]} u_i z_i\Big\}.\]
    \noindent In other words, two $r$-tuples $y_{[r]}$ and $z_{[r]}$ constitute an element of $\mathcal{D}$ if they yield the value of the given linear combination.\\
    \indent Note that the condition $\on{gcd}(u_1, \dots, u_r) = 1$ implies that $|\mathcal{D}| = |G_1|^{2r - 1}$.\\
    \indent Next, for all $\gamma \in \Gamma$ and $i \in [r]$, define characters $\tilde{\gamma}_i, \tilde{\gamma}'_i \in  \hat{\mathcal{D}}$ by
    \[\tilde{\gamma}_i (y_{[r]}, z_{[r]}) = \gamma(y_i),\hspace{1cm}\tilde{\gamma}'_i (y_{[r]}, z_{[r]}) = \gamma(z_i),\]
    and write $\tilde{\Gamma}$ for the set of these $2rd$ characters. For $\ell \in [d]$ and $i \in [2r]$, we define a map $\Xi_{\ell, i} : C \to \hat{\mathcal{D}}$,  by 
    \[\Xi_{\ell, i}(x)\Big(y_{[2r]}\Big) = \Theta_\ell(x)(y_i).\]
    
    \vspace{\baselineskip}
    
    \begin{claim}
         For $\ell \in [d]$ and $i \in [2r]$, the map $\Xi_{\ell, i} : C \to \hat{\mathcal{D}}$ is well-defined and Freiman-linear.
    \end{claim}
    
    \begin{proof}
        Clearly, the map $y_{[2r]}\mapsto \Xi_{\ell, i}(x)\Big(y_{[2r]}\Big)$ is a homomorphism from $\mathcal{D}$ to $\mathbb{T}$ for all $\ell \in [d], i \in [2r]$ and $x \in C$, i. e. $\Xi_{\ell, i}$ is indeed a map from $C$ to $\hat{\mathcal{D}}$. For Freiman-linearity, take $x, x' \in C$ such that $x + x' \in C$. Then $\Xi_{\ell, i}(x + x')\Big(y_{[2r]}\Big) = \Theta_\ell(x+x')(y_i) = \Theta_\ell(x)(y_i) + \Theta_\ell(x')(y_i) = \Xi_{\ell, i}(x)\Big(y_{[2r]}\Big) + \Xi_{\ell, i}(x')\Big(y_{[2r]}\Big)$. This holds for all $y_{[2r]} \in \mathcal{D}$ so $\Xi_{\ell, i}(x + x') = \Xi_{\ell, i}(x) + \Xi_{\ell, i}(x')$, proving that $\Xi_{\ell, i}$ is Freiman-linear.
    \end{proof}
    
    \vspace{\baselineskip}
    
    \indent Consider the bilinear Bohr variety $\tilde{V}$ inside $C \times \mathcal{D}$ given by columns
    \[B(\Xi_{[d] \times [2r]}(x); \rho) \cap B(\tilde{\Gamma}, \rho).\]
    
    Define map $\phi^{\text{err}}$ on $\tilde{V}$ by setting
    \begin{equation}\phi^{\text{err}}(x; y_{[r]}, z_{[r]}) = \sum_{i\in [r]} u_i \phi(x, y_i) - \sum_{i\in [r]} u_i \phi(x, z_i).\label{errorfunctiondefinition}\end{equation}
    
    Thus, $\phi^{\text{err}}(x; y_{[r]}, z_{[r]})$ controls the error in the naive extension for tuples $y_{[r]}$ and $z_{[r]}$ in the column indexed by $x$.
    
    \vspace{\baselineskip}
    
    \begin{claim}
        The map $\phi^{\text{err}}$ is well-defined and Freiman-bilinear on $\tilde{V}$.
    \end{claim}
    
    \begin{proof}
        To show that $\phi^{\text{err}}$ is well-defined, we just need to ensure that all arguments of $\phi$  in~\eqref{errorfunctiondefinition} belong to $V$. If $(x; y_{[r]}, z_{[r]}) \in \tilde{V}$, then $\uc{\Xi_{\ell, i}(x)(y_{[r]}, z_{[r]})}\leq \rho$ so $\uc{\Theta_{\ell}(x)(y_i)} \leq \rho$ for all $\ell \in [d]$, so $y_i \in B(\Theta_{[d]}(x), \rho)$. Furthermore, for all $\gamma \in \Gamma$ $\uc{\tilde{\gamma}_\ell(y_{[r]}, z_{[r]})} \leq \rho$, so $\uc{\gamma(y_i)} \leq \rho$, and thus $y_i \in B(\Gamma, \rho)$. The same holds for $z_i$.\\
        The check that $\phi^{\text{err}}$ is Freiman-bilinear is direct. Namely, suppose first that $x, x' \in C$ are such that $(x; y_{[r]}, z_{[r]}), (x'; y_{[r]}, z_{[r]}), (x + x'; y_{[r]}, z_{[r]}) \in \tilde{V}$. Then
        \begin{align*}\phi^{\text{err}}(x + x'; y_{[r]}, z_{[r]}) = &\sum_{i\in [r]} u_i \phi(x + x', y_i) - \sum_{i\in [r]} u_i \phi(x + x', z_i) \\
        = &\sum_{i\in [r]} u_i \phi(x, y_i) + \sum_{i\in [r]} u_i \phi(x', y_i) - \sum_{i\in [r]} u_i \phi(x, z_i) - \sum_{i\in [r]} u_i \phi(x', z_i)\end{align*}
        as $\phi$ is itself Freiman-bilinear and all points in arguments of $\phi$ belong to $V$. The last expression equals $\phi^{\text{err}}(x; y_{[r]}, z_{[r]}) + \phi^{\text{err}}(x'; y_{[r]}, z_{[r]})$.\\
        Finally, similar algebraic manipulation shows that $\phi^{\text{err}}$ is Freiman-linear in vertical direction. Namely, assume that $x \in C$ and $(y_{[r]}, z_{[r]}), (y'_{[r]}, z'_{[r]}) \in \mathcal{D}$ are such that $(x; y_{[r]}, z_{[r]}), (x; y'_{[r]}, z'_{[r]}), (x; (y + y')_{[r]}, (z + z')_{[r]}) \in \tilde{V}$. Then
        \begin{align*}\phi^{\text{err}}(x; (y+y')_{[r]}, (z + z')_{[r]}) = & \sum_{i\in [r]} u_i \phi(x, y_i + y'_i) - \sum_{i\in [r]} u_i \phi(x, z_i + z'_i)\\
        = & \sum_{i\in [r]} u_i \phi(x, y_i )+ \sum_{i\in [r]} u_i \phi(x,  y'_i) - \sum_{i\in [r]} u_i \phi(x, z_i) - \sum_{i\in [r]} u_i \phi(x, z'_i)\\
        & = \phi^{\text{err}}(x; y_{[r]}, z_{[r]}) + \phi^{\text{err}}(x; y'_{[r]}, z'_{[r]})\end{align*}
        once again as $\phi$ is itself Freiman-bilinear and all points in arguments of $\phi$ belong to $V$.
    \end{proof}
    
    \vspace{\baselineskip}
    
    Let $\sigma = \rho / 2U$, (recall that $U = \sum_{i \in [r]} |u_i|$,) and let $Z$ be the subvariety of $\tilde{V}$ whose columns, for $x \in C$, are
    \[B(\Xi_{[d] \times [2r]}(x); \sigma) \cap B(\tilde{\Gamma}, \sigma),\]
    i.e. where the Bohr sets in columns are given by radius $\sigma$ instead of $\rho$. Then $\phi^{\text{err}}$ vanishes on $Z$.
    
    \vspace{\baselineskip}
    
    \begin{claim}\label{phierrvanishingz}
        The map $\phi^{\text{err}}$ vanishes on $Z$.
    \end{claim}
    
    \begin{proof}
        If $(x; y_{[r]}, z_{[r]}) \in Z$, then $\sum_{i\in [r]} \lambda_i y_i \in V_{x \bcdot}$ holds for all coefficients $0 \leq \lambda_i \leq u_i$, and so $\sum_{i\in [r]} u_i \phi(x, y_i)  = \phi(x, \sum_{i\in [r]} u_i y_i)$ holds by Freiman-linearity in vertical direction. Since $\sum_{i\in [r]} u_i y_i= \sum_{i\in [r]} u_i z_i$, we get the same value for $z_i$ in place of $y_i$, so we get $\phi^{\text{err}}(x; y_{[r]}, z_{[r]}) = 0$.
    \end{proof}
    
    \vspace{\baselineskip}
    
    We claim that $\phi^{\text{err}}$ takes few values on $\tilde{V} \cap (C' \times B(\tilde{\Gamma}, \sigma'))$ for some radius $\sigma' > 0$ and some slightly smaller coset progression $C' \subseteq C$. The proof relies on the algebraic regularity lemma (Theorem~\ref{algreglemmaintro}). We want to apply the algebraic regularity lemma to make bilinear Bohr varieties $\tilde{V}$ and $Z$ simultaneously quasirandom. However, the lemma is stated for a single bilinear Bohr variety, so we need to consider a further auxiliary variety.\\

    \noindent\textbf{Finding quasirandom pieces of $\tilde{V}$ and $Z$.} Let $\varepsilon > 0$ be a parameter to be chosen later. Let us consider a joint variety $\mathcal{J} = \bigcup_{(x, y) \in C \times C} \{(x, y)\} \times \tilde{V}_{x \bcdot} \times Z_{y \bcdot}$ inside $G^2_1 \times \mathcal{D}^2$. The algebraic regularity lemma with error parameter $\varepsilon$ gives us radii $\rho' \in [\rho/4, \rho/2], \sigma' \in [\sigma/4, \sigma/2]$ and a further symmetric coset progression $C' \subseteq C$ of rank at most $d$ and size $|C'| \geq \exp(-(2d\log(c^{-1} \rho^{-1}\sigma^{-1} \varepsilon^{-1}))^{O(1)}) |G_1|$, such that $\mathcal{J}' \cap \Big((C' \times C') \times (B(\tilde{\Gamma}, \rho') \times B(\tilde{\Gamma}, \sigma'))\Big)$ is $\varepsilon$-quasirandom in the sense of theorem, where $\mathcal{J}'_{(x,y) \bcdot} = B(\Xi_{[d] \times [2r]}(x); \rho') \times B(\Xi_{[d] \times [2r]}(y); \sigma')$ , i.e. the columns of $\mathcal{J}$ are slightly shrunk. Let us stress that we may assume that the resulting coset progression has the product structure $C' \times C'$ due to the final part of statement of Theorem~\ref{algreglemmaintro}.\\
    
    Let us now check that $\tilde{V} \cap (C' \times B(\tilde{\Gamma}, \rho'))$ and $Z \cap (C' \times B(\tilde{\Gamma}, \sigma'))$ are quasirandom in the sense of Theorem~\ref{algreglemmaintro} with error parameter $\varepsilon' = 2(4\sigma^{-1})^{2rd} \varepsilon$. By averaging, there exists $y_0$ for which there are at least $(1-\varepsilon) |C'|^2$ pairs $(x, x') \in C' \times C'$ such that 
    \begin{align*}\Big||B(\Xi_{[d] \times [2r]}(x); \rho') \cap &B(\tilde{\Gamma}, \rho')| |B(\Xi_{[d] \times [2r]}(y_0); \sigma') \cap B(\tilde{\Gamma}, \sigma')| \\
    &- |B(\Xi_{[d] \times [2r]}(x'); \rho') \cap B(\tilde{\Gamma}, \rho')| |B(\Xi_{[d] \times [2r]}(y_0); \sigma') \cap B(\tilde{\Gamma}, \sigma')|\Big| \leq 2\varepsilon |\mathcal{D}|^2.\end{align*}
    By Lemma~\ref{basicbohrsizel}, we have $|B(\Xi_{[d] \times [2r]}(y_0); \sigma') \cap B(\tilde{\Gamma}, \sigma')| \geq {\sigma'}^{4rd} |\mathcal{D}|$, so
    \[\Big||B(\Xi_{[d] \times [2r]}(x); \rho') \cap B(\tilde{\Gamma}, \rho')|  \,\,- \,\,|B(\Xi_{[d] \times [2r]}(x'); \rho') \cap B(\tilde{\Gamma}, \rho')| \Big| \leq 2\varepsilon{\sigma'}^{-4rd}  |\mathcal{D}|^2.\] 
    Thus, there exists a value $v \geq 0$ such that $\Big||B(\Xi_{[d] \times [2r]}(x); \rho') \cap B(\tilde{\Gamma}, \rho')|   - v\Big| \leq \varepsilon' |\mathcal{D}|$, so we may take density parameter, denoted $\delta_{\tilde{V}}$, to be $\frac{v}{|B(\tilde{\Gamma}, \rho')|}$ to get property \textbf{(i)} of the theorem. The other properties are deduced similarly, with density parameter for $Z$ denoted $\delta_Z$. Note that $\delta_{\tilde{V}} \geq (\rho/4)^{4rd}$ and $\delta_Z \geq (\sigma/4)^{4rd}$.\\
    
    Let us remark the key fact that the codimension of the Bohr sets defining rows in the quasirandom pieces, namely $B(\tilde{\Gamma}, \rho')$ and $B(\tilde{\Gamma}, \sigma')$, are independent of the quasirandomness parameter $\varepsilon$. The parameter $\varepsilon$ will be quite small (namely $\varepsilon \leq (2^{-r}{\rho}^{O(1)})$) in order to ensure that $\tilde{V}$ remains very quasirandom even when intersected with $C' \times B(\tilde{\Gamma}, \sigma')$.\\ 
    
    To complete the proof, we use quasirandomness to show that every value taken by $\phi^{\text{err}}$ appears quite frequently inside a slight shrinking of $V \cap (C' \times B')$. This will imply that the set of values is small. Furthermore, we shall show that each value appears in a vast majority of columns. Then a random column takes all values. Let us now turn to details.\\
    
    \noindent\textbf{Values of $\phi^{\text{err}}$ are frequent.} Recall the definition of the sets $E_x(\tau, \tau')$ from the statement of the proposition. Since $\phi$ is Freiman-bilinear on $V$, we have that $E_x(\tau_1, \tau'_1) + E_x(\tau_2, \tau'_2) \subseteq E_x(\tau_1 + \tau_2, \tau'_1 + \tau'_2)$ holds for all $\tau_1, \tau_2, \tau'_1, \tau'_2 \geq 0$ with $\tau_1 + \tau_2, \tau'_1 + \tau_2' \leq \rho$, and also, from Claim~\ref{phierrvanishingz}, $E_x(\sigma, \sigma) = \{0\}$.\\

    \begin{claim}
        For each $x \in C$, we have $|E_x(\rho/2, \rho/2)| \leq \sigma^{-4rd}$.
    \end{claim}

    \begin{proof}
        Let $y_i, z_i \in B(\Theta_{[d]}(x), \rho/2) \cap B(\Gamma, \rho/2)$ be such that $\sum_{i = 1}^r u_i y_i = \sum_{i = 1}^r u_i z_i$. Write $w = \sum_{i = 1}^r u_i \phi(x, y_i) - \sum_{i = 1}^r u_i \phi(x, z_i)$.\\        
        \indent Take any $y'_i, z'_i \in \Big(B(\Theta_{[d]}(x), \sigma) \cap B(\Gamma, \sigma)\Big)$ such that $\sum_{i = 1}^r u_i y'_i = \sum_{i = 1}^r u_i z'_i$. Due to arguments of Claim~\ref{phierrvanishingz}, we have $\sum_{i = 1}^r u_i \phi(x, y'_i) = \sum_{i = 1}^r u_i \phi(x, z'_i)$. Then $y_i + y'_i, z_i + z_i' \in \Big(B(\Theta_{[d]}(x), \rho) \cap B(\Gamma, \rho)\Big)$ and 
        \[\sum_{i = 1}^r u_i \phi(x, y_i + y'_i) - \sum_{i = 1}^r u_i \phi(x, z_i + z'_i) = \Big(\sum_{i = 1}^r u_i \phi(x, y_i) - \sum_{i = 1}^r u_i \phi(x, z_i)\Big) + \Big(\sum_{i = 1}^r u_i \phi(x, y'_i) - \sum_{i = 1}^r u_i \phi(x, z'_i)\Big) = w.\]
        By Lemma~\ref{basicbohrsizel} and Cauchy-Schwarz inequality, the number of $2r$-tuples $(y'_{[r]}, z'_{[r]})$ above is at least $\sigma^{4rd}|G_1|^{2r-1}$. Hence, $w$ is attained for at least $\sigma^{4rd}|G_1|^{2r-1}$ $2r$-tuples of elements in $B(\Theta_{[d]}(x), \rho) \cap B(\Gamma, \rho)$, each satisfying the fixed linear combination. Hence $|E_x(\rho/2, \rho/2)| \leq \sigma^{-4rd}$.
    \end{proof}
    
    Let $x \in \frac{1}{10}\cdot C'$ be arbitrary, where we recall $\frac{1}{k}\cdot C' = \{a \in C' : ka \in C'\}$. We claim that every value in $E_x(\rho'/10, \sigma'/10)$ for such an $x$ appears in $E_{x'}(\rho', \sigma')$ for a vast majority of $x' \in \frac{1}{10} \cdot C'$.\\
    
    Consider any value $w \in E_x(\rho'/10, \sigma'/10)$. Let $Y$ be the collection of all $(y_{[r]}, z_{[r]}) \in B(\Xi_{[d] \times [2r]}(x); \rho'/5) \cap B(\tilde{\Gamma}, \sigma'/5)$ which have $\phi^{\text{err}}(y_{[r]}, z_{[r]}) = w$. As in the previous claim, we have $|Y| \geq (\sigma'/10)^{4rd}|\mathcal{D}|$.\\

    Consider $t \in \frac{1}{2}\cdot C'$. Claim~\ref{phierrvanishingz} implies that $\phi^{\text{err}}(x; y_{[r]}, z_{[r]}) = 0$ for all $(y_{[r]}, z_{[r]}) \in B(\Xi_{[d] \times [2r]}(t); \sigma') \cap B(\tilde{\Gamma}, \sigma')$. Note that $Y \subseteq B(\tilde{\Gamma}, \sigma')$. By Lemmas~\ref{genappendknhoods} and~\ref{appendonesided}, we have 
    \[|Y \cap B(\Xi_{[d] \times [2r]}(t); \sigma')| \geq \frac{1}{2}\delta_Z |Y|\]
    for all but at most $O(\varepsilon^{\frac{1}{8}} (\sigma'/10)^{-8rd}) \delta_Z^{-2}|C'|$ elements $t \in \frac{1}{2}\cdot C'$. Take any such that $t$. Then $w \in E_{x + t}(\rho', \sigma')$. Indeed, taking any $(y_{[r]}, z_{[r]}) \in Y \cap B(\Xi_{[d] \times [2r]}(t); \sigma')$ we get 
    \[\phi^{\text{err}}(x + t; y_{[r]}, z_{[r]}) = \phi^{\text{err}}(x; y_{[r]}, z_{[r]}) + \phi^{\text{err}}(t; y_{[r]}, z_{[r]}) = \phi^{\text{err}}(x; y_{[r]}, z_{[r]}) = w.\]

    In particular, we can obtain all but at most $O(\varepsilon^{\frac{1}{8}} (\sigma'/10)^{-16rd}))|C'|$ elements of $\frac{1}{10}\cdot C'$ as $x + t$ for some $t$ as above. Hence, the total number of values of $\phi^{\text{err}}$ is at most $2 (\sigma'/10)^{-4rd}$. Provided we choose $\varepsilon = \bsc (\sigma'/10)^{128rd})100^{-8d}$, we get a column in which all values appear, proving the claim. We set $\tilde{\rho} = \rho'/10, \tilde{\sigma} = \sigma'/10$ and final set of columns $\frac{1}{10} \cdot C'$.
\end{proof}

\subsection{Obtaining subgroups in columns}\label{subgroupsincolumnssubsec}

In this subsection, we ensure that columns of the Bohr variety become subgroups.\\

Recall that a set $E$ has rank at most $r$ if $E = \langle S\rangle_{\{-1,0,1\}}$ for a set $S$ of size $r$.\\

\begin{proposition}\label{finalExtensionStep2}
    Let $C$ be a symmetric proper coset progression of rank $d$ and density $c$ in $G_1$. Let $\Theta_1, \dots, \Theta_d : C \to \hat{G}_2$ be Freiman-linear maps. Let $\Gamma \subseteq \hat{G}_2$ be a set of characters of size $d$ and let $\rho > 0$. Let $V$ be the bilinear Bohr variety given by
    \begin{equation}\label{detaileddomaindefn}V = \bigcup_{x \in C} \{x\} \times \Big(B(\Gamma, \rho) \cap B(\Theta_{[d]}(x), \rho)\Big).\end{equation}
    Let $\phi : V \to H$ be a Freiman-bilinear map on $V$.\\
    \indent There exist a set $A \subset [-R, R]^d$ of size at most $O((2d \log \rho^{-1})^{O(1)})$, where $R \leq \exp((2d \log \rho^{-1})^{O(1)})$, a coset progression $C' \subseteq C$, a radius $\sigma \geq \exp(-(2d\log(c^{-1}\rho^{-1}))^{O(1)}))$, and an $E$-bihomomorphism $\psi : V' \to H$, where $V'$ has columns $B(\Gamma, \sigma) \cap \cap_{\lambda \in A} \on{ker} \lambda \cdot \Theta(x)$, for $x \in C'$, for a set $E$ of rank $(2d\log(c^{-1}\rho^{-1}))^{O(1)}$, such that 
    \[\psi(x, y_0 + 2y_1 + \dots + 2^r y_r) - \sum_{i = 0}^r 2^i \phi(x, y_i) \in E\]
    for all $(x, y_i) \in V \cap (C' \times B(\Gamma, \sigma))$.
\end{proposition}

\vspace{\baselineskip}

\begin{proof}
    Let $r = (2 d \log \rho^{-1})^{O(1)}$ and $\rho' = (2d)^{-4d}\rho$ be the quantities appearing in Corollary~\ref{subgroupgenerationiteratedsumsbohr} with the property that, for any Bohr set $B$ of codimension $2d$ and radius $\rho'/100$, the subgroup $\langle B \rangle$ equals $4B + 2\cdot B + \dots + 2^{r - 4} \cdot B$,\footnote{Note that we replaced $r$ by $r-4$. This is done so that the linear combinations have $r$ summands, which simplifies the notation slightly.} and every element $x \in \langle B \rangle$ has a at least $\alpha |G_2|^{r-1}$ representations as $y_1 + y_2 + y_3 + y_4 + 2y_5 + \dots + 2^{r-4} y_r$ with $y_i \in B$, where $\alpha = \exp(-(2d \log \rho^{-1})^{O(1)})$. Let us define coefficients $u_1 = u_2 = u_3 = u_4 = 1,  u_i = 2^{i - 4}$ for $i \in [5,r]$.\\

   Noting that the sum of coefficients is at most $2^r$, apply Proposition~\ref{controllingerrorscolsprop} to find a coset progression $C' \subseteq C$ of rank at most $d$ and size $|C'| \geq \exp(-(2rd\log(c^{-1} \rho^{-1}))^{O(1)}) |G_1|$, an element $x_0 \in C'$, a radius $\tilde{\rho} \in [\rho' / 100, \rho'/20]$, and a radius $\tilde{\sigma} \in [\rho' / (100\cdot 2^r), \rho' / (20\cdot 2^r)]$, such that $E_x(\tilde{\rho}, \tilde{\sigma}) \subseteq E_{x_0}(10\tilde{\rho}, 10\tilde{\sigma})$ for all $x \in C'$, where $E_x(\tau, \tau')$ was defined in the proposition.\\

   \noindent\textbf{Values of $\phi^{\text{err}}$ belong to a set of small rank.} Write $E = E_{x_0}(10\tilde{\rho}, 10\tilde{\sigma})$. Since we passed to $\rho'$ instead of $\rho$, we may apply Proposition~\ref{generalextensionserror} to the Freiman-linear map $y \mapsto \phi(x_0, y)$ on domain $B(\Gamma, 10\tilde{\rho}) \cap B(\Theta(x_0), 10\tilde{\rho})$, showing that the slightly larger set $E(10\tilde{\rho}, 10\tilde{\rho})$ is contained a set $E$ of rank at most $(r + d )^{O(1)} \leq (2 d \log \rho^{-1})^{O(1)}$.\\

   Since $E_x(\tilde{\rho}, \tilde{\sigma}) \subseteq E$, we conclude that we conclude that whenever $y_1, \dots, z_r \in \Big(B(\Theta(x); \tilde{\rho}) \cap B(\Gamma, \tilde{\sigma})\Big)$ satisfy $\sum_{i \in [r]} u_i y_i = \sum_{i \in [r]} u_i z_i$, then
\[\sum_{i \in [r]} u_i \phi(x, y_i) -  \sum_{i \in [r]} u_i \phi(x, z_i) \in E.\]

  In order words, we may naively extend $\phi$ to an $E$-homomorphism in the column indexed by $x$ to the set
  
    \begin{equation}4\Big(B(\Theta(x); \tilde{\rho}/2) \cap B(\Gamma, \tilde{\sigma}/2)\Big) + 2 \cdot \Big(B(\Theta(x); \tilde{\rho}/2) \cap B(\Gamma, \tilde{\sigma}/2)\Big)\Big) + \dots + 2^{r-4} \cdot \Big(B(\Theta(x);\tilde{\rho}/2) \cap B(\Gamma, \tilde{\sigma}/2)\Big)\Big).\label{finalcolumnsbeforesubgroups}\end{equation}
    
    Corollary~\ref{subgroupgenerationiteratedsumsbohr} implies that every element of $B(\Gamma, \tilde{\sigma}/4)$ can be represented in at least $\alpha' |G_2|^{r-1}$ many ways as $x = y_1 + y_2 +y_3 + y_4 + 2y_5 + \dots + 2^{r-4}y_r$ for elements $y_1, \dots, y_r \in B(\Gamma, \tilde{\sigma}/2)$, and every element of $S_x = \langle B(\Theta(x), \tilde{\rho}/2) \rangle$ can be represented in at least $\alpha' |G_2|^{r-1}$ many ways as analogous sum using elements in $B(\Theta(x), \tilde{\rho}/2)$, for some $\alpha' \geq \exp(-(2d \log \rho^{-1})^{O(1)})$. We misuse the notation, and write $\alpha$ instead of $\on{min}(\alpha, \alpha')$.\\
    
    Let $K = 2^{20}r^2\mathsf{C}_{\on{spec}}d (\alpha^4 \tilde{\rho}\tilde{\sigma})^{-2}$. Proposition~\ref{sumsofgenbohrs} implies that if
    \begin{equation}\langle \Gamma \rangle_{[-K, K]} \cap \langle \Theta(x)\rangle_{[-K, K]} = \{0\},\label{fourierconditionforsumsofbohrsets}\end{equation}
    then the set in~\eqref{finalcolumnsbeforesubgroups} contains $S_x \cap  B(\Gamma, \tilde{\sigma}/4)$.\\
    
    We now pass to a subprogression $C'' \subseteq C'$ of the same rank, in which a vast majority of columns has the desired structure. We carry out an iterative argument, setting $C'' = C'$ initially. We keep track of a lattice $\Lambda \leq \mathbb{Z}^d$, generated by elements $\lambda_1, \dots, \lambda_s \in [-K, K]^d$ after $s$\tss{th} step, such that $\lambda \cdot \Theta(x) = 0$ holds for all $x \in C''$ and $\lambda \in \Lambda$. Additionally, in order to control the behaviour of subgroups in columns, we keep track of another lattice $M \leq \mathbb{Z}^d$. Our eventual goal is to have
    \begin{equation}
        \label{lowordergeneratorsbgps} \cap_{\mu \in M \cap [-K, K]^d} \ker \mu \cdot \Theta(x) \subseteq \langle B(\Theta(x); \tilde{\rho}/10000) \rangle
    \end{equation}
    for all but at most $\eta |C''|$ elements of $C''$ and $M$ is generated by vectors $\mu$ with property that $a \mu \in \Lambda$ for some $a \in \exp((2d\log \rho^{-1})^{O(1)})$.\\
    
    \noindent\textbf{Ensuring that~\eqref{fourierconditionforsumsofbohrsets} and~\eqref{lowordergeneratorsbgps} hold.} Let $\eta > 0$ be a parameter to be chosen later. For technical reasons, we consider a slightly more general condition that 
    \begin{equation}\langle \Gamma \rangle_{[-K, K]} \cap \langle \Theta(x), \Theta(y), \Theta(z)\rangle_{[-K, K]} = \{0\},\label{fourierconditionforsumsofbohrsets2}\end{equation}
    and 
    \begin{equation}(\forall \lambda, \mu, \nu \in [-K, K]^d) \,\, \lambda \cdot \Theta(x) + \mu \cdot \Theta(y) + \nu \cdot \Theta(z) = 0\implies \lambda , \mu, \nu \in \Lambda,\label{fourierconditionforsumsofbohrsets3}\end{equation}
    hold for all but at most $\eta |C''|^3$ choices of $x,y,z \in C''$. Suppose that this does not yet hold. By averaging, we get $\eta (2K)^{-4d}|C''|$ elements $x \in C''$, a linear combination $\lambda \in [-K, K]^d \setminus \Lambda$ and an element $v \in \hat{G}_2$ such that $\lambda \cdot \Theta(x) = v$. We use Lemma~\ref{fewvalsCPFLin}, we may pass to a further subprogression of same shape, losing a factor of $2^{-O(d^2)}(\eta K^{-4d})^{d + 1}$ in the size, on which $\lambda \cdot \Theta$ vanishes, so we may replace $\Lambda$ by $\Lambda + \langle \lambda \rangle$.\\ 
    On the other hand, if~\eqref{lowordergeneratorsbgps} fails, by Proposition~\ref{subgroupgeneratedbybohrset}, there exist $m, R \leq \exp((2d \log \rho^{-1})^{O(1)})$ and $\mu \in [-R, R]^d$, such that for at least $\exp(-(2d \log \rho^{-1})^{O(1)})|C''|$ elements $x \in C''$, we have 
    $\langle B(\Theta(x); \rho'/10000)\rangle \subseteq \on{ker} \mu \cdot \Theta(x)$ and  $m \mu \cdot \Theta(x) = 0$. If $m \mu \notin \Lambda$, add $m \mu$ to $\Lambda$, add $\mu$ to $M$ and pass to shrinking of $C''$ on which $m\mu \cdot \Theta(x) = 0$. Otherwise, add $\mu$ to $M$.\\
    
    Due to Lemma~\ref{nestedLattices}, as either $M$ or $\Lambda$ increases in each step, the procedure terminates after at most $O(d^2 (\log d + \log K))$ steps. Moreover, taking product of all $m$ such that $m\mu \in \Lambda$ at each step, we conclude that we have $\tilde{m} \leq \exp(O(d \log (K \rho^{-1}))^{O(1)})$ such that $\tilde{m} \mu \in \Lambda$ for all $\mu \in M$.\\
    
    Having passed to the desired coset progression $C''$, we conclude that $\phi$ naively extends to $\psi$ defined on subgroups in each column. It remains to check the $E$-homomorphism property in the horizontal direction.\\
    
    For any $y \in S_{x_1} \cap S_{x_2} \cap S_{x_3} \cap S_{x_4} \cap B(\Gamma, \sigma')$, where $x_1 + x_2 = x_3 + x_4$, it suffices to find a single common $r$-tuple belonging to all 4 columns $B(\Theta(x_i); \rho'') \cap B(\Gamma, \sigma')$. Shrinking the radii $\rho''$ and $\sigma'$ by a factor of 3, the $r$-tuple is guaranteed to belong to the fourth column. Use~\eqref{fourierconditionforsumsofbohrsets2}. A single common tuple suffices by the previous work, which tells us that all $m$-tuples give roughly the same value.
    
    \begin{claim}\label{horizontalcheckclaimsubgroupinter}
        Suppose that triple $(x_1, x_2, x_3)$ satisfies~\eqref{fourierconditionforsumsofbohrsets3}. Then 
        \[S_{x_1} \cap S_{x_2} \cap S_{x_3} = \langle B(\Theta(x_1), \Theta(x_2), \Theta(x_3); \tilde{\rho}/2)\rangle.\]
    \end{claim}

    \begin{proof}
        Since $B(\Theta(x_1), \Theta(x_2), \Theta(x_3); \tilde{\rho}/2) \subseteq S_{x_i}$, we have $\langle B(\Theta(x_1), \Theta(x_2), \Theta(x_3); \tilde{\rho}/2)\rangle \subseteq S_{x_1} \cap S_{x_2} \cap S_{x_3}$, so we need to show reverse inclusion. Thus, let $y \in S_{x_1} \cap S_{x_2} \cap S_{x_3}$ be arbitrary. We know that $y$ can be written as $y_1 + y_2 + y_3 + y_4 + +2z_1 + \dots + 2^{r-4}z_{r-4}$ for $y_1, \dots, z_{r-4} \in B(\Theta(x_i), \tilde{\rho}/2)$ in at least $\alpha|G_2|^{r-1}$ many ways. By Proposition~\ref{sumsofgenbohrs}, we get the desired inclusion.
    \end{proof}
    
    \begin{claim}\label{horizontalcheckclaimiteratedsum}
        Suppose that $x_1, x_2, x_3$ satisfy~\eqref{fourierconditionforsumsofbohrsets2}. Then 
        \begin{align*}\langle B(\Theta(x_1), \Theta(x_2), \Theta(x_3); \tilde{\rho})\rangle \cap B(\Gamma, \sigma'/2) \subseteq 4 (B(\Theta(x), &\Theta(y), \Theta(z); \tilde{\rho}) \cap B(\Gamma, \sigma')) \\
        +& 2 \cdot  (B(\Theta(x), \Theta(y), \Theta(z); \tilde{\rho}) \cap B(\Gamma, \sigma')) + \dots \\
        +& 2^{r-4} \cdot (B(\Theta(x), \Theta(y), \Theta(z); \tilde{\rho})\cap B(\Gamma, \sigma')).\end{align*}
    \end{claim}

    \begin{proof}
        Similarly, we rely on~\eqref{fourierconditionforsumsofbohrsets2} and Proposition~\ref{sumsofgenbohrs}.
    \end{proof}
    
    We are now ready to show that whenever  $x_1, x_2, x_3$ satisfy~\eqref{fourierconditionforsumsofbohrsets2} and~\eqref{fourierconditionforsumsofbohrsets3}, then $\psi(x_1, y) + \psi(x_2, y) - \psi(x_3, y) - \psi(x_4, y) \in 2E - 2E$ holds for all $y \in S_{x_1} \cap S_{x_2} \cap S_{x_3} \cap S_{x_4} \cap B(\Gamma, \sigma'/2)$. Take any such $y$. By Claim~\ref{horizontalcheckclaimsubgroupinter}, we have that $y \in \langle B(\Theta(x), \Theta(y), \Theta(z); \tilde{\rho})\rangle \cap B(\Gamma, \sigma')$. By Claim~\ref{horizontalcheckclaimiteratedsum} we have $z_1, \dots, z_r \in B(\Theta(x_1), \Theta(x_2), \Theta(x_3); \tilde{\rho}) \cap B(\Gamma, \sigma')$ such that $z_1 + \dots + z_4 + 2z_5 + \dots + 2^{r-4} z_r = y$. Then $(x_4, z_i) \in B(\Theta(x_4); \tilde{\rho}) \cap B(\Gamma, \sigma')$, which is in the domain of $\phi$. Hence
    \begin{align*}\psi(x_1, y) +& \psi(x_2, y) - \psi(x_3, y) - \psi(x_4, y)\\
    &\in \Big(\sum_{i \in [r]} u_i\phi(x_1, z_i)\Big) + \Big(\sum_{i \in [r]} u_i\phi(x_2, z_i)\Big) - \Big(\sum_{i \in [r]} u_i\phi(x_3, z_i)\Big)\\
    &\hspace{2cm}-\Big(\sum_{i \in [r]} u_i\phi(x_4, z_i)\Big) + (2E - 2E) = 2E - 2E.\end{align*}

    It follows that $\psi$ is defined on a vast majority of the variety. Let us show that the groups $S_x$ have the claimed structure. By~\eqref{lowordergeneratorsbgps} we have $\cap_{\mu \in M \cap [-K, K]^d} \on{ker} \lambda \cdot \Theta(x) \subseteq S_x$. By Theorem~\ref{quantLattice}, we may find a generating set $A$ of $M \cap [-K, K]^d$ of desired size.\\
    
    Finally, pick suitable $\eta$ and apply Proposition~\ref{almostfullbilinearmap} to complete the proof.
\end{proof}

\subsection{Extending domain to structured product}\label{extendingtoproductsubsec}

In this subsection, we extend the domain to a product of a Bohr set and coset progression. 

\begin{proposition}\label{finalExtensionStep3}
    Let $\Theta_1, \dots, \Theta_d : C \to \hat{G}_2$ be Freiman linear maps on a symmetric proper coset progression $C$ of rank $d$ and density $c$ inside $G_1$ and suppose that there exists $m$ such that $m \Theta_i(x) = 0$ holds for all $x \in C$, $i \in [d]$. Let $\Gamma \subseteq \hat{G}_2$ be a set of size $r$ and let $\rho > 0$. Define a bilinear Bohr variety $V \subseteq C \times G_2$ whose columns are
    \[\Big(\bigcap_{i \in [d]}\on{ker} \Theta_i(x) \Big)\cap B(\Gamma, \rho)\]
    for $x \in C$. Let $\phi : V \to H$ be an $E$-bihomomorphism with $\phi(x, 0) \in E$ for a set $E$ of rank $d'$. Then, there exist 
    \begin{itemize}
        \item a coset progression $C' \subseteq C$ of rank at most $d$ and density $\exp(-2rd\log(c^{-1}\rho^{-1})^{O(1)})$,
        \item a set $\Gamma' \subseteq \hat{G}_2$ of size at most $2rd\log(c^{-1}\rho^{-1})^{O(1)}$,
        \item an $E'$-bihomomorphism $\tilde{\phi}$ on $C' \times B(\Gamma \cup \Gamma', \rho/100)$, an $E'$-homomorphism $\psi : B(\Gamma \cup \Gamma', \rho /2) \to H$ and an element $x_0 \in G_1$ such that for each $(x,y)$ in the domain we have 
        \[\tilde{\phi}(x,y) \in \phi(x, y - z) + \psi(z) + \phi(x - x_0, z) + E'.\]
    \end{itemize}
\end{proposition}

\begin{proof}
    Write $S_x = \cap_{i \in [d]}\on{ker} \Theta_i(x)$ for the subgroup which defines column of $x$ after intersection with $B(\Gamma, \rho)$. Since $m \Theta_i(x) = 0$ holds for all $x \in C$, we have $|G_2 : S_x| \leq m^d$. Our strategy for defining the extension map $\tilde{\phi}$ will be as follows. Pick an arbitrary element $x_0 \in C$, and extend $\phi_{x \bcdot}$ from its column to $B(\Gamma, \rho/16)$ using Proposition~\ref{ehomextensionsfromsubgroups}, giving an $E'$-homomorphism $\psi$ for some set $E'$ of rank at most $O(d d'\log m)$. Then, for arbitrary $x \in C$, given a coset of $S_x$ defined by $\Theta(x)(\bcdot) = t$ we extend $\phi_{x \bcdot}$ to $\Theta(x)^{-1}(t) \cap B(\Gamma, \rho/128)$ by taking an element $z \in S_{x - x_0} \cap \{\Theta(x)(\bcdot) = t\} \cap B(\Gamma, \rho/32)$ and then defining
    \begin{equation}\tilde{\phi}(x, y) = \phi(x, y - z) + \psi(z) + \phi(x - x_0, z).\label{tildephiproductextndefn}\end{equation}
    Note that $(x, y- z), (x - x_0, z) \in V$ and note that $\Theta(x_0)(z) = \Theta(x_0 - x)(z) + \Theta(x)(z) = t$, first summand being 0, the other being $t$.\\

    Let us now turn to details. Let $\varepsilon > 0$ be a positive parameter to be chosen later. We apply algebraic regularity lemma (Theorem~\ref{algreglemmaintro}) with error parameter $\varepsilon$, which gives us a proper symmetric coset progression $C' \subseteq C$ of rank $d$ and size 
    \[|C'| \geq \exp\Big(-d^{O(1)} r^{O(1)} \log^{O(1)} (\varepsilon^{-1}) \log^{O(1)} (\rho^{-1}) \log^{O(1)}m \Big) |C|,\]
    such that the bilinear Bohr variety is sufficiently quasirandom on $C' \times B(\Gamma, \rho')$, with $\rho' \in [\rho/2, \rho]$.\\
    
    Let $R_1 = \mathsf{C}_{\on{spec}} r 2^{20r + 8} \rho^{-2r-1} m^{4r^2 + 2rd}$, which is the quantity that stems from Proposition~\ref{ehomextensionsfromsubgroups}. We shall also apply Proposition~\ref{bohrsizeLargeFC} at some point, leading to $R_2 = \mathsf{C}_{\on{spec}} 2^{10 + 12r} r \rho^{-2d} m^{2d}$. Let $R = \on\{R_1, R_2\}$.\\
    \indent Note that the proof of Theorem~\ref{algreglemmaintro} shows that the lattice $\Lambda$ of vanishing linear combinations of $\Theta(x)$ is the same for all but at most $(1-\varepsilon)|C'|$ points in $C'$. The proof also shows that, recalling that $m\Theta(x) = 0$, $\langle\Theta_1(x), \dots, \Theta_r(x) \rangle_{[0, m-1]} \cap \langle \Gamma \rangle_R = \{0\}$ for at least $(1-\varepsilon)|C'|$ points in $C'$. Let $X_0 \subseteq C'$ be the set of all points for which these two properties hold. Hence $\on{Im} \Theta(x)$ is the same subgroup $T \leq \mathbb{T}^d$ for all points in $X_0$. Note that $mt = 0$ for all $t \in T$.\\
    
    We say that a pair $(x_0, x)$ in $X_0 \cap \frac{1}{2}C'$ is \emph{good} if for all $t \in T$ we have 
    \begin{equation}
        |S_{x - x_0} \cap \Theta(x)^{-1}(t) \cap B(\Gamma, \rho/32)| \geq \frac{1}{4}m^{-2d} (\rho/32)^{r} |G_2|.\label{gooddefinpair}
    \end{equation}

    By quasirandomness, we show that a vast majority of pairs are good.

    \begin{claim}
        All but at most $1000 2^{12r}m^{2d} \rho^{-2r}\sqrt[8]{\varepsilon}$ pairs $(x_0, x)$ in $X_0 \cap \frac{1}{2}C'$ are good.
    \end{claim}

    \begin{proof}
        Slightly more generally, it suffices to bound the number of pairs $(x_1, x_2) \in C' \times C'$ such that $S_{x_1} \cap \Theta(x_2)^{-1}(t) \cap B(\Gamma, \rho/32)$ is small. Firstly, we show that $\Theta(x_2)^{-1}(t) \cap B(\Gamma, \rho/32)$ has size at least $\frac{1}{2}m^{-d} (\rho/32)^{r} |G_2|$ most of the time. Take $\tilde{\rho} \in [\rho/64, \rho/32]$ such that $B(\Gamma, \tilde{\rho})$ is weakly regular. By Proposition~\ref{bohrsizeLargeFC}, all Fourier coefficients $\chi$ such that $|\widehat{\id_{B(\Gamma, \tilde{\rho})}}| \geq \frac{1}{2}m^{-d}\frac{|B(\Gamma, \tilde{\rho})|}{|G_2|}$ lie in the span $\langle \Gamma\rangle_{R_2}$.\\
        
        We now show that $\Theta(x_2)^{-1}(t) \cap B(\Gamma, \rho/32)$ is large. Namely, if $\langle\Theta_1(x_2), \dots, \Theta_r(x_2) \rangle_{[0, m-1]} \cap \langle \Gamma \rangle_R = \{0\}$, then we have
        \begin{align*}|\Theta(x_2)^{-1}(t) \cap B(\Gamma, \tilde{\rho})| = &|G_2| \exx_{y \in G_2} \id(\Theta(x_2)(y) = t) \id_{B(\Gamma, \tilde{\rho})}(y) \\
        =  &m^{-d} |G_2|\sum_{a_1, \dots, a_d \in [0, m-1]} \id_{B(\Gamma, \tilde{\rho})}(y) \on{e}\Big(\sum_{i \in [d]} a_i \Theta(x_2)( y)\Big) \\
        =& m^{-d} |G_2| \sum_{a_1, \dots, a_d \in [0, m-1]} \widehat{\id_{B(\Gamma, \tilde{\rho})}}\Big(\sum_{i \in [d]} a_i \Theta(x_2)\Big)\end{align*}
        and the contribution from $a_1, \dots, a_d = 0$ is $m^{-d} |B(\Gamma, \tilde{\rho})|$, while the rest have a total absolute value at most $\frac{1}{2}m^{-d} |B(\Gamma, \tilde{\rho})|$.\\
        Now take any $x_2$ for which $\Theta(x_2)^{-1}(t) \cap B(\Gamma, \rho/32)$ is large. By Lemmas~\ref{genappendknhoods} and~\ref{appendonesided}, we are done.
    \end{proof}

    We pick $x_0$ uniformly at random from $\frac{1}{2}C'$. Thus, with high probability, we have that $(x_0, x)$ is good for almost all $x\in X_0 \cap \frac{1}{2}C'$. More precisely, if $X_1$ is the set of $x \in X_0 \cap \frac{1}{2}C'$ such that $(x_0, x)$ is good, then with probability at least $1-\sqrt[32]{\varepsilon}$ we have $|X_1| \geq (1 - \sqrt[32]{\varepsilon})|\frac{1}{2}C'|$, provided $2^{-12r - 2d-10}m^{-2d} \rho^{2r} \geq \sqrt[16]{\varepsilon}$. Define $\psi$ as an arbitrary extension to $B(\Gamma, \rho/16)$ using Proposition~\ref{ehomextensionsfromsubgroups}. Hence, we may extend $\phi$ using the definition~\eqref{tildephiproductextndefn} to all columns for $x \in X_1$. For this, when extending as in~\eqref{tildephiproductextndefn} we take a random element $z$ uniformly inside $S_{x - x_0} \cap \Theta(x)^{-1}(t) \cap B(\Gamma, \rho/32)$. It remains to check that $\tilde{\phi}$ has desired properties with high probability. Moreover, we restrict rows to $S_{x_0}$, so our domain becomes
    \[\Big(X_1 \times (S_{x_0} \cap B(\Gamma, \rho/128))\Big) \cap \Big(\bigcap_{x \in C'} \{x\} \times S_x\Big).\]

    \noindent\textbf{Vertical check.} Let $x \in X_1$, $y_1, y_2, y_3 \in B(\Gamma, \rho/128)$ be such that $y_3 = y_1 + y_2$. By definition, we take some $z_1, z_2, z_3 \in B(\Gamma, \rho/32) \cap S_{x - x_0}$ such that $y_i - z_i \in S_{x \bcdot}$ and set
    \[\tilde{\phi}(x, y_i) = \phi(x, y_i - z_i) + \psi(z_i) + \phi(x - x_0, z_i).\]
    Note firstly that, due the fact that all elements below belong to $\{x\} \times (V_{x \bcdot})$
    \[\phi(x, y_1 - z_1) +\phi(x, y_2 - z_2) - \phi(x, y_3 - z_3) \in \phi(x, y_1 + y_2 - y_3 - z_1 - z_2 + z_3) + E = \phi(x, z_3 - z_1 -z_2) + E.\]
    Also, let us stress that $z_3 - z_1 - z_2 \in S_{x \bcdot}$. Furthermore,
    we have $z_1, z_2, z_3 \in S_{x - x_0 \bcdot} \cap B(\Gamma, \rho/32)$, so we have $z_1 + z_2 - z_3 \in S_{x - x_0 \bcdot} \cap B(\Gamma, \rho/2)$ and thus
    \[\phi(x - x_0, z_1) + \phi(x - x_0, z_2) - \phi(x - x_0, z_3) \in \phi(x - x_0, z_1 + z_2 - z_3) + E.\]
    But $\Theta_i(x)(z_1 + z_2 - z_3) = 0$ and $\Theta_i(x - x_0)(z_1 + z_2 - z_3) = 0$ imply that $\Theta_i(x_0)(z_1 + z_2 - z_3) = 0$ as well, so $(x_0, z_1 + z_2 - z_3) \in V$. Hence, $\psi(z_1 + z_2 - z_3) - \phi(x_0, z_1 + z_2 - z_3) \in E'$. Putting everything together, we get
    \begin{align*}\tilde{\phi}(x, y_1) +& \tilde{\phi}(x, y_2) - \tilde{\phi}(x, y_3)\\
    &\in (\phi(x, z_3 - z_1 -z_2) + E) 
    + (\phi(x_0, z_1 + z_2 - z_3) + 2 E') + (\phi(x - x_0, z_1 + z_2 - z_3) + E) \\
    &\subseteq 2E' + 4E,\end{align*}
    where we used symmetry $\phi(x, -t) + \phi(x, t) - \phi(x, 0) \in E$ and $\phi(x, 0) \in E$.\\

    \noindent\textbf{Horizontal check.} This part of the argument is more subtle than the previous check and we need an approximate cocycle identity first. A similar result was proved in~\cite{extensionspaper}, appearing as Lemma 10.

    \begin{claim}\label{cocycleEqnBil}
        Let $x,z, u \in \frac{1}{3}C'$ and $y, w, v \in B(\Gamma, \rho/10)$ be given. Suppose that $\Theta(x)(y) = \Theta(z)(w) = \Theta(u)(v)$ and that $\Theta(a)(b) = 0$ for remaining 6 points $(a,b)$ in $\{x,z,u\} \times \{y, w, v\}$. Then we have identity (and all arguments belong to $V$)
        \begin{align*}\phi(x, w) \in\,\, &\phi(x + z, -y + w) + \phi(x  - u, y + v) - \phi(z - u, w + v)\\
        &\hspace{1cm} -\phi(x,v) + \phi(z,y) + \phi(u,y) + \phi(z,v) - \phi(u,w) + 5E.\end{align*}
    \end{claim}

    \begin{proof} Using the fact that $\phi$ is $E$-bilinear we have
        \begin{align*}\phi(x, w) \in\,\, & \phi(x, w + v) - \phi(x, v) + E\\
        \subseteq\,\,&\phi(x + z - u, w + v) - \phi(z - u, w + v) - \phi(x, v) + 2E \\
        \subseteq\,\,&\phi(x + z - u, -y + w) + \phi(x + z - u, y + v) - \phi(z - u, w + v) - \phi(x, v) + 3E \\
        \subseteq\,\,&\phi(x + z - u, - y + w) + \phi(x  - u, y + v) + \phi(z, y) + \phi(z, v) - \phi(z - u, w + v) - \phi(x, v) + 4E\\
        \subseteq\,\,&\phi(x + z, -y + w)  + \phi(u, y) - \phi(u,w) + \phi(x  - u, y + v)\\
        &\hspace{1cm}+ \phi(z, y) + \phi(z, v) - \phi(z - u, w + v) - \phi(x, v) + 5E\\
        =\,\,&\phi(x + z, -y + w) + \phi(x  - u, y + v) - \phi(z - u, w + v)\\
        &\hspace{1cm} -\phi(x,v) + \phi(z,y) + \phi(u,y) + \phi(z,v) - \phi(u,w) + 5E.\qedhere\end{align*}
    \end{proof}
    
    Let $x_1, x_2, x_3, x_4 \in C'$ be arbitrary additive quadruple with $x_1 - x_2 + x_3 - x_4 = 0$. Let $y \in B(\Gamma, \rho/128)$ be given. We show that the corresponding horizontal additive quadruple is respected with high probability. Then we extend $\phi$ at these points using elements $z_1, z_2, z_3, z_4$ such that $y - z_i \in S_{x_i}$ and $z_i \in S_{x_i - x_0} \cap B(\Gamma, \rho/32)$. Note that 
    
    \[\Theta(x_0)(z_1 - z_2 + z_3 - z_4) = \sum_{i \in [4]} (-1)^{i + 1} \Theta(x_0)(z_i) = \sum_{i \in [4]} (-1)^{i + 1} \Theta(x_i)(z_i) = \sum_{i \in [4]} (-1)^{i + 1} \Theta(x_i)(y) = 0,\]
    so $(x_0, z_1 - z_2  + z_3 - z_4) \in V$. Hence $\sum_{i \in [4]} (-1)^{i+1}\psi(z_i) \in \psi(z_1 - z_2 + z_3 - z_4) + 3E' = \phi(x_0, z_1 - z_2 + z_3 - z_4) + 3E'$.\\
    \indent Thus, we need to check that the cocycle expression
    \begin{equation}
        \sum_{i \in [4]} (-1)^{i+1} \phi(x_i, y - z_i) + \sum_{i \in [4]} (-1)^{i+1}\phi(x_i - x_0, z_i) + \phi(x_0, z_1  - z_2 + z_3 - z_4)\label{approxcocycleeqnmain}
    \end{equation}
    belongs to $O(1)E'$ when $x_1 - x_2 + x_3 - x_4 = 0$.\\

    \noindent\textbf{Reducing to a special case.} Write $t_i = \Theta(x_i)(y)$. Then we also have $\Theta(x_i)(z_i) = \Theta(x_0)(z_i) = t_i$. Thus $t = (t_1, t_2, t_3, t_4)$ is an additive quadruple in $T$. In order to complete the proof, we reduce to the case when $t$ has two zeros, and another value appearing twice.\\
    
    Define further quadruples $\tilde{t}^1 = (t_1, t_1, 0, 0), \tilde{t}^2 = (0, 0, t_4, t_4)$ which $\tilde{t}^3 = (0, t_2 - t_1, t_2 - t_1, 0)$ add up to $t$. Using quasirandomness, we observe that with probability at least $1 - \sqrt[32]{\varepsilon}$ we may find $\tilde{z}^j_i, \tilde{y}^j \in B(\Gamma, \rho/128)$ for $j \in [2], i \in [4]$ with $\Theta(x_0)(\tilde{y}^j) = 0$ and $\tilde{t}^j_i = \Theta(x_i)(\tilde{y}^j) = \Theta(x_i)(\tilde{z}^j_i) = \Theta(x_0)(\tilde{z}^j_i)$. This follows from Lemma~\ref{genappendknhoods} and~\ref{appendonesided}, and moreover, we have at least  choices $2^{-20} m^{-26d}(\rho/128)^{10r}|G_2|^{10}$ of these 10-tuples, as long as $2^{-\blc (r +d + 1)}m^{-\blc d} \rho^{ \blc r} \geq \varepsilon$.\\
    \indent Additionally, set $\tilde{z}^3_i = z_i - \tilde{z}^1_i - \tilde{z}^2_i$ and $\tilde{y}^3 = y - \tilde{y}^1 - \tilde{y}^2$, which satisfy $\tilde{t}^3_i = \Theta(x_i)(\tilde{y}^3) = \Theta(x_i)(\tilde{z}^3_i) = \Theta(x_0)(\tilde{z}^3_i)$. If we control the cocycle expressions~\eqref{approxcocycleeqnmain} for $(\tilde{y}^j, \tilde{z}^j_{[4]})$ in place of $(y, z_{[4]})$, we are done. Hence, without loss of generality $t_1=  t_2 = t$ and $t_3 = t_4 = 0$.\\

    \noindent\textbf{Proving the approximate cocycle identity.} We first need to find $u_1, u_2 \in \frac{1}{4}C'$ and $v \in B(\Gamma, \rho/128)$ such that 
    \[\Theta(u_1)(v) = t, \Theta(u_2)(v) = 0, \Theta(u_1)(y) = \Theta(u_2)(y) = \Theta(u_1)(z_1) = \Theta(u_2)(z_1) = \dots= \Theta(u_2)(z_4) = 0\]
    and
    \[\Theta(x_0)(v) = \Theta(x_1)(v) = \dots = \Theta(x_4)(v) = 0.\]

    By Lemma~\ref{genappendknhoods} and~\ref{appendonesided}, with probability at least $1 - \sqrt[32]{\varepsilon}$ we may first find $u_1, u_2 \in \frac{1}{4}C'$ such that 
    \[\Theta(u_1)(y) = \Theta(u_2)(y) = \Theta(u_1)(z_1) = \Theta(u_2)(z_1) = \dots= \Theta(u_2)(z_4) = 0.\]
    Then, with probability at least $1 - \sqrt[32]{\varepsilon}$ we may find $v$ so that the remaining conditions hold.\\

    Hence, recalling that $y \in S_{x_0}$, so $\Theta(x_0)(y) = 0$, and equalities
    \begin{align*}\Theta(x_1)(y) =& \Theta(x_1)(z_1) = \Theta(x_0)(z_1) = t,\,\,\Theta(x_2)(y) = \Theta(x_2)(z_2) = \Theta(x_0)(z_2) = t,\\
    \Theta(x_3)(y) =& \Theta(x_3)(z_3) = \Theta(x_0)(z_3) = 0,\,\,\Theta(x_4)(y) = \Theta(x_4)(z_4) = \Theta(x_0)(z_4) = 0,\end{align*}
    Claim~\ref{cocycleEqnBil} applies to the following four pairs of triples: $\Big((x_1 - x_0, x_0, u_1), (y, z_1, v)\Big)$,  $\Big((x_2 - x_0, x_0, u_1), (y, z_2, v)\Big)$,  $\Big((x_3 - x_0, x_0, u_2), (y, z_3, v)\Big)$  and  $\Big((x_4 - x_0, x_0, u_2), (y, z_4, v)\Big)$. Thus, subtracting the approximate identities from Claim~\ref{cocycleEqnBil} for the first two pairs, we get 
    \begin{align*}\phi(x_1 - x_0, z_1) -& \phi(x_2 - x_0, z_2) - \phi(x_1, z_1 - y) + \phi(x_2, z_2 - y)\\
    &\in \,\,\phi(x_1 - x_2, y + v) - \phi(x_0 - u_1, z_1 - z_2) - \phi(x_1 - x_2, v) - \phi(u_1, z_1 - z_2) + 10 E\\
    &=\,\,\phi(x_1 - x_2, y + v) - \phi(x_0, z_1 - z_2) - \phi(x_1 - x_2, v) + 10 E\end{align*}
    and similarly for the second two
    \begin{align*}\phi(x_3 - x_0, z_3) -& \phi(x_4 - x_0, z_4) - \phi(x_3, z_3 - y) + \phi(x_4, z_4 - y)\\
    &\in \,\,\phi(x_3 - x_4, y + v) - \phi(x_0, z_3 - z_4) - \phi(x_3 - x_4, v) + 10 E.\end{align*}
    Thus, after adding these expressions together and recalling that $x_1 - x_2 + x_3 - x_4 = 0$, we get the desired condition. We pick $\varepsilon = 2^{-\blc (r +d + 1)}m^{-\blc d} \rho^{ \blc r}$ so that the relevant bounds hold and thus obtain a map $\tilde{\phi}$. Apply Lemma~\ref{almostfullbilinearmap} to complete the proof.
\end{proof}

\section{Inverse theorem for Freiman bihomomorphisms}\label{putevtogsection}

We are now ready to prove the main structural result of this paper.

\begin{proof}[Proof of Theorem~\ref{maininversetheorem}]
    The proof is just the matter of going through the main results of this chapter. The quantitative aspects will mostly be trivial, namely all codimensions and ranks of objects will be bounded by at most $(2\log c^{-1})^{O(1)}$ and radius of Bohr sets and various densities will be at least $\exp(-(2\log c^{-1})^{O(1)})$. The only two points requiring some care are applications of Propositions~\ref{almostalladd16tupleshavesmallimage} and~\ref{changeofperspectivestep1}.\\
    
    Let $\Phi : A \to H$ be the given Freiman bihomomorphism. Firstly apply Proposition~\ref{freimanbihomomStep1passtolineaersystem} to pass to a system of Freiman linear maps $\phi_x$ indexed by a dense set $X$ in $G_1$ where many additive quadruples are image-respected in the sense that $\range(\phi_{x_1} - \phi_{x_2} + \phi_{x_3} - \phi_{x_4})$ is small. Let $c_1$ be the density of $X$. Then apply Proposition~\ref{almostalladd16tupleshavesmallimage} with error parameter $\varepsilon = \exp (-(2\log c_1^{-1})^K_1)$ for a sufficiently large positive constant $K_1$ in terms of the technical condition on $\varepsilon$ in Proposition~\ref{smallimgoncosetprogressionlinsys}, to pass to a subset $X' \subseteq X$ where all but very few additive 16-tuples are image-respected. By Proposition~\ref{smallimgoncosetprogressionlinsys}, as $K_1$ is sufficiently large, we may assume that additionally the system of maps $\phi_x$ is indexed by a coset progression $C$ where all additive quadrples are image-respected.\\

    We now make a change of perspective and say that an additive quadruple is Bohr-respected if $\phi_{x_1} - \phi_{x_2} + \phi_{x_3} - \phi_{x_4}$ vanishes on the intersection of their domains.  Apply Proposition~\ref{changeofperspectivestep1} with error parameter $10^{-10d}$, where $d$ is the rank of coset progression $C$, to deduce that a vast majority of additive 16-tuples are Bohr-respected. By Proposition~\ref{bilinearBogolyubovStep}, we may assume that domains of maps $\phi_x$ are columns of a bilinear Bohr variety, however only a dense collection of additive quadruples are now Bohr-respected. Apply Proposition~\ref{absgStep2bilBohrvar} to once again obtain a set of indices in which all additive quadruples are Bohr-respected.\\

    We make another change of perspective and now view $\phi$ as function of two variables defined on a subset of $G_1 \times G_2$. Apply Proposition~\ref{secondrobustBRstep} to deduce that we have a Freiman-bilinear map on a bilinear Bohr variety. Apply Propositions~\ref{controllingerrorscolsprop},~\ref{finalExtensionStep2} and~\ref{finalExtensionStep3} to get a desired $E$-bihomomorphism $\Phi : C_1 \times B_1 \to H$ with the property that for each $(x,y) \in C_1 \times B_1$ there are at least $\exp(-(2\log c^{-1})^{O(1)})$ proportions of all $(d_1, \dots, d_k)$-arrangements $(a_1, b_1), \dots, (a_m, b_m)$, where $m = d_1 \cdots d_k$, $k = O(1)$, $d_1, \dots, d_k \leq (2\log c^{-1})^{O(1)}$, of lengths $(x,y)$ such that appropriate $\pm$-linear combination of $\varphi(a_i, b_i)$ equals $\Phi(x,y)$. The results follows by averaging and Theorem~\ref{approxFreimanHom}.
\end{proof}

\pagebreak

\begin{center}\noindent{\large\bfseries{\scshape Chapter 3: Inverse theory for $\mathsf{U}^4$ norm}}
\end{center}\normalsize

In this chapter, we deduce the inverse theorems for $\mathsf{U}^4$ norm from the structural result for Freiman bihomomorphisms. We begin with an equidistribution theory for almost trilinear forms.

\section{Equidistribution theory of almost trilinear forms}\label{sectionequidistributiontheory}

In this section, we study almost multilinear forms. Namely, we say that a map $\phi : B^k \to \mathbb{T}$, defined on product of $k$ copies of a Bohr set $B$, is an \textit{$\varepsilon$-multilinear form} if for all choices of index $i \in [k]$ and elements $x_1, \dots, x_{i-1}, x_{i+1}, \dots, x_k \in B$, the map $\psi(y) = \phi(x_1, \dots, x_{i -1}, y, x_{i + 1}, \dots, x_k)$ satisfies $\varepsilon$-linearity, i.e. $\|\psi(y + z) - \psi(y) - \psi(z)\|_{\mathbb{T}} \leq \varepsilon$, whenever $y,z,y+z \in B$.\\

In this and the next section, we require the standard notion of regularity of Bohr sets and related facts, which we first recall.\\

\subsection{Regular Bohr sets}

Before stating the key definitions and results, let us note a couple of useful elementary inequalities regarding the exponential map.

\begin{lemma}\label{elemExp} Let $t \in [-1/2,1/2]$. Then
    \begin{itemize}
        \item $4|t| \leq |\on{e}(t) - 1| = 2\sin(\pi |t|)\leq 2 \pi |t|$, and
        \item $1 - \on{Re}(\on{e}(t)) \geq 8t^2$.
    \end{itemize}
\end{lemma}

\begin{proof}
    The first inequality is standard. For the second one, we have $1 - \on{Re}(\on{e}(t)) = 1- \cos(2 \pi t) = 2 \sin^2(\pi t) \geq 8 |t|^2$. 
\end{proof}

For a Bohr set $B = B(\Gamma, \rho)$ of codimension $d$, write $B_\delta$ for its $\delta$-dilate, i.e. $B_\delta = B(\Gamma, \delta \rho)$. Following~\cite{TaoVuBook} (Definition 4.24), we say that $B$ is \textit{regular} if
\[1 - 100d\delta\leq \frac{|B_{1 + \delta}|}{|B|} \leq 1+100d\delta\]
holds for all $|\delta| \leq \frac{1}{100d}$.\\

A key fact due to Bourgain (e.g. Lemma 4.25 in~\cite{TaoVuBook}) is that there exists $\delta \in [1/2, 1]$ for which $B_\delta$ is regular.\\

\noindent\textbf{Remark.} Let us note that for the most of this and the following subsection, we use sums instead of averages over the group. That affects the notion of convolution, as well as the choice of $\ell^p$ norms instead of $L^p$.\\

We need a few facts allowing us to carry out usual expression manipulations, such as averaging and change of variables, relative to Bohr sets. The arguments are standard. 

\begin{lemma}[Convolution of Bohr sets]\label{regbohrconvolution}
    Let $B$ be a regular Bohr set of codimension $d$ and radius $\rho$. Then, for all $\delta \leq 1/200d$
    \[\Big\|\id_B - \frac{1}{|B_{\delta}|}\id_{B_\delta} \ast \id_{B_{1-\delta}}\Big\|_{\ell^1} \leq 200 \delta d |B|.\]
\end{lemma}

\begin{proof}
    Note that $\id_B(x) - \frac{1}{|B_{\delta}|}\id_{B_\delta} \ast \id_{B_{1-\delta}}(x)$ always takes values in $[0,1]$ and equals 0 outside of $B$ as well as inside $B_{1-2\delta}$. Hence
    \[\Big\|\id_B - \frac{1}{|B_{\delta}|}\id_{B_\delta} \ast \id_{B_{1-\delta}}\Big\|_{\ell^1} \leq |B \setminus B_{1-2\delta}| \leq 200 d\delta |B|,\]
    by the regularity of the Bohr set $B$.
\end{proof}

\begin{lemma}[Radius reduction]\label{radiusReductionLemma}
    Let $B$ be a Bohr set of codimension $d$, let $\delta_1, \dots, \delta_r \in (0,1]$. Let $F : B^r \to \mathbb{D}$ be a function such that
    \[\Big|\sum_{x_1 \in B_{\delta_1}, \dots, x_r \in B_{\delta_r}} F(x_1, \dots, x_r) \Big|\geq c |B_{\delta_1}|\dots |B_{\delta_r}|.\]
    Let $\varepsilon > 0$, let $\delta_i' \leq \frac{\varepsilon \delta_i}{200d}$ for some index $i$ and suppose that $B_{\delta_i}$ is regular. Then there exists an element $t \in B_{\delta_i - \delta'_i}$ such that
    \[\Big|\sum_{\ssk{x_1 \in B_{\delta_1}, \dots, x_{i-1} \in B_{\delta_{i-1}},\\ x_i \in B_{\delta'_{i}}, x_{i+1} \in B_{\delta_{i+1}}, \dots, x_r \in B_{\delta_r}}} F(x_1, \dots, x_{i-1}, x_i + t, x_{i + 1}, \dots,  x_r) \Big| \geq (c - \varepsilon) \frac{|B_{\delta'_i}|}{|B_{\delta_i}|} |B_{\delta_1}|\dots |B_{\delta_r}|.\]
\end{lemma}

\begin{proof}
    Without loss of generality, $i = r$. We have
    \begin{align*}&\Big|\sum_{x_1, \dots, x_r \in G} \id_{B_{\delta_1}}(x_1) \dots \id_{B_{\delta_r}}(x_r) F(x_1, \dots, x_r) \\
    &\hspace{2cm}- \frac{1}{|B_{\delta'_r}|} |\sum_{x_1, \dots, x_r \in G, t\in G} \id_{B_{\delta_1}}(x_1) \dots \id_{B_{\delta_{r-1}}}(x_{r-1}) \id_{B_{\delta'_r}}(x_r)\id_{B_{\delta_r - \delta'_r}}(t) F(x_1, \dots, x_r + t)\Big|\\
    &\hspace{1cm}= \Big|\sum_{x_1, \dots, x_r \in G} \id_{B_{\delta_1}}(x_1) \dots \id_{B_{\delta_{r-1}}}(x_{r-1}) \Big(\id_{B_{\delta_r}}(x_r) - \frac{1}{|B_{\delta'_r}|}\id_{B_{\delta'_r}} \ast \id_{B_{\delta_r - \delta'_r}}(x_r)\Big)F(x_1, \dots, x_r)\Big|\\
    &\hspace{1cm}\leq |B_{\delta_1}|\dots |B_{\delta_{r-1}}| \Big\|\id_{B_{\delta_r}}(x_r) - \frac{1}{|B_{\delta'_r}|} \id_{B_{\delta'_r}} \ast \id_{B_{\delta_r - \delta'_r}}(x_r)\Big\|_{\ell^1}.
    \end{align*}

    By the previous lemma, the expression above is at most $200d \delta'_r \delta_r^{-1} |B_{\delta_1}|\dots |B_{\delta_r}| \leq \varepsilon |B_{\delta_1}|\dots |B_{\delta_r}|$, by assumption on $\delta_r'$. \\

    By triangle inequality, we get

    \[\Big|\frac{1}{|B_{\delta'_r}|} |\sum_{x_1, \dots, x_r \in G, t \in G} \id_{B_{\delta_1}}(x_1) \dots \id_{B_{\delta_{r-1}}}(x_{r-1}) \id_{B_{\delta'_r}}(x_r)\id_{B_{\delta_r - \delta'_r}}(t) F(x_1, \dots, x_r + t)\Big| \geq (c - \varepsilon)|B_{\delta_1}|\dots |B_{\delta_r}|\]
    from which the claim follows after averaging over $t$.
\end{proof}

\begin{lemma}[Change of variables]\label{regularbohrchangeVar}
    Let $B$ be a Bohr set of codimension $d$, let $\delta_1,\dots, \delta_r, \rho \in (0,1/2]$, $\varepsilon > 0$ and $\lambda_1, \dots, \lambda_r \in \mathbb{Z}$. Let $F : B^{r+1} \to \mathbb{D}$ be a function. If $|\lambda_i|\delta_i \leq \frac{\varepsilon \rho}{200dr}$ for each $i \in [r]$ and $B_\rho$ is regular, then
    \[\Big|\sum_{\ssk{x_1 \in B_{\delta_1}, \dots, \\x_r \in B_{\delta_r}, y \in B_\rho}}  F(x_1, \dots, x_r, y)  - \sum_{\ssk{x_1 \in B_{\delta_1}, \dots,\\ x_r \in B_{\delta_r}, y \in B_\rho}} F(x_1, \dots, x_r, y + \lambda_1 x_1 + \dots + \lambda_r x_r)  \Big| \leq \varepsilon |B_{\delta_1}| \dots |B_{\delta_r}| |B_\rho|.\]
\end{lemma}

\begin{proof}
    Fix $x_1\in B_{\delta_1}, \dots, x_r\in B_{\delta_r}$ and write $f(y) = F(x_1, \dots, x_r, y)$. For $t = \lambda_1 x_1 + \dots + \lambda_r x_r$, we have 
    \[\Big|\sum_{y \in B_\rho} f(y) - \sum_{y \in B_\rho} f(y + \lambda_1 x_1 + \dots + \lambda_r x_r)\Big| \leq 2|B_\rho \setminus (B_\rho + t)|.\]
    Let $\delta' = |\lambda_1|\delta_1 + \dots + |\lambda_r|\delta_r$. Since $B_{\rho - \delta'} \subseteq B_\rho + t$, we have $|B_\rho \setminus (B_\rho + t)| \leq |B_\rho \setminus (B_{\rho - \delta'})| \leq 100 d \rho^{-1}\delta' |B_\rho|$. The claim follows by summing over $x_1, \dots, x_r$ and the triangle inequality.
\end{proof}

\subsection{Symmetry argument}

In this subsection we derive symmetry argument for almost bilinear forms. The argument originates in the paper~\cite{GreenTaoU3} of Green and Tao and is an important ingredient in the inverse theory of uniformity norms. A particularly clean formulation of the symmetry argument was given by Tidor, which we follow.\\

\begin{lemma}[Symmetry argument]\label{symmetryArgumentBohr}
    Let $\theta : B_{2\rho} \times B_{2\rho} \to \mathbb{T}$ be a $\varepsilon$-bilinear map on a regular Bohr set $B_\rho$ of codimension $d$ and suppose that
    \[\Big|\sum_{x, y \in B_\rho} f_1(x) f_2(y) f_3(x + y) \on{e}(\theta(x,y)) \Big| \geq c |B|^2\] 
    holds for some functions $f_1, f_2, f_3 : B_{2\rho} \to \mathbb{D}$. Let $\rho_1 \leq \frac{c^2 \rho}{2000 d}$ be such that $B_{\rho_1}$ is regular. Then, provided $\varepsilon \leq 2^{-20}c^8$, we have 
    \[\on{Re}\bigg(\sum_{y_1, y_2 \in B_{2\rho_1}}  \id_{B_{\rho_1}}\ast\id_{B_{\rho_1}}(y_1) \id_{B_{\rho_1}}\ast\id_{B_{\rho_1}}(y_2) \on{e}\Big(\theta(y_1, y_2) - \theta(y_2, y_1)\Big)\bigg)\geq 2^{-13}c^8|B_{\rho_1}|^4.\]
\end{lemma}

\begin{proof}
    By Cauchy-Schwarz inequality, we have
    \[ \sum_{x, y_1, y_2 \in B_\rho} f_2(y_1) f_3(x + y_1) \overline{f_2(y_2) f_3(x + y_2)} \on{e}(\theta(x,y_1) - \theta(x,y_2)) \geq c^2 |B_\rho|^3.\]
    We apply Lemma~\ref{radiusReductionLemma} twice with error parameter $c^2/4$ to find elements $t_1, t_2 \in B_{\rho - \rho_1}$ such that
    \[\Big|\sum_{x \in B_\rho, y_1, y_2 \in B_{\rho_1}} f_2(y_1 + t_1) f_3(x + y_1 + t_1) \overline{f_2(y_2 + t_2) f_3(x + y_2 + t_2)} \on{e}(\theta(x,y_1 + t_1) - \theta(x,y_2 + t_2))\Big|\geq \frac{c^2}{2}|B||B_{\rho_1}|^2.\]
    Looking at the function, since $\rho_1 \leq \frac{c^2 \rho}{2000 d}$,
    \[(x, y_1, y_2) \mapsto f_2(y_1 + t_1) f_3(x + y_1 + t_1) \overline{f_2(y_2 + t_2) f_3(x + y_2 + t_2)} \on{e}(\theta(x,y_1 + t_1) - \theta(x,y_2 + t_2))\]
    we use Lemma~\ref{regularbohrchangeVar} with error parameter $c^2/4$ to make a change of variables $x - y_1 - y_2$ in place of $x$, so
    \begin{align*}\Big|\sum_{x \in B_\rho, y_1, y_2 \in B_{\rho_1}} f_2(y_1 + t_1) f_3(x - y_2 + t_1) &\overline{f_2(y_2 + t_2) f_3(x - y_1 + t_2)} \\
    &\on{e}(\theta(x - y_1 - y_2, y_1 + t_1) - \theta(x - y_1 - y_2, y_2 + t_2))\Big|\geq \frac{c^2}{4}|B||B_{\rho_1}|^2.\end{align*}

    Using $\varepsilon$-bilinearity a few times, we get
    \begin{align*}\Big\|\Big(\theta(x - y_1 - y_2, y_1 + t_1) - \theta(x - y_1 - y_2, y_2 + t_2)\Big) - \Big(\theta(y_1, y_2) - &\theta(y_2, y_1) + \theta(x - y_1, y_1 + t_1) - \theta(y_2, t_1)\\
    - &\theta(x - y_2, y_2 + t_2) + \theta(y_1, t_2)\Big)\Big\|_{\mathbb{T}} \leq 4\varepsilon.\end{align*}

    By elementary estimate $|\on{e}(t) - | \leq 2 \pi |t|$ and the triangle inequality, we get
    \begin{align*}\Big|\sum_{x \in B_\rho, y_1, y_2 \in B_{\rho_1}}  f_2(y_1 + t_1)&\overline{f_3(x - y_1 + t_2)} \on{e}\Big(\theta(x - y_1, y_1 + t_1) + \theta(y_1, t_2)\Big) \\
    &f_3(x - y_2 + t_1) \overline{f_2(y_2 + t_2)} \on{e}\Big(-\theta(x - y_2, y_2 + t_2) -  \theta(y_2, t_1)\Big)\\
    &\on{e}(\theta(y_1, y_2) - \theta(y_2, y_1))\Big|\geq \Big(\frac{c^2}{4} - 8 \pi \varepsilon\Big)|B_\rho||B_{\rho_1}|^2.\end{align*}

    Recall that $\varepsilon \leq \frac{c^2}{256}$, so the factor on the right hand side can be replaced by $\frac{c^2}{8}$. By averaging over $x$ and writing $\alpha(y_1, y_2) = \theta(y_1, y_2) - \theta(y_2, y_1)$, we obtain two functions $u,v : B_{\rho_1} \to \mathbb{D}$ such that
    \[\Big|\sum_{y_1, y_2 \in B_{\rho_1}}  u(y_1) v(y_2) \on{e}\Big(\alpha(y_1, y_2)\Big)\Big|\geq \frac{c^2}{8}|B_{\rho_1}|^2.\]

    By Cauchy-Schwarz inequality, we have
    \begin{align*}\Big(\frac{c^2}{8}|B_{\rho_1}|^2\Big)^2 \leq &\Big|\sum_{y_1\in B_{\rho_1}}  u(y_1) \Big(\sum_{y_2 \in B_{\rho_1}} v(y_2) \on{e}\Big(\alpha(y_1, y_2)\Big)\Big|^2\\
    \leq& \Big(\sum_{y_1\in B_{\rho_1}}  |u(y_1)|^2\Big) \Big(\sum_{y_1\in B_{\rho_1}}  \Big|\sum_{y_2 \in B_{\rho_1}} v(y_2) \on{e}\Big(\alpha(y_1, y_2)\Big)\Big|^2\Big),\end{align*}
    so after expansion, we get
    \[2^{-6}c^4|B_{\rho_1}|^3 \leq \sum_{y_1, y_2, y_2' \in B_{\rho_1}} v(y_2) \overline{v(y'_2)}\on{e}\Big(\alpha(y_1, y_2) - \alpha(y_1, y'_2)\Big).\]
    Another such step with $y_1$ in place of $y_2$ gives
    \[2^{-12}c^8|B_{\rho_1}|^4 \leq \sum_{y_1, y_1', y_2, y_2' \in B_{\rho_1}}\on{e}\Big(\alpha(y_1, y_2) - \alpha(y_1, y'_2) - \alpha(y_1, y'_2) + \alpha(y'_1, y'_2)\Big).\]

    By $\varepsilon$-bilinearity, the value $\alpha(y_1, y_2) - \alpha(y_1, y'_2) - \alpha(y_1, y'_2) + \alpha(y'_1, y'_2)$ differs from $\alpha(y_1 - y_1', y_2 - y_2')$ by at most $3\varepsilon$. Hence
   \[2^{-13}c^8|B_{\rho_1}|^4 \leq \on{Re}\bigg(\sum_{y_1, y_1', y_2, y_2' \in B_{\rho_1}}\on{e}\Big(\alpha(y_1 - y_1', y_2 - y_2')\Big)\bigg).\]
    After a change of variables, where we take $z_i = y_i - y'_i$ in place of $y_i$, we get
    \[2^{-13}c^8|B_{\rho_1}|^4 \leq \on{Re}\bigg(\sum_{z_1, z_2 \in G, y_1', y_2' \in B_{\rho_1}} \id_{B_{\rho_1}} (y'_1 + z_1)\id_{B_{\rho_1}} (y'_2+ z_2)\on{e}\Big(\alpha(z_1, z_2)\Big)\bigg),\]
    from which the claim follows.
\end{proof}

\subsection{Equidistribution theory}

This subsection is devoted to understanding the structure of $\varepsilon$-trilinear maps which do not have uniform distribution of values.

\begin{theorem}\label{equidistributionalmosttrilinearmaps}
    Let $B$ be a Bohr set of codimension $d$ and radius $\rho$. Let $\delta \leq 1/2$ be such that $B_\delta$ is regular. Suppose that $\phi : B \times B \times B \to \mathbb{T}$ is a $\varepsilon$-trilinear map such that 
    \[\Big|\sum_{x,y,z \in G} \id_{B_\delta} \ast \id_{B_\delta}(x) \id_{B_\delta} \ast \id_{B_\delta}(y) \id_{B_\delta} \ast \id_{B_\delta}(z) \on{e}(\psi(x,y,z))\Big| \geq c |B_\delta|^6.\]
    \indent Then there exists a bilinear Bohr variety $V \subseteq B \times B$ of codimension at most $(2d \log(\varepsilon^{-1} c^{-1} \delta^{-1} \rho^{-1}))^{O(1)}$ and a radius $\tilde{\rho} \geq \exp(-(2d \log(\varepsilon^{-1} c^{-1} \delta^{-1}\rho^{-1}))^{O(1)})$ such that $\|\phi(x,y,z)\|_{\mathbb{T}} \leq (2d)^{O(d)} \sqrt{\varepsilon}$ for all $(x,y) \in V$ and $z \in B(\Gamma, \tilde{\rho})$. 
\end{theorem}

Before proceeding with the proof, we need to understand almost linear forms. Eventually, we shall be able to relate them to characters, but firstly we need a lemma about the special case of biased linear forms defined on groups and arithmetic progressions.

\begin{lemma}\label{biasedalmostlinearformstructure}
    Let $X \subseteq G$. Suppose that $\phi : X \to \mathbb{T}$ is a $\varepsilon$-linear map such that $\Big|\sum_{x \in X} \on{e}(\phi(x))\Big| \geq c|X|$. 
    \begin{itemize}
        \item If $X \leq G$ and $c \geq 36 \pi \sqrt{\varepsilon}$, then $|\phi(x)| \leq 9\sqrt{\varepsilon}$ on $X$.
        \item If $X = [-L, L] \cdot a$ is a proper arithmetic progression in $G$, then $|\phi(x)| \leq 1000c^{-2} \varepsilon$ for $|x| \leq \varepsilon L$.
    \end{itemize}
\end{lemma}

A consequence of $\varepsilon$-linearity and elementary estimates on $\on{e}(t) - 1$ that we shall frequently use in the proof is
\[\Big|\on{e}(\phi(a_1 + \dots + a_r)) - \on{e}(\phi(a_1))\cdots \on{e}(\phi(a_r))\Big| \leq 2\pi r \varepsilon\]
for all $a_1, \dots , a_r \in X$ such that $a_i + \dots + a_r \in X$ for all $i \in [r]$.

\begin{proof}
\textbf{Subgroup case.} Suppose first that $X \leq G$. Consider the function $f : X \to \mathbb{C}$ given by $f(x) = \on{e}(\phi(x))$. Then, on the group $X$, we have $\sum_{\chi \in \hat{X}} |\hat{f}(\chi)|^4 = \|f\|_{\mathsf{U}^2}^4 = \exx_{x, a, b \in X} \on{e}(\phi(x + a + b) - \phi(x + a) - \phi(x + b) + \phi(x)) \geq 1 - 4 \pi \varepsilon$. Hence, there exists $\chi \in \hat{X}$ such that $|\hat{f}(\chi)| \geq 1 - 4 \pi \varepsilon$. By looking at the argument of $\hat{f}(\chi)$ we find an element $t \in \mathbb{T}$ such that $\sum_{x \in X}\on{e}(\phi(x) - \chi(x) - t)$ is real and at least $(1 - 4 \pi \varepsilon)|X|$. Thus
\[\sum_{x \in X}\Big( 1 - \on{Re}\on{e}(\phi(x) - \chi(x) - t)\Big) \leq 4 \pi \varepsilon |X|.\]
Let $X'$ be the set of all $x \in X$ such that $\|\phi(x) - \chi(x) - t\|_{\mathbb{T}} \leq 4\sqrt{\varepsilon} $. By Lemma~\ref{elemExp}, we get
\[\sum_{x \in X \setminus X'}\Big( 1 - \on{Re}\on{e}(\phi(x) - \chi(x) - t)\Big) \geq 128 \varepsilon |X \setminus X'|.\]
Hence, $|X'| \geq \frac{2}{3}|X|$ and so every element $x \in X$ can be written as $y - z$ for $y,z \in X'$. By $\varepsilon$-linearity, we get $\|\phi(x) - \chi(x)\|_{\mathbb{T}} \leq 9\sqrt{\varepsilon}$ for all $x \in X$. The bias assumption on $\phi$ and Lemma~\ref{elemExp} imply that 
\[\Big|\sum_{x \in X} \on{e}(\chi(x))\Big| \geq \Big|\sum_{x \in X} \on{e}(\phi(x))\Big|  - \Big|\sum_{x \in X} \on{e}(\chi(x) - \phi(x))\Big| \geq c |X| - 18 \pi \sqrt{\varepsilon}|X|  > 0,\] 
as long as $c \geq 36 \pi \sqrt{\varepsilon}$, so $\chi = 0$.\\

\textbf{Progression case.} Set $\varepsilon' = 1000c^{-2} \varepsilon$ and $k = c \varepsilon^{-1} /100$. Assume on the contrary, that there exists an element $x \in [- \varepsilon'' L,  \varepsilon'' L]$ such that $\|\phi(x)\|_{\mathbb{T}} \geq \varepsilon'$. Let us partition $X$ into translates of progression $[kx]$. This can be achieved with at most $kx$ elements remaining not covered. On the other hand, using $\varepsilon$-linearity gives for $i \leq k$
\[\Big|\on{e}(\phi(t + i x + j)) - \on{e}(\phi(t))\on{e}(i\phi(x))\on{e}(\phi(j))\Big| \leq 2 \pi (k+2) \varepsilon\]
so on each translate we have the estimate
\begin{align*}\Big|\sum_{y \in t + [kx]} \on{e}(\phi(y))\Big| = &\Big|\sum_{i = 0}^{k-1} \sum_{j \in [x]} \on{e}(\phi(t + i x + j))\Big| \leq kx 2 \pi (k+2) \varepsilon + \Big|\sum_{i = 0}^{k-1} \sum_{j \in [x]} \on{e}(\phi(t))\on{e}(i\phi(x))\on{e}(\phi(j))\Big| \\
= & 2 \pi k(k+2)x \varepsilon + \Big|\sum_{i = 0}^{k-1}\on{e}(i\phi(x))\Big|  \Big|\sum_{j \in [x]}\on{e}(\phi(j))\Big|\\
\leq &2 \pi  k(k+2)x \varepsilon +  x\Big|\sum_{i = 0}^{k-1}\on{e}(i\phi(x))\Big|\\
\leq & 2 \pi k(k+2)x \varepsilon +  x \frac{1}{\|\phi(x)\|_{\mathbb{T}}} \leq \Big(2 \pi k(k+2)\varepsilon + \frac{1}{\|\phi(x)\|_{\mathbb{T}}}\Big) x \leq \frac{c}{2} kx,
\end{align*}
where we used inequalities $2 \pi(k+2)\varepsilon \leq \frac{c}{4}$ and $\|\phi(x)\|_{\mathbb{T}} \geq 4c^{-1}/k$, that follow from our choices of parameters, in the last step.

As $[-L, L]$ can be partitioned into translates of $[kx]$ and at most $kx \leq k \varepsilon''L \leq  \frac{c}{4}L$ leftover elements, it follows that $\Big|\sum_{y \in [-L, L]} \on{e}(\phi(y))\Big|\leq \frac{c}{2} (2L + 1) + kx < c (2L + 1)$, which is a contradiction.
\end{proof}

We may now relate almost linear form on Bohr sets with characters on dilates.

\begin{lemma}
    Let $\phi : B(\Gamma, \rho) \to \mathbb{T}$ be a $\varepsilon$-linear form on a regular Bohr set of codimension $d$ and radius $\rho$. Then there exists a character $\chi$ such that $\|\phi(x) - \chi(x)\|_{\mathbb{T}} \leq O(d^{O(d)}\sqrt{\varepsilon})$ on $B(\Gamma, \Omega( \varepsilon d^{-O(d)}\rho))$.
\end{lemma}

\begin{proof}
    Let us note that we may assume that $\varepsilon \leq (2d)^{-\blc d}$ during the proof, as otherwise the statement is trivial.\\

    Define function $f(x) = \id_{B}(x) \on{e}(\phi(x))$, where $B = B(\Gamma, \rho)$. Note that
    \begin{align*}\Big|\sum_{x,a,b\in G}& f(x + a + b) \overline{f(x + a)}\,\overline{f(x + b)} f(x)\, - \,\id_{B}(x + a + b) \id_{B}(x + a) \id_{B}(x + b) \id_{B}(x)\Big| \\
    = & \Big| \sum_{x,a,b\in G}\id_{B}(x + a + b)\id_{B}(x + a) \id_{B}(x + b) \id_{B}(x)\Big(\on{e}(\phi(x + a + b) - \phi(x + a) - \phi(x + b) + \phi(x))- 1\Big) \Big| \\
    \leq &O(\varepsilon) \sum_{x,a,b\in G}\id_{B}(x + a + b) \id_{B}(x + a) \id_{B}(x + b) \id_{B}(x).\end{align*}

    Hence, provided $\varepsilon \leq \bsc$, we get
    \[\sum_\chi |\hat{f}(\chi)|^4 \geq \frac{1}{2} \exx_{x,a,b\in G}\id_{B} (x + a + b) \id_{B}(x + a) \id_{B}(x + b) \id_{B}(x) \geq \frac{1}{2|G|^3}|B(\Gamma, \rho/3)|^3 \geq \frac{1}{2}3^{-3d}(|B|/|G|)^3,\]
    due to small doubling property of Bohr sets.\\

    Hence, there exists $\chi \in \hat{G}$ such that 
    \[\Big|\sum_{x \in B} \on{e}(\phi'(x))\Big| \geq c |B|,\]
    where $\phi' = \phi - \chi$, which is also $\varepsilon$-linear, and $c \geq \frac{1}{2}3^{-3d}$.\\

    By Proposition~\ref{cpsinbohrsets}, for $\rho' = d^{-2d}\rho$, we may find a proper symmetric coset progression $C = [-L_1, L_1] \cdot a_1 + \dots + [-L_d, L_d] \cdot a_d + K$ of rank at most $d$ such that $C \subseteq B(\Gamma, \rho')$s and for all $\alpha \in (0,1/2)$
    \begin{equation}B(\Gamma, \alpha d^{-2d} \rho')  \subseteq [-2\alpha L_1, 2\alpha L_1] \cdot a_1 + \dots + [-2\alpha L_d, 2\alpha L_d] \cdot a_d + K.\label{shringbohrcpcontainment}\end{equation}

    Let $c' = c/400d$ and set $C' = [- c'L_1, c' L_1] \cdot a_1 + \dots + [-c' L_d, c' L_d] \cdot a_d + K$. Note that $C'$ is contained inside $B(\Gamma, 2c'  \rho')$. Namely, looking at any $x \in C'$, we have $x, 2x, \dots, m x \in C \subseteq B(\Gamma, \rho')$ for $m = \lfloor  {c'}^{-1} \rfloor$. Taking any character $\gamma \in \Gamma$, we have $\|\gamma(\ell x)\|_{\mathbb{T}} \leq \rho'$ for $\ell \in [m]$. This can only occur if $\|\gamma(x)\|_{\mathbb{T}} \leq 2c' \rho'$.\\ 
       
    Since $B$ is regular and $\phi'$ is $\varepsilon$-linear, we also have that 
    \begin{align*}\Big|\sum_{x \in B} \on{e}(\phi'(x)) - &\frac{1}{|C'|}\sum_{x \in B, y \in C'} \on{e}(\phi'(x) + \phi'(y))\Big| \\
    \leq &\Big|\sum_{x \in B} \on{e}(\phi'(x)) - \frac{1}{|C'|}\sum_{x \in B, y \in C'} \on{e}(\phi'(x + y))\Big| + 2 \pi \varepsilon |B|\\
    \leq & \frac{1}{|C'|}\sum_{y \in C'} |B \Delta (B + y)| + 2 \pi \varepsilon |B|  \leq (100d c' + 2 \pi \varepsilon) |B|. \end{align*}

    By our choice of $c'$ and averaging over $x \in B$, we get

    \[\Big|\sum_{y \in C'} \on{e}(\phi'(y))\Big| \geq \Big(\frac{c}{2} - O(\varepsilon)\Big) |C'|.\]

    Recall that 
    \[\Big|\on{e}(\phi(\lambda_1 a_1 + \dots + \lambda_r a_r + h)) - \on{e}(\phi(\lambda_1 a_1)) \cdots \on{e}(\phi(\lambda_r a_r))\on{e}(\phi(h))\Big| \leq O(r \varepsilon).\]
    
    As $C'$ can be partitioned into translates of $[- c' L_i, c' L_i] \cdot a_i$, as well as into translates of $K$, for each $i \in [d]$, we have
    
    \[\Big|\sum_{\lambda_i \in [- c'L_i, c' L_i]} \on{e}(\phi(\lambda_i a_i))\Big| \geq \Big(\frac{c}{2} - O(r\varepsilon)\Big) c'L_i,\]

    and 

    \[\Big|\sum_{x\in K} \on{e}(\phi(x))\Big| \geq \Big(\frac{c}{2} - O(r\varepsilon)\Big) |K|.\]

    By Lemma~\ref{biasedalmostlinearformstructure}, it follows that $\|\phi(x)\|_{\mathbb{T}} \leq O(c^{-2}r\varepsilon)$ holds for all $x$ inside
    \[[- c'\varepsilon L_1, c'\varepsilon L_1] \cdot a_1 + \dots + [-c'\varepsilon L_d,c'\varepsilon L_d] \cdot a_d + K.\]
    Recalling that $c \geq \frac{1}{2}3^{-3d}$ and using~\eqref{shringbohrcpcontainment}, we obtain $\|\phi - \chi\|_{\mathbb{T}} \leq O(d^{O(d)}\sqrt{\varepsilon})$ on $B(\Gamma, \frac{1}{4000}3^{-3d} \varepsilon d^{-2d}\rho)$.
\end{proof}

\vspace{\baselineskip}

We may now prove the equidistribution theorem.

\vspace{\baselineskip}

\begin{proof}[Proof of Theorem~\ref{equidistributionalmosttrilinearmaps}]
    Since $\phi$ is $\varepsilon$-trilinear, we have that $z\mapsto \phi(x,y, z)$ is $\varepsilon$-linear on $B$ for each $x, y \in B$. Let $\rho' \geq \Omega(\varepsilon d^{-O(d)}\rho)$ be such that $B(\Gamma, \rho')$ is regular and the conclusion of the previous lemma applies on $B(\Gamma, \rho')$. Thus, for each $(x,y) \in B^2$, we get a character $\chi(x,y) \in \hat{G}$ such that $\|\phi(x,y, z) - \chi(x,y)(z)\|_{\mathbb{T}} \leq \varepsilon'$ for all $z \in B(\Gamma, \rho')$, where $\varepsilon' = O(d^{O(d)}\sqrt{\varepsilon})$. \\

    Let $E = \langle \Gamma \rangle_{[-R, R]}$ for $R = \mathsf{C}_{\on{spec}}(\varepsilon\rho d^{-d}/2)^{O(d)}$, where $\mathsf{C}_{\on{spec}}$ is the constant from Proposition~\ref{bohrsizeLargeFC} and $R$ is the bound on absolute value of coefficients of linear combinations of characters coming from Proposition~\ref{bohrsizeLargeFC} when applied with choices $\varepsilon \leftarrow \frac{1}{2}(\rho'/4)^d$  and $\eta \leftarrow  \frac{1}{4}(\rho'/4)^{d+1}$.\\
    \indent Similarly to the notion of $E$-bihomomorphism, we define an \textit{$E$-bilinear map} as a map $\Theta: A \subseteq G \times G \to \hat{G}$ such that $\Theta(x_1, y) + \Theta(x_2, y) - \Theta(x_3, y) \in E$ whenever all three points in the arguments are in the domain and $x_1  +x_2 = x_3$, and analogously in the vertical direction. 

    \begin{claim}\label{biasedtrilinearEbilinear}
        The map $\chi$ is an $E$-bilinear on $B_{1/2}\times B_{1/2}$.
    \end{claim}

    \begin{proof}
        We prove this in the horizontal direction, the vertical direction is analogous. Let $x_1, x_2, x_3 \in B_{1/2}$ be such that $x_1 + x_2 = x_3 $ and let $y \in B$. Then for all $z \in B(\Gamma, \rho')$, we have
        \begin{align*}\Big\|\Big(\chi(x_1,y) + &\chi(x_2, y) - \chi(x_3, y)\Big)(z)\Big\|_{\mathbb{T}} \\
        \leq &\Big\| \phi(x_1,y, z) + \phi(x_2, y, z) - \phi(x_3, y, z)\Big\|_{\mathbb{T}} + 3\varepsilon' \leq 4\varepsilon'.\end{align*}

        Hence, $\xi = \chi(x_1,y) + \chi(x_2, y) - \chi(x_3, y)$ is a character on $G$, taking values inside $(-4\varepsilon', 4\varepsilon')$ on $B(\Gamma, \rho')$. As long as $\varepsilon' \leq 1/100$, taking weakly regular $B(\Gamma, \rho'')$ for some $\rho'' \in (1/4\rho',1/2\rho')$, in the sense that $|B(\Gamma,\rho'' + \eta) \setminus B(\Gamma,\rho'')| \leq \frac{1}{4}(\rho'/4)^d$, for $\eta = \frac{1}{4}(\rho'/4)^{d+1}$. We get $|\widehat{\id_{B(\Gamma,\rho'')}}(\xi)| \geq \frac{1}{2}\frac{|B(\Gamma,\rho'')|}{|G|} \geq \frac{1}{2}{\rho''}^d$.\\ 
        
        By Proposition~\ref{bohrsizeLargeFC}, we see that $\chi(x_1,y) + \chi(x_2, y) - \chi(x_3, y) \in E$. 
    \end{proof}

    Moreover, $\chi$ has a large approximate kernel.

    \begin{claim}
        Set $\delta = 2^{-7}$. We have $\chi(x,y) \in E$ for at least $(c/2d)^{O(1)}| |B|^2$ points $(x,y) \in B_{\delta}^2$.
    \end{claim}

    \begin{proof}
        Let $\delta' = \frac{\delta c}{400d}$. By regularity of $B$, we have 
        \[\Big|\frac{1}{|B_{\delta'}|^2}\sum_{x, y \in B_{\delta'}, s, t \in B_{1-\delta'}, z \in B} \on{e}\Big(\phi(x + s, y + t, z)\Big) - \sum_{ x,y, z \in B} \on{e}\Big(\phi(x, y, z)\Big)\Big| \leq 200d \delta' |B|^3 \leq \frac{c}{2}|B|^3.\]

        By assumptions of the theorem, by averaging we find $s,t \in B$ such that 
        \[\Big|\frac{1}{|B_{\delta'}|^2}\sum_{x, y \in B_{\delta'}, z \in B} \on{e}\Big(\phi(x + s, y + t, z)\Big)\Big| \geq \frac{c}{2}|B_{\delta'}|^2|B|.\]
        Hence, we get at least $\frac{c}{4}|B_{\delta'}|^2$ pairs $(x,y) \in (s +B_{\delta'}) \times  (t +B_{\delta'})$ such that $\Big|\sum_{z \in B} \on{e}(\phi(x + s, y + t, z))\Big| \geq \frac{c}{4} |B|$. Hence $\Big|\sum_{z \in B(\Gamma, \rho')} \on{e}(\chi(x +S, y + t)(z))\Big| \geq \frac{c}{16} |B|$, as long as $\varepsilon' \leq c / (32\pi)$. Like in Claim~\ref{biasedtrilinearEbilinear}, the claim follows from Proposition~\ref{bohrsizeLargeFC}. 
    \end{proof}

    Apply bilinear Bogolyubov argument (Theorem~\ref{bogruzsabilinearintro}) to finish the proof.
\end{proof}

\section{Groups of order coprime to 6}\label{generalabelianinverseunifsection}

By an \emph{almost-cubic polynomial} $\phi : B = B(\Gamma, \rho_0) \to \mathbb{T}$ we mean a map such that, for each $\rho \leq \rho_0$, we have
\[\|\Delta_{a,b,c,d} q(x)\|_{\mathbb{T}} \leq  2^{10} \rho\]
for all $x, a,b,c,d \in B(\Gamma, \rho)$, where $\rho \leq 1/8$.

\begin{theorem}\label{generalU4}
    Let $G$ be a finite abelian group of order coprime to 6. Let $f : G \to \mathbb{D}$ be such that $\|f\|_{\mathsf{U}^4} \geq c$. Then there exists a Bohr set $B$ of codimension $(2 \log c^{-1})^{O(1)}$ and radius $\exp(-(2 \log c^{-1})^{O(1)})$, an element $t \in G$ and an almost-cubic polynomial $\phi : B  \to \mathbb{T}$ such that 
    \[\Big|\exx_{x \in G} \id_{B}(x) f(x + t) \on{e}(q(x))\Big| \geq \exp(-(2 \log c^{-1})^{O(1)}).\]
\end{theorem}

At the final step of the proof, we shall use the $\mathsf{U}^3$ inverse theorem for general finite abelian groups of odd order of Green and Tao~\cite{GreenTaoU3}. The bounds are due to Sanders~\cite{Sanders}.

\begin{theorem}\label{greentaou3inversethm}
    Let $G$ be a finite abelian group of odd order. Let $f : G \to \mathbb{D}$ be such that $\|f\|_{\mathsf{U}^3(G)} \geq c$. Then there exists a Bohr set $B$ of codimension $(2 \log c^{-1})^{O(1)}$ and radius $\exp(-(2 \log c^{-1})^{O(1)})$, an element $t \in G$ and a locally quadratic function $\phi : B \to \mathbb{T}$ such that 
    \[\Big|\exx_{x \in G} \id_{B}(x) f(x + t) \on{e}(q(x))\Big| \geq \exp(-(2 \log c^{-1})^{O(1)}).\]
\end{theorem}

\begin{proof}[Proof of Theorem~\ref{generalU4}]
    As in the previous deductions of the inverse theorem for uniformity norms based on Freiman bihomomorphisms~\cite{U4paper,newU4}, we consider the large Fourier coefficients of $\partial_{a, b} f$. We have
    \[c^{16} \leq \|f\|_{\mathsf{U}^4}  = \exx_{a_1, a_2, a_3 ,a_4, x} \partial_{a_1, a_2, a_3, a_4} f(x) = \exx_{a,b} \|\partial_{a,b} f\|_{\mathsf{U}^2}^4.\]

    By the inverse theorem for the $\mathsf{U}^2$ norm and averaging, there exists a set $A \subseteq G \times G$ of size $(c/2)^{O(1)}|G|^2$ such that for each $(a,b) \in S$ we have some $\phi(a,b) \in \hat{G}$ with $|\widehat{\partial_{a,b} f}(\phi(a,b))| \geq (c/2)^{O(1)}$. We view $\phi$ as a map on the set $A$ and show that it respects many horizontal additive quadruples. Namely,

    \[(c/2)^{O(1)} \leq \exx_{a,b} \id_A(a,b)|\widehat{\partial_{a,b} f}(\phi(a,b))|^2
    \leq \exx_{a,b,x,y} \id_A(a,b)\partial_bf(x) \overline{\partial_{b} f(x)}\, \overline{\partial_{b} f(y + a)} \overline{\partial_{b} f(y)} \on{e}(\phi(a, b)(x - y)).\]

    Fixing $x,y,b$ and applying Cauchy-Schwarz inequality yields

    \[(c/2)^{O(1)} \leq \exx_{a,a', b,x,y} \id_A(a,b) \id_A(a',b)\partial_bf(x +a ) \overline{\partial_{b} f(x + a')}\, \overline{\partial_{b} f(y + a)} \partial_{b} f(y + a') \on{e}(\phi(a, b)(x - y) - \phi(a', b)(x - y)).\]

    Making a change of variables $a \mapsto a - x, a' \mapsto a'-x$ and applying Cauchy-Schwarz inequality one more time for fixed $x,y, b$ as above shows that $\phi$ respects $(c/2)^{O(1)}$ proportion of all horizontal additive quadruples. We may apply Theorem~\ref{approxFreimanHom} for all rows that have $(c/2)^{O(1)}$ proportion of respected horizontal additive quadruples to conclude that $\phi$ is Frieman homomorphism in horizontal direction on a subset $A' \subseteq A$ of size $|A'| \geq \exp(-(2 \log c^{-1})^{O(1)}) |G|^2$.\\

    Repeat the same procedure in the vertical direction to obtain a Freiman bihomomorphism. Recall that $(|G|, 6) = 1$. Thus, the map $g \mapsto 6g$ is an isomorphism. By the structure theorem for Freiman bihomomorphisms (Theorem~\ref{maininversetheorem}), there exist an integer $d \leq (2\log c^{-1})^{O(1)}$, a set $E \subseteq \hat{G}$ of rank $d$, a Bohr set $B = B(\Gamma, \rho_0) \subseteq G$ of codimension $d$ and radius $\rho_0 \geq \exp(-(2\log c^{-1})^{O(1)})$, elements $s, t \in G$ and an $E$-bihomomorphism $\Phi : B \times B \to H$ such that $6\Phi(x,y) = \phi(x + s,y + t)$ and $(x + s,y + t) \in A$ hold for at least $\exp(-(2\log c^{-1})^{O(1)})|G|^2$ points $(x,y) \in B  \times B$, where we used the fact that $g \mapsto 6g$ is an isomorphism. Thus, for some $c_1 \geq \exp(-(2\log c^{-1})^{O(1)})$,
    \[c_1 \leq \exx_{a_1, a_2 \in G} \id_B(a_1)\id_B(a_2) \Big| \exx_x \partial_{s + a_1, t + a_2} f(x) \on{e}(6\Phi(a_1, a_2)(x)) \Big|^2.\]
    Expanding, we get
    \[c_1 \leq \exx_{a_1, a_2, a_3, x\in G} \id_B(a_1)\id_B(a_2) \partial_{a_1 + s, a_2 + t, a_3} f(x) \on{e}(6\Phi(a_1, a_2)(a_3)).\]

    We now switch the notation to sum instead of expectations. Namely
    
    \begin{equation}c_1 |G|^4 \leq \sum_{a_1, a_2, a_3, x\in G} \id_B(a_1)\id_B(a_2) \partial_{a_1 + s, a_2 + t, a_3} f(x) \on{e}(6\Phi(a_1, a_2)(a_3)).\label{u4genIneq1}\end{equation}

    \noindent\textbf{Passing to a regular Bohr set.} Before proceeding further, we need to ensure that we have a regular Bohr set in place of $B$. We phrase the next step as a separate claim. As $E$ is a set of rank $d$, let $E_0$ be a subset of size $d$ such that $E = \langle E_0 \rangle_{\{-1,0,1\}}$
    
    \begin{claim}\label{regularbohrpassu4inverse}
        Let $\delta_0 \leq \frac{1}{288\pi}\Big(\frac{c_1}{2}\Big)^{64}$. There exist a regular Bohr set $B'$ of frequency set $\Gamma \cup E_0$ and radius $\rho_1 \geq \frac{c_1}{2^{11}d}\rho_0^{3d}$, a map $\psi : B' \times B' \times B' \to \mathbb{T}$ and a quantity $\delta_1 \in [\delta_0, 2\delta_0]$ such that $B'_{\delta_1}$ is also regular,
        \begin{align}\Big|\sum_{u_1, u_2, u_3 \in G} \id_{B'_{\delta_1}} \ast \id_{B'_{\delta_1}}(u_1)&\id_{B'_{\delta_1}} \ast \id_{B'_{\delta_1}}(u_2)\id_{B'_{\delta_1}} \ast \id_{B'_{\delta_1}}(u_3)\nonumber\\
        &\on{e}(6\psi(u_1, u_2, u_3)) \Big(\sum_{x \in G} \partial_{u_1, u_2, u_3} f(x)\Big)\Big| \geq \frac{1}{2}\Big(\frac{c_1}{2}\Big)^{64}  |B'_{\delta_1}|^6|G|,\label{symmetriccorelationU4inverse}\end{align}
        and for each $\varepsilon > 0$ we have that the restriction $\psi|_{B'_{\varepsilon} \times B'_\varepsilon \times B'_\varepsilon} \to \mathbb{T}$ is an $\varepsilon$-trilinear map.
    \end{claim}
    
    \begin{proof}
        Recall that $E$ is the set obtained above such that $\Phi$ is $E$-bihomomorphism. Let $\rho_1 \in [\frac{c_1}{2^{11}d}\rho_0^{3d}, \frac{c_1}{2^{10}d}\rho_0^{3d}]$ be such that $B(\Gamma \cup E_0, \rho_1)$ is regular. Define $B' = B(\Gamma \cup E_0, \rho_1)$.\\
        \indent Let us apply Lemma~\ref{bohrwreg1} to $B$ to find a radius function $\tilde{\rho}_0 : \Gamma \to [\rho_0 - 2\rho_1, \rho_0 + 2\rho_1]$ such that $B = B(\Gamma, \tilde{\rho}_0)$ and $|B(\Gamma, \tilde{\rho}_0 + 2\rho_1) \setminus B(\Gamma, \tilde{\rho}_0 - 2\rho_1)| \leq \frac{c_1}{100 d}\rho_0^{3d} |G|$. Let $\delta_1 \in [\delta_0, 2\delta_0]$ be such that $B'_{\delta_1}$ is regular.\\
        
        In particular, every element of $B(\Gamma, \tilde{\rho}_0 - 2\rho_1)$ can be written in $|B'_{\delta_1}|$ ways as a sum of an element in $B'_{\delta_1}$ and another one in $B(\Gamma, \tilde{\rho}_0 - \rho_1)$. Moreover, each such sum belongs to $B$. Lemma~\ref{basicbohrsizel} implies that $|B| \geq \rho_0^d|G|$. Hence
        
        \[\Big\|\id_B - \frac{1}{|B'_{\delta_1}|} \id_{B(\Gamma, \tilde{\rho}_0 - \rho_1)} \ast \id_{B'_{\delta_1}}\Big\|_{\ell^1} \leq |B(\Gamma, \tilde{\rho}_0) \setminus B(\Gamma, \tilde{\rho}_0 - 2\rho_1)| \leq \frac{c_1}{100 d}\rho_0^{3d} |G|.\]
    
        Using this approximation inside~\eqref{u4genIneq1}, we get
    
        \begin{align*}\frac{c_1}{2} |G|^4 |B'_{\delta_1}|^2 \leq & \sum_{a_1, a_2, a_3, x\in G} \id_{B(\Gamma, \tilde{\rho}_0 - \rho_1)} \ast \id_{B'_{\delta_1}}(a_1)\id_{B(\Gamma, \tilde{\rho}_0 - \rho_1)} \ast \id_{B'_{\delta_1}}(a_2) \partial_{a_1 + s, a_2 + t, a_3} f(x) \on{e}(6\Phi(a_1, a_2)(a_3))\\
        = & \sum_{a_1, a_2, a_3, x, u_1, u_2\in G} \id_{B(\Gamma, \tilde{\rho}_0 - \rho_1)}(u_1) \id_{B'_{\delta_1}}(a_1 - u_1)\id_{B(\Gamma, \tilde{\rho}_0 - \rho_1)}(u_2) \id_{B'_{\delta_1}}(a_2 - u_2) \\
        &\hspace{4cm}\partial_{a_1 + s, a_2 + t, a_3} f(x) \on{e}(6\Phi(a_1, a_2)(a_3)).\end{align*}
    
        By averaging, we get $u_1, u_2 \in B(\Gamma, \tilde{\rho}_0 - \rho_1)$ such that, after a slight change of variables where we replace $a_i$ by $a_i + u_i$,
        \[\frac{c_1}{2} |G|^2 |B'_{\delta_1}|^2 \leq \Big|\sum_{a_1, a_2, a_3, x\in G} \id_{B'_{\delta_1}}(a_1) \id_{B'_{\delta_1}}(a_2) \partial_{a_1 +u_1+ s, a_2 +u_2+ t, a_3} f(x) \on{e}(6\Phi(a_1 + u_1, a_2 + u_2)(a_3))\Big|.\]

         By making a slight change of variables where we replace $a_3$ by $a_3 + v$ for some dummy variable $v$ varying over $G$ and restricting $a_3$ to $B'$, after averaging, we obtain 

         \begin{align*}\frac{c_1}{2} |G| |B'_{\delta_1}|^3 \,\,\leq\,\, &\Big|\sum_{a_1, a_2, a_3 \in B'_{\delta_1}, x \in G}\partial_{a_1 +u_1+ s, a_2 +u_2+ t, a_3 + v} f(x) \on{e}(6\Phi(a_1 + u_1, a_2 + u_2)(a_3 + v))\Big|\end{align*}

         for some $v \in G$. Next, we apply Cauchy-Schwarz inequality three times, similarly to Gowers-Cauchy-Schwarz inequality for three variables. Namely,

          \begin{align*}\Big(\frac{c_1}{2} |G| |B'_{\delta_1}|^3\Big)^2 \,\,\leq\,\, &\Big|\sum_{a_2, a_3 \in B'_{\delta_1}, x\in G}\overline{\partial_{a_2 +u_2+ t, a_3 + v} f(x)} \Big( \sum_{a_1 \in B'_{\delta_1}} \partial_{a_2 +u_2+ t, a_3 + v} f(x+ a_1 +u_1+ s) \\
          &\hspace{4cm}\on{e}(6\Phi(a_1 + u_1, a_2 + u_2)(a_3 + v))\Big)\Big|^2\\
          \leq  \,\,&\Big(\sum_{a_2, a_3 \in B'_{\delta_1}, x\in G} \Big|\overline{\partial_{a_2 +u_2+ t, a_3 + v} f(x)}\Big|^2\Big) \Big(\sum_{a_2, a_3 \in B'_{\delta_1}, x\in G}\Big | \sum_{a_1 \in B'_{\delta_1}} \partial_{a_2 +u_2+ t, a_3 + v} f(x+ a_1 +u_1+ s)\\
          &\hspace{5cm}\on{e}(6\Phi(a_1 + u_1, a_2 + u_2)(a_3 + v))\Big|^2\Big)\\
          \leq\,\,& |G||B'_{\delta_1}|^2 \sum_{a_1, a'_1, a_2, a_3 \in B'_{\delta_1}, x\in G}\partial_{a_1 - a'_1, a_2 +u_2+ t, a_3 + v} f(x+ a'_1 +u_1+ s)\\
          &\hspace{2cm}\on{e}(6\Phi(a_1 + u_1, a_2 + u_2)(a_3 + v)-6\Phi(a'_1 + u_1, a_2 + u_2)(a_3 + v)).\end{align*}

         Making a change of variables $x - a'_1 - u_1 - s$ in place of $x$ allows us to write $x$ as the argument of $f$ above. Applying the same step to $a_2$ and $a_3$, allows us to conclude that

         \begin{align*}\Big(\frac{c_1}{2}\Big)^8 |G| |B'_{\delta_1}|^6 \,\,\leq\,\,& \Big|\sum_{a_1, a'_1, a_2, a_2', a_3, a_3' \in B'_{\delta_1}} \Big(\sum_{x \in G} \partial_{a_1 - a'_1, a_2 - a_2', a_3 - a_3'} f(x)\Big) \\
         &\on{e}(6\Phi(a_1 + u_1, a_2 + u_2)(a_3 - a_3') - 6\Phi(a'_1 + u_1, a_2 + u_2)(a_3 - a_3') \\
         &\hspace{3cm}-6\Phi(a_1 + u_1, a'_2 + u_2)(a_3 - a_3') + 6\Phi(a'_1 + u_1, a'_2 + u_2)(a_3 - a_3'))\Big|.\end{align*}

         By averaging, we obtain $a'_1, a'_2, a'_3 \in B'_{\delta_1}$ such that
         \begin{align*}\Big(\frac{c_1}{2}\Big)^8 |G| |B'_{\delta_1}|^3 \,\,\leq\,\, \Big|\sum_{a_1, a_2, a_3 \in B'_{\delta_1}} &\Big(\sum_{x \in G} \partial_{a_1 - a'_1, a_2 - a_2', a_3 - a_3'} f(x)\Big) \\
         &\on{e}(6\Phi(a_1 + u_1, a_2 + u_2)(a_3 - a_3') - 6\Phi(a'_1 + u_1, a_2 + u_2)(a_3 - a_3')  \\
         &\hspace{1cm}-6\Phi(a_1 + u_1, a'_2 + u_2)(a_3 - a_3') + 6\Phi(a'_1 + u_1, a'_2 + u_2)(a_3 - a_3'))\Big|\end{align*}

         We make a change of variables and use $d_i = a_i - a'_i$ in place of $a_i$, ranging over $a_i' + B'_{\delta_1}$ giving
           \begin{align}\Big(\frac{c_1}{2}\Big)^8 |G| |B'_{\delta_1}|^3 \,\,\leq\,\, \Big|\sum_{d_1, d_2, d_3 \in  G}& \id_{B'_{\delta_1}}(d_1 + a'_1) \id_{B'_{\delta_1}}(d_2 + a'_2)\id_{B'_{\delta_1}}(d_3 + a'_3)\Big(\sum_{x \in G} \partial_{d_1, d_2, d_3} f(x)\Big) \label{u4inversecoprime6eqnbeforepsi}\\
         &\on{e}(6\Phi(d_1 + a'_1 + u_1, d_2 + a'_2 + u_2)(d_3') - 6\Phi(a'_1 + u_1, d_2 + a'_2 + u_2)(d_3) \nonumber\\
         &\hspace{1cm}-6\Phi(d_1 + a'_1 + u_1, a'_2 + u_2)(d_3) + 6\Phi(a'_1 + u_1, a'_2 + u_2)(d_3))\Big|.\nonumber\end{align}

         We define $\psi : B'_{1/2} \times B'_{1/2} \times B'_{1/2} \to \mathbb{T}$ by
         \begin{align*}\psi(d_1, d_2, d_3) = & \Phi(d_1 + a'_1 + u_1, d_2 + a'_2 + u_2)(d_3) - \Phi(a'_1 + u_1, d_2 + a'_2 + u_2)(d_3)  \\
         &\hspace{2cm}-\Phi(d_1 + a'_1 + u_1, a'_2 + u_2)(d_3) + \Phi(a'_1 + u_1, a'_2 + u_2)(d_3).\end{align*}

        Note that $\psi$ is well-defined as $\Phi$ is defined on $B \times B$ and $a'_1 + u_1 \in B(\Gamma, \tilde{\rho}_0 - (1 - \delta_1)\rho_1)$.\\
         \indent Let us observe that $\psi$ is almost-trilinear in the sense of the claim. By definition of $\psi$, it is easily seen to be linear in the third coordinate. We prove this property for the first coordinate, the deduction is analogous for the second coordinate. To that end, let $x, x',y, z\in B'_{\varepsilon}$ be given. Since $\varepsilon \leq 1/4$, we have $x + x' \in B_{1/2}$. Then
         \begin{align*}\psi(x& + x', y, z) - \psi(x, y, z) - \psi(x', y, z) \\
         &= \Big(\Phi(x + x' + a'_1 + u_1, y + a'_2 + u_2)(z) - \Phi(x + a'_1 + u_1, y + a'_2 + u_2)(z) \\
         &\hspace{4cm}- \Phi(x' + a'_1 + u_1, y + a'_2 + u_2)(z) + \Phi(a'_1 + u_1, y + a'_2 + u_2)(z)\Big)\\
         &\hspace{1cm}\Big(\Phi(x + x' + a'_1 + u_1, a'_2 + u_2)(z) - \Phi(x + a'_1 + u_1, a'_2 + u_2)(z) \\
         &\hspace{4cm}- \Phi(x' + a'_1 + u_1, a'_2 + u_2)(z) + \Phi(a'_1 + u_1, a'_2 + u_2)(z)\Big).\end{align*}

        Since $\Phi$ is an $E$-bihomomorphism, the expression above equals $\chi(z)$ for some $\chi \in 2E$. Since $z \in B'_\varepsilon$, we have $\|\chi(z)\|_{\mathbb{T}} \leq 2 \varepsilon \rho_1 \leq 2^{-9}c_1\varepsilon$.\\

        Going back to~\eqref{u4inversecoprime6eqnbeforepsi}, we get
         \[\Big(\frac{c_1}{2}\Big)^8 |G| |B'_{\delta_1}|^3 \,\,\leq\,\, \Big|\sum_{d_1, d_2, d_3 \in  G}\id_{B'_{\delta_1}}(d_1 + a'_1) \id_{B'_{\delta_1}}(d_2 + a'_2)\id_{B'_{\delta_1}}(d_3 + a'_3)\Big(\sum_{x \in G} \partial_{d_1, d_2, d_3} f(x)\Big) \on{e}(6\psi(d_1, d_2, d_3))\Big|.\]

         Applying the Cauchy-Schwarz steps as above to variables $d_1, d_2, d_3$, we get
         \begin{align*}\Big(\frac{c_1}{2}\Big)^{64} |G| |B'_{\delta_1}|^6 \,\,\leq\,\, &\Big|\sum_{d_1,d_1', d_2,d_2', d_3,d_3' \in  G} \id_{B'_{\delta_1}}(d_1 + a'_1)\id_{B'_{\delta_1}}(d'_1 + a'_1) \id_{B'_{\delta_1}}(d_2 + a'_2)\id_{B'_{\delta_1}}(d'_2 + a'_2)\\
         &\hspace{5cm}\id_{B'_{\delta_1}}(d_3 + a'_3)\id_{B'_{\delta_1}}(d'_3 + a'_3)\Big(\sum_{x \in G} \partial_{d_1 - d_1', d_2 - d_2', d_3 - d_3'} f(x)\Big)\\
         &\hspace{1cm} \on{e}\Big(6\psi(d_1, d_2, d_3) - 6\psi(d'_1, d_2, d_3) - \dots + 6\psi(d_1, d'_2, d'_3) - 6\psi(d'_1, d'_2, d'_3)\Big)\Big|.\end{align*}

         By the almost linearity of $\psi$ in the first two variables, as well as linearity in the final variable, we see that 
         \[6\psi(d_1, d_2, d_3) - 6\psi(d'_1, d_2, d_3) - \dots + 6\psi(d_1, d'_2, d'_3) - 6\psi(d'_1, d'_2, d'_3)\]
         equals 
         \[6\psi(d_1, d_2, d_3-d'_3) - 6\psi(d'_1, d_2, d_3 - d'_3) -  6\psi(d_1, d'_2, d_3 - d'_3) + 6\psi(d'_1, d'_2, d_3 -d'_3)\]
         which differs from $6\psi(d_1 - d_1', d_2 - d_2', d_3 - d_3')$ by quantity $6\chi(d_3 - d'_3)$ for some $\chi \in 6E$ which is at most $36\delta_1\rho_1$ distant from from $0$ in $\mathbb{T}$. Using the elementary estimate on $\on{e}(t) - 1$, triangle inequality and inequality $\delta_1 \leq 2\delta_0 \leq \frac{1}{144\pi}\Big(\frac{c_1}{2}\Big)^{64}$, it follows that

         \begin{align*}\frac{1}{2}\Big(\frac{c_1}{2}\Big)^{64} |G| |B'_{\delta_1}|^6 \,\,\leq\,\, &\Big|\sum_{d_1,d_1', d_2,d_2', d_3,d_3' \in  G} \id_{B'_{\delta_1}}(d_1 + a'_1)\id_{B'_{\delta_1}}(d'_1 + a'_1) \id_{B'_{\delta_1}}(d_2 + a'_2)\id_{B'_{\delta_1}}(d'_2 + a'_2)\\
         &\hspace{5cm}\id_{B'_{\delta_1}}(d_3 + a'_3)\id_{B'_{\delta_1}}(d'_3 + a'_3)\Big(\sum_{x \in G} \partial_{d_1 - d_1', d_2 - d_2', d_3 - d_3'} f(x)\Big) \\
         &\hspace{1cm}\on{e}(6\psi(d_1 - d_1', d_2 - d_2', d_3 - d_3'))\Big|.\end{align*}

         Making a change of variables $u_i = d_i - d'_i$ in place of $d_i$, we finally get the claimed inequality after a slight misuse of notation (writing $B'$ instead of $B'_{1/2}$).
    \end{proof}

    Apply the claim above with the choice $\delta_0 = 2^{-29}\Big(\frac{c_1}{2}\Big)^{512}$. Hence, for $\varepsilon_1 = 2^{-28}\Big(\frac{c_1}{2}\Big)^{512}$, $\psi$ is $\varepsilon_1$-trilinear on $B'_{2\delta_0}\times B'_{2\delta_0} \times B'_{2\delta_0}$, which is important for the application of Lemma~\ref{symmetryArgumentBohr}.\\

    \noindent\textbf{Definition of almost-cubic.} Let us define $q : B' \to \mathbb{T}$ by $q(x) = \psi(x,x,x)$, which is well-defined. Note that 
    \begin{equation}\Big\|\Delta_{a,b,c}q(x) - \sum_{\pi \in \on{Sym}} \psi \circ \pi(a,b,c)\Big\|_{\mathbb{T}} \leq 2^9\eta,\label{almostcubicderivativerelationship}\end{equation}
    whenever $x, a,b,c \in B'_\eta$ with $\eta \leq 1/4$, as expansion of terms like $\psi(x + a + b + c, x + a + b + c, x + a + b + c)$ can result in at most $2^6$ simpler terms with single variable at each argument and we have 8 terms in $\Delta_{a,b,c}q(x)$. In the equality above, for a permutation $\pi \in \on{Sym}_3$ we use notation $\psi \circ \pi(u_1, u_2, u_3)$ defined by $\psi(u_{\pi(1)}, u_{\pi(2)}, u_{\pi(3)})$. In particular, $q$ is an almost-cubic polynomial, as the inequality above implies $\Big\|\Delta_{a,b,c,d}q(x)\Big\|_{\mathbb{T}} \leq 2^{10}\eta$ whenever $\eta \leq 1/8$.\\
    
    \noindent\textbf{Symmetry argument.} We now use a symmetry argument. Let $0 < \varepsilon_2 < 2^{-100}\Big(\frac{c_1}{2}\Big)^{1024}\delta_1$ be a small quantity to be specified later. Take $\delta_2\in [\varepsilon_2/2, \varepsilon_2]$ such that $B'_{\delta_2}$ is regular. Since $\varepsilon_1 = 2^{-28}\Big(\frac{c_1}{2}\Big)^{512}$, Lemma~\ref{symmetryArgumentBohr} implies that for for each transposition $\pi$ 
     \begin{align*}\on{Re}\Big(\sum_{u_1, u_2, u_3 \in G} \id_{B'_{\delta_2}} \ast \id_{B'_{\delta_2}}(u_1)\id_{B'_{\delta_2}} \ast \id_{B'_{\delta_2}}(u_2)\id_{B'_{\delta_2}} \ast &\id_{B'_{\delta_2}}(u_3) \on{e}(6\psi(u_1, u_2, u_3) - 6 \psi\circ\pi(u_1, u_2, u_3)) \Big)\\
     &\geq 2^{-31}\Big(\frac{c_1}{2}\Big)^{512}|B'_{\delta_2}|^6.\end{align*}

    By Theorem~\ref{equidistributionalmosttrilinearmaps} applied to each transposition $\pi$, as such permutations generate $\on{Sym}_3$, we may find some quantities $\rho_2, \delta_3 \geq \exp(-(2d \log(\varepsilon_2^{-1} c_1^{-1} \rho_1^{-1}))^{O(1)})$, a bilinear Bohr variety $S \subseteq B'_{\delta_2} \times B'_{\delta_2}$ of codimension $d_2 \leq (2d \log(\varepsilon_2^{-1} c_1^{-1}\rho_1^{-1})^{O(1)})$ and radius at least $\rho_2$, such that for all $(x,y) \in S, z \in B'_{\delta_3}$, we have $6\psi(x,y,z)$, $6\psi(x, z, y)$, etc. all $\varepsilon_3 = (2d)^{Cd}\sqrt{\varepsilon_2}$ close to each other, where $C$ is the final implicit constant in Theorem~\ref{equidistributionalmosttrilinearmaps}.\\

    Hence, as long as $\varepsilon_3 \leq \frac{1}{100}$, for each $\pi$ and each $(x,y,z) \in S \times B'_{\delta_3}$, we have some value $\sigma_\pi(x,y,z) \in \mathbb{Z}/6\mathbb{Z}$ such that $\Big\|\psi(x,y,z) - \psi \circ \pi(x,y,z) - \frac{\sigma_\pi(x,y,z)}{6}\Big\|\leq \varepsilon_2$, for a map $\sigma_\pi : S \times B'_{\delta_3} \to \mathbb{Z}/6\mathbb{Z}$.\\

    \begin{claim}
        For each $\pi \in \on{Sym}_3$, the map $\sigma_\pi$ is trilinear on $S \times B'_{\delta_3}$.
    \end{claim}

    \begin{proof}
        Clear as $\psi$ is $\varepsilon_1$-trilinear.
    \end{proof}

    Thus, if have $(x,y,z) \in S \times (6 \cdot B'_{\delta_3/6})$, then $\sigma_\pi(x,y,z)$ vanishes for each $\pi$. In particular, for such $(x,y,z)$ we obtain
    \begin{equation}\Big\|6\psi(x,y,z) - \Big(\sum_{\pi \in \on{Sym}_3}\psi \circ \pi(x,y,z)\Big)\Big\|_{\mathbb{T}} \leq 6\varepsilon_3.\label{sixpsitopermpsi}\end{equation}
    Let us now use this fact to replace $6\psi(x,y,z)$ in~\eqref{symmetriccorelationU4inverse} by $\sum_{\pi \in \on{Sym}_3}\psi \circ \pi(x,y,z)$ and eventually obtain a correlation with an almost-cubic. \\

    Let the bilinear Bohr variety $S$ be defined as $\cup_{x \in C} B(\Lambda, \Theta_1(x), \dots, \Theta_{d_2}(x), \rho_2)$ for a set of characters $\Lambda$ of size at most $d_2$ and Freiman-linear maps $\Theta_i : C \to \hat{G}$. We write $S^{\rho'}$ for related bilinear Bohr variety in which the columns are replaced $B(\Lambda, \Theta_1(x), \dots, \Theta_{d_2}(x), \rho')$.

    \begin{claim}
        Provided $\varepsilon_3 \leq 2^{-20}\Big(\frac{c_1}{2}\Big)^{512}$, we have a coset progression $C \subseteq B'_{\delta_2}$ of  rank at most $(2d\,d_2\log(c_1^{-1}\rho_2^{-1}))^{O(1)}$ and size $|C| \geq \exp(-(2d\,d_2\log(c_1^{-1}\rho_2^{-1}))^{O(1)})|G|$ such that for some $\rho_3 \geq \rho_2/4$
        \begin{align*}&\Big|\sum_{u_1, u_2, u_3 \in G} \id_{C} \ast \id_{C}(u_1)\id_{S^{\rho_3}_{u_1 \bcdot}} \ast \id_{S^{\rho_3}_{u_1 \bcdot}}(u_2)\id_{6 \cdot B'_{\delta_3/{12}}} \ast \id_{6 \cdot B'_{\delta_3/12}}(u_3)\\
        &\hspace{2cm}\on{e}\Big(\sum_{\pi \in \on{Sym}_3}\psi\circ\pi(u_1, u_2, u_3)\Big) \Big(\sum_{x \in G} \partial_{u_1, u_2, u_3} f(x)\Big) \Big| \geq 2^{-15}\Big(\frac{c_1}{2}\Big)^{512}  |C| \Big(\sum_{u_1 \in C} |S^{\rho_3}_{u_1 \bcdot}|^2\Big)  |B'_{\delta_3/12}|^2.
        \end{align*}
    \end{claim}

    \begin{proof} 
    Let $\eta = 2^{-40}\Big(\frac{c_1}{2}\Big)^{1024}(\rho_2/2)^{d_2}$ and apply algebraic regularity lemma (Theorem~\ref{algreglemmaintro}) to find a symmetric proper coset progression $C' \subseteq C$ of same rank and size $|C'| \geq \exp(-(2d\log(\eta^{-1}\rho_2^{-1}))) |C|$ and a quantity $\alpha$ such that $(1-\eta)|C'|$ columns of $S^{\rho_3}$ have size $\Big||S^{\rho_3}_x| - \alpha |B(\Lambda, \rho_3)|\Big| \leq \eta |G|$.\\
    
    We use Cauchy-Schwarz inequality to essentially ensure that $u_1 \in C'$, $(u_1, u_2) \in S$ and $u_3 \in 6 \cdot B'_{\delta_3/6}$. The conditions on $u_1$ and $u_3$ are easy to achieve in the same way as in the proof of Claim~\ref{regularbohrpassu4inverse} using two applications of the Cauchy-Schwarz inequality. We thus obtain
    \begin{align*}\Big|\sum_{u_1, u_2, u_3 \in G} \id_{\frac{1}{2}C'} &\ast \id_{\frac{1}{2}C'}(u_1)\id_{B'_{\delta_1}} \ast \id_{B'_{\delta_1}} (u_2)\id_{6 \cdot B'_{\delta_3/{12}}} \ast \id_{6 \cdot B'_{\delta_3/12}}(u_3)\\
        &\on{e}\Big(6\psi(u_1, u_2, u_3)\Big) \Big(\sum_{x \in G} \partial_{u_1, u_2, u_3} f(x)\Big) \Big| \geq 2^{-4}\Big(\frac{c_1}{2}\Big)^{256}  \Big|\frac{1}{2}C'\Big|^2 |B'_{\delta_1}|^2  |B'_{\delta_3/12}|^2.
        \end{align*}

    For each $u_1 \in C'$, using the regularity of $B'_{\delta_3}$ we may replace any variable $z$ ranging over $B'_{\delta_3} = B(\Gamma \cup E_0, \delta_3\rho_1)$ by a sum of variables $z' + w$, where $z' \in B'' = B(\Gamma \cup E_0, \delta_3\rho_1 - \rho_3)$ and $w \in S^{\rho_3}_{u_1 \bcdot}$. Defining $M = \alpha |B(\Lambda, \rho_3)|$, recall that $|S^{\rho_3}_{u_1 \bcdot}| \in [M/2, 2M]$ for all but at most $\eta|C'|$ elements $u_1 \in C$. Hence we get
    \begin{align*}&\Big|\sum_{u_1, u_2, z_2, w_2,  u_3 \in G} \id_{\frac{1}{2}C'} \ast \id_{\frac{1}{2}C'}(u_1)
    \id_{B'_{\delta_1}}(u_2)\id_{B''}(z_2) \id_{S^{\rho_3}_{u_1 \bcdot}}(w_2)
    \id_{6 \cdot B'_{\delta_3/{12}}} \ast \id_{6 \cdot B'_{\delta_3/12}}(u_3)\\
        &\hspace{2cm}\on{e}\Big(6\psi(u_1, u_2, u_3)\Big) \Big(\sum_{x \in G} \partial_{u_1, z_2 + w_2 - u_2, u_3} f(x)\Big) \Big| \geq 2^{-6}\Big(\frac{c_1}{2}\Big)^{256}  \Big|\frac{1}{2}C'\Big|^2 M |B'_{\delta_1}|^2  |B'_{\delta_3/12}|^2.
        \end{align*}

    Averaging over $u_2$ and $z_2$, we get a choice of those elements such that
     \begin{align*}&\Big|\sum_{u_1, w_2, u_3 \in G} \id_{\frac{1}{2}C'} \ast \id_{\frac{1}{2}C'}(u_1)
    \id_{S^{\rho_3}_{u_1 \bcdot}}(w_2)
    \id_{6 \cdot B'_{\delta_3/{12}}} \ast \id_{6 \cdot B'_{\delta_3/12}}(u_3)\\
        &\hspace{2cm}\on{e}\Big(6\psi(u_1, u_2, u_3)\Big) \Big(\sum_{x \in G} \partial_{u_1, z_2 + w_2 - u_2, u_3} f(x)\Big) \Big| \geq 2^{-6}\Big(\frac{c_1}{2}\Big)^{256}  \Big|\frac{1}{2}C'\Big|^2 M |B'_{\delta_3/12}|^2.
        \end{align*}

    By Cauchy-Schwarz inequality and $\varepsilon_1$-trilinearity of $\psi$, we get
    \begin{align*}&\Big|\sum_{u_1, w_2, u_3 \in G} \id_{\frac{1}{2}C'} \ast \id_{\frac{1}{2}C'}(u_1)
    \id_{S^{\rho_3}_{u_1 \bcdot}} \ast \id_{S^{\rho_3}_{u_1 \bcdot}}(w_2)
    \id_{6 \cdot B'_{\delta_3/{12}}} \ast \id_{6 \cdot B'_{\delta_3/12}}(u_3)\\
        &\hspace{2cm}\on{e}\Big(6\psi(u_1, u_2, u_3)\Big) \Big(\sum_{x \in G} \partial_{u_1, w_2, u_3} f(x)\Big) \Big| \geq 2^{-13}\Big(\frac{c_1}{2}\Big)^{512}  \Big|\frac{1}{2}C'\Big|^2 M^2 |B'_{\delta_3/12}|^2.
        \end{align*}

    Using~\eqref{sixpsitopermpsi}, provided $\varepsilon_3 \leq 2^{-20}\Big(\frac{c_1}{2}\Big)^{512}$, we get
    \begin{align*}&\Big|\sum_{u_1, w_2, u_3 \in G} \id_{\frac{1}{2}C'} \ast \id_{\frac{1}{2}C'}(u_1)
    \id_{S^{\rho_3}_{u_1 \bcdot}} \ast \id_{S^{\rho_3}_{u_1 \bcdot}}(w_2)
    \id_{6 \cdot B'_{\delta_3/{12}}} \ast \id_{6 \cdot B'_{\delta_3/12}}(u_3)\\
        &\hspace{2cm}\on{e}\Big(\sum_{\pi \in \on{Sym}_3}\psi\circ\pi(u_1, u_2, u_3)\Big) \Big(\sum_{x \in G} \partial_{u_1, w_2, u_3} f(x)\Big) \Big| \geq 2^{-14}\Big(\frac{c_1}{2}\Big)^{512}  \Big|\frac{1}{2}C'\Big|^2 M^2 |B'_{\delta_3/12}|^2.
        \end{align*}

    Use the quasirandomness of $S^{\rho/3}$ to replace $\Big|\frac{1}{2}C'\Big| M^2$ by $\Big(\sum_{u_1 \in C} |S^{\rho_3}_{u_1 \bcdot}|^2\Big)$.\end{proof}

    Using~\eqref{almostcubicderivativerelationship} inside $B'_{\delta_4}$ for $\delta_4 = 2^{-20}\Big(\frac{c_1}{2}\Big)^{512} $, we get
    
    \begin{align*}&\Big|\sum_{u_1, u_2, u_3 \in G} \id_{C} \ast \id_{C}(u_1)\id_{S^{\rho_3}_{u_1 \bcdot}} \ast \id_{S^{\rho_3}_{u_1 \bcdot}}(u_2)\id_{6 \cdot B'_{\delta_3/{12}}} \ast \id_{6 \cdot B'_{\delta_3/12}}(u_3)\\
        &\hspace{2cm}\Big(\sum_{y \in B'_{\delta_4}}\on{e}\Big(\Delta_{u_1, u_2, u_3}q(y)\Big)\Big) \Big(\sum_{x \in G} \partial_{u_1, u_2, u_3} f(x)\Big) \Big| \geq 2^{-15}\Big(\frac{c_1}{2}\Big)^{512}  |C| \Big(\sum_{u_1 \in C} |S^{\rho_3}_{u_1 \bcdot}|^2\Big)  |B'_{\delta_3/12}|^2.
        \end{align*}

    Making a change of variables where $x$ is replaced by $x + y$, taking a suitable $x$ and defining $\tilde{f}(t) = \id_{B'}(t) f(t + x) \on{e}(q(t))$, we get

     \begin{align*}&\Big|\sum_{u_1, u_2, u_3 \in G} \id_{C} \ast \id_{C}(u_1)\id_{S^{\rho_3}_{u_1 \bcdot}} \ast \id_{S^{\rho_3}_{u_1 \bcdot}}(u_2)\id_{6 \cdot B'_{\delta_3/{12}}} \ast \id_{6 \cdot B'_{\delta_3/12}}(u_3)\\
        &\hspace{2cm}\Big(\sum_{y \in B'_{\delta_4}} \partial_{u_1, u_2, u_3} \tilde{f}(y)\Big) \Big| \geq 2^{-15}\Big(\frac{c_1}{2}\Big)^{512}  |C| \Big(\sum_{u_1 \in C} |S^{\rho_3}_{u_1 \bcdot}|^2\Big)  |B'_{\delta_3/12}|^2 | B'_{\delta_4}|.
        \end{align*}

    Choose a suitable $\varepsilon_2 \geq \Big(\frac{c_1}{2}\Big)^{O(1)}d^{-O(d)}$ so that $\varepsilon_3 \leq 2^{-20}\Big(\frac{c_1}{2}\Big)^{512}$. Applying Gowers-Cauchy-Schwarz inequality implies that $\|\tilde{f}\|_{\mathsf{U}^3(G)} \geq \exp(-\log^{O(1)}(2c^{-1}))$, and we are done by Theorem~\ref{greentaou3inversethm} (noting that locally quadratic functions on Bohr sets are almost-cubic polynomials). 
\end{proof}

\section{Abelian 2-groups}\label{abelian2groupssection}

In this section, we deduce the inverse theorem for $\mathsf{U}^4$ norm when the ambient group is $(\mathbb{Z}/2^d\mathbb{Z})^n$.

\begin{theorem}\label{abeliantwogroupsinverse}
    Fix $d \in \mathbb{N}$ and let $G = (\mathbb{Z}/2^d\mathbb{Z})^n$. Suppose that $f : G \to \mathbb{D}$ satisfies $\|f\|_{\mathsf{U}^4} \geq c$. Then there exists a cubic polynomial $q : G \to \mathbb{T}$ such that 
    \[\Big|\exx_{x} f(x) \on{e}(q(x))\Big| \geq \exp(-\log^{O_d(1)}(2c^{-1})).\]
\end{theorem}

Unlike finite vector spaces in which all subgroups are direct summands, this is not necessarily the case with bounded torsion groups. The next lemma recovers that property to a large extent.

\begin{lemma}\label{directsummandstorsion}
    Let $H \leq (\mathbb{Z}/2^d\mathbb{Z})^n$ be a subgroup of density $c$. Then there exist subgroups $U, V \leq (\mathbb{Z}/2^d\mathbb{Z})^n$ such that $U \leq H$, $U$ has density at least $c^d$ and $(\mathbb{Z}/2^d\mathbb{Z})^n = U \oplus V$. In particular, there exists a surjective projection $\pi : (\mathbb{Z}/2^d\mathbb{Z})^n \to U$, taking values inside $H$ and having a kernel of size at most $c^{-d}$.
\end{lemma}

\begin{proof}
    The proof relies on the Smith normal form. Let $e_1, \dots, e_n$ be the standard basis of $(\mathbb{Z}/2^d\mathbb{Z})^n$, in the sense that they generate the group and $\sum_{j \in [n]} \lambda_j e_j = 0$ for integers $\lambda_j$ implies $2^d | \lambda_j$. Let $h_1, \dots, h_m \in H$ be any generating set of $H$. We set $M$ to be $m \times n$ matrix obtained by expressing $h_1, \dots, h_m$ in terms of $e_1, \dots, e_n$, i.e. such that $h_i = \sum_{j \in [n]} M_{ij} e_j$.\\
    \indent We claim that, whenever $A$ is an $n \times n$ invertible integer matrix, we have that $a_i = \sum_{j \in [n]} A_{i j} e_j$, where $i \in [n]$, is also a basis of the group $(\mathbb{Z}/2^d\mathbb{Z})^n$ in the above sense. Indeed, as every element $x$ can be written as $x = \sum_{i \in [n]} \mu_i e_i$, we have $x = \sum_{i, j \in [n]} \mu_i(A^{-1})_{ij} a_j$. On the other hand, if $\sum_{i \in [n]} \lambda_i a_i = 0$, then $\sum_{i, j \in [n]} \lambda_i A_{i j} e_j= 0$, so $2^d | \sum_{i \in [n]} \lambda_i A_{i j}$ for all $j \in [n]$. So the entries of $A^T \lambda$ are divisible by $2^d$. Multiplying by the inverse of $A^T$, which is an integral matrix, we see that $2^d | \lambda_i$, as desired.\\
    \indent Similarly, whenever $B$ is an $m \times m$ invertible integer matrix, we have that $b_i = \sum_{j \in [m]} B_{i j} h_j$, where $i \in [m]$, is a generating set of $H$. The matrix of coefficients of $b_1, \dots, b_m$ with respect to $a_1, \dots, a_n$ is $BMA^{-1}$. We may find $A, B$ so that $M$ is in its Smith normal form.\\
    \indent Hence, we may find a basis $a_1, \dots, a_n$ of $(\mathbb{Z}/2^d\mathbb{Z})^n$ and a generating set $b_1, \dots, b_m$, such that $m \leq n$ and $b_i = k_i a_i$, $k_1, \dots, k_m \not= 0$ (as we may ignore zeros in the generating set), where $k_1 | k_2 | \dots | k_m$. If $2^{r_i}||k_i$ for each $i$, we get a formula for the size of $H$, namely $|H| = 2^{-r_1 - \dots -r_m} 2^{dm}$.\\
    \indent Since $k_i | k_{i + 1}$, we have $r_i \leq r_{i+1}$, so let $\ell \in [m]$ be the largest index with $r_\ell = 0$. We may now define our desired subgroups, namely $U = \langle a_i : i \in [\ell]\rangle$, $V = \langle a_i : i \in [\ell + 1,n]\rangle$. Note that when $r_i = 0$, we have $k_i$ invertible, so $U \leq H$. Hence $|U| = 2^{d\ell}$. It remains to prove the final size estimate.\\
    \indent By the work above, we have $1\leq r_{\ell + 1} \leq \dots \leq r_m \leq d$, $2^{-r_{\ell + 1} - \dots -r_m} 2^{dm} \geq c 2^{dn}$ and  $|U| = 2^{d\ell}$. Write $s = r_{\ell + 1} + \dots  + r_m$. Hence $2^{-s} 2^{dm} \geq c 2^{dn}$, so in particular $2^{-s} \geq c$. Thus
    \[|U| =2^{d\ell} = 2^{-d(m - \ell)}2^{dm}  = 2^{-(d-1)(m - \ell)} 2^{-(m-\ell)} 2^{dm} \geq 2^{-(d-1)s} \Big(2^{-s} 2^{dm}\Big) \geq c^{d-1} \cdot c 2^{dn} = c^d 2^{dn}.\qedhere\]
\end{proof}

\vspace{\baselineskip}

We need an analogous result for small subgroups.

\begin{lemma}\label{directsummandstorsionextension}
    Let $H \leq (\mathbb{Z}/2^d\mathbb{Z})^n$ be a subgroup of size $K$. Then there exist subgroups $U, V \leq (\mathbb{Z}/2^d\mathbb{Z})^n$ such that $H \leq U$, $U$ has size at most $K^d$ and $(\mathbb{Z}/2^d\mathbb{Z})^n = U \oplus V$. In particular, there exists a surjective projection $\pi : (\mathbb{Z}/2^d\mathbb{Z})^n \to V$, having a kernel of size at most $K^{d}$ and vanishing on $H$.
\end{lemma}

\begin{proof}
    A slight modification of the last step in the previous proof implies the proposition.
\end{proof}

We observe that the structure theorem for Freiman bihomomorphisms simplifies in the case of bounded torsion groups.

\begin{corollary}[Global structure theorem for Freiman bihomomorphisms in $(\mathbb{Z}/2^d\mathbb{Z})^n$]
    \label{maininversetheoremspectorsion}
    Let $G = (\mathbb{Z}/2^d\mathbb{Z})^n$ and let exponent of $H$ be $2^d$. Let $A \subseteq G \times G$ be a set of density $c$ and let $\phi : A \to H$ be a Freiman bihomomorphism. Then there exists a global Freiman bihomomorphism $\Phi : G \times G \to H$ such that $\phi(x,y) = \Phi(x,y)$ holds for at least $\exp(-\log^{O(1)}(c^{-1}))|G|^2$ points $(x,y) \in A$. 
\end{corollary}

\begin{proof}
    Apply Theorem~\ref{maininversetheorem} to get a parameters $r \leq \log^{O(1)}(2c^{-1})$ and $c' \geq \exp(-\log^{O(1)}(c^{-1}))$, a set $E$ of rank at most $r$, Bohr sets $B_1 \subseteq G_1, B_2 \subseteq G_2$ of codimension at most $r$ and radius $c'$, elements $s \in G_1,t\in G_2$ and an $E$-bihomomorphism $\Phi : B_1 \times B_2 \to H$ such that $\Phi(x,y) = \phi(x + s,y + t)$ holds for at least $c'|G_1|G_2|$ points $(x,y) \in B_1 \times B_2$. Since $G = (\mathbb{Z}/2^d\mathbb{Z})$, $B_1$ and $B_2$ contain subgroups $U_1$ and $U_2$ of density at least $2^{-dr}$. Using Lemma~\ref{directsummandstorsion}, we may pass to further subgroups $U'_1 \leq U_1$ and $U'_2 \leq U_2$ such that we have surjective projections $\pi_i : G \to U'_i$. Hence, $\Phi'(x,y) = \Phi(\pi_1(x), \pi_2(y))$ is a global $E$-bihomomorphism.\\
    Let $S = \langle E\rangle$ be the subgroup generated by $E$ inside $H$, which has size at most $2^{dr}$. By Lemma~\ref{directsummandstorsionextension}, we may find a homomorphism $\tilde{\pi} : H \to H$ which vanishes on $S$, but has at most $2^{d^2r}$ elements in the kernel. Then $\tilde{\pi} \circ \Phi'$ is a global Freiman bihomorphism. The result follows by an averaging argument.
\end{proof}

The following lemma will be key to completing the integration argument in the inverse theorem. We remark that the inverse theorem is known in $(\mathbb{Z}/2\mathbb{Z})^n$.

\begin{lemma}\label{correlationforsymmetrytorsionLemma} Let $G = (\mathbb{Z}/2^d\mathbb{Z})^n$, let $f_1, f_2, f_3, f_4 : G \to \mathbb{D}$ be functions and let $\alpha : G \times G \times G \to \mathbb{T}$ be a symmetric trilinear form
    \begin{equation}\Big|\exx_{u , v \in G} f_1(u) f_2(v) f_3(u + v) f_4(2u + v) \on{e}\Big(\alpha(u,u, v)\Big)\Big| \geq \delta.\label{keycorrelationforsymmetrytorsion}\end{equation}
    Then $2^{d-1} \alpha(u,u,v)$ is a symmetric bilinear form on some subgroup $K \leq G$ of size $|K| \geq ( \delta/2)^{O_d(1)}|G|$.
\end{lemma}

\begin{proof} We prove the claim by induction on $d$. The base case $d = 1$ was covered in~\cite{Tidor}. Now suppose that $d \geq 2$ and that the claim holds for $d - 1$. We note that the inductive hypothesis implies a similar conclusion from a condition very similar to~\eqref{keycorrelationforsymmetrytorsion}, in which $f_4$ has argument $u + 2v$ instead.

\begin{claim}\label{symmetryargumnetexpressionreversedclaim}
    Let $G'$ be a group isomorphic to $(\mathbb{Z}/2^{d-1}\mathbb{Z})^n$. Let $g_1, g_2, g_3, g_4 : G' \to \mathbb{D}$ be functions and let $\beta : G'\times G' \times G' \to \mathbb{T}$ be a symmetric trilinear form
    \begin{equation}\Big|\exx_{u , v \in G'} g_1(u) g_2(v) g_3(u + v) g_4(u + 2v) \on{e}\Big(\beta(u,u,v)\Big)\Big| \geq \delta'.\label{keycorrelationforsymmetrytorsionReversed}\end{equation}
    Then $2^{d-2} \beta(u,u,v)$ is a symmetric bilinear form on some subgroup $K' \leq G'$ of size $|K'| \geq  (\delta'/2)^{O_d(1)}|G'|$.
\end{claim}

\begin{proof}
    Symmetry of $\beta$ implies the identity
    \[\beta(u + v, u + v, u+v) = \beta(u,u,u) + 3\beta(u,u,v) + 3\beta(u,v,v) + \beta(v,v,v).\]
    Using the fact that 3 is invertible modulo $2^{d-1}$ we get
    \[\beta(u,u,v) = 3^{-1}\beta(u + v, u + v, u+v) - 3^{-1}\beta(u,u,u) - 3^{-1}\beta(v,v,v) - \beta(v,v,u).\]
    Thus,
    \begin{align*}\delta' \leq &\Big|\exx_{u , v \in G'} g_1(u) g_2(v) g_3(u + v) g_4(u + 2v) \on{e}\Big(\beta(u,u,v)\Big)\Big|\\
    = &\Big|\exx_{u , v \in G'} g_1(u) g_2(v) g_3(u + v) g_4(u + 2v) \on{e}\Big(3^{-1}\beta(u + v, u + v, u+v) \\
    &\hspace{6cm}- 3^{-1}\beta(u,u,u) - 3^{-1}\beta(v,v,v) - \beta(v,v,u)\Big)\Big|\\
    = &\Big|\exx_{u , v \in G'} \tilde{g}_1(u) \tilde{g}_2(v) \tilde{g}_3(u + v) \tilde{g}_4(u + 2v) \on{e}\Big(- \beta(v,v,u)\Big)\Big|\end{align*}
    for $\tilde{g}_1(x) = g_1(x) \on{e}(- 3^{-1}\beta(x,x,x))$, $\tilde{g}_2(x) = g_2(x) \on{e}(- 3^{-1}\beta(x,x,x))$, $\tilde{g}_3(x) = g_(x) \on{e}(3^{-1}\beta(x,x,x))$ and $\tilde{g}_4(x) = g_4(x)$. After a changing the roles of $u$ and $v$ we get a correlation of the same shape as~\eqref{keycorrelationforsymmetrytorsion}, so the inductive hypothesis can be applied, proving the claim.
\end{proof}

We begin the proof by observing that certain directional uniformity norms of $f_1$ and $f_3$ control the expression.\\

\noindent\textbf{Control by $f_1$ and $f_3$.} Let us observe that $\mathsf{U}(G, 2^{d-1}\cdot G)$ norm of $f_1$ and $f_3$ controls the expression above. We phrase this as a separate claim for arbitrary functions.
    
    \begin{claim}
        Let $h_1, \dots, h_4 : G \to \mathbb{D}$ be functions. Then
         \[\Big|\exx_{u , v \in G} h_1(u) h_2(v) h_3(u + v) h_4(2u + v) \on{e}\Big(\alpha(u,u, v)\Big)\Big| \leq \|h_1\|_{\mathsf{U}(G, 2^{d-1}\cdot G)}\|h_3\|_{\mathsf{U}(G, 2^{d-1}\cdot G)}.\]
    \end{claim}

    \begin{proof}
    We introduce a dummy variable $w \in 2^{d-1} \cdot G$ and replace $u$ by $u + w$ and observe that 

    \begin{align*}
        &\Big|\exx_{u , v \in G} h_1(u) h_2(v) h_3(u + v) h_4(2u + v) \on{e}\Big(\alpha(u,u, v)\Big)\Big|^4 \\
        = &\Big|\exx_{u , v \in G, w \in 2^{d-1} \cdot G} h_1(u + w) h_2(v) h_3(u + w + v) h_4(2u + 2w + v) \on{e}\Big(\alpha(u + w,u + w, v)\Big)\Big|^4 \\
        = &\Big|\exx_{u , v \in G, w \in 2^{d-1} \cdot G} h_1(u + w) h_2(v) h_3(u + w + v) h_4(2u + v) \on{e}\Big(\alpha(u,u, v)\Big)\Big|^4\\
        = &\Big|\exx_{u , v \in G}  h_2(v)  h_4(2u + v) \on{e}\Big(\alpha(u,u, v)\Big) \exx_{w \in 2^{d-1} \cdot G} h_1(u + w)h_3(u + w + v)\Big|^4\\
        \leq &\Big| \exx_{u , v \in G, w, a \in 2^{d-1} \cdot G} \partial_a h_1(u + w) \partial_a h_3(u + w + v) \Big|^2\\
        \leq & \Big(\exx_{x, b \in G, a \in 2^{d-1} \cdot G} \partial_{a,b} h_1(x)\Big) \Big(\exx_{x, b \in G, a \in 2^{d-1} \cdot G} \partial_{a,b} h_3(x)\Big)\\
        = & \|h_1\|_{\mathsf{U}(G, 2^{d-1}\cdot G)}^4\|h_3\|_{\mathsf{U}(G, 2^{d-1}\cdot G)}^4.
    \end{align*}

    Note that we used $\alpha(w,w, v) = 0$ since $w = 2^{d-1} w'$ for some element $w' \in G$, hence $\alpha(w,w, v) = \alpha(2^{d-1} w', 2^{d-1} w', v) = 2^{2d - 2} \alpha(w',w', v) = 0$ since $2d - 2 \geq d$ for $d \geq 2$. Also $\alpha(u,w,v) + \alpha(w, u, v) = 2\alpha(u, w, v) = 2^d \alpha(u, w', v) = 0$. We also used $2w = 2^d w' = 0$. The final two inequalities are applications of Cauchy-Schwarz inequality.\end{proof}

    Going back to~\eqref{keycorrelationforsymmetrytorsion}, we hope to apply inverse theorem for $\mathsf{U}(G, 2^{d-1}\cdot G)$ to functions $f_1$ and $f_3$ and replace them by structured functions appearing as obstructions for that norm. However, doing that naively gives a correlation for $f_1$ and $f_3$, which may not simply be inserted back into~\eqref{keycorrelationforsymmetrytorsion}. Instead, we carry out a maneuver where $f_1$ and $f_3$ are replaced by a function for which this strategy works.\\

    \noindent\textbf{Replacing $f_1$ and $f_3$.} Define $F_1(u) = \exx_{v \in G} \overline{f_2(v) f_3(u + v) f_4(2u + v)} \on{e}\Big(-\alpha(u,u, v)\Big)$. Hence, inequality~\eqref{keycorrelationforsymmetrytorsion} becomes $|\ex_{u \in G}f_1(u) \overline{F_1(u)} | \geq \delta.$ By Cauchy-Schwarz inequality, we get
    $|\ex_{u \in G}F_1(u) \overline{F_1(u)} | \geq \delta^2$. Now expand the definition of $\overline{F_1}$ to get
    \[\Big|\exx_{u , v} F_1(u) f_2(v) f_3(u + v) f_4(2u + v) \on{e}\Big(\alpha(u,u, v)\Big)\Big| \geq \delta^2.\]
    By the claim above, we have $\|F_1\|_{\mathsf{U}(G, 2^{d-1}\cdot G)} \geq \delta^2$. Theorem~\ref{u2inversesbgp} gives us a function $f'_1$ and a character $\chi_1 \in \hat{G}$ such that
    \[\Big|\exx_{u \in G}\overline{F_1(u)} f'_1(2u) \on{e}(\chi_1(u))\Big| \geq (\delta/2)^{O(1)}.\]
    Expanding the definition of $\overline{F_1}$ gives us     
    \[\Big|\exx_{u,v \in G} f'_1(2u) \on{e}(\chi(u))f_2(v) f_3(u + v) f_4(2u + v) \on{e}\Big(\alpha(u,u, v)\Big)\Big| \geq (\delta/2)^{O(1)}.\]

    Applying the same maneuver to $f_3$, we may find a character $\chi_2 \in \hat{G}$ and a function $f'_3 : G \to \mathbb{D}$ such that
    \[\Big|\exx_{u , v \in G} f'_1(2u) \on{e}(\chi_1(u))f_2(v) f'_3(2u + 2v) \on{e}(\chi_2(u + v)) f_4(2u + v) \on{e}\Big(\alpha(u,u, v)\Big)\Big| \geq (\delta/2)^{O(1)}.\]
    Write $\chi = \chi_1 + \chi_2$, and, misusing the notation slightly, consume $\on{e}(\chi_2(v))$ into $f_2(v)$, to obtain
    \begin{equation}\label{preintegrationtorsionlemmaeqfreplaced}\Big|\exx_{u , v \in G} \on{e}(\chi(u))f'_1(2u) f_2(v) f'_3(2u + 2v) f_4(2u + v) \on{e}\Big(\alpha(u,u, v)\Big)\Big| \geq (\delta/2)^{O(1)}.\end{equation}

    \noindent\textbf{Removing character $\chi$.} Let us introduce a dummy variable $a$ ranging over $2^{d-1}\cdot G$, which we add to $u$. Note that $2a =0$ and that $\alpha(u + a,u + a, v) = \alpha(u,u,v) + 2\alpha(u,a,v) + \alpha(a,a,v) = \alpha(u,u,v)$. Hence
     \[\Big|\exx_{u , v \in G, a\in 2^{d-1}\cdot G} \on{e}(\chi(u + a))f'_1(2u) f_2(v) f'_3(2u + 2v) f_4(2u + v) \on{e}\Big(\alpha(u,u, v)\Big)\Big| \geq (\delta/2)^{O(1)}.\]
     But then we have $|\ex_{a\in 2^{d-1}\cdot G} \on{e}(\chi(a))| > 0$ which occurs precisely when $\chi = 0$ on $2^{d-1}\cdot G$. The structure of the group $G$ implies that $\chi = 2\chi'$ for some further character $\chi' \in \hat{G}$. Hence, defining $f''_1(x) = f'_1(x) \on{e}(\chi'(x))$,~\eqref{preintegrationtorsionlemmaeqfreplaced} becomes
     \[\Big|\exx_{u , v \in G} f''_1(2u) f_2(v) f'_3(2u + 2v) f_4(2u + v) \on{e}\Big(\alpha(u,u, v)\Big)\Big| \geq (\delta/2)^{O(1)}.\]

     \noindent\textbf{Applying inductive hypothesis.} Let us introduce another dummy variable $b$ ranging over $G$, and add $2b$ to $v$. Thus
     \[\Big|\exx_{u , v , b\in G} f''_1(2u) f_2(v + 2b) f'_3(2u + 2v + 4b) f_4(2u + v + 2b) \on{e}\Big(\alpha(u,u, v + 2b)\Big)\Big| \geq (\delta/2)^{O(1)}.\]
     Averaging over $v$, there exists some choice such that, for $g_1(x) = f''_1(x)$, $g_2(x) = f_2(v + x), g_3(x) = f'_3(2v + x)$ and $g_4(x) = f_4(x + v)$,
     \begin{equation}\label{preintegrationtorsionlemmaeqbeforeIH}\Big|\exx_{u, b\in G} g_1(2u) \on{e}\Big(\alpha(u,u, v)\Big) g_2(2b) g_3(2u + 4b) g_4(2u + 2b) \on{e}\Big(2\alpha(u,u, b)\Big)\Big| \geq (\delta/2)^{O(1)}.\end{equation}

     Finally, we interpret these terms as functions on $G/2^{d-1} \cdot G$, which is isomorphic to $(\mathbb{Z}/2^{d-1}\mathbb{Z})^n$. Namely, define 
     \[\tilde{g}_1(x + 2^{d-1}\cdot G) = g_1(2x) \on{e}\Big(\alpha(x,x, v)\Big),\,\,\tilde{g}_2(x + 2^{d-1}\cdot G) = g_2(2x),\,\,\tilde{g}_3(x + 2^{d-1}\cdot G) = g_3(2x),\,\,\tilde{g}_4(x + 2^{d-1}\cdot G) = g_4(2x).\]
     To see that these are well-defined, we only need to discuss $\tilde{g}_1$. Namely if $x - y \in  2^{d-1}\cdot G$, then $y = x + 2^{d-1}w$ and so $2x = 2y$ and
     \[\alpha(y,y, v) = \alpha(x + 2^{d-1}w,x + 2^{d-1}w, v) = \alpha(x,x,v) + 2^{d}\alpha(x, w, v) + 2^{2d-2}\alpha(w,w,v) = \alpha(x,x,v).\]
     Also, the form $\tilde{\alpha} : (G/2^{d-1} \cdot G) \times (G/2^{d-1} \cdot G) \times (G/2^{d-1} \cdot G) \to \mathbb{T}$ defined by $\tilde{\alpha}(x + 2^{d-1}\cdot G,y + 2^{d-1}\cdot G,z + 2^{d-1}\cdot G) = 2\alpha(x,y,z)$ is well-defined and is a symmetric trilinear form on $G/2^{d-1} \cdot G$. Hence, inequality~\eqref{preintegrationtorsionlemmaeqbeforeIH} becomes
     \begin{align*}\Big|\exx_{u + 2^{d-1} \cdot G, b + 2^{d-1} \cdot G\in G/2^{d-1} \cdot G} \tilde{g}_1(u + 2^{d-1} \cdot G) &\tilde{g}_2(b + 2^{d-1} \cdot G) \tilde{g}_3(u + 2b + 2^{d-1} \cdot G) \tilde{g}_4(u + b + 2^{d-1} \cdot G) \\
     &\on{e}\Big(\tilde{\alpha}(u + 2^{d-1} \cdot G,u + 2^{d-1} \cdot G, b + 2^{d-1} \cdot G)\Big)\Big| \geq (\delta/2)^{O(1)}.\end{align*}

    This expression resembles~\eqref{keycorrelationforsymmetrytorsion} very closely, however, the coefficients of the term in $\tilde{g}_3$ are reversed in comparison to the initial one. We may apply Claim~\ref{symmetryargumnetexpressionreversedclaim} to conclude that there exists a subgroup $\tilde{K} \leq G/2^{d-1} \cdot G$ of density $(\delta/2)^{O_d(1)}$ on which $2^{d-2}\tilde{\alpha}(u + 2^{d-1} \cdot G,u + 2^{d-1} \cdot G, v + 2^{d-1} \cdot G)$ is a symmetric bilinear form. But that means that $2^{d-1}\alpha(u,u,v)$ is symmetric bilinear form on $K = \tilde{K} + 2^{d-1} \cdot G$.

\end{proof}

Finally, we need a structural result for biased trilinear forms.

\begin{theorem}[Structure of biased trilinear forms in $(\mathbb{Z}/2^d\mathbb{Z})^n$]\label{biasedinversetorsion}
    Let $G = (\mathbb{Z}/2^d\mathbb{Z})^n$ and $\phi : G \times G \times G \to \mathbb{T}$ be a trilinear form such that 
    \[\exx_{x,y,z \in G} \on{e}(\phi(x,y,z)) \geq c.\]
    Then there exists a subgroup $H \leq G$ of density $(c/2)^{O_d(1)}$ such that $\phi(x,y,z) = 0$ whenever $x,y,z \in H$.
\end{theorem}

A qualitative version of this result can be deduced from a work of Eberhard~\cite{Eberhard}, but we include a proof as we need a quantitative version.

\begin{proof}
    Consider first $2\phi$. Namely, since $2x + t$ ranges uniformly over $G$ as $x,t$ vary over $G$, we have
    \[c \leq \exx_{x,y,z \in G} \on{e}(\phi(x,y,z)) = \exx_{x,y,z, t \in G} \on{e}(\phi(2x + t,y,z)) = \exx_{y,z,t \in G}  \on{e}(\phi(t,y,z))\Big(\exx_{x \in G} \on{e}(2\phi(x,y,z))\Big).\]
    Applying Cauchy-Schwarz inequality, we see that $c^2 \leq \ex_{x,y,z \in G} \on{e}(2\phi(x,y,z))$. Let us define $\theta : (G / 2^{d-1}\cdot G) \times (G / 2^{d-1}\cdot G) \times (G / 2^{d-1}\cdot G) \to \mathbb{T}$ by $\theta(x + 2^{d-1}\cdot G, y + 2^{d-1}\cdot G, z + 2^{d-1}\cdot G) = 2\phi(x,y,z)$. Since $\phi$ takes values in $\frac{\mathbb{Z}}{2^d}$, it follows that $\theta$ is well-defined. But
    \[\exx_{x+ 2^{d-1}\cdot G,y+ 2^{d-1}\cdot G,z + 2^{d-1}\cdot G\in G/ 2^{d-1}\cdot G}\on{e}\Big(\theta(x + 2^{d-1}\cdot G, y + 2^{d-1}\cdot G, z + 2^{d-1}\cdot G)\Big) = \exx_{x,y,z \in G} \on{e}(2\phi(x,y,z)) \geq c^2.\]

    Since $G/2^{d-1} \cdot G \cong (\mathbb{Z}/2^{d-1}\mathbb{Z})^n$, by induction hypothesis for parameters $d-1$ and $c^2$ in place of $d$ and $c$, we have a subgroup $U$ of density $(c/2)^{O_d(1)}$ inside $G/2^{d-1} \cdot G$ such that $\theta$ vanishes on $U \times U \times U$. Hence $2\phi$ vanishes on $U' \times U' \times U'$, where $U' = U + 2^{d-1} \cdot G$ and thus $\frac{|U'|}{|G|} = \frac{|U|}{|G/2^{d-1} \cdot G|}$.\\ 

    Now define $\psi : (U' /2 \cdot U') \times (U' /2 \cdot U') \times (U' /2 \cdot U') \to \mathbb{F}_p$ by $\psi(x + 2 \cdot U', y + 2 \cdot U', z + 2 \cdot U') = \phi(x,y,z)$, which is well-defined as $2\phi = 0$ on $U' \times U' \times U'$. Then $U' /2 \cdot U' \cong \mathbb{F}_2^m$ for some $m$ and $\psi$ becomes a trilinear form on a finite-dimensional $\mathbb{F}_2$-vector space. Since 
    \[\exx_{x + 2 \cdot U', y + 2 \cdot U', z + 2 \cdot U' \in U'/2 \cdot U'} \on{e}(\psi(x + 2 \cdot U', y + 2 \cdot U', z + 2 \cdot U')) = \exx_{x,y,z \in U'} \on{e}(\phi(x,y,z)) \geq \exx_{x,y,z \in G} \on{e}(\phi(x,y,z)) \geq c,\]
    by Theorem 1 of~\cite{CohenMosh}, we get $W \leq U' /2 \cdot U'$ of density at least $(c/2)^{O(1)}$, such that $\psi$ vanishes on $W \times W \times W$. Finally, $\phi$ vanishes on $(W + 2 \cdot U') \times (W + 2 \cdot U') \times (W + 2 \cdot U')$, completing the proof.
\end{proof}

\begin{proof}[Proof of Theorem~\ref{abeliantwogroupsinverse}]
    Let us slightly modify the notation from the statement. Since $c$ will be used for shifts in discrete multiplicative derivatives, we assume that $\|f\|_{\mathsf{U}^4} \geq \delta$.\\
    
    We prove the theorem by induction on $d$. The case $d = 1$ was proved in~\cite{newU4}. Let some $d \geq 2$ be given, and assume the theorem for smaller values of $d$. Analogous arguments as in previous works, appearing also in the proof of Theorem~\ref{generalU4}, show that we have a Freiman bihomomorphism $\phi : A \to \hat{G}$, defined on a set $A \subseteq G \times G$ of density $\exp(-\log^{O(1)}(2\delta^{-1}))$, such that for each $(a,b) \in A$, $|\widehat{\partial_{a,b}f}(\phi(a,b))| \geq \delta^{O(1)}$. Corollary~\ref{maininversetheoremspectorsion} gives us a global Freiman bihomomorphism $\Phi : G\times G \to \hat{G}$ that coincides with $\phi$ at at least $\exp(-\log^{O(1)}(2\delta^{-1}))|G|^2$ points of $A$. We may write $\Phi(a,b) = B(a,b) + \Theta_1(a) + \Theta_2(b) + \chi$, for a bilinear map $B : G \times G \to \hat{G}$, homomorphisms $\Theta_1, \Theta_2 : G \to \hat{G}$ and element $\chi \in \hat{G}$. Hence
    \begin{align*}\exp(-\log^{O(1)}(2\delta^{-1})) \leq &\exx_{a,b \in G} \Big|\widehat{\partial_{a,b}f}(B(a,b) + \Theta_1(a) + \Theta_2(b) + \chi)\Big|^2\\
    = &\exx_{x,a,b,c \in G}  \partial_{a,b,c} f(x) \on{e}\Big(B(a,b)(c) + \Theta_1(a)(c) + \Theta_2(b)(c) + \chi(c)\Big).\end{align*}

    Let us define $\beta : G \times G \times G \to \mathbb{T}$ as $\beta(a,b,c) = B(a,b)(c)$, which is a trilinear form. By Gowers-Cauchy-Schwarz inequality we have
   \[\exp(-\log^{O(1)}(2\delta^{-1})) \leq \Big|\exx_{x,a,b,c \in G}  \partial_{a,b,c} f(x) \on{e}\Big(\beta(a,b, c) \Big)\Big|.\] 

    By Theorem~\ref{symmetryArgumentBohr}, $|\ex_{a,b,c} \on{e}(\beta(a,b, c) - \beta(b,a, c)| \geq \exp(-\log^{O(1)}(2\delta^{-1}))$, and a similar inequality holds for other permutations of variables.\\
    
    Combining Theorem~\ref{biasedinversetorsion} and Lemma~\ref{directsummandstorsion}, we have a symmetric trilinear map $\alpha : G \times G \times G \to \mathbb{T}$ such that 
    \begin{equation}\label{maincorrelationtorsioninverse}\Big|\exx_{a,b,c,x \in G} \partial_{a,b,c} f(x) \on{e}(\alpha(a,b,c))\Big| \geq \delta_1,\end{equation}
    for some $\delta_1 \geq \exp(-\log^{O(1)}(2\delta^{-1}))$.\\

    In order to be able to integrate $\alpha$, due to Lemma~\ref{trilinearintegration}, we need to replace that trilinear form by another symmetric form $\alpha'$ for which $2^{d-1}\alpha'(u,u,v)$ is a symmetric bilinear form. That is the main goal of the proof.\\

    Let us introduce additional variables $u$ and $v$, and make a change of variables by replacing $a,b,c$ with $a - u, b- u, c- v$ instead. Then

    \[\Big|\exx_{u, v, a,b,c,x \in G} \partial_{a + u,b + u,c + v} f(x) \on{e}(\alpha(a + u,b + u,c + v))\Big| \geq \delta_1.\]

    Expanding out, by averaging, there exist $x,a,b,c$ such that

    \begin{align*}\Big|\exx_{u, v \in G} &f(2u + v + x + a + b + c) \overline{f(2u + x + a + b)}\,\overline{f(u + v + x + a + c)}\,\overline{f(u + v + x + b + c)}\\
    &f(u + x + a) f(u + x + b) f(v + x + c) \overline{f(x)}\\
    &\on{e}(\alpha(u,u,v) + \alpha(u,u,c) + \alpha(u,b,v) + \alpha(u,b,c) +\alpha(a,u,v) +\alpha(a,u,c) + \alpha(a,b,v) +\alpha(a,b,c))\Big| \geq \delta_1.\end{align*}
    
    Defining 
    \begin{align*}
        f_1(t) = &\overline{f(2t + x + a + b)}f(t + x + a) f(t + x + b) \overline{f(x)} \on{e}(\alpha(t,t,c) + \alpha(t,b,c)+ \alpha(a,t,c)+\alpha(a,b,c)),\\
        f_2(t) = &f(t + x + c) \on{e}(\alpha(a,b,t) ),\\
        f_3(t) = &\overline{f(t + x + a + c)}\,\overline{f(t + x + b + c)},\\
        f_4(t) = &f(t + x + a + b + c),
    \end{align*}
    and recalling that $\alpha$ is symmetric, we obtain

    \[\Big|\exx_{u , v} f_1(u) f_2(v) f_3(u + v) f_4(2u + v) \on{e}\Big(\alpha(u,u, v) + \alpha(a + b, u, v)\Big)\Big| \geq \delta_1.\]

    Since $(u,v) \mapsto \alpha(a + b, u, v)$ is a symmetric bilinear map on $G$, by Lemma~\ref{bilinearintegration}, we may find a quadratic polynomial $q$ such that $q(u + v) - q(u) - q(v) + q(0) = \alpha(a + b, u, v)$. Hence, by consuming quadratic phases in suitable $f_i$ and misusing the notation slightly, we get

    \begin{equation}\Big|\exx_{u , v} f_1(u) f_2(v) f_3(u + v) f_4(2u + v) \on{e}\Big(\alpha(u,u, v)\Big)\Big| \geq \delta_1.\label{firstCorrelationComplexityDecTorsion}\end{equation}

    Apply Lemma~\ref{correlationforsymmetrytorsionLemma} to conclude that $2^{d-1}\alpha(u,u, v)$ is symmetric on a subgroup $K$ of density $(\delta_1/2)^{O_d(1)}$. Apply Lemma~\ref{directsummandstorsion} to obtain a decomposition $K' \oplus T = G$, with $|K'| \geq (\delta_1/2)^{O_d(1)}|G|$ and let $\pi_1 : G \to K'$ and $\pi_2 : G \to T$ be the associated projections. Define $\tilde{\alpha} :G \times G \times G \to \mathbb{T}$ by $\tilde{\alpha}(x,y,z) = \alpha(\pi_1(x), \pi_1(y), \pi_1(z))$. Then $\alpha$ is a symmetric trilinear form such that $2^{d-1}\tilde{\alpha}(u,u,v)$ is symmetric bilinear form. By Theorem~\ref{trilinearintegration} there exists a cubic polynomial $q :G \to \mathbb{T}$ such that $\Delta_{a,b,c}q(x) = \tilde{\alpha}(a,b,c)$ for all $x,a,b,c \in G$.\\
    We finally go back to~\eqref{maincorrelationtorsioninverse} and use the decomposition above to conclude that 
    \begin{align*}\delta_1 \leq & \Big|\exx_{a,b,c,x \in G} \partial_{a,b,c} f(x) \on{e}(\tilde{\alpha}(a,b,c) + \alpha(\pi_2(a), \pi_1(b), \pi_1(c)) + \dots + \alpha(\pi_2(a), \pi_2(b), \pi_2(c)))\Big|\\
    = & \Big|\exx_{a,b,c,x \in G} \sum_{t_1, t_2, t_3 \in T} \id(\pi_2(a) = t_1, \pi_2(b) = t_2, \pi_2(c) = t_3) \partial_{a,b,c} f(x) \\
    &\hspace{4cm}\on{e}(\tilde{\alpha}(a,b,c) + \alpha(t_1, \pi_1(b), \pi_1(c)) + \dots + \alpha(t_1, t_2, t_3))\Big|.\end{align*}
    By triangle inequality and averaging, we may find $t_1, t_2, t_3$ and some functions $s_1, s_2, s_3 : G\times G  \to \mathbb{D}$ coming from phases of terms $\alpha(t_1, \pi_1(b), \pi_1(c)), \dots$ such that 
    \[\delta_1 |T|^{-3} \leq \Big|\exx_{a,b,c,x \in G} \id(\pi_2(a) = t_1) \id(\pi_2(b) = t_2) \id(\pi_2(c) = t_3) s_1(a,b) s_2(a,c) s_3(b,c) \partial_{a,b,c} f(x) \on{e}(\tilde{\alpha}(a,b,c))\Big|.\]
    We may use Gowers-Cauchy-Schwarz inequality to eliminate all the terms except $\partial_{a,b,c} f(x) \on{e}(\tilde{\alpha}(a,b,c))$. Use $\Delta_{a,b,c}q(x) = \tilde{\alpha}(a,b,c)$ to consume $\on{e}\Big(\alpha(a,b, c)\Big)$ into $f$ and finish the proof by using the inverse  theorem for the $\mathsf{U}^3$ norm (Theorem~\ref{torsionu3powertwo}).
    
\end{proof}

\appendix

\titleformat{\section}[hang]{\scshape\large\bfseries\filcenter}{Appendix \thesection.}{4pt}{}

\section{Cubic polynomials on $(\mathbb{Z}/2^d\mathbb{Z})^n$}

Recall that a \textit{polynomial of degree at most $m$} between abelian groups $G$ and $H$ is a map $q: G \to H$ such that $\Delta_{a_1}\dots \Delta_{a_{m+1}} q(x) = 0$ for all $a_1, \dots, a_{m+1}, x \in G$, where we iterate the discrete additive derivative operator defined as $\Delta_af(x) = f(x +a) - f(x)$. We shall consider polynomials whose codomain is $\mathbb{T} = \mathbb{R}/\mathbb{Z}$.\\

Let $G$ be a finite abelian group. By the classification of such groups, we know that $G \cong \mathbb{Z}/N_1\mathbb{Z} \oplus \dots \oplus \mathbb{Z}/N_r\mathbb{Z}$. For the time being, we make no assumptions on $N_1, \dots, N_r$. Later, we shall specialize to the case when $N_i$ are powers of 2. Let $\pi: \mathbb{Z}^r \to G$ be the natural projection $\pi(x_1, \dots, x_r) = (x_1 + N_1 \mathbb{Z}, \dots, x_r + N_r \mathbb{Z})$.\\

The next two lemmas give a relationship between polynomials on $G$ and $\mathbb{Z}^r$.

\begin{lemma}
    Let $\phi : G \to \mathbb{T}$ be a polynomial of degree at most $m$. Then $\tilde{\phi} : \mathbb{Z}^r \to \mathbb{T}$ given by $\tilde{\phi} = \phi \circ \pi$ is a polynomial of degree at most $m$. Moreover, $\tilde{\phi}$ is \textit{$N_{[r]}$-periodic}, in the sense that $\tilde{\phi}(x_1 + a_1 N_1, \dots, x_r + a_r N_r) = \tilde{\phi}(x_1, \dots, x_r)$ holds for all $x_1, \dots, x_r,a_1, \dots, a_r \in \mathbb{Z}$.
\end{lemma}

\begin{proof}
    The first claim follows trivially from the fact that $\pi$ is a homomorphism of abelian groups. Namely, let $a_1, \dots, a_{m + 1}, x \in \mathbb{Z}$. Then $\Delta_{a_1, \dots, a_{m+1}} \tilde{\phi}(x) = \Delta_{\pi(a_1), \dots, \pi(a_{m+1})} \phi(x) = 0$. Second part follows from $(a_1N_1, \dots, a_rN_r)$ being in $\ker \pi$. 
\end{proof}

\begin{lemma}
    Suppose that $\tilde{\phi} : \mathbb{Z}^r \to \mathbb{T}$ is an $N_{[r]}$-periodic polynomial of degree at most $m$. Then there exists a unique $\phi : G \to \mathbb{T}$ such that $\tilde{\phi} = \phi \circ \pi$. Moreover, such $\phi$ is a polynomial of degree at most $m$ on $G$.
\end{lemma}

\begin{proof}
    Since $\tilde{\phi}(x_1 + a_1 N_1, \dots, x_r + a_r N_r) = \tilde{\phi}(x_1, \dots, x_r)$ holds for all $x_1, \dots, x_r,a_1, \dots, a_r \in \mathbb{Z}$, we may define $\phi(x_1 + N_1 \mathbb{Z}, \dots, x_r + N_r \mathbb{Z})$ to be $\tilde{\phi}(x_1, \dots, x_r)$, which is independent of the choice of the representative of the coset. Clearly, $\phi$ is unique and it remains to check that it is a polynomial on $G$. To that end, let $x, a_1, \dots, a_{m+1} \in G$. Take any $\tilde{x}, \tilde{a}_1, \dots, \tilde{a}_{m+1} \in \mathbb{Z}$ such that $\pi(\tilde{x}) = x$ and $\pi(\tilde{a}_i) = a_i$. Then
    \[\Delta_{a_1, \dots, a_{m+1}} \phi(x) = \sum_{I \subseteq [m+1]} (-1)^{m + 1 - |I|} \phi\Big(x + \sum_{i \in I} a_i\Big) = \sum_{I \subseteq [m + 1]} (-1)^{m + 1 - |I|} \tilde{\phi}\Big(\tilde{x} + \sum_{i \in I} \tilde{a}_i\Big) = \Delta_{\tilde{a}_1, \dots, \tilde{a}_{m+1}} \tilde{\phi}(\tilde{x}) = 0.\qedhere\]
\end{proof}

The following lemma 

\begin{lemma}\label{zpolynomials}
    If $\phi : \mathbb{Z}^r \to \mathbb{T}$ is a polynomial of degree at most $m$ then there exist coefficients $\lambda_{i_{[\ell]}} \in \mathbb{R}$, indexed by non-decreasing sequences $i_{[\ell]}$ of elements in $[r]$, with $0 \leq \ell \leq m$, such that for all $x\in \mathbb{Z}$
    \[\phi(x) = \sum_{\ssk{0 \leq \ell \leq m\\1\leq i_1 \leq i_2 \leq \dots \leq i_\ell \leq r}} \lambda_{i_{[\ell]}} x_{i_1} \dots x_{i_\ell} + \mathbb{Z}.\]
\end{lemma}

\begin{proof} We use slightly different, but equivalent representation. Define $\binom{x}{k} = \frac{x(x-1)\dots (x-k + 1)}{k!}$, which is a polynomial on $\mathbb{Z}$, taking integer values, and satisfies Pascal's triangle identity $\binom{x + 1}{ k +1} - \binom{x}{k + 1} = \binom{x}{k}$. It suffices to show that 
\[\phi(x) = \sum_{d_1 + \dots + d_r \leq m} \lambda_{d_{[r]}} \binom{x_1}{d_1} \cdots \binom{x_r}{d_r} + \mathbb{Z}.\]
We prove the lemma by induction on $r + m$. The base case is $r = 1, m = 0$, when $\phi$ satisfies $\Delta_a \phi(x) = 0$ for all $a, x \in \mathbb{Z}$, so $\phi$ is constant.\\
Consider $\Delta_{e_r}f(x)$, which is a polynomial of degree at most $m-1$, so by induction hypothesis 
\[\Delta_{e_r}\phi(x) = \sum_{d_1 + \dots + d_r \leq m - 1} \lambda_{d_{[r]}} \binom{x_1}{d_1} \cdots \binom{x_r}{d_r} + \mathbb{Z}.\]
Consider $\psi(x) = \phi(x) - \sum_{d_1 + \dots + d_r \leq m - 1} \lambda_{d_{[r]}} \binom{x_1}{d_1} \cdots \binom{x_r}{d_r + 1} + \mathbb{Z}$. Then $\Delta_{e_r}\psi(x) = 0$, so $\psi$ is independent of the last coordinate, and we may consider $\psi' : \mathbb{Z}^{r-1} \to \mathbb{T}$ given by $\psi'(x_{[r-1]}) = \psi(x_{[r-1]}, 0)$, which is then also a polynomial of degree at most $m$. We are done by induction hypothesis applied to $\psi'$.\end{proof}

We now specialize to $(\mathbb{Z}/2^d\mathbb{Z})^n$ as the ambient group and classify polynomials of degree at most 3.

\begin{proposition}[Classification of cubic polynomials on $(\mathbb{Z}/2^d\mathbb{Z})^r$]
    Suppose that $d \geq 2$. Let $\phi : (\mathbb{Z}/2^d\mathbb{Z})^r \to \mathbb{T}$. Let $|\cdot| : \mathbb{Z}/2^d\mathbb{Z} \to \mathbb{Z}$ be the map sending $x + 2^d\mathbb{Z}$ to the unique integer $a \in \{0, 1, \dots, 2^d - 1\}$ such that $a \equiv x \pmod{2^d}$. Then
    \begin{itemize}
        \item $\phi$ is a degree $\leq 1$ polynomial if and only if it is of the form
        \[\phi(x) = \sum_{i \in [r]} \frac{\lambda_i}{2^{d}}|x_i| + \alpha + \mathbb{Z},\]
        \item $\phi$ is a degree $\leq 2$ polynomial if and only if it is of the form
        \[\phi(x) = \sum_{i \in [r]} \frac{\lambda_{i,i}}{2^{d+1}}|x_i|^2 + \sum_{1 \leq i < j \leq r} \frac{\lambda_{i, j}}{2^{d}}|x_i||x_j| + \sum_{i \in [r]} \frac{\lambda_i}{2^{d}}|x_i| + \alpha + \mathbb{Z},\]
        \item $\phi$ is a degree $\leq 3$ polynomial if and only if it is of the form
        \begin{align*}\phi(x) = &\sum_{1\leq i \leq r} \lambda_{i,i,i}\frac{2|x_i|^3 - 3 |x_i|^2 + 4|x_i|}{3\cdot 2^{d +2}} + \sum_{1 \leq i < j \leq r} \mu_{i,j} \frac{|x_i|^|x_j| + |x_i||x_j|^2 - |x_i||x_j|}{2^{d+1}} + \sum_{1 \leq i \not= j \leq r} \lambda_{i,j}\frac{|x_i|^2|x_j|}{2^d} \\
        &\hspace{2cm} + \sum_{\ssk{1 \leq i,j,k\leq r\\\on{distinct}}} \lambda_{i,j, k}\frac{|x_i||x_j||x_k|}{2^d} + \sum_{i \in [r]} \frac{\lambda_{i,i}}{2^{d+1}}|x_i|^2 + \sum_{1 \leq i < j \leq r} \frac{\lambda_{i, j}}{2^{d}}|x_i||x_j| + \sum_{i \in [r]} \frac{\lambda_i}{2^{d}}|x_i|  +  \alpha + \mathbb{Z},\end{align*}
    \end{itemize}
    where coefficients $\lambda_\bcdot, \mu_\bcdot \in \mathbb{Z}$ and $\alpha\in \mathbb{R}$.
\end{proposition}

\textbf{Remark.} The case $d = 1$, is the case of so-called non-classical polynomials in finite vector space $\mathbb{F}_2^r$, which were classified by Tao and Ziegler. In that case, we get 
\[\sum_{\ssk{j, i_1, \dots, i_r \geq 0\\i_1 + \dots + i_m \leq d - j}} \
\lambda_{j, i_1, \dots, i_r} \frac{|x_1|^{i_1} \cdots |x_r|^{i_r}}{2^{j+1}}.\]

\begin{proof}
    By previous lemmas, $\phi$ is of degree at most $s$ if and only if it comes from a map $f : \mathbb{Z}^r \to \mathbb{T}$, which is $2^d$-periodic in each of its $r$ variables and has the form 
    \[f(x) = \sum_{\ssk{\ell \leq s\\1\leq i_1 \leq \dots \leq i_\ell \leq r}} \alpha_{i_1, \dots, i_\ell} x_{i_1} \dots x_{i_\ell}.\] 
    We classify such maps instead.\\
    
    We shall consider the discrete additive derivatives, so observe that if $g :\mathbb{Z}^r \to \mathbb{T}$ is $2^d$-periodic in its variables, so is $\Delta_a g(x)$, both in $a$ and in $x$, which holds for iterated derivatives as well. Let us now specialize to cubic polynomials. The proof will proceed by taking care of the cubic monomials, then quadratic, etc., thus giving all three parts of the conclusion.\\

    Hence, we are given $f$, $2^d$-periodic in all its variables, of the form

    \[f(x) = \sum_{i} \alpha_{i,i,i} x_i^3 + \sum_{i < j} (\alpha_{i, i, j}x_i^2 x_j + \alpha_{i, j, j} x_i x_j^2) + \sum_{i < j < k} \alpha_{i,j,k}x_ix_jx_k + \sum_{i} \alpha_{i,i} x_i^2 + \sum_{i < j} \alpha_{i,j} x_i x_j + \sum_i \alpha_i x_i + \alpha.\]

    \noindent\textbf{Coefficients $\alpha_{i,j,k}$.} By looking at $\Delta_{ae_i, be_j, ce_k}f(x)$ for $i < j < k$, which equals $\alpha_{i,j,k} abc$, we obtain $\alpha_{i,j,k} \in \frac{\mathbb{Z}}{2^d}$. Conversely, for such a coefficient $\alpha_{i,j,k} x_ix_jx_l$ is $2^d$-periodic, so may assume $\alpha_{i,j,k} = 0$.\\

    \noindent\textbf{Coefficients $\alpha_{i,i,i}$.} Firstly, the derivative $\Delta_{ae_i, be_i, ce_i}f(x) = 6\alpha_{i,i,i} abc$, implies that $6\cdot 2^d\alpha_{i,i,i} \in \mathbb{Z}$. On the other hand, the polynomial $\frac{2x^3 - 3x^2 + 4x}{12\cdot 2^d}$ is $2^d$-periodic: subtracting values at $x + 2^d a$ and $x$, we have, modulo $3\cdot 2^{d+2}$,
    \begin{align*}&2(x+2^da)^3  -2x^3- 3(x + 2^da)^2  + 3 x^2 + 4\cdot2^da\\
    &\hspace{2cm}= 6 \cdot 2^d x^2 a + 6 \cdot 2^{2d} x a^2 + 2^{3d + 1} a^3 - 6\cdot 2^d xa - 3 \cdot 2^{2d} a^2 + 4\cdot2^da\\
    &\hspace{2cm}\equiv 6 \cdot 2^d x(x-1) a + 2^d(2^{2d + 1} a^3 - 3 \cdot 2^{d} a^2 + 4a).\end{align*}
    Since $2|x(x-1)$, and both 3 and 4 divide $2^{2d + 1} a^3 - 3 \cdot 2^{d} a^2 + 4a$ when $d \geq 2$, this is congruent to 0 modulo $12\cdot 2^d$. Hence, we may assume $\alpha_{i,i,i} = 0$.\\
    
   \noindent\textbf{Coefficients $\alpha_{i,i,j}$ and $\alpha_{i,j,j}$.}  Similarly, for derivatives $\Delta_{ae_i, be_i, ce_j}f(x), \Delta_{ae_i, be_j, ce_j}f(x)$ we get $2^{d+1}\alpha_{i,i,j}, $ $2^{d+1}\alpha_{i,j,j} \in \mathbb{Z}$. Now consider $\Delta_{2^dae_i, be_i}f(ye_i + ze_j) = 0$. We get
   \[\alpha_{i,i,j}2^{d+1}abz + \alpha_{i,i}2^{d+1}ab = 0.\]
   Take $z = 0, a = b = 1$ so $2^{d+1}\alpha_{i,i}$. Since $\frac{y^2}{2^{d+1}}$ is $2^d$-periodic, we may assume that $\alpha_{i,i} = 0$. Similarly, $\alpha_{j,j} = 0$.\\
   \indent Next, consider $\Delta_{2^dae_i, be_j}f(ye_i + ze_j) = 0$. We get
   \[\alpha_{i,i,j}2^{d+1}aby + \alpha_{i,j,j}2^{d+1}abz +  \alpha_{i,j,j}2^{d}ab^2 + \alpha_{i,j}2^{d}ab = 0.\]
   It follows that $2^d  (\alpha_{i,j} + \alpha_{i,j,j}) \in \mathbb{Z}$. Hence, either $2^d \alpha_{i,i,j}, 2^d \alpha_{i,j,j}, 2^d\alpha_{i,j} \in \mathbb{Z}$ or $2^d \alpha_{i,i,j}, 2^d \alpha_{i,j,j}, 2^d\alpha_{i,j} \in \frac{1}{2} + \mathbb{Z}$. Adding $\frac{x_i^2x_j + x_ix_j^2 - x_ix_j}{2^{d+1}}$, reduces us to the former case, and the coefficients can again be made to vanish. Let us just check that $\frac{x^2y + xy^2 - xy}{2^{d+1}}$ is $2^d$-periodic. It is symmetric in $x$ and $y$, so it suffices to consider $x$ only. We have
   \[(x + 2^d)^2 y + (x + 2^d)y^2 - (x + 2^d)y - (x^2y + xy^2 - xy) = 2^{d+1}xy + 2^{2d}y + 2^d y^2 - 2^d y \equiv 2^dy(y-1) \equiv 0 \pmod{2^{d+1}}.\]

   The argument above also shows that coefficients of quadratic monomials behave as described in the statement. For the linear terms, the claim is obvious.
\end{proof}

Using the classification of cubic polynomials, we may give criteria for checking whether multilinear forms, which are maps $\phi : (\mathbb{Z}/2^d\mathbb{Z})^n \tdt (\mathbb{Z}/2^d\mathbb{Z})^n \to \mathbb{T}$, being homomorphisms in each variable separately, are integrable or not. Before that, we record a simple lemma considering such maps.

\begin{lemma}
    Let $\phi : (\mathbb{Z}/2^d\mathbb{Z})^n \tdt (\mathbb{Z}/2^d\mathbb{Z})^n \to \mathbb{T}$ be a multilinear form in $k$ variables. Then there exist unique coefficients $\lambda_{i_1, \dots, i_k} \in \mathbb{Z}/2^d\mathbb{Z}$ such that
    \[\phi(x_1, \dots, x_k) = \sum_{i_1, \dots, i_k \in [n]} \frac{\lambda_{i_1, \dots, i_k} x_{1,i_1}\cdots x_{k, i_k}}{2^d}\]
    holds for all $x_1, \dots, x_k \in (\mathbb{Z}/2^d\mathbb{Z})^n$.\\
    \indent Additionally, if $q$ is a polynomial of degree at most $k$, then $\Delta_{a_1, \dots, a_k}q(x)$ is independent of $x$ and defines a multilinear form $\phi(a_1, \dots, a_k)$, whose coefficients are given by $\lambda_{i_1, \dots, i_k} \Delta_{e_1, \dots, e_k}f(x)$.
\end{lemma}

\begin{proof}
    \noindent\textbf{Existence of coefficients.} Observe firstly that all values taken by $\phi$ lie inside $\frac{\mathbb{Z}}{2^d} + \mathbb{Z}$. Namely, fixing $x_1, \dots, x_{k-1}$, we have that $y \mapsto \phi(x_1, \dots, x_{k-1}, y)$ is a homomorphism from $(\mathbb{Z}/2^d\mathbb{Z})^n$ to $\mathbb{T}$, which only takes such values. Using homomorphism property in each variable, we have 
    \[\phi(x_1, \dots, x_k) = \sum_{i_1}x_{1, i_1}\phi(e_1, x_2, \dots, x_k) = \dots = \sum_{i_1, \dots, i_k} \phi(e_{i_1}, \dots, e_{i_k}) x_{1, i_1} \dots x_{k, i_k}.\]
    
    \noindent\textbf{Uniqueness of coefficients.} By subtracting two possible representations of a given multilinear form in the above form, we need to show that, if $\psi(x_1, \dots, x_k) = \sum_{i_1, \dots, i_k \in [n]} \frac{\lambda_{i_1, \dots, i_k} x_{1,i_1}\cdots x_{k, i_k}}{2^d}$ always vanishes, then all coefficients vanish as well. But this follows from plugging in elements of the standard basis, as $\lambda_{i_1, \dots, i_k} = \psi(e_{i_1}, \dots e_{i_k}) = 0$.\\

    \noindent\textbf{Relationship with polynomials.} Since $q$ is a polynomial of degree at most $k$, that means that $\Delta_{a_1, \dots, a_k, y - x}q(x) = 0$ for all $a_1, \dots, a_k ,x ,y$ so $\Delta_{a_1, \dots, a_k}q(x) = \Delta_{a_1, \dots, a_k}q(y)$, thus $\phi(a_1, \dots, a_k)$ is well-defined. To see that it is multilinear, we show that it is a homomorphism in the last variable, the same argument works for other variables. Namely, 
    \begin{align*}\phi(a_1, \dots, a_k + b_k) =  &\Delta_{a_1, \dots, a_k + b_k}q(x) = \Delta_{a_1, \dots, a_{k-1}}q(x + a_k + b_k) - \Delta_{a_1, \dots, a_{k-1}}q(x) \\
    =&\Delta_{a_1, \dots, a_{k-1}}q(x + a_k + b_k) - \Delta_{a_1, \dots, a_{k-1}}q(x + a_k) + \Delta_{a_1, \dots, a_{k-1}}q(x + a_k) - \Delta_{a_1, \dots, a_{k-1}}q(x)\\
    =&\Delta_{a_1, \dots, a_{k-1}, b_k}q(x + a_k) + \Delta_{a_1, \dots, a_{k-1}, a_k}q(x)\\
    =&\phi(a_1, \dots, a_k) + \phi(a_1, \dots, b_k).
    \end{align*}
    The claim about coefficients stems from uniqueness and definition of $\phi$.
\end{proof}

\begin{theorem}[Integration of bilinear maps]\label{bilinearintegration}
    Let $\phi : (\mathbb{Z}/2^d\mathbb{Z})^n \times (\mathbb{Z}/2^d\mathbb{Z})^n \to \mathbb{T}$ be a bilinear map. Then there exists a quadratic polynomial $q : (\mathbb{Z}/2^d\mathbb{Z})^n \to \mathbb{T}$ such that $\Delta_{a,b}q(x) = \phi(a,b)$ for all $x,a,b \in (\mathbb{Z}/2^d\mathbb{Z})^n$ if and only if $\phi(a,b)$ is a symmetric bilinear map.
\end{theorem}

\begin{proof}
    The case $d = 1$ was studied in~\cite{TaoZiegler, Tidor}, so we may assume $d \geq 2$. Clearly, $\Delta_{a,b}q(x)$ is a symmetric bilinear map for any quadratic $q$. On the other hand, given a symmetric bilinear map $\phi(a,b)$, we may write it as $\phi(a,b) = \sum_{i, j} \frac{\lambda_{i,j} a_i b_j}{2^d}$. From the uniqueness of coefficents and symmetry of $\phi$, we have $\lambda_{i,j} = \lambda_{j, i}$. But, for $i \not= j$, we have $\Delta_{a, b} q(x) = \frac{\lambda_{ij}(a_ib_j + a_j b_i)}{2^d}$ for $q(x) = \frac{\lambda_{i,j} x_ix_j}{2^d}$, and, when $i = j$, we have $\Delta_{a, b} q(x) = \frac{\lambda_{i,i}a_ib_i}{2^d}$ for $q(x) = \frac{|\lambda_{i,i}||x_i|^2}{2^{d+1}}$.
\end{proof}

\begin{theorem}[Integration of trilinear maps]\label{trilinearintegration}
    Let $\phi : (\mathbb{Z}/2^d\mathbb{Z})^n \times (\mathbb{Z}/2^d\mathbb{Z})^n \times (\mathbb{Z}/2^d\mathbb{Z})^n \to \mathbb{T}$ be a trilinear map (meaning a group homomorphism in each variable separately). Then there exists a cubic polynomial $q : (\mathbb{Z}/2^d\mathbb{Z})^n \to \mathbb{T}$ such that $\Delta_{a,b,c}q(x) = \phi(a,b,c)$ for all $x,a,b,c \in (\mathbb{Z}/2^d\mathbb{Z})^n$ if and only if $\phi(a,b,c)$ is a symmetric trilinear map and $2^{d-1}\phi(a,a,b)$ is a symmetric bilinear map.
\end{theorem}

\begin{proof}
    The case $d = 1$ was studied in~\cite{TaoZiegler, Tidor}, so we may assume $d \geq 2$. Suppose first that $\Delta_{a,b,c}q(x) = \phi(a,b,c)$ holds for a cubic polynomial $q$. Then $\phi$ is a symmetric trilinear form. Furthermore, $2^{d-1}\phi(a,a,b)$ is a bilinear form. Indeed,
    \[2^{d-1}\phi(a + a',a + a',b) = 2^{d-1}\phi(a,a,b) + 2^d \phi(a, a', b) + 2^{d-1}\phi(a',a',b) = 2^{d-1}\phi(a,a,b) + 2^{d-1}\phi(a',a',b),\]
    since $\phi$ takes values in $\frac{\mathbb{Z}}{2^d} + \mathbb{Z}$. Finally, to see symmetry, we need an additional claim.
    
    \begin{claim}
        For a given $k$, let $U_k = 2^{2k - 1} + 2^{k-1}, L_k = 2^{2k - 1} - 2^{k-1}$. Then there exist integers $\lambda_i, \mu_i \in \mathbb{Z}$, $i \in [U_k]$, and $\lambda'_i, \mu'_i \in \mathbb{Z}$, for $i \in [L_k]$, such that for all maps $F : H_1 \to H_2$ between two abelian groups $H_1$ and $H_2$ ($F$ is not necessarily a homomorphism), we have two identities 
        \begin{align*}&\sum_{i \in [U_k]} \Delta_{a, b, a+b} F(x + \lambda_i a + \mu_i b) = -F(x) + F(x + a)  - F(x + 2^ka + (2^k - 1)b) \\
        &\hspace{4cm}+ F(x  + 2^ka + 2^k b) - F(x + a + 2^k b) + F(x + (2^k - 1)b)\end{align*}
        and
        \begin{align*}&\sum_{i \in [L_k]} \Delta_{a, b, a+b} F(x + \lambda'_i a + \mu'_i b) = -F(x) + F(x + (2^k - 1)a)  - F(x + 2^ka + b) \\
        &\hspace{4cm}+ F(x  + 2^ka + 2^k b) - F(x + (2^k - 1)a + 2^k b) + F(x + b)\end{align*}
        for all $x, a, b \in H_1$.
    \end{claim}

    \begin{proof}
        The claim follows by induction. For the base case $k = 1$, we take $\lambda_1 = \mu_1 = 0$, $\lambda_2 = 0, \mu_2 = 1, \lambda_3 = \mu_3 = 1$, to get
        
        and for the second identity, we take $\lambda'_1 = \mu'_1 = 0$, and expand $\Delta_{a,b,a+b}F(x)$.\\
        Suppose now that the claim holds for some $k \geq 1$, let $\lambda_i, \mu_i$, $i \in [L_k]$ be the coefficients.
    \end{proof}
    
    The claim above for the second identity gives
    \[ 2^{d-1}\phi(a,b,a+b) = 2^{d-1}\Delta_{a, b, a + b}q(x) = ... = 0.\]
    Since $2^{d-1}\phi(a,a,b)$ takes values in $\frac{1}{2} + \mathbb{Z}$, we get the desired property.\\

    Let us now assume that $\phi$ is a symmetric trilinear form such that $2^{d-1}\phi(a,a,b)$ is symmetric bilinear form. Expressing $\phi$ using coefficients, we get
    \[\phi(a,b,c) = \sum_{i,j,k \in [n]} \frac{\lambda_{i,j,k} a_ib_jc_k}{2^d}.\]
    Uniqueness of coefficients and symmetry of $\phi$ imply that $\lambda_{i,j,k}$ is invariant under permutations of coefficients. Furthermore,
    \begin{align*}2^{d-1}\phi(a,a,b) = \sum_{i \in [n]} \frac{\lambda_{i,i,i} a_i^2b_i}{2} + & \sum_{1\leq i < j \leq n} \frac{\lambda_{i,i,j} (a_i^2b_j +2a_ia_jb_i ) + \lambda_{i,j,j}(a_j^2b_i + 2a_ia_jb_j)}{2}\\
    + &\sum_{i < j < k \leq n} \frac{2\lambda_{i,j,k}(a_ia_jb_k + a_i a_k b_j + a_j a_k b_i)}{2}.\end{align*}
    The first sum is symmetric as modulo 2 we have $a_i^2b_i$ congruent to $a_ib_i$ and thus to $a_i b_i^2$. The third sum vanishes. Hence, symmetry of $2^{d-1}\phi(a,a,b)$ and uniqueness of coefficients imply that $\lambda_{i,i,j} \equiv \lambda_{j,j,i} \pmod{2}$ whenever $i \not=j$.\\

\end{proof}

\section{Some lower order inverse results}

In this appendix, we gather a couple of lower order inverse theorems, one concerning basic directional uniformity norms and the other concerning $\mathsf{U}^3$ norm in bounded torsion groups. Before that, we observe that characters of subgroups can be extended to characters on the full group.

\begin{lemma}\label{charextnlemma}
    Let $H \leq G$ be finite groups and let $\chi : H \to \mathbb{T}$ be a homomorphism. Then $\chi$ extends to a homomorphism on $G$.
\end{lemma}

\begin{proof}
    Let $x \in G \setminus H$. It suffices to extend $\chi$ to $H' = H + \langle x \rangle$. Let $k$ be smallest positive integer such that $kx \in H$. Define $\chi' : H' \to \mathbb{T}$ by $\chi'(y) = \frac{\ell}{k} + \chi(y - \ell x)$, where $\ell$ is such that $y - \ell x \in H$.\\
    \indent To see that $\chi'$ is well-defined, note first that for each $y \in H'$ there exists $\ell$ is such that $y - \ell x \in H$. Furthermore, if $\ell'$ is another such an integer, then $(\ell' - \ell) x \in H$, so $k | \ell' - \ell$ by the choice of $k$. Hence $\chi'$ is well-defined.\\
    \indent Let now $y, z\in H'$, and let $y - \ell x, z - mx \in H$. Then $(y + z) - (\ell + m)x \in H$ and $\chi'(y+z) = \chi'(y)+\chi'(z)$.
\end{proof}

\begin{theorem}\label{u2inversesbgp}
Let $H \leq G$ be a subgroup and suppose that $f : G \to \mathbb{D}$ satisfies $\|f\|_{\mathsf{U}(H, G)} \geq c$. Then there exist a character $\chi : G \to \mathbb{T}$ and a function $h : G/H \to \mathbb{D}$ such that 
\[\Big|\exx_{x \in G} f(x) \overline{h(x + H)} \on{e}(\chi(x))\Big| \geq (c/2)^{O(1)}.\]
\end{theorem}

\begin{proof}
    Let $T \subset G$ be an arbitrary set of representatives of cosets of $H$, i.e. $|T| = |G| / |H|$ and $G = T + H$. For each $t \in T$, we define a map $f_t : H \to \mathbb{D}$ by $f_t(x) = f(t + x)$. Let $S_t$ be the large spectrum of $f_{t}$, namely the set of all $\chi \in \hat{H}$ such that $|\ex_{x \in H} f_t(x) \on{e}(-\chi(x))| \geq c^4/2$. By Plancharel's theorem, we have $|S_t| \leq 4 c^{-8}$. We claim that $S_t \cap S_{t'} \not= \emptyset$ for at least $c^4 |T|^2$ pairs $(t, t') \in T^2$.\\

    Expanding the norm in the assumptions, we get $c^4 \leq \exx_{x, a \in G, b \in H} f(x) \overline{f(x + a)}\,\overline{f(x + b)} f(x + a + b)$. Note that $x$ and $x + b$ belong to the same coset of $H$, and so do $x + a$ and $x + a + b$. Moreover, all pairs of cosets are uniformly covered in this fashion. Hence 
    \begin{align*}c^4 \leq &\exx_{t, t' \in T} \Big(\exx_{x ,a,b \in H} f_{t}(x) \overline{ f_{t'}(x + a) }\,\overline{ f_{t}(x + b) }  f_{t'}(x + a  +b)\Big)\\
    = &\exx_{t, t' \in T} \sum_{\chi \in \hat{H}} |\widehat{f_t}(\chi)|^2 |\widehat{f_{t'}}(\chi)|^2 \leq \exx_{t, t' \in T} \Big(\max_{\chi \in \hat{H}} |\widehat{f_t}(\chi)||\widehat{f_{t'}}(\chi)|\Big)\Big(\sum_{\chi \in \hat{H}} |\widehat{f_t}(\chi)| |\widehat{f_{t'}}(\chi)|\Big).\end{align*}

    On the other hand, for any $t, t' \in T$, by Cauchy-Schwarz inequality and Plancharel's theorem
    \[\sum_{\chi \in \hat{H}} |\widehat{f_t}(\chi)| |\widehat{f_{t'}}(\chi)| \leq \sqrt{\sum_{\chi \in \hat{H}} |\widehat{f_t}(\chi)|^2}\sqrt{\sum_{\chi \in \hat{H}} |\widehat{f_{t'}}(\chi)|^2} = \|f_t\|_{L^2}\|f_{t'}\|_{L^2} \leq 1.\]

    Hence, we have at least $\frac{c^4}{2}|T|^2$ pairs $(t,t') \in T$ such that $\max_{\chi \in \hat{H}} |\widehat{f_t}(\chi)||\widehat{f_{t'}}(\chi)| \geq c^4/2$, so $S_{t} \cap S_{t'} \not= \emptyset$.\\ 
    
    By averaging, there exist $\chi \in \hat{H}$ that belongs to at least $\frac{c^{12}}{8}|T|$ sets $S_{t}$. Let $T'$ be the set of such $t$. By Lemma~\ref{charextnlemma}, we get a character $\tilde{\chi} : G \to \mathbb{T}$ extending $\chi$. Let $\tau : G \to T$ be the map given by unique $\tau(x) \in T$ such that $x- \tau(x) \in H$. Thus, $\tau(x) = t$ for all $x \in t + H$, when $t \in T$. Hence, 
    \begin{align*}2^{-5}c^{20} \leq &\exx_{t \in T} \id_{T'}(t) \Big|\exx_{x \in H} f_t(x) \on{e}(-\tilde{\chi}(x))\Big|^2 \leq \exx_{t \in T} \Big|\exx_{x \in H} f_t(x) \on{e}(-\tilde{\chi}(x))\Big|^2 \\
    = &\exx_{t \in T} \overline{\widehat{f_t}(\chi)} \exx_{x \in H} f(x + t) \on{e}(-\tilde{\chi}(x)) = \exx_{x \in G} \overline{\widehat{f_{\tau(x)}}(\chi)} \on{e}(\tilde{\chi(\tau(x))}) f(x) \on{e}(-\tilde{\chi}(x)).\end{align*}
    We define $h : G/ H \to \mathbb{D}$ by $h( x + H) = \widehat{f_{\tau(x)}}(\chi)  \on{e}(-\tilde{\chi(\tau(x))})$, which is well-defined as $\tau$ is constant on $x + H$, completing the proof.
\end{proof}

We include a quick deduction of the $\mathsf{U}^3$ inverse theorem for the group $(\mathbb{Z}/2^d\mathbb{Z})^n$.

\begin{theorem}\label{torsionu3powertwo}
    Let $G = (\mathbb{Z}/2^d\mathbb{Z})^n$. Let $f : G \to \mathbb{D}$ be a function such that $\|f\|_{\mathsf{U}^3} \geq c$. Then there exists a quadratic polynomial $q : G \to \mathbb{T}$ such that $|\ex_{x \in G} f(x) \on{e}(q(x))| \geq \exp(-O_d(\log^{O(1)}(2c^{-1})))$.
\end{theorem}

\begin{proof}
    Standard steps show that there exists a map $\Phi : A \to \hat{G}$, defined on a set $A$ of density $(c/2)^{O(1)}$ such that $\Phi$ respects $(c/2)^{O(1)}|G|^3$ additive quadruples and $|\widehat{\partial_af}(\Phi(a))| \geq (c/2)^{O(1)}$ for each $a \in A$. Combining the structure theorem for approximate homomorphisms (Theorem~\ref{approxFreimanHom}), the fact that every dense, bounded rank coset progression in $(\mathbb{Z}/2^d\mathbb{Z})^n$ contains a translate of a dense subgroup and Lemma~\ref{directsummandstorsion}, we get a homomorphism $\Psi : G \to \hat{G}$ and an element $\chi \in \hat{G}$ such that 
    \[\exp(-O_d(\log^{O(1)}(2c^{-1})))\leq \exx_{a \in G} |\widehat{\partial_af}(\Psi(a) + \chi)|^2 = \exx_{x,a,b \in G} \partial_{a,b}f(x) \on{e}(\Psi(a)(b) + \chi(b)).\]
    Defining $\beta  : G\times G \to \mathbb{T}$ by $\beta(a,b) = \Psi(a)(b)$ gives a bilinear map. Using the Gowers-Cauchy-Schwarz inequality, we get
    \[\exp(-O_d(\log^{O(1)}(2c^{-1})))\leq \Big|\exx_{x,a,b \in G} \partial_{a,b}f(x) \on{e}(\beta(a,b))\Big|.\]
    Theorem~\ref{symmetryArgumentBohr} and Lemma~\ref{directsummandstorsion} allow us to assume that $\beta$ is symmetric. Finally, we may use Theorem~\ref{bilinearintegration} to reduce the proof to the case of the inverse $\mathsf{U}^2$ norm.
\end{proof}

\section{Comparison with cyclic groups}\label{cyclicgroupsappendix}

In the case of cyclic groups, in this appendix, we give a more explicit description of almost trilinear maps on Bohr sets.\\
\indent When $x \in \mathbb{R}$, we write $\{x\}$ for the unique real in $(-1/2, 1/2]$ such that $x - \{x\}$ is an integer. We also write $\lr{x} = x - \{x\}$. Thus $\lr{x} \in \mathbb{Z}$ for all $x$.\\

\begin{proposition}[Corollary 10.5 in~\cite{GreenTaoU3}]\label{goodprogzn}
    Let $N$ be a prime. Let $r_1, \dots, r_d \in \mathbb{Z}/N\mathbb{Z}$ and $\rho > 0$. Then there exist a proper symmetric progression $P$ with a basis $v_1, \dots, v_d$ such that 
    \[B(r_1, \dots, r_d, d^{-2d}\rho) \subseteq P \subseteq B(r_1, \dots, r_d, \rho)\]
    and the vectors $\Big(\Big\{\frac{r_1 v_i}{N}\Big\}, \dots, \Big\{\frac{r_d v_i}{N}\Big\} \Big)$ are independent in $\mathbb{R}^d$ for $i \in [d]$.
\end{proposition}

Let $\iota : \mathbb{Z}/N\mathbb{Z} \to [0, N-1] \subseteq \mathbb{Z}$ be such that $x = \iota(x) + N \mathbb{Z}$. Next, we describe integer-valued Freiman-linear maps on Bohr sets inside cyclic groups.

\begin{proposition}\label{zhommcyclicbohr}
    Let $N$ be a prime. Let $r_1, \dots, r_d \in \mathbb{Z}/N\mathbb{Z}$ and $\rho \in (0, 1/4)$. Suppose that $\phi : B(r_1, \dots, r_d; \rho) \to \mathbb{Z}$ is a Freiman-linear map. Then there exist reals $a_1, \dots, a_d$ such that 
    \[\phi(x) = \sum_{i \in [d]} a_i \blr{\frac{\iota(r_i) \iota(x)}{N}}\]
    for all $x \in B(r_1, \dots, r_d; d^{-2d}\rho)$. Moreover, for $i \in [d]$, if $\lambda_1 r_1 + \dots + \lambda_d r_d = 0$ in $\mathbb{Z}/N\mathbb{Z}$ for $|\lambda_1|, \dots, |\lambda_d| \leq (2\rho^{-1}d^d)^{\blc d}$ implies that $\lambda_i = 0$, then $a_i \in \mathbb{Z}$ and is unique.
\end{proposition}

\begin{proof}
    Let $P = [-L_1, L_1] \cdot v_1 + \dots + [-L_d, L_d] \cdot v_d$ be the progression obtained in Proposition~\ref{goodprogzn}. Let $A = (\alpha_{ij})_{i, j \in [d]}$ be the inverse for the real-valued matrix $(\{r_i v_j/N\})_{i,j \in [d]}$. Thus, for all $i, k \in [d]$ we have
    \[\sum_{j \in [d]} \alpha_{i j} \{r_j v_k/N\} = \id(i = k).\]

    For all $x \in P$, define $\lambda_i(x) = \sum_{j \in [d]} \alpha_{ij} \{r_j x/N\}$. Observe that $\lambda_i(v_k) = \id(i = k)$ and that $\lambda_i$ is Freiman-linear on $B(r_1, \dots, r_d; \rho)$. Hence, for each $x \in P$, we have
    \[x = \sum_{i \in [d]} \lambda_i(x) v_i.\]
    Furthermore, since $\phi$ is Freiman-linear on $B(r_1, \dots, r_d; \rho)$, we get 
    \[\phi(x) = \sum_{i \in [d]} \lambda_i(x) \phi(v_i).\]
    Hence, there are reals $\beta_1, \dots, \beta_d$ such that 
    \begin{equation}
        \phi(x) = \sum_{i \in [d]} \beta_i \{r_i x/N\}\label{phibetaexpansion}
    \end{equation}
    holds for all $x \in B(r_1, \dots, r_d; d^{-2d} \rho)$.\\

    Next, we show that $\beta_i$ is an integer for the described index $i$ in the statement.\\
    \indent Set $k = \lceil d^{2d}  \rho^{-1} \rceil$.  We first find $x \in B(r_1, \dots, r_d; d^{-2d} \rho)$ such that $|\{r_j x / N\}| \leq d^{-4d-10}\rho^2$ for $j \not= i$, but $\{r_ix /N\} \in [\frac{1}{k}, \frac{1}{k} + d^{-4d-10}\rho^2]$. By assumptions and Proposition~\ref{approximatelatticeimage} applied to the value $\Big(0,0, \dots, 0,  \frac{1}{k} + \frac{1}{2}d^{-4d-10}\rho^2\Big)$ and approximation parameter $\frac{1}{2}d^{-4d-10}\rho^2$, we obtain such $x$.\\

    Once we have such $x$, and observe that $k x \in B(r_1, \dots, r_d; d^{-2d} \rho)$ and $ k\{r_j x/N\} =  \{kr_j x/N\}$ for $j \not= i$, but $\{kr_ix/N\} = k  \{r_i x/N\} - 1$. Thus, from~\eqref{phibetaexpansion} we have
    \[\beta_i = k \phi(x) - \phi(kx) \in \mathbb{Z}.\]

    It remains to put $\phi$ in the desired form. Note that $\{r_i x / N\} = \{\iota(r_i) \iota(x) / N\} = \iota(r_i) \iota(x) / N - \lr{\iota(r_i) \iota(x) / N}$.

    Hence, we have 
    \[\phi(x) = \frac{\Big(\sum_{i \in [d]} \beta_i \iota(r_i)\Big) \iota(x) }{N} - \sum_{i \in [d]} \beta_i\lr{\iota(r_i) \iota(x) /N},\]
    so it follows that 
    \[\frac{\Big(\sum_{i \in [d]} \beta_i \iota(r_i)\Big) \iota(x) }{N}  \in \mathbb{Z}\]
    for all $x \in B(r_1, \dots, r_d, d^{-2d}\rho)$. The only way this can happen is if $\sum_{i \in [d]} \beta_i \iota(r_i) \in N \mathbb{Z}$, so it is identically zero and the proof is complete.
    
\end{proof}

We may now relate $\varepsilon$-trilinear forms to generalized polynomials.

\begin{theorem}\label{almosttrilineartogenpolys}
    Let $B = B(\Gamma, \rho) \subseteq \mathbb{Z}/N\mathbb{Z}$ be a Bohr set of codimension $d$ and let $\phi : B \times B \times B \to \mathbb{T}$ be a $\varepsilon$-trilinear form. Then there exists a generalized polynomial $g : [0, N-1]^3 \to \mathbb{T}$, sum of at most $ (2d\log(\rho^{-1}\varepsilon^{-1}))^{O(1)}$ terms, each of the form 
    \[\{\{\alpha x\}\beta y\} \gamma z,\,\,\{\alpha x y\} \gamma z,\,\,,\{\alpha x\}\beta yz,\,\,\alpha xyz,\,\,\{\alpha x\}\{\beta y\} \gamma z, \{\beta y\} \alpha xz\]
    for some reals $\alpha, \beta, \gamma$, bilinear Bohr variety $W$ and a Bohr set $B'$ of codimension at most $(2d\log(\rho^{-1}\varepsilon^{-1}))^{O(1)}$ and radius at least $\exp(-(2d\log(\rho^{-1}\varepsilon^{-1}))^{O(1)})$ such that $\|\phi(x, y, z) - g(\iota(x),\iota(y),\iota(z)) \|_{\mathbb{T}} \leq O(d^d \sqrt{\varepsilon})$ for all $(x,y) \in S,z \in B'$.
\end{theorem}

\begin{proof}
    As in the proof of Theorem~\ref{equidistributionalmosttrilinearmaps}, there exists $\rho' \geq \Omega(\varepsilon d^{-O(d)} \rho)$ such that $B(\Gamma, \rho')$ is regular and there is an $E$-bilinear map $\chi : B(\Gamma, \rho') \times B(\Gamma,\rho') \to \mathbb{Z}/N\mathbb{Z}$ such that $\|\phi(x,y,z) - \frac{\chi(x,y)z}{N}\|_{\mathbb{T}} \leq \varepsilon'$ holds for all $x,y,z \in B(\Gamma, \rho')$, where $\varepsilon' = O(d^{O(d)}\sqrt{\varepsilon})$ and $E = \langle \Gamma \rangle_R$, $R \leq (2\varepsilon^{-1} \rho^{-1} d^{d})^{O(d)}$.\\
    In particular, $\chi$ is a Freiman-bihomomorphism, so by the proof of the structure theorem for such maps, up to Proposition~\ref{secondrobustBRstep}, using $E$-bilinearity and the bilinear Bogolyubov argument (Theorem~\ref{bogruzsabilinearintro}), we get a bilinear Bohr set $S \subseteq B(\Gamma, \rho') \times B(\Gamma,\rho')$ of codimension $r_1 \leq (2d\log(\rho^{-1}\varepsilon^{-1}))^{O(1)}$ and radius at least $\rho_1 \geq \exp(-(2d\log(\rho^{-1}\varepsilon^{-1}))^{O(1)})$, and a Freiman-bilinear map $\psi : S \to \mathbb{Z}/N\mathbb{Z}$ such that $\psi(x,y) - \chi(x,y) \in O(1)E$ for all $(x,y) \in S$. It remains to express $\psi$ using generalized polynomials.\\

    Thus, we get a symmetric proper coset progression $C$, Freiman-linear maps $\Theta_1, \dots, \Theta_{r_1} : C \to \mathbb{Z}/N\mathbb{Z}$, frequency set $\Gamma_1$ containing $\Gamma$ and of size at most $r_1$ such that $S = \bigcup_{x \in C}\{x\} \times B(\Gamma_1, \Theta_1(x), \dots, \Theta_{r_1}(x))$. Let $\eta > 0$ be a parameter to be chosen later. By applying the algebraic regularity lemma (Theorem~\ref{algreglemmaintro}), we may assume that the given bilinear Bohr variety is quasirandom with parameter $\eta$, at the cost of weakening passing to a symmetric proper coset progression $C' \subseteq C$ of rank $r_2 \leq (2d\log(\rho^{-1}\varepsilon^{-1}\eta^{-1}))^{O(1)}$, while $\rho_1$ and $r_1$ are essentially unchanged. We misuse the notation and write $C$ instead of $C'$ and use the same notation $\rho_1$ and $r_1$. Let $\Gamma_1 = \{\gamma_1, \dots, \gamma_{r_1}\} \subseteq \mathbb{Z}/N\mathbb{Z}$.\\

    Let $X \subseteq C$ be the set of all $x \in C$ such that
    \[\lambda_1 \gamma_1 + \dots + \lambda_{r_1}\gamma_{r_1} + \lambda_1'\Theta_1(x) + \dots + \lambda'_{r_1}\Theta_{r_1}(x) = 0\]
    for $|\lambda_1|, \dots, |\lambda'_{r_1}| \leq (2\rho_1 r_1^{r_1})^{10Cr_1}$, where $C$ is the implicit constant in Proposition~\ref{zhommcyclicbohr}, implies that $\lambda'_1, \dots, \lambda'_{r_1} = 0$. By quasirandomness of $S$ (tweaking $\eta$ slightly), we have that $|X| \geq (1 - \eta) |C|$.\\
    Apply Proposition~\ref{zhommcyclicbohr} to $S_{x \bcdot}$ and map\footnote{This map is $\mathbb{Z}/N\mathbb{Z}$ valued, but the proposition shows that coordinate functions on progression $P$, which lies between the Bohr set and its smaller dilate, are of the desired form.} $y \mapsto \psi(x, y)$ for each $x \in C$ to obtain $a_1(x), \dots, a_{r_1}(x) \in \mathbb{R}/N\mathbb{Z}$ and $b_1(x), \dots, b_{r_1}(x) \in \mathbb{Z}/N\mathbb{Z}$ such that for each $x \in X$ and $y \in B(\Gamma_1, \Theta_1(x), \dots, \Theta_{r_1}(x); (2r_1)^{-4r_1}\rho_1)$ we have
    \[\psi(x,y) = \sum_{i \in [r_1]} a_i(x) \blr{\frac{\iota(\gamma_i) \iota(y)}{N}} + \sum_{i \in [r_1]} b_i(x) \blr{\frac{\iota(\Theta_i(x)) \iota(y)}{N}}.\]

    Take an additive triple $(x_1, x_2, x_1 + x_2)$ of elements in $X$ and let $y \in S_{x_1 \bcdot} \cap S_{x_2 \bcdot} \cap S_{x_1 + x_2 - \bcdot}$. Then
    \begin{align*}
        0 = \psi(x_1 + x_2,y) - \psi(x_1,y) - \psi(x_2,y)  = &\sum_{i \in [r_1]} a_i(x_1 + x_2) \blr{\frac{\iota(\gamma_i) \iota(y)}{N}} + \sum_{i \in [r_1]} b_i(x_1 + x_2) \blr{\frac{\iota(\Theta_i(x_1 + x_2)) \iota(y)}{N}}\\
        &\hspace{1cm}- \sum_{i \in [r_1]} a_i(x_1) \blr{\frac{\iota(\gamma_i) \iota(y)}{N}} -\sum_{i \in [r_1]} b_i(x_1) \blr{\frac{\iota(\Theta_i(x_1)) \iota(y)}{N}}\\
        &\hspace{1cm} -\sum_{i \in [r_1]} a_i(x_2) \blr{\frac{\iota(\gamma_i) \iota(y)}{N}} - \sum_{i \in [r_1]} b_i(x_2) \blr{\frac{\iota(\Theta_i(x_2)) \iota(y)}{N}}
    \end{align*}

   Note that, as $(x_1, y), (x_2, y) \in S$, $\Big|\Big\{\frac{\iota(\Theta_i(x_2)) \iota(y)}{N}\Big\}\Big| = \Big|\Big\{\frac{\Theta_i(x_2)y}{N}\Big\}\Big|\leq \rho_1 \leq 1/8$. Then 
    \[\blr{\frac{\iota(\Theta_i(x_1)) \iota(y)}{N}} + \blr{\frac{\iota(\Theta_i(x_2)) \iota(y)}{N}} = \blr{\frac{\iota(\Theta_i(x_1 + x_2)) \iota(y)}{N}},\]
    so we have

    \begin{align}
        0  = &\sum_{i \in [r_1]} (a_i(x_1 + x_2) - a_i(x_1) - a_i(x_2)) \blr{\frac{\iota(\gamma_i) \iota(y)}{N}} + \sum_{i \in [r_1]} (b_i(x_1 + x_2)  - b_i(x_1)) \blr{\frac{\iota(\Theta_i(x_1)) \iota(y)}{N}}\nonumber\\
        &\hspace{1cm}+ \sum_{i \in [r_1]} (b_i(x_1 + x_2)  - b_i(x_2)) \blr{\frac{\iota(\Theta_i(x_2)) \iota(y)}{N}}.\label{abcoeffseqn}
    \end{align}

    But uniqueness of coefficients in the second part of the statement of Proposition~\ref{zhommcyclicbohr}, implies that $b_i(x_1 + x_2) - b_i(x_1) \in N\mathbb{Z}$ holds a vast majority of time. Hence, we get
     \[\psi(x,y) = \sum_{i \in [r_1]} a_i(x) \blr{\frac{\iota(\gamma_i) \iota(y)}{N}} + \sum_{i \in [r_1]} b_i \blr{\frac{\iota(\Theta_i(x)) \iota(y)}{N}}.\]

     Going back to~\eqref{abcoeffseqn}, we also see that $a_i$ respect almost all additive triples in $X$, so they can be assumed to be Freiman-linear on $\frac{1}{2}C$. Applying Proposition~\ref{zhommcyclicbohr} to maps $a_i$ and $\Theta_i$ proves the theorem.
\end{proof}

\end{document}